\documentclass[a4paper]{amsart}

\usepackage[utf8]{inputenc}
\usepackage[english]{babel}
\usepackage{tikz-cd}
\usetikzlibrary{positioning,decorations.text,quotes,calc}
\usetikzlibrary{arrows,babel}
\usepackage{amsmath,amsfonts,amssymb}
\usepackage{mathtools}
\usepackage{float}
\usepackage{amsthm}
\usepackage{stackrel}
\usepackage{pgfplots}
\usepackage{graphicx}
\usepackage{bbm}
\usepackage{amssymb}
\usepackage{geometry,enumerate}
\usepackage{graphicx}
\usepackage{psfrag}
\usepackage{amscd}
\usepackage{comment}
\usepackage[all]{xy}
\usepackage{url}
\usepackage{leftindex}
\usepackage{mathrsfs}

\usepackage{fullpage}

\usepackage[dvipsnames]{xcolor}

\usepackage[backref=page]{hyperref}

\hypersetup{
 colorlinks,
 citecolor=Green,
 linkcolor=blue,
 urlcolor=Blue}

\usepackage{booktabs}
\usepackage{color}
\usepackage{todonotes}

\usepackage{braket}

\usepackage{stmaryrd} 
\usepackage{mathabx}
\usepackage{mathbbol}
\usepackage{xspace}
\usepackage{enumerate}
\usepackage{caption}
\usepackage{subcaption}

 \usepackage{cleveref}
\crefformat{section}{#2#1#3} 
\crefformat{subsection}{#2#1#3}

\newcommand{\ZZ}{\mathbb{Z}}
\newcommand{\QQ}{\mathbb{Q}}
\newcommand{\NN}{\mathbb{N}}
\newcommand{\RR}{\mathbb{R}}

\newcommand{\PP}{\mathbb{P}}
\newcommand{\CC}{\mathbb{C}}

\renewcommand{\AA}{\mathbb{A}}

\newcommand{\A}{\mathcal{A}}
\newcommand{\C}{\mathcal{C}}
\renewcommand{\O}{\mathcal{O}}

\newcommand{\I}{\mathcal{I}}

\newcommand{\D}{\mathcal{D}}
\newcommand{\E}{\mathcal{E}}
\newcommand{\X}{\mathcal{X}}

\newcommand{\J}{\mathcal{J}}
\newcommand{\U}{\mathcal{U}}
\newcommand{\V}{\mathcal{V}}
\newcommand{\M}{\mathcal{M}}
\renewcommand{\S}{\mathcal{S}}

\renewcommand{\P}{\mathcal{P}}

\newcommand{\un}{\underline}
\newcommand{\ov}{\overline}
\newcommand{\wh}{\widehat}
\newcommand{\wt}{\widetilde}

\newcommand{\m}{\mathfrak{m}}

\newcommand{\s}{\mathfrak{s}}

\newcommand{\Mbargn}{\overline{\mathcal{M}}_{g,n}}
\newcommand{\Cbargn}{\overline{\mathcal{C}}_{g,n}}
\newcommand{\Mbargnbis}{\overline{\mathcal{M}}_{g,n+1}}
\newcommand{\Cbargnbis}{\overline{\mathcal{C}}_{g,n+1}}
\newcommand{\Mgn}{{\mathcal{M}}_{g,n}}
\newcommand{\Cgn}{{\mathcal{C}}_{g,n}}
\newcommand{\Bgn}{\mathbb{B}_{g,n}}
\newcommand{\Bgnbis}{\mathbb{B}_{g,n+1}}

\newcommand{\Mbarg}{\overline{\mathcal{M}}_{g}}
\newcommand{\Cbarg}{\overline{\mathcal{C}}_{g}}
\newcommand{\Mg}{{\mathcal{M}}_{g}}
\newcommand{\Cg}{{\mathcal{C}}_{g}}
\newcommand{\Dg}{\mathbb{D}_{g}}

\newcommand{\Sigmas}{\operatorname{^{s}\Sigma}}
\newcommand{\Sigmans}{\operatorname{^{ns}\Sigma}}

\newcommand{\Ds}{\operatorname{^{s}\D}}
\newcommand{\Dns}{\operatorname{^{ns}\D}}

\newcommand{\Deg}{\operatorname{Deg}}
\newcommand{\Degs}{\operatorname{^{s}\operatorname{Deg}}}
\newcommand{\Degns}{\operatorname{^{ns}\operatorname{Deg}}}

\newcommand{\Dgn}{\mathbb{D}_{g,n}}
\newcommand{\Dgns}{{^{s}\mathbb{D}_{g,n}}}
\newcommand{\Dgnns}{{^{ns}\mathbb{D}_{g,n}}}

\newcommand{\DD}{\mathbb{D}}

\newcommand{\DDns}{{^{ns}\mathbb{D}}}

\newcommand{\Dguno}{\mathbb{D}_{g,1}}
\newcommand{\Dgunos}{{^{s}\mathbb{D}_{g,1}}}
\newcommand{\Dgunons}{{^{ns}\mathbb{D}_{g,1}}}
\newcommand{\Dgunou}{{^{ns}\mathbb{D}^e_{g,1}}}
\newcommand{\Dgunom}{{^{ns}\mathbb{D}^f_{g,1}}}
\newcommand{\Dgpr}{\mathbb{D}_{g,1}^{pr}}
\newcommand{\Dgbig}{\mathbb{D}_{g,1}^{\geq g}}

\newcommand{\Dgnbis}{\mathbb{D}_{g,n+1}}
\newcommand{\Dgnwh}{\wh{\mathbb{D}}_{g,n}}

\newcommand{\wPR}{\wt{\operatorname{PR}}}
\newcommand{\PR}{\operatorname{PR}}

\newcommand{\Dyn}{\operatorname{Dyn}}

\newcommand{\fX}{\mathfrak{X}}

\newcommand{\supp}{\operatorname{supp}}
\newcommand{\sm}{\operatorname{sm}}

\newcommand{\id}{\operatorname{id}}

\newcommand{\Def}{\operatorname{Def}}
\newcommand{\taut}{\operatorname{taut}}

\newcommand{\VStab}{\operatorname{VStab}}

\newcommand{\Pol}{\operatorname{Pol}}

\newcommand{\BCon}{\operatorname{BCon}}

\newcommand{\Pic}{\operatorname{Pic}}
\renewcommand{\Im}{\operatorname{Im}}

\newcommand{\Aut}{\operatorname{Aut}}
\newcommand{\Spec}{\operatorname{Spec}}

\newcommand{\type}{\operatorname{type}}
\newcommand{\PicRel}{\operatorname{PicRel}}
\newcommand{\cHom}{{\mathcal Hom}}
\newcommand{\Tot}{\operatorname{Tot}}

\newcommand{\NF}{\operatorname{NF}}
\newcommand{\TF}{\operatorname{TF}}
\newcommand{\TFo}{^{o}\operatorname{TF}}
\newcommand{\TFp}{^{+}\operatorname{TF}}
\newcommand{\TFm}{^{-}\operatorname{TF}}

\newcommand{\Gm}{\operatorname{\mathbb{G}_m}}
\newcommand{\st}{\operatorname{st}}
\newcommand{\St}{\operatorname{St}}

\newcommand{\Z}{\mathbb{Z}}
\newcommand{\R}{\mathbb{R}}

\pgfplotsset{compat=1.17}

\definecolor{LivGreen}{cmyk}{68 0 100 0}

\newtheorem{theorem}{Theorem}[section]
\newtheorem{corollary}[theorem]{Corollary}
\newtheorem{lemma}[theorem]{Lemma}
\newtheorem{fact}[theorem]{Fact}
\newtheorem{proposition}[theorem]{Proposition}
\newtheorem{proposition-definition}[theorem]{Proposition-Definition}
\newtheorem{lemma-definition}[theorem]{Lemma-Definition}
\newtheorem{question}[theorem]{Question}

\newtheorem{theoremalpha}{Theorem}

\newtheorem{corollaryalpha}[theoremalpha]{Corollary}

\newtheorem*{corollary*}{Corollary}

\theoremstyle{definition}
\newtheorem{definition}[theorem]{Definition}
\newtheorem{example}[theorem]{Example}
\newtheorem{remark}[theorem]{Remark}

\numberwithin{equation}{section}

\newenvironment{sis}{\left\{\begin{aligned}}{\end{aligned}\right.}

\begin{document}

\title{A complete classification of modular compactifications of the universal Jacobian}

%\author{Marco Fava, Nicola Pagani, Filippo Viviani}
%\\ Dipartimento di Matematica, Universit\`a di Roma ``Tor Vergata'', Via della Ricerca Scientifica 1, 00133 Roma, Italy\\ viviani@mat.uniroma2.it}

%\author{}
%\date{}
\author{Marco Fava}
\address{Marco Fava, Mathematics Institute, University of Warwick, Coventry, CV4 7AL, United Kingdom}
\email{marco.fava@warwick.ac.uk}

\author{Nicola Pagani}
\address{Nicola Pagani, Department of Mathematical Sciences, University of Liverpool, Liverpool, L69 7ZL, United Kingdom}
\email{pagani@liv.ac.uk}
\urladdr{http://pcwww.liv.ac.uk/~pagani/}

\author{Filippo Viviani}
\address{Filippo Viviani, Dipartimento di Matematica, Universit\`a di Roma ``Tor Vergata'', Via della Ricerca Scientifica 1, 00133 Roma, Italy}
\email{viviani@mat.uniroma2.it}



%		\begin{abstract}
%We study compactified universal Jacobians (stacks or spaces) over the moduli stack $\Mbargn$ of stable $n$-pointed curves of genus $g$. Firstly, we give a complete combinatorial classification of all compactified universal Jacobians, fine and not. Then we define a subclass of compactified universal Jacobian stacks, called \emph{classical}, which, in the fine case, coincides with the compactified universal Jacobians studied by Kass-Pagani \cite{Kass_2019} and by Melo \cite{melo2019}, and we prove that their relative good moduli spaces are locally projective over $\Mbargn$. Next, we investigate when two compactified universal Jacobians (the stacks or their good moduli spaces) are isomorphic over $\Mbargn$ and we describe a resolution of the universal family of a compactified universal Jacobian over $\Mbargn$ in terms of a compactified universal Jacobian over $\Mbargnbis$. Finally, we show that for $n=0$ there is only one compactified universal Jacobian space for every degree, namely the one constructed by Caporaso \cite{Cap94}, while for $n\geq 1$ we study the poset of compactified universal Jacobian stacks, determining in particular the maximal and submaximal elements. 
	%\end{abstract} %%%%%%%%%

\begin{abstract}
This is the third paper in a series, following \cite{FPV1} and \cite{FPV2}.

We classify all modular compactifications of the universal Jacobian over $\Mbargn$, both as stacks and as their relative good moduli spaces. Our main result gives a combinatorial parametrization of compactified universal Jacobian stacks by \emph{$V$-functions} on a  stability domain $\mathbb{D}_{g,n}$ of half-vine types (two-components topological types with a chosen side); under this correspondence, fine compactifications are exactly the general $V$-functions. 

We single out the \emph{classical} compactified universal Jacobians, namely those induced by numerical polarizations (relative $\mathbb{R}$-line bundles on the universal curve $\Cbargn/\Mbargn$), recovering the constructions of Kass--Pagani \cite{Kass_2019} and Melo \cite{melo2019} in the fine case, and we prove that their good moduli spaces are locally projective over $\Mbargn$. 

We  determine when two compactified universal Jacobians are isomorphic over $\Mbargn$ and describe a resolution of the universal family via a compactified Jacobian over $\Mbargnbis$. 

Finally, we analyse the poset $\Sigma_{g,n}$ of compactified universal Jacobians, an extension of the poset of regions of the hyperplane arrangement of classical stability conditions $\mathcal{A}_{g,n}$ studied in \cite{Kass_2019}.  We prove that for $n=0$ all compactified universal Jacobians are those constructed by Caporaso \cite{Cap94}. We then give an explicit description of the submaximal elements of $\Sigma_{g,n}$ for all $n$, generalizing the stability walls in the classical stability space $\mathcal{A}_{g,n}$ from  \cite{Kass_2019}. 
\end{abstract}

\maketitle
	\bigskip
	
	\tableofcontents

\section{Introduction}    

For any pair of integers $(g,n)\in \NN^2$ such that $2g-2+n>0$, the \emph{universal Jacobian stack} of type $(g,n)$ is the stack $\J_{g,n}$ parametrizing pairs $(C,p_1, \ldots, p_n,L)$, where $(C,p_1, \ldots, p_n)$ is an element of the stack $\M_{g,n}$ of $n$-pointed smooth projective connected curves of genus $g$, and $L$ is a line bundle on $C$. The stack $\J_{g,n}$ admits a decomposition into a disjoint union of irreducible stacks $\coprod_{\chi \in \Z}\J_{g,n}^\chi$, where $\J_{g,n}^\chi$ parametrizes line bundles of Euler characteristic $\chi$. We will denote by $J_{g,n}:=\J_{g,n}\fatslash \Gm$, that is, the rigidification of $\J_{g,n}$ by the group $\Gm$ of scalar automorphisms (and similarly $J_{g,n}^\chi:=\J_{g,n}^\chi\fatslash \Gm$) and call $J_{g,n}$ the \emph{universal Jacobian space}. 

%The relative good moduli space of $\J_{g,n}\to \M_{g,n}$ is the universal Jacobian space $J_{g,n}=\coprod_{\chi\in \Z}J_{g,n}^\chi$ and the morphism $\J_{g,n}\to J_{g,n}$ is the $\Gm$-rigidification. 

\vspace{0.2cm}

The aim of this paper is to study all modular compactifications of $\J_{g,n}^\chi$ over the moduli stack $\Mbargn$ of stable $n$-pointed nodal curves of genus $g$, by which we mean the following. 

   A \textbf{compactified universal Jacobian stack} of characteristic $\chi\in \Z$ over $\Mbargn$ is an open substack $\ov \J_{g,n}^{\chi}$ of the stack $\TF^{\chi}_{g,n}$ of relative rank-$1$ torsion-free sheaves of characteristic $\chi$ on $\Cbargn/\Mbargn$, admitting a relative proper good moduli space $\ov J_{g,n}^\chi\rightarrow \Mbargn$, called a \textbf{compactified universal Jacobian space}.  A compactified universal Jacobian stack $\ov \J_{g,n}^\chi$ (resp. space $\ov J_{g,n}^\chi$) is called \emph{fine} if $\ov J_{g,n}^{\chi}=\ov \J_{g,n}^\chi\fatslash \Gm$.
   This terminology  is justified, a posteriori, by the fact that if $\ov \J_{g,n}^{\chi}$ is fine then the universal sheaf $\I_{g,n}$ on the universal family $\ov \J_{g,n}^\chi\times_{\Mbargn} \Cbargn\to \ov \J_{g,n}^\chi$ descends (non uniquely) to a tautological sheaf $\I_{g,n}^{\taut}$ on $\ov J_{g,n}^\chi\times_{\Mbargn} \Cbargn\to \ov J_{g,n}^\chi$. This follows from the fact that the $\Gm$-gerbe $\ov \J_{g,n}^\chi \to \ov J_{g,n}^{\chi}=\ov \J_{g,n}^\chi\fatslash \Gm$ is trivial, which is obvious if $n\geq 1$ (because the universal family has sections), and which we will  show in Corollary \ref{C:fineMg} for $n=0$.

   The set of compactified universal Jacobian stacks over $\Mbargn$ is a poset under the natural inclusion relation. Note that any such compactified universal Jacobian stack $\ov \J_{g,n}^\chi$ (resp. space $\ov J_{g,n}^\chi$) is a "compactification" of the universal Jacobian stack (resp. space) since the restriction of $\ov \J_{g,n}^\chi$ (resp. of $\ov J_{g,n}^\chi$) to $\M_{g,n}$ is $\mathcal{J}_{g,n}^\chi$ (resp. $J_{g,n}^\chi$).

\subsection*{History and applications}

The search for compactified universal Jacobians (stacks and spaces) has been carried out by many authors using different approaches over the last thirty years. The first such compactification is due to Caporaso \cite{Cap94}, who constructed a compactified universal Jacobian space over $\ov\M_g$ using GIT of suitable Hilbert schemes of curves (see also \cite{pandharipande1995compactification} for another construction of this space using slope semistability). The corresponding compactified universal Jacobian stack over $\ov\M_g$ has been later studied by Caporaso \cite{Cap08} and Melo \cite{melo2009}, and it has been extended to $\Mbargn$ by Melo \cite{melo2011}. Later on, plenty of fine compactified universal Jacobians have been constructed over $\Mbargn$, independently, by Kass-Pagani \cite{Kass_2019} (building upon the work of Oda-Seshadri \cite{Oda1979CompactificationsOT} and Simpson \cite{simpson}) and by Melo \cite{melo2019} (building upon work of Esteves \cite{esteves}). Indeed, the two constructions produce the same set of fine compactified universal Jacobians  (see \cite[Rmk. 4.6]{Kass_2019} and \cite[Prop.~4.17]{melo2019}), that we call \emph{classical} fine compactified universal Jacobians.  Finally, Fava \cite{fava2024} (building upon the work of Pagani-Tommasi \cite{pagani2023stability} and Viviani \cite{viviani2023new}) has recently classified all \emph{fine} compactified universal Jacobians  over $\Mbargn$, discovering, in particular, the existence of non-classical fine compactified universal Jacobians for $g,n>0$ (outside of a finite, small range of pairs $(g,n)$).

\vspace{0.2cm}

Compactified universal Jacobians have found several applications, among which we mention: a modular extension of the Torelli map (see \cite{alexeev}) and a Torelli theorem for stable curves (see \cite{caporasoviviani}); the study of the birational geometry (e.g. Kodaira dimension and Iitaka fibration) of $J_g$ (see \cite{BFV} and \cite{CMKVbirational}); the study of the tropicalization of $\J_{g,n}$ (see \cite{AP20}, \cite{MMUV}, \cite{AAPT}), its relation to the logarithmic universal Jacobian (see \cite{MW}, \cite{MMUVW}); its relation with the double ramification cycles (see \cite{dudin}, \cite{KPH}) and with the logarithmic double ramification cycles (see \cite{Mol23}, \cite{HMPPS}, \cite{MolchoFourier}), and the calculation and wall-crossing phenomena for universal Brill-Noether classes (see \cite{kp2}, \cite{PRvZ}, and \cite{APag}). Recently, the cohomology of compactified universal Jacobians has also been the subject of investigation (see \cite{Yin}, \cite{wood}, \cite{Maulikcohom}, \cite{PPSW}).

\subsection*{Our results}

The main goal of this paper is to classify \emph{all} (fine and non-fine) compactified universal Jacobians, building upon our previous works \cite{FPV1} and \cite{FPV2}.  The first observation is that, in the universal case, the stability inequalities can be defined on the components of curves with two smooth irreducible components, and they depend only on their topological types plus a choice of one of the two components: we call this a 'half-vine type'. However, the stability inequalities cannot be set independently on all half-vine types: they have to abide two types of relations, which come from the fact that the Euler characteristic of the sheaves we are considering is fixed. The first type is a compatibility between the two halves of the same vine curve. The second is a compatibility that occurs when three half-vine type have a common degeneration that is as general as possible --that is-- a curve with three smooth irreducible components, which we call 'a triangle'.   %In order to state our classification result, we  introduce some definitions.  

 The \textbf{stability domain of type $(g,n)$} is the set $\Dgn$ that parametrizes \emph{half vine graphs} of type $(g,n)$: for any stable vine graph of type $(g,n)$, i.e. a stable graph of genus $g$ and $n$ legs with two vertices and no loops,  together with the choice of one of its two vertices, we associate the element $(e;h,A)\in \Dgn$ where $e$ is the number of edges, $h$ is the genus of the chosen vertex and $A\subseteq [n]$ is the set of legs rooted at the chosen vertex. 
 
 The set $\Dgn$ comes with two natural structures:
\begin{itemize}
\item the complement $(e;h,A)^c$ of an element $(e;h,A)\in \Dgn$ is defined to be the complementary half vine graph.
\item a triangle in $\Dgn$ is a multiset (i.e. repetitions are allowed) $\Delta=[(e_1;h_1,A_1),(e_2;h_2,A_2), (e_3;h_3,A_3)]$ of $3$ elements of $\Dgn$ such that there exists a stable graph of type $(g,n)$ with three vertices $\{v_1,v_2,v_3\}$ and having no loops and at least one edge in between any pair of vertices, and such that each vertex $v_i$ has genus $h_i$, it is joined by $e_i$ edges to the other two vertices, and the legs rooted at $v_i$ are marked by the subset $A_i$.
\end{itemize}
See \eqref{E:Dgn}, and the discussion following it, for a more explicit description of $\Dgn$ together with its two natural structures outlined above.

   For any $\chi \in \Z$, denote by $\Sigma^\chi_{g,n}$ the set of all  \textbf{vine (or V-)functions}  of type $(g,n)$ and characteristic $\chi$, i.e. functions
    \begin{align*}
        \sigma:\Dgn&\to \ZZ\\
        (e;h,A)&\mapsto \sigma(e;h,A)
    \end{align*}
    satisfying the following properties:
\begin{enumerate}
\item  for any $(e;h,A)\in \Dgn$, we have 
\begin{equation*}
\sigma(e;h,A)+\sigma((e;h,A)^c)-\chi
\in \{0,1\}.
\end{equation*}
An element $(e;h,A)\in \Dgn$ is said to be \emph{$\sigma$-degenerate} if $\sigma(e;h,A)+\sigma((e;h,A)^c)=\chi$,
and \emph{$\sigma$-nondegenerate} otherwise.

\item   for each triangle $\Delta=[(e_1;h_1,A_1), (e_2;h_2,A_2), (e_3;h_3,A_3)]$ of $\Dgn$, we have that:
\begin{enumerate}
 \item if two among the elements $\Delta$ are $\sigma$-degenerate, then so is  the third. 
            \item the following holds
            \begin{equation*}
            \sum_{i=1}^{3}\sigma(e_i; h_i,A_i)-\chi
            \in \begin{cases}
                \{1,2\} \quad \textup{ if $(e_i;h_i,A_i)$ is $\sigma$-nondegenerate for all $i$};\\
                \{1\} \quad \textup{ if there exists a unique $i$ such that } 
                 \textup{ $(e_i;h_i,A_i)$ is $\sigma$-degenerate};\\
                \{0\} \quad \textup{ if $(e_i;h_i,A_i)$ is $\sigma$-degenerate for all $i$}.
            \end{cases}
        \end{equation*}
\end{enumerate}
\end{enumerate}
We  write $\chi=|\sigma|$ and we set  $\Sigma_{g,n}:=\coprod_{\chi \in \ZZ}\Sigma_{g,n}^\chi$.  

 The \emph{degeneracy subset} of $\sigma$ is the collection
\begin{equation*}
\D(\sigma):=\{(e;h,A)\in \Dgn: (e;h,A) \text{ is $\sigma$-degenerate}\}.
\end{equation*}
(This is stable under taking complements, and it encodes the information of which half-vine types admit strictly semistable sheaves). We say that $\sigma$ is \emph{general}  if $\D(\sigma)=\emptyset$.

The space of V-functions $\Sigma_{g,n}$ of type $(g,n)$ is a poset under the following order relation (see Definition~\ref{D:Sigma-pos})
  $$
  \sigma_1\geq \sigma_2 \Longleftrightarrow 
  \begin{sis}
  &|\sigma_1|=|\sigma_2|,\\
  & \sigma_1(e;h,A)\geq \sigma_2(e;h,A) \text{ for any } (e;h,A)\in \Dgn.\\
  \end{sis}
  $$

In order to state our main classification result, we first observe that the set $\Dgn$ describes the combinatorial type of the biconnected subcurves (i.e. subcurves that are  connected and whose complement is also connected) of \emph{all} stable curve of type $(g,n)$. Indeed, for any $(X,p_i)\in \Mbargn$ there is a function
\begin{equation*}
    \begin{aligned}
      \type=\type_X: \BCon(X) & \longrightarrow \Dgn\\
      Y & \mapsto \type(Y):=(|Y\cap Y^c|; g(Y),\{i\in [n]\: : p_i\in Y\}),
    \end{aligned}
\end{equation*}
where $\BCon(X)$ denotes the set of biconnected subcurves of $X$. 

\begin{theoremalpha}\label{T:thmA}(see Theorem \ref{T:cJUniv} and Proposition \ref{P:Stabgn})
    There is an anti-isomorphism of posets  
$$
\begin{aligned}
\Sigma_{g,n} & \xrightarrow{\cong} \left\{\text{Compactified universal Jacobian stacks over $\Mbargn$}\right\},\\
\sigma &\mapsto \ov \J_{g,n}(\sigma):=\left\{
\begin{aligned} 
& I \in  \TF_{g,n}^{|\sigma|}: \: \chi((I_{|X})_Y)\geq \sigma(\type_X(Y)) \\ 
& \text{ for any } X\in \Mbargn \text{ and any } Y\in \BCon(X) 
\end{aligned} 
\right\},
\end{aligned}
$$
where $(I_{|X})_Y$ denotes the torsion-free quotient of the restriction of the sheaf $I_{|X}$ on $Y$. 

Moreover, $\sigma$ is general if and only if $\ov \J_{g,n}(\sigma)$ is fine. 
\end{theoremalpha}
The bijection between the set of \emph{fine} compactified Jacobians and the set of \emph{general} universal V-stability conditions was shown by Fava in \cite[Thm. A]{fava2024}, building upon the results of \cite{pagani2023stability} and \cite{viviani2023new}. The compactified universal Jacobian space  associated to $\ov \J_{g,n}(\sigma)$ will be denoted by $\ov J_{g,n}(\sigma)$.

The proof of Theorem \ref{T:thmA} goes as follows. Given any compactified universal Jacobian $\ov \J_{g,n}^\chi$, its fibers over $\Mbargn$ are universally smoothable compactified Jacobians and hence they are V-compactified Jacobians by the results of \cite{FPV1} and \cite{FPV2} (see Theorem \ref{T:VcJ-nod}). This implies that $\ov \J_{g,n}^\chi$ is the V-compactified Jacobian for a unique relative V-stability condition on $\Cbargn/\Mbargn$ (see Theorem \ref{T:cJUniv}). Finally, we conclude by identifying the set of relative V-stability conditions on $\Cbargn/\Mbargn$ with the set of  V-functions of type $(g,n)$, see Proposition \ref{P:Stabgn}.

%We will denote by $\ov J_{g,n}(\sigma)$ the relative good moduli space of $\ov \J_{g,n}(\sigma)\to \Mbargn$, which is a proper algebraic space over $\Mbargn$. 

\vspace{0.1cm}

As a corollary of the above classification result, we can prove that any two fine compactified universal Jacobian spaces have the same coohomology groups and the same derived category.

\begin{corollaryalpha}\label{C:stesso-HD}(see the end of Section \ref{Sec:cuJ}) 
    Let $\ov J_{g,n}(\sigma_i)$ (for $i=1,2$) be two fine compactified universal Jacobian spaces over $\Mbargn$.  Then we have that:
    \begin{enumerate}
        \item \label{C:isoH} There is an isomorphism 
        $$
        H^*(\ov J_{g,n}(\sigma_1)_{\CC}, \CC)\cong H^*(\ov J_{g,n}(\sigma_2)_{\CC}, \CC)
        $$
        as graded modules over the cohomology algebra $H^*((\Mbargn)_{\CC}, \CC)$.
        \item \label{C:isoD} There is a $\D^b(\Mbargn)$-linear  isomorphism of derived categories
        $$
        \D^b(\ov J_{g,n}(\sigma_1))\cong \D^b(\ov J_{g,n}(\sigma_2)).
        $$ 
        \end{enumerate}
\end{corollaryalpha}
Part \eqref{C:isoH} follows for classical fine compactified Jacobian spaces (assuming $n\geq 1$) by \cite[Thm.~0.6]{Maulikcohom} and it was also proved (purely combinatorially) in \cite[Thm. 1.2]{PPSW}\footnote{The proof of loc. cit. is given for any two universal PT-assignments \cite{pagani2023stability}; however, universal PT-assignments over $\Mbargn$ and general V-functions of type $(g,n)$ coincide by Proposition \ref{P:Stabgn} and \cite[Ex. 4.25]{FPV1}.}; Part \eqref{C:isoD} was proved for classical fine compactified Jacobian spaces (assuming $n\geq 1$) in \cite[Thm. 1.2(d)]{MolchoFourier}.

\vspace{0.2cm}

We then restrict to a subclass of compactified universal Jacobians coming from $\R$-line bundles on the universal family $\Cbargn/\Mbargn$, which we call \emph{classical} compactified universal Jacobians. We follow \cite{Kass_2019}, where the authors identified the real relative Picard group of $\Cbargn\to\Mbargn$ as a vector space of stability conditions for compactified universal Jacobians, and endowed it with a wall-and-chamber decomposition given by stability hyperplanes and stability polytopes.

Denote by  $\PicRel_{g,n}^{\ZZ}(\RR)$ the $\R$-vector space consisting of relative $\R$-line bundles on $\pi:\Cbargn\to \Mbargn$ whose $\pi$-relative degree $\deg_{\pi}$ is an integer. It is well-known (see Section \ref{Sec:classcJ} and the references therein) that the relative Picard group of $\Cbargn/\Mbargn$ is the abelian group generated by:
\begin{itemize}
    \item the relative dualizing line bundle $\omega_\pi$; 
    \item the image $\Sigma_i$ of the $i$-th section of $\Cbargn/\Mbargn$ (for all $1\leq i \leq n$);
    \item the boundary line bundles $\{\O(C_{(h,A)})\}_{(h,A)\in \Bgn}$ on $\Cbargn$, where  \begin{equation*}
        \Bgn:=\{(h,A)\: :  0\leq h\leq g, A\subseteq [n], 2h-2+|A|>0, 2g-2h+|A^c|>0\}, 
        \end{equation*}
    and $C_{(h,A)}$ is the divisor of $\Cbargn$ whose generic point is a curve made of two smooth irreducible components $C_1$ and $C_2$ of genera, respectively, $h$ and $g-h$, meeting at a node, and containing the marked points $p_i$ such that, respectively, $i\in A$ or $i\in A^c$, and in such a way that if $(h,n)\neq (\frac{g}{2},0)$ then the tautological point lies on $C_1$;
\end{itemize}
subject to the following relations:
\begin{itemize}
    \item $\O(C_{(h,A)})+\O(C_{(h,A)^c})=0$ where $(h,A)^c:=(g-h,A^c)$; 
    \item if $g=1$ then $\omega_\pi=0$;
    \item if $g=0$ then $\Sigma_1=\ldots=\Sigma_n$ and $\omega_\pi=-2\Sigma_1$.
\end{itemize}

From the above description of $\PicRel_{g,n}(\ZZ)$, we deduce the equality  
\begin{equation*}
    \PicRel_{g,n}^\ZZ(\RR)=\left\{L
    =\beta\omega_{\pi}+\sum_{i=1}^n \alpha_i\Sigma_i+\sum_{(h,A)\in \Bgn}\gamma_{(h,A)}\O(C_{(h,A)})\: : \deg_{\pi}(L)=(2g-2)\beta+\sum_{i=1}^n \alpha_i\in \ZZ\right\}.
\end{equation*}

By combining Lemma \ref{L:arr-Uni} with Proposition \ref{P:Stabgn}, we obtain a map
 \begin{equation}\label{E:map-sigma2}
 \begin{aligned}
 \sigma_-:  \PicRel_{g,n}^{\ZZ}(\RR)  & \rightarrow \Sigma_{g,n}\\
   L & \mapsto \sigma_L(e;h,A):=
   \begin{cases}
\lceil \beta(2h-2+1)+\sum_{i\in A} \alpha_i- \gamma_{(h,A)}+\gamma_{(h,A)^c} \rceil & \text{ if } e=1,\\
\lceil \beta(2h-2+e)+\sum_{i\in A} \alpha_i \rceil& \text{ if } e\geq 2,
\end{cases}
   \end{aligned}
 \end{equation}
 such that $|\sigma_L|=\deg_{\pi}(L)$ and whose fibers are the regions of $\PicRel_{g,n}^{\ZZ}(\RR)$ with respect to the following arrangement of hyperplanes  \begin{equation*}
\begin{aligned} 
\A_{g,n}:=& \bigcup_{\substack{(1;h,A)\in \Dgn \\ k\in \ZZ}}\Bigg\{(2h-2+1)\omega_{\pi}^{\vee}+\sum_{i\in A}\Sigma_i^{\vee}+\O(C_{(h,A)})^{\vee}=k\Bigg\}\\
& \bigcup_{\substack{(e;h,A)\in \Dgn \text{ with } e\geq 2\\ m\in \ZZ}}\Bigg\{(2h-2+e)\omega_{\pi}^\vee+\sum_{i\in A}\Sigma_i^\vee=m\Bigg\},
\end{aligned} 
\end{equation*} 
where $(-)^\vee\in \PicRel_{g,n}^\Z(\R)^\vee$ denotes the functional dual to a certain element. 
We will denote by $[L]$ the region of $\PicRel_{g,n}^{\ZZ}(\RR)$ containing $L$ and we set $\sigma_{[L]}:=\sigma_L$. Observe that  $\sigma_{[L]}$ is general if and only if $[L]$ is a chamber, i.e. a maximal dimensional region.

%The aim of this section is to describe the classical compactified universal Jacobians over $\Mbargn$, extending the work of \cite{Kass_2019} from the fine case to the general case. 

We can now state the classification result of classical compactified Jacobians over $\Mbargn$.

\begin{theoremalpha}\label{T:thmB}(see Theorem \ref{T:cla-cJUni}) 
\noindent 
\begin{enumerate}
\item \label{T:thmB1}  We have an order-reversing injection of posets  
  $$
  \begin{aligned}
   \left\{
   \begin{aligned}
   & \text{Regions of } \PicRel_{g,n}^{\ZZ}(\RR)\\
   &\text{ with respect to } \A_{g,n}
   \end{aligned}\right\}  & \hookrightarrow
   \left\{\text{Compactified universal Jacobian stacks over } \Mbargn \right\}\\
   [L] &\mapsto \ov \J_{g,n}([L]):=
   \left\{
\begin{aligned} 
& I \in  \TF_{g,n}^{\deg_{\pi}(L)}: \: \chi((I_{|X})_Y)\geq \deg_Y(L_{|X}) \\ 
& \text{ for any } X\in \Mbargn \text{ and any } Y\in \BCon(X)
\end{aligned} 
\right\}.
  \end{aligned}
  $$
 \item \label{T:thmB2}  For any $L\in \PicRel_{g,n}^{\ZZ}(\RR)$, the compactified universal Jacobian space $\ov J_{g,n}([L])$, associated to the stack $\ov \J_{g,n}([L])$, is  locally projective over $\Mbargn$. 
  \end{enumerate}  
\end{theoremalpha}
We call the compactified universal Jacobian stacks of the form $\ov \J_{g,n}([L])$ \textbf{classical compactified universal Jacobian stacks} of type $(g,n)$  and their associated good moduli spaces $\ov J_{g,n}([L])$ \textbf{classical compactified universal Jacobian spaces}. 

Part \eqref{T:cla-cJUni1} of the above Theorem follows from \cite[Sec. 4, 5]{Kass_2019} in the case of fine classical compactified universal Jacobians, which correspond to the chambers (i.e. the maximal dimensional regions) of $\PicRel_{g,n}^{\ZZ}(\RR)$  with respect to $\A_{g,n}$. See also \cite{melo2019} for another construction of the classical fine compactified universal Jacobians over $\Mbargn$.

It was shown in \cite[Thm. 3.9]{fava2024} that if (and only if) $n>0$, $g>0$ and $2g+n\geq 8$ there exist fine compactified Jacobians that are not classical. We improve on this by showing that in the complementary range all compactified universal Jacobians (including the non fine ones) are classical (see Section \ref{Sec:equiv-cUJ}).

\vspace{0.2cm}

We then investigate when two  among the universal compactified Jacobian stacks $\{\ov \J_{g,n}(\sigma)\}$ or their associated compactified Jacobian spaces  $\{\ov J_{g,n}(\sigma)\}$ are isomorphic over $\Mbargn$, generalizing what proved by Kass-Pagani \cite[Section~6.2]{Kass_2019} for classical fine compactified universal Jacobians. In order to answer this question, we consider the following group 
\begin{equation*}
    \wPR_{g,n}:=\PicRel_{g,n}(\ZZ)\rtimes (\ZZ/2\ZZ),
\end{equation*}
where $\PicRel_{g,n}(\ZZ)$ is the integral relative Picard group of $\Cbargn/\Mbargn$ and  $ \ZZ/2\ZZ$ acts on $\PicRel_{g,n}(\ZZ)$ by mapping a line bundle to its inverse. The group $\PicRel_{g,n}(\ZZ)$ is generated by the relative dualizing sheaf $\omega_\pi$ of the universal family $\pi:\Cbargn\to \Mbargn$, the images $\{\Sigma_i\}$ of the universal $n$ sections of $\pi$ and the boundary line bundles $\{\O(C_{(h,A)})\}$ of $\Cbargn$, subject to some explicit relations (see Section \ref{Sec:classcJ}).

The group $\wPR_{g,n}$ acts on the stack $\TF_{g,n}$ in the following way: (see Proposition \ref{P:PR-TF})
  \begin{itemize}
    \item an element $L\in \PicRel_{g,n}(\ZZ)$ acts by sending $\I\in \TF_{g,n}$ to 
    $$L\cdot \I:=\I\otimes L.$$
     \item the generator $\iota$ of $\Z/2\Z$  acts by sending $\I\in \TF_{g,n}$ to 
     $$\iota \cdot \I:=\I^*:={\mathcal Hom}(\I,\omega_{\Cbargn/\Mbargn}).$$
    \end{itemize}
 
 In Proposition \ref{P:PR-TF}, we show that the action of $\wPR_{g,n}$ permutes the compactified universal Jacobian stacks in such a way that the bijection of Theorem \ref{T:thmA} becomes $\wPR_{g,n}$-equivariant with respect to the action of $\wPR_{g,n}$  on $\Sigma_{g,n}$ given by (see Remark \ref{R:PR-Pol}):
\begin{itemize}
    \item an element $L=\beta\omega_{\pi}+\sum_{i=1}^n \alpha_i\Sigma_i+\sum \gamma_{(h,A)}\O(C_{(h,A)})\in \PicRel_{g,n}(\ZZ)$ acts by 
$$(L\cdot \sigma)(e;h,A):=\sigma(e;h,A)+
    \begin{cases}
 \beta(2h-2+1)+\sum_{i\in A} \alpha_i- \gamma_{(h,A)}+\gamma_{(h,A)^c} & \text{ if } e=1,\\
 \beta(2h-2+e)+\sum_{i\in A} \alpha_i& \text{ if } e\geq 2;
\end{cases}$$
    \item the generator $\iota$ of $\Z/2\Z$ acts by 
    $$(\iota \cdot \sigma)(e;h,A):=
     \begin{cases}
      -\sigma(e;h,A) & \text{ if } (e;h,A)\in \D(\sigma), \\
-\sigma(e;h,A)+1 & \text{ if } (e;h,A)\not \in \D(\sigma).    
     \end{cases}$$
\end{itemize}

%Moreover, we show in Proposition \ref{P:finite-orb} that there are only finitely many orbits for the action of $\wt \PR_{g,n}$ on $\Sigma_{g,n}$.

We will also need a decomposition of the poset $\Sigma_{g,n}$ of  V-functions of type $(g,n)$ into a separating and a non-separating part.  First of all, we can partition the stability domain $\Dgn$  into a separating and a non-separating domain
$$
\Dgn=\Dgns\bigsqcup \Dgnns,
$$
where 
\begin{equation*}
  \begin{aligned}
& \Dgns:=\{(1;h,A): (1; h,A)\in \Dgn\} \\
&  \Dgnns:= \{(e;h,A): (e; h,A)\in \Dgn \text{ and } e\geq 2\}.
  \end{aligned}  
\end{equation*}

Then we can define the poset $\Sigmans$ (resp. $\Sigmas$) of \emph{non-separating} (resp. \emph{separating}) V-function of type $(g,n)$ as the set of functions from $\Dgnns$ (resp. $\Dgns$) to $\Z$ satisfying the same properties as in the above definition of V-functions and the same order relation (see Definition \ref{D:Sigmagn2} for more details). We therefore get an isomorphism of posets  (see Lemma \ref{L:Sigma-s-ns})
$$
\begin{aligned}
    \Sigma_{g,n} & \xrightarrow{\cong}  {}\Sigmas_{g,n}\times \Sigmans_{g,n}\\
    \sigma&\mapsto (\sigma^s:=\sigma_{|\Dgns},\sigma^{ns}:=\sigma_{|\Dgnns}).
\end{aligned}
$$

%We will call $\sigma^s$ (resp. $\sigma^{ns}$) the separating (resp. non-separating) component of $\sigma\in \Sigma_{g,n}$.

The action of $\wPR_{g,n}$ on $\Sigma_{g,n}$ preserves the above decomposition and its restriction to $\Sigmans_{g,n}$ factors via the quotient (see Lemma \ref{L:Sigma-ac})
$$
\PR_{g,n}:=\wPR_{g,n}/\langle \O(C_{(h,A)})\rangle_{(h,A)}\cong \PicRel^{op}_{g,n}(\ZZ)\rtimes (\ZZ/2\ZZ),
$$
where $\PicRel^{op}_{g,n}(\ZZ)$ is the relative Picard group of the universal curve $\C_{g,n}/\M_{g,n}$ (see \eqref{E:PRbis} and the discussion following it).

\begin{theoremalpha}\label{T:thmC}(see Corollary \ref{C:iso-UniSt} and Theorem \ref{T:iso-UniSp})
Let $\sigma_1,\sigma_2\in \Sigma_{g,n}$.
 \begin{enumerate}
\item Then $\ov \J_{g,n}(\sigma_1)$ and $\ov \J_{g,n}(\sigma_2)$ are isomorphic over $\Mbargn$ if and only if $\sigma_1$ and $\sigma_2$ lie in the same orbit for the action of $\wPR_{g,n}$ on $\Sigma_{g,n}$.
\item The following conditions are equivalent:
     \begin{enumerate}
         \item 
         the V-functions $\sigma_1^{ns}$ and $\sigma_2^{ns}$ lie in the same orbit for the action of $\PR_{g,n}$ on $\Sigmans_{g,n}$.
         \item  the compactified universal Jacobian stacks $\ov \J_{g,n}(\sigma_1)$ and $\ov \J_{g,n}(\sigma_2)$ are isomorphic over the open locus $\Mbargn^{ns}\subset \Mbargn$ parametrizing curves with no separating nodes.
         \item  the relative good moduli spaces $\ov J_{g,n}(\sigma_1)$ and $\ov J_{g,n}(\sigma_2)$ are isomorphic over $\Mbargn$.
          \item  the relative good moduli spaces $\ov J_{g,n}(\sigma_1)$ and $\ov J_{g,n}(\sigma_2)$ are isomorphic over $\Mbargn^{ns}$.
     \end{enumerate}
 \end{enumerate}
\end{theoremalpha}
Part (1) of the above Theorem is proved by Kass-Pagani in \cite[Sec. 6.2]{Kass_2019} for classical fine compactified universal Jacobian stacks, under the weaker assumption that there exists a birational morphism over $\Mbargn$ between them. Along the same lines, we can also characterize the pairs of compactified universal Jacobian stacks (not necessarily fine) such that there exists a birational morphism over $\Mbargn$ from one to the other (see Theorem \ref{T:bir-UniSt}). 

Since we show in Proposition \ref{P:finite-orb} that there are finitely many orbits for the action of $\wt \PR_{g,n}$ on $\Sigma_{g,n}$, the above Theorem implies that there are finitely many isomorphism classes of compactified universal Jacobian stacks (and spaces) over $\Mbargn$.

\vspace{0.2cm}

Our next result is the description of the resolution of singularities of the universal family over a compactified Jacobian stack over  $\Mbargn$ in terms of a compactified Jacobian stack over $\ov\M_{g,n+1}$. 

Recall that there is a canonical isomorphism between the universal family $\pi:\Cbargn\to \Mbargn$ (endowed with the $n$ sections $\sigma_i$ for $1\leq i \leq n$) and $\Mbargnbis$:
\begin{equation}\label{E:univ-Mg2}
\begin{tikzcd}
  \Mbargnbis \arrow["\Phi", rd]  \arrow[rr, "\Upsilon", "\cong"']& & \Cbargn \arrow[ld, "\pi"']   \\
     & \Mbargn  \arrow[ur, bend right=30, "\sigma_i"']\arrow[ul, bend left=30, "\sigma_i'"]&  
\end{tikzcd}
\end{equation}
where the isomorphism $\Upsilon$ and the morphism $\Phi$ are defined on geometric points by
$$
\Upsilon(C):=(C^{\st}, \st(p_{n+1})) \quad \text{ and } \quad \Phi(C)=C^{\st},$$
with $\st_C=\st: C=(C,p_1,\ldots,p_{n+1})\to C^{\st}=(C,p_1,\ldots, p_n)^{\st}$ being the \emph{stabilization} morphism that forgets the last marked point $p_{n+1}$ and then it stabilizes the resulting $n$-pointed curve. 
%The universal curve $\Cbargn/\Mbargn$ is endowed with $n$ canonical sections defined on geometric points by  (for any $1\leq i \leq n$)
%$$
%\sigma_i(C)=(C,p_i).
%$$
%Via the isomorphism $\Upsilon^{-1}$, the section $\sigma_i$ is sent to the section $\sigma_i':=\Upsilon^{-1}\circ \sigma_i$ given on geometric points by 
%$$
%\sigma_i'(C)=B_{p_i}(C),
%$$
%where $B_{p_i}(C)$ is the bubbling of $C$ at $p_i$, i.e. the stable $n+1$-pointed curve obtained by gluing a smooth rational curve $E$ with the $n$-pointed curve $C$ at the old marked point $p_i$, and then putting the new $i$-th and the $(n+1)$-th marked point on $E$. 

The next result provides a lifting of the diagram in \eqref{E:univ-Mg2} to the stack $\TF_{g,n}^\chi$ or to any compactified universal Jacobian stack $\ov\J_{g,n}(\sigma)$. Observe that the universal family over $\TF_{g,n}^\chi$, together with its $n$ canonical sections $\wh{\sigma_i}$
\begin{equation*}
\begin{tikzcd}
  \Cbargn\times_{\Mbargn} \TF_{g,n}^\chi \arrow["\wh{\pi}", r]  & \TF_{g,n}^\chi \arrow[l, bend left=30, "\wh{\sigma_i}"]
\end{tikzcd}
\end{equation*}
is given by pulling back the universal family $\Cbargn/\Mbargn$, together with its $n$ canonical sections, along the forgetful morphism $\TF_{g,n}\to \Mbargn$. We will denote by $\I_{g,n}$ the universal sheaf on $\Cbargn\times_{\Mbargn} \TF_{g,n}^\chi$. 

%We now want to lift the morphisms $\Phi$ and $\Upsilon$, together with the sections $\sigma_i'$, over $\TF_{g,n}^\chi$. It turns out that we cannot lift $\Phi$ and $\Upsilon$ to the entire stack $\TF_{g,n+1}^\chi$ but only to an open substack, as we now show.

\begin{theoremalpha}\label{T:thmD} (see Theorem \ref{T:univ-TF}, Proposition \ref{P:inv-Phi}, Remark \ref{R:Omegapm})
\noindent 
\begin{enumerate}
    \item \label{T:thmD1} There exists an open substack $\TFo_{g,n+1}^\chi$ of $\TF_{g,n+1}^\chi$ whose geometric points are given by 
$$
    \TFo_{g,n+1}^\chi(k):=\{(C,I)\in \TF_{g,n+1}^\chi(k)\: : \chi(I_E)\geq 0 \text{ and } \chi(I_{E^c})\geq \chi \text{ for any exceptional  } E \subset C\}
    $$
    fitting into a commutative diagram 
\begin{equation*}
\begin{tikzcd}
  \TFo_{g,n+1}^\chi \arrow["\wh{\Phi}", rd]  \arrow[rr, "\wh{\Upsilon}"]& & \Cbargn\times_{\Mbargn} \TF_{g,n}^\chi \arrow[ld, "\wh{\pi}"']   \\
     & \TF_{g,n}^\chi  \arrow[ur, bend right=30, "\wh{\sigma_i}"']\arrow[ul, bend left=30, "\wh{\sigma_i}'"]& 
\end{tikzcd}
\end{equation*}
lying over the diagram in \eqref{E:univ-Mg2}, defined on geometric points by
$$
\begin{sis}
& \wh{\Phi}(C,I)=(C^{\st},\st_*(I)),\\   
& \wh{\Upsilon}(C,I)=(C^{\st},\st(p_{n+1}),\st_*(I)),\\
& \wh{\sigma_i}'(C,I)=(B_{p_i}(C),\st^*(I)).
\end{sis}
$$
Moreover, $\wh{\Upsilon}$ is a good moduli space morphism and $\wh{\sigma_i}'$ (for $1\leq i \leq n$) are sections of $\wh{\Phi}$. 

\item \label{T:thmD2} 
The restriction of $\wh{\Upsilon}$ to the following open substacks of $\TFo_{g,n+1}^\chi$ 
$$
\begin{sis}
    & \TFp_{g,n+1}^\chi(k):=\{(C,I)\in \TF_{g,n+1}^\chi(k)\: : \chi(I_E)\geq 1 \text{ and } \chi(I_{E^c})\geq \chi \text{ for any exceptional  } E \subset C\},\\
    & \TFm_{g,n+1}^\chi(k):=\{(C,I)\in \TF_{g,n+1}^\chi(k)\: : \chi(I_E)\geq 0 \text{ and } \chi(I_{E^c})\geq \chi+1 \text{ for any exceptional  } E \subset C\},\\
\end{sis}
$$ 
induces the following isomorphisms over $\Cbargn\times_{\Mbargn} \TF_{g,n}^\chi$:
$$
\TFp_{g,n+1}^\chi\cong \PP(\I_{g,n})   \text{ and } \TFm_{g,n+1}^\chi\cong \PP(\I_{g,n}^\vee).
$$
%under which the tautological line bundle $\O(1)$ becomes isomorphic to, respectively,  
%$$
%\O_{\PP(\I_{g,n})}(1)=\wh{\sigma}_{n+1}^*(\I_{g,n+1})_{|\TFp_{g,n+1}^\chi} \text{ and } \O_{\PP(\I_{g,n}^\vee)}(1)=\wh{\sigma}_{n+1}^*(\I_{g,n+1}^\vee)_{|\TFm_{g,n+1}^\chi}
%$$
%where $\I_{g,n}$ is the universal sheaf on $\Cbargn\times_{\Mbargn} \TF_{g,n}^\chi$.
\item \label{T:thmD3} For any $\sigma\in \Sigma_{g,n}^\chi$, there exists  $ \Omega(\sigma)\in \Sigma_{g,n+1}^\chi$ such that 
 $$
 \ov \J_{g,n+1}(\Omega(\sigma))=\wh{\Phi}^{-1}(\ov\J_{g,n}(\sigma))=\wh{\Upsilon}^{-1}(\Cbargn\times_{\Mbargn} \ov\J_{g,n}(\sigma)).
 $$
Moreover, there exist two general elements $\Omega(\sigma)^+,\Omega(\sigma)^-\geq \Omega(\sigma)$ such that 
 $$
 \ov \J_{g,n+1}(\Omega(\sigma)^+)\subseteq \PP(\I_{g,n})_{|\Cbargn\times_ {\Mbargn}\ov \J_{g,n}(\sigma)} \text{ and } \ov \J_{g,n+1}(\Omega(\sigma)^-)\subseteq \PP(\I_{g,n}^\vee)_{|\Cbargn\times_ {\Mbargn}\ov \J_{g,n}(\sigma)}
 $$
 with equality if $\sigma$ is general.
\end{enumerate}    
\end{theoremalpha}
Observe that the above result provides a desingularization of the universal family $\Cbargn\times_{\Mbargn} \TF_{g,n}$, which has ordinary double point singularities (=$A_1$) in codimension $3$: 
%the stacks $\TFo_{g,n+1}^\chi$ (and its open substacks $\TFp^\chi_{g,n+1}$ and $\TFm_{g,n+1}^\chi$) are regular and the morphism $\Upsilon$ is a stacky resolution, while 
the morphisms $\Upsilon^+$ and $\Upsilon^-$ are the two canonical small crepant resolutions of the $A_1$ singularity in dimension $3$ (which are related by the Atiyah flop) and they can be realized as a local variation of GIT inside the stacky resolution $\Upsilon$ (see Proposition \ref{P:push-loc}).

\vspace{0.2cm}

We now specialize to the case $n=0$ and we show that there is essentially one compactified universal Jacobian over $\Mbarg$, namely the one constructed by Caporaso \cite{Cap94} which we now recall.

Consider the \emph{canonical V-function} of genus $g$ and characteristic $\chi$
\begin{equation*}
\begin{aligned}
\sigma_{g}^{\chi}=\sigma_{\frac{\chi}{2g-2}\omega_{\pi}}: \Dg& \longrightarrow \ZZ\\
(e;h,\emptyset)=:(e;h)& \mapsto \Big\lceil \frac{\chi}{2g-2}(2h-2+e) \Big\rceil.
\end{aligned}
\end{equation*}

The compactified universal Jacobian stack (resp. space) associated to the canonical universal V-stability of genus $g$ will be called the \emph{Caporaso's compactified universal Jacobian stack (resp. space)} and it will be denoted by 
\begin{equation*}
\ov \J_g^{\text{Cap}, \chi}:=\ov \J_g(\sigma_g^\chi)=\ov \J_g\left(\left[\frac{\chi}{2g-2}\omega_{\pi}\right]\right) \quad \left(\text{resp. } \ov J_g^{\text{Cap}, \chi}:=\ov J_g(\sigma_g^\chi)=\ov J_g\left(\left[\frac{\chi}{2g-2}\omega_{\pi}\right]\right)\right).
\end{equation*}
Indeed, the absolute good moduli space of $\ov J_g^{\text{Cap}, \chi}$ is isomorphic to Caporaso's \cite{Cap94} compactified universal Jacobian over the coarse moduli space $\ov M_g$ of $\ov \M_g$, and the stack $\ov \J_g^{\text{Cap}, \chi}$ is the one studied in \cite{Cap08} and \cite{melo2009}. 

Note that $\ov \J_g^{\text{Cap}, \chi}$ is fine (or, equivalently,  $\sigma_g^\chi$ is general) if and only if $\gcd(\chi, 2g-2)=1$.

\begin{theoremalpha}\label{T:thmE}(see Corollaries \ref{C:class-n0} and \ref{C:iso-n0})
\begin{enumerate} 
\item \label{T:thmE1} Let $\ov \J_g^{\chi}$ be a compactified universal Jacobian stack of characteristic $\chi \in \Z$ over $\ov \M_g$ and let $\ov J_g^\chi$ be its associated compactified universal Jacobian space.  Then we have that
 \begin{enumerate}[(i)]
     \item  \label{T:thmE1i} ${\ov \J_{g}^\chi}_{|\Mbarg^{ns}}={\ov \J_{g}^{Cap, \chi}}_{|\Mbarg^{ns}}$.
     \item  \label{T:thmE1ii} $\ov J_{g}^\chi$ is isomorphic to $\ov J_{g}^{Cap, \chi}$ over $\ov \M_g$. 
 \end{enumerate}
    \item \label{T:thmE2} Let $\sigma_1\in \Sigma_g^{\chi_1}$ and $\sigma_2\in \Sigma_g^{\chi_2}$. Then we have that:
   \begin{enumerate}[(i)]
       \item \label{T:thmE2i} $\ov \J_g(\sigma_1)$ is isomorphic to $\ov \J_g(\sigma_2)$ over $\ov \M_g$ if and only if $\D(\sigma_1^s)=\D(\sigma_2^s)$ and $\chi_1\equiv \pm \chi_2 \mod 2g-2$. 
        \item \label{T:thmE2ii} $\ov J_g(\sigma_1)$ is isomorphic to $\ov J_g(\sigma_2)$ over $\ov \M_g$ if and only if $\chi_1\equiv \pm \chi_2 \mod 2g-2$. 
   \end{enumerate}
\end{enumerate}    
\end{theoremalpha}
 Parts \eqref{T:thmE1i} and \eqref{T:thmE2i} were announced in \cite[Sec. 9.3]{pagani2023stability} in the case of fine compactified universal Jacobians, but the proof in loc.cit. contains an error (more details in Remark~\ref{rem:correcterror}).

\vspace{0.2cm}

In  Section \ref{Sec:poset} of the paper, we study the poset $\Sigma_{g,n}$ of  V-functions of type $(g,n)$ for $n\geq 1$, or equivalently the poset of universal compactified Jacobian stacks over $\Mbargn$ for $n\geq 1$. 
%We first show in Corollary \ref{C:max-vfun} that for $n\geq 1$ the maximal elements of $\Sigma_{g,n}$ are the general V-functions (a result that is false for $n=0$ by Theorem \ref{T:thmC}). 
 To this end, we define the poset $\Deg_{g,n}$ of degeneracy subsets of type $(g,n)$ (see Definition \ref{D:degsub}) and we show in Proposition \ref{D:Deg-map} that the degeneracy map 
\begin{equation*}
    \begin{aligned}
        \D: \Sigma_{g,n}& \longrightarrow \Deg_{g,n}\\
        \sigma & \mapsto \D(\sigma)
    \end{aligned}
\end{equation*}
is order preserving, invariant under the action of $\wt \PR_{g,n}$ and upper lifting. Next, we factor the degeneracy map into a separating and a non-separating degeneracy map (see \eqref{E:Ds-Dns} and \eqref{E:Deg-dec}) and in Proposition~\ref{P:Dsep} we give a complete description of the poset of separating  V-functions and of the separating degeneracy map. The poset of non-separating V-functions and the non-separating degeneracy map turns out to be much more intricate. In Subsection \ref{Sub:n1}, we give a complete description of the image of the degeneracy map for $n=1$ (see Theorem \ref{T:pos-Degn1}) and in Subsection \ref{Sub:g1} we give partial results in the case $g=1$.

In Subsection \ref{sub:max-wall} we describe the maximal elements and submaximal elements (i.e. those that are dominated only by maximal elements) of $\Sigma_{g,n}^\chi$.

\begin{theoremalpha}\label{T:thmF}(see Corollaries \ref{C:max-vfun}, \ref{C:walls-Vfun} and \ref{C:2max-wall})
Let $n\geq 1$ and fix $\chi\in \Z$.
\begin{enumerate}
    \item \label{T:thmF1} The maximal elements of $\Sigma_{g,n}^\chi$ are exactly the general V-functions, i.e. those V-functions $\sigma$ such that $\D(\sigma)=\emptyset$.
    \item \label{T:thmF2} The submaximal elements of $\Sigma_{g,n}^\chi$ are the V-functions $\sigma$ such that $\D(\sigma)$ is equal to one of the following 
    \begin{enumerate}[(a)]
        \item $\{(e;h,A),(e;h,A)^c\}$ with either $e=1$, or $\emptyset\subsetneq A\subsetneq [n]$, or $A=\emptyset$ and $2h-2+e\geq g$, or $A=[n]$ and $2h+e\leq g$;
        \item $\displaystyle W_{\delta}:=\bigcup_{\substack{\delta \mid (2h-2+e) \\ e\geq 2}} \{(e;h,\emptyset), (e;h,\emptyset)^c\}$ for some $1\leq \delta\leq g-1$ .
    \end{enumerate}
    Moreover, each of the above subsets is a universal degeneracy subset of some element of $\Sigma_{g,n}^\chi$. 
    \item \label{T:thmF3} Every submaximal element of $\Sigma_{g,n}^\chi$ is dominated by exactly two maximal elements. 
\end{enumerate}
\end{theoremalpha}

Using the hyperplane arrangement $\mathcal A_{g,n}$ introduced after Equation~\eqref{E:map-sigma2}, one obtains a natural poset of stability regions and their walls (ordered by inclusion); the assignment $L\mapsto\sigma_L$ defined in \eqref{E:map-sigma2} induces an order-preserving injection of this poset into $\Sigma^\chi_{g,n}$, and Remark~\ref{R:classic-walls} compares the resulting classical walls with the submaximal elements classified in Theorem~\ref{T:thmF}.

The above description of maximal and submaximal elements is particularly relevant when studying how cohomology classes on universal Jacobians vary under wall-crossing. For example, one can define universal Brill-Noether classes
$$
\mathsf{w}_d(\sigma) \in A^{g-d}(\ov \J_{g,n}(\sigma))
$$
as in \cite[Sec. 3.d]{APag} and then study the wall-crossing behaviour of these classes when a submaximal element is crossed. The case of classical  compactified universal Jacobians and of crossing classical walls have been studied in \cite{APag}.

\subsection*{Open Questions} 

This paper leaves open some natural questions:

\begin{enumerate}
\item Are all compactified Jacobian spaces $\ov J_{g,n}(\sigma)$ locally projective over $\Mbargn$?

The answer is positive for classical compactified Jacobian spaces (see Theorem \ref{T:thmB}\eqref{T:thmB2}), but unknown in the non-classical case (even for fine compactified Jacobian spaces). See also \cite[Open Question (2)]{FPV1} for a related question. 

\item  How to characterize the V-functions of type $(g,n)$ that are classical (or equivalently, the universal compactified Jacobians over $\Mbargn$ that are classical)?

We give an answer to the above question for $n=0$, in which case all the V-functions are classical (see Theorem \ref{T:class-n0} and Remark \ref{R:class-s-ns}), and for $n=1$ (see Theorem \ref{T:class-n1}).

\item What is the structure of the poset $\Sigma_{g,n}$ of V-functions of type $(g,n)$ (and of its degeneracy map $\D$) for $n\geq 1$? We single out two natural questions:

$(a)$ Is the poset $\Sigma_{g,n}^\chi$  connected through height one, i.e. are any two maximal elements connected through a path made only of maximal and submaximal elements?

$(b)$ Is the poset $\Sigma_{g,n}^\chi$ upper-graded, i.e. is it true that all the ascending maximal chains starting from a given element have the same length?

Both questions (a) and (b) have a positive answer for $n=1$ (see  Theorem \ref{T:conn-h1} and Remark~\ref{R:oss-poset}). 
%As discussed above, in Section \ref{Sec:poset}, we give some partial results on the structure of $\Sigma_{g,n}$ and $\D$, but there seems to be much more to be discovered. 

\end{enumerate}

\subsection*{Outline of the paper} 

\vspace{0.1cm}

The paper is organized as follows.  In Section \ref{Sec:VcJ-fam}, we recall   the definition of V-compactified Jacobians for families of nodal curves and then we prove two new classification results for a fixed nodal curve (see Theorem \ref{T:VcJ-nod}) and for suitable "large" families of nodal curves (see Theorem~\ref{T:cJ-nodal}). In Section~\ref{Sec:cuJ}, we first classify universal compactified Jacobian stacks in terms of  universal V-stabilities (see Theorem \ref{T:cJUniv}) and then we show that the poset of universal V-stabilities is isomorphic to the poset of V-functions (see Proposition \ref{P:Stabgn}). In Section~\ref{Sec:classcJ}, we consider the universal V-stabilities coming from the real vector space of universal numerical polarizations (see Fact \ref{F:RelPic} and Lemma \ref{L:arr-Uni}) and their associated classical universal compactified Jacobians (see Theorem \ref{T:cla-cJUni}). In Section~\ref{Sec:equiv-cUJ}, we determine when two compactified universal Jacobian stacks or spaces are isomorphic over $\Mbargn$ (see Theorems~\ref{T:bir-UniSt} and \ref{T:iso-UniSp}). In particular, we deduce that  there are only a finite number of isomorphism classes of compactified universal Jacobians over $\Mbargn$ (see Proposition \ref{P:finite-orb}). In Section~\ref{Sec:univ-fam}, we provide a resolution of singularities of the universal compactified Jacobian stacks over $\Mbargn$ in terms of certain universal compactified Jacobian stacks over $\Mbargnbis$ (see Theorem \ref{T:univ-TF} and Proposition \ref{P:push-loc}).
In Section~\ref{Sec:n0}, we classify compactified universal Jacobians over $\ov \M_g$ (see Corollaries \ref{C:class-n0} and \ref{C:iso-n0}). In Section \ref{Sec:poset}, we give some partial results on the poset of V-functions and the degeneracy map for $n>0$ (see Propositions~\ref{P:Deg-map} and \ref{P:Dsep}): in Subsection~\ref{Sub:n1} we give a description of these posets for $n=1$; in Subsection~\ref{sub:max-wall} we describe the maximal and submaximal elements; in Subsection~\ref{Sub:g1}  we examine the case $g=1$ in more detail.

\vspace{0.1cm}

\subsection*{Acknowledgements}

We thank Sam Molcho for raising the question of the relative projectivity of the compactified Jacobian spaces (see Theorem \ref{T:cla-cJUni} and the discussion following it). We thank Y. Bae, D. Maulik, R. Pandharipande, J. Shen for discussions that lead us to add Corollary \ref{C:stesso-HD} to the second arXiv version of our paper. We thank  Alex Abreu, Dan Petersen, Orsola Tommasi for  helpful discussions related to this project.

MF was supported by the DTP/EPSRC award 
EP/W524001/1, and he is supported by the EPSRC Fellowship EP/X02752X/1.

NP is funded by the PRIN 2022 ``Geometry Of Algebraic Structures: Moduli, Invariants, Deformations'' funded by MUR, and he is a member of the GNSAGA section of INdAM. 

FV is funded by the MUR  ``Excellence Department Project'' MATH@TOV, awarded to the Department of Mathematics, University of Rome Tor Vergata, CUP E83C18000100006, by the  PRIN 2022 ``Moduli Spaces and Birational Geometry''  funded by MUR,  and he is a member of  the GNSAGA section of INdAM.
%and of the Centre for Mathematics of the University of Coimbra (funded by the Portuguese Government through FCT/MCTES, DOI 10.54499/UIDB/00324/2020).

\section*{Notation}\label{Sec:nota}

Let $X$ be a connected nodal projective curve over $k=\ov k$.

A \emph{subcurve} $Y\subseteq X$ is a closed subscheme of $X$ that is either a curve or the empty scheme. In other words, a subcurve $Y\subseteq X$ is the union of some irreducible components of $X$.
%so that there is a bijection 
%\begin{equation}\label{E:sub-Vert}
% \begin{aligned}
%   \left\{\text{Subsets of } I(X)\right\} & \leftrightarrow \left\{\text{Subcurves of } X \right\}:=\Sub(X)  \\
%   W &\mapsto X[W]:=\bigcup_{v\in W} X_v  \\
%   I(Y) & \mapsfrom Y.
% \end{aligned}   
%\end{equation}
We say that a subcurve $Y$ is non-trivial if $Y\neq \emptyset, X$.
%which happens if and only if $I(Y)\neq \emptyset, I(X)$. 
The complementary subcurve of $Y$ is 
$Y^\mathsf{c}:=\ov{X-Y}$.  
A subcurve $Y$ of $X$ is called \emph{biconnected} if it is connected and if its complementary subcurve is also connected (in particular, $Y$ is non-trivial). The set of biconnected subcurves of $X$ is denoted by $\BCon(X)$. 
%The set of connected subcurves of $X$ (which does not include the empty subcurve) is denoted by $\Con(X)$. Note that if $Y\in \Con(X)$ then $Y^c=\coprod_i W_i$ with each $W_i\in \BCon(X)$. 

%We define the \emph{join} and the \emph{meet} of two subcurves by 
%$$
%\begin{sis}
%& Y_1\cup Y_2:=X[I(Y_1)\cup I(Y_2)], \\
%& Y_1\wedge Y_2:=X[I(Y_1)\cap I(Y_2)].
%\end{sis}
%$$
%In other words, the join of two subcurves is simply their union, while the meet of two subcurves is the union of their common irreducible components. 

%A \emph{family of connected reduced curves} $\pi:X\to S$  over an algebraic stack $S$ is a projective and flat morphism $\pi$ whose geometric fibers are connected reduced curves. 

A coherent sheaf $I$ on $X$ is said to be:
\begin{itemize}
\item  \emph{torsion-free on $X$}  if its associated points are the generic points of $X$, or equivalently if $I$ is pure of dimension $1$ (i.e. $I$ does not have torsion subsheaves), and with support $\supp(I)$ equal to $X$.
%The support of $I$, denoted by $\supp(I)$, is a subcurve of $X$. We say that $I$ is a torsion-free sheaf \emph{on $X$} if $I$ is torsion-free with $\supp(I)=X$. 
\item \emph{rank-$1$} if $I$ has rank one on each generic point of $\supp(I)$. 
%\item  \emph{simple} if $\End(I) = k $, or equivalently if $\Aut(I)=\Gm$.
\end{itemize}
Note that a torsion-free sheaf $I$ on $X$ is locally free away from the singular locus of $X$ and that each line bundle on $X$ is a  rank-$1$, torsion-free sheaf on $X$. 

%The \emph{degree} of a torsion-free rank-$1$ sheaf $I$ on $X$ is defined to be
%\begin{equation}\label{E:degI}
%\deg(I)=\chi(I)-\chi(\O_X). 
%\end{equation}

There are two natural ways of "restricting" a  torsion-free (resp. rank-$1$) sheaf $I$ on $X$ to a torsion-free (resp. rank-$1$) sheaf on a subcurve $Y\subseteq X$: 
    \begin{itemize}
        \item $I_Y$ the quotient of the restriction $I_{|Y}$ modulo the torsion subsheaf, so that $I_Y$ is torsion-free  on $Y$ and it is the smallest quotient of $I$ with support equal to $Y$.
        \item $\leftindex_{Y}{I}$ the kernel of the surjection $I\twoheadrightarrow I_{Y^c}$, so that $\leftindex_{Y}{I}$ is torsion-free  on $Y$ and it is the largest subsheaf of $I$ that is supported on $Y$.
    \end{itemize}
    Hence,  we have an exact sequence 
\begin{equation}\label{E:resY}
0\to \leftindex_{Y^c}{I}\to I \to I_Y\to 0,
\end{equation}
from which we deduce the equality 
\begin{equation}\label{E:add-chi}
    \chi(I)=\chi(I_Y)+\chi(\leftindex_{Y^c}{I}).
\end{equation}
%We let $\deg_Y (I)$ denote the degree of $I_Y$, that is, $\deg_Y(I) := \chi(I_Y )-\chi(\O_Y)$.

%The relation between the two constructions of Definition \ref{D:IY} is clarified by the following 

%\begin{lemma}\label{L:IY}
%    Let $I$ be a torsion-free sheaf on $X$ and let $Y$ be a subcurve of $X$. 
%    \begin{enumerate}[(i)]
%    \item \label{L:IY1} The following maps are injective: 
%    \begin{enumerate}[(a)]
%    \item the compositions $\leftindex_{Y}{I}\hookrightarrow I \twoheadrightarrow I_Y$ and  $\leftindex_{Y^c}{I}\hookrightarrow I \twoheadrightarrow I_{Y^c}$;
%    \item the simultaneous restriction map $I\to I_Y\oplus I_{Y^c}$;
%    \item the simultaneous inclusion map $\leftindex_{Y}{I}\oplus \leftindex_{Y^c}{I}\to I$.
%   \end{enumerate}
%    \item \label{L:IY2} We have the following canonical isomorphisms 
%    $$\frac{I_Y}{\leftindex_{Y}{I}}\xleftarrow{\cong}\frac{I_Y\oplus I_{Y^c}}{I} \xrightarrow{\cong}  \frac{I_{Y^c}}{\leftindex_{Y^c}{I}} \xleftarrow{\cong}\frac{I}{\leftindex_{Y}{I}\oplus \leftindex_{Y^c}{I}} \xrightarrow{\cong}  \frac{I_{Y}}{\leftindex_{Y}{I}}$$
%    \end{enumerate}
 %   of sheaves supported on the $0$-dimensional subscheme $Y\cap Y^c$. 
%\end{lemma}

It turns out that, for any  rank-$1$ torsion-free sheaf $I$ on $X$ and any subcurve $Y\subseteq X$, we have that (see \cite[Ex. 3.4]{FPV1})
    \begin{equation}\label{E:2restY}
    \leftindex_Y{I}=I_Y(-(Y\cap Y^c\cap \NF(I)^c)),
    \end{equation}
    where $\NF(I)$ is the set of nodes of $X$ at which $I$ is not free. 
    
    %Therefore, $\displaystyle \frac{I_Y}{\leftindex_Y{I}}$ is the structure sheaf of the $0$-dimensional reduced scheme $Y\cap Y^c\cap \NF(I)^c$. In particular, $I$ splits at $Y$ (i.e. $I=I_Y\oplus I_{Y^c}$) if and only if $I$ is not free at all nodes of $Y\cap Y^c$.

\section{V-compactified Jacobians for families of nodal curves}\label{Sec:VcJ-fam}

V-compactified Jacobians for families of reduced curves were introduced in \cite{FPV1} and then studied in detail for a fixed nodal curve in \cite{FPV2} (following the earlier treatment in \cite{pagani2023stability} and \cite{viviani2023new} in the fine case).
The aim of this subsection is to complete the above results for families of nodal curves. We will use the same notation of loc. cit., limiting ourselves to recall the main definitions and referring to loc. cit. for more details. 

Let us first recall the definition of V-stabilities and of V-compactified Jacobians for families of reduced curves, following \cite{FPV1}.

\begin{definition}\label{D:VStabX}
\noindent 
\begin{enumerate}
    \item 
Let $X$ be a connected reduced curve over $k=\ov k$. 
   A \emph{stability condition of vine type} (or simply a \textbf{V-stability condition)}    of characteristic $\chi\in \ZZ$ on  $X$ is a function
    \begin{align*}
        \s:\BCon(X)&\to \ZZ\\
        Y&\mapsto \s_Y
    \end{align*}
    satisfying the following properties:
\begin{enumerate}
\item \label{D:VStabX1} for any $Y\in \BCon(X)$, we have 
\begin{equation}\label{E:sum-s}
\mathfrak s_Y+\mathfrak s_{Y^\mathsf{c}}-\chi
\in \{0,1\}.
\end{equation}
%\begin{equation}\label{E:sum-n}
%\mathfrak n_Y+\mathfrak n_{Y^\mathsf{c}}-g(Y)-g(Y^c)+g(X)-d+1
%\in \{0,1\}.
%\end{equation}
A subcurve $Y\in \BCon(X)$ is said to be \emph{$\s$-degenerate} if $\s_Y+\s_{Y^c}-\chi=0$,
%$\mathfrak n_Y+\mathfrak n_{Y^\mathsf{c}}-g(Y)-g(Y^c)+g(X)-d+1=0$, 
and \emph{$\s$-nondegenerate} otherwise.

\item  \label{D:VStabX2} given subcurves $Y_1,Y_2,Y_3\in \BCon(X)$ without pairwise common irreducible components such that $X=Y_1\cup Y_2\cup Y_3$, we have that:
\begin{enumerate}
 \item if two among the subcurves $\{Y_1,Y_2,Y_3\}$ are $\s$-degenerate, then so is  the third. %is $\s$-degenerate;
            \item the following condition holds
            \begin{equation} \label{E:tria-s}
            \sum_{i=1}^{3}\s_{Y_i}-\chi
            %\sum_{i=1}^{3}\n_{Y_i}-\sum_{i=1}^3 g(Y_i)+g(X)-d+2
            \in \begin{cases}
                \{1,2\} \textup{ if $Y_i$ is $\s$-nondegenerate for all $i=1,2,3$};\\
                \{1\} \textup{ if there exists a unique   } i\in \{1,2,3\} \text{ such that $Y_i$ is $\s$-degenerate};\\
                \{0\} \textup{ if $Y_i$ is $\s$-degenerate for all $i=1,2,3$}.
            \end{cases}
        \end{equation}
\end{enumerate}
\end{enumerate}

The \emph{degeneracy set} of $\s$ is the collection
\begin{equation*}
\D(\s):=\{Y\in \BCon(X): Y \text{ is $\s$-degenerate}\}.
\end{equation*}
   
\item Let $\pi:X\to S$ be a family of connected reduced curves. 

 A \emph{(relative) stability condition of vine type} (or simply a \textbf{V-stability condition)} of characteristic $\chi$ on $X/S$ is a collection of V-stabilities 
 $$\s=\{\s^s\in \VStab^\chi(X_s)\: : s \text{ is a geometric point of } S\},$$
such that, for  any \'etale specialization $\xi: s\rightsquigarrow t$ of geometric points of $S$, we have that $\xi^*(\s^t)=\s^s$, where $\xi^*(\s^t)$ is defined by 
$$
\xi^*(\s^t)_Y:=\s^t_{\xi_*(Y):=\ov Y\cap X_t} \text{ for any } Y\in \BCon(X_s).
$$
\end{enumerate}
\end{definition}

The characteristic $\chi$ of a (relative) V-stability $\s$ will also be denoted by $|\s|$. The collection of all V-stability conditions of characteristic $\chi$ on $X/S$ is denoted by $\VStab^\chi(X/S)$ and the collection of all V-stability conditions on $X/S$ is denoted by 
$$\VStab(X/S)=\coprod_{\chi \in \ZZ}\VStab^\chi(X/S).$$
A V-stability condition $\s\in \VStab(X/S)$ is called \emph{general}  if  $\D(\s^s)=\emptyset$ for any geometric point $s$ of $S$.

We now recall how to associate to a V-stability condition on $X/S$ of characteristic $\chi$ a \emph{compactified Jacobian stack} for $X/S$  of characteristic $\chi$, i.e. an open substack $\ov \J_{g,n}^{\chi}$ of the stack $\TF^{\chi}_{X/S}$ of relative rank-$1$ torsion-free sheaves of characteristic $\chi$ on $X/S$, admitting a relative proper good moduli space $\ov J_{X/S}^\chi\xrightarrow{f} S$, called a \emph{compactified Jacobian space} for $X/S$. A compactified  Jacobian stack $\ov \J_{X/S}^\chi$ is called fine if the good moduli morphism $\Xi:\ov \J_{X/S}^\chi\to \ov J_{X/S}^{\chi}$ is a $\Gm$-gerbe.

\begin{fact}\label{F:VcJ} (\cite{FPV1}) 
Let $\pi:X\to S$ be a family of connected reduced curves over a  quasi-separated and locally Noetherian algebraic stack $S$.
For any V-stability condition $\s= \{\s^s\}$ on $X/S$ of characteristic $\chi$,  the substack $\ov \J_{X/S}(\s)$  of  $\TF^\chi_{X/S}$ defined  by
\begin{equation*}
\ov \J_{X/S}(\s):=\{I \in \TF_{X/S}^\chi\: : \:  \chi((I_{|X_s})_{Y_s})\geq \mathfrak \s^s_{Y_s}  \text{ for any geometric point $s$ of $S$ and any $ Y_s\in \BCon(X_s)$}\}.
\end{equation*}
is a compactified Jacobian stack of $X/S$ of characteristic $\chi$, called  the \textbf{V-compactified Jacobian stack} associated to $\s$, whose $S$-relative proper good moduli space is denoted by $\ov \J_{X/S}(\s)$ and is it called the \textbf{V-compactified Jacobian space} associated to $\s$.

Moreover, $\s$ is general if and only if $\ov \J_{X/S}(\s)$ is fine.
\end{fact}
\begin{proof}
  This follows from \cite[Thm. A]{FPV1} which is stated for an algebraic space $S$: the case of an arbitrary stack $S$ follows by taking a smooth atlas.     
\end{proof}

We can now combine the results of \cite{FPV1} and \cite{FPV2} in order to give the following characterization of V-compactified Jacobians for nodal curves. 

\begin{theorem}\label{T:VcJ-nod}
Let $X$ be a nodal curve over $k=\ov k$ and let $\ov \J_X^{\chi}\subseteq \TF_X^{\chi}$ be a compactified Jacobian stack of $X$. Then the following conditions are equivalent:
\begin{enumerate}[(i)]
    \item \label{T:VcJ-nod1} $\ov \J_X^{\chi}$ is a V-compactified Jacobian, i.e. there exists a V-stability condition $\s\in \VStab(X)$ of characteristic $|\s|=\chi$ such that $\ov \J_X^{\chi}=\ov \J_X(\s)$.
     \item \label{T:VcJ-nod2} $\ov \J_X^{\chi}$ is \emph{smoothable}, i.e. for any $1$-parameter smoothing $\X/\Delta$ of $X$, the open substack $\ov \J_{\X}^{\chi}:=\ov \J_X^{\chi}\cup J_{\X_{\eta}}^{\chi}\subset \TF_{\X/\Delta}^{\chi}$ is a compactified Jacobian stack for $\X/\Delta$. 
     \item \label{T:VcJ-nod3} $\ov \J_X^{\chi}$ is \emph{universally smoothable}, i.e. there exists a compactified Jacobian stack $\ov \J_{\fX}^{\chi}$ for the effective semiuniversal deformation $\fX\to \Spec R_X$ of $X$ such that $(\ov \J_{\fX}^{\chi})_o=\ov \J_{X}^{\chi}$.
\end{enumerate}
\end{theorem}
Note that the equivalence $\eqref{T:VcJ-nod1}\Longleftrightarrow \eqref{T:VcJ-nod2}$ was shown in \cite[Thm. B]{FPV2} and the full theorem was shown for fine compactified Jacobians in \cite[Thm. D]{viviani2023new}, building upon the results of \cite{pagani2023stability}.
\begin{proof}
 The equivalence $\eqref{T:VcJ-nod1}\Longleftrightarrow \eqref{T:VcJ-nod2}$ is shown in \cite[Thm. B]{FPV2}.

 The implication $\eqref{T:VcJ-nod3}\Longrightarrow \eqref{T:VcJ-nod2}$ follows from the fact that any $1$-parameter smoothing of $X$ is a pull-back of the semiuniversal deformation of $X$.

The implication $\eqref{T:VcJ-nod1}\Longrightarrow \eqref{T:VcJ-nod3}$
 follows from \cite[Lemma 8.4 and Thm. A]{FPV1}. 
\end{proof}

Furthermore, we can show that all compactified Jacobian stacks  are V-compactified Jacobian stacks  for families of nodal curves satisfying certain assumptions.

\begin{theorem}\label{T:cJ-nodal} 
    Let $\pi:X\to S$ be a family of connected nodal curves over a quasi-separated and locally Noetherian algebraic stack $S$ and let $\ov \J_{X/S}^\chi$ be a compactified Jacobian stack for $X/S$.
    Assume that one of the following two assumptions is satisfied:
    \begin{enumerate}[(A)]
        \item \label{Hyp1} The family $X/S$ is versal at any point of $S$.
        \item \label{Hyp2} The compactified Jacobian stack $\ov \J_{X/S}^{\chi}$ is fine and the open subset 
        $$
        S^{\sm}:=\{s\in S: X_s=\pi^{-1}(s) \text{ is smooth}\}
        $$
        is dense in $S$. 
    \end{enumerate}    
Then $\ov \J_{X/S}^\chi=\ov \J_{X/S}(\s)$  for a unique V-stability condition $\s$ on $X/S$ of characteristic $\chi$.
\end{theorem}
Note that if $X/S$ is versal at any point of $S$ then the open subset $S^{\sm}$ is dense in $S$.

\begin{proof}
Assume first that Condition \eqref{Hyp1} is satisfied. Then, any geometric fiber $(\ov \J_{X/S}^{\chi})_s$, for $s$ a geometric point of $S$, is a universally smoothable compactified Jacobian of $X_s$ in the sense of Theorem \ref{T:VcJ-nod}\eqref{T:VcJ-nod3}. Hence we can apply Theorem \ref{T:VcJ-nod} in order to get a unique V-stability condition $\s^s$ of characteristic $\chi$ on $X_s$ such that 
\begin{equation}\label{E:fiberJXS}
(\ov \J_{X/S}^{\chi})_s=\ov \J_{X_s}(\s^s).
\end{equation}

Assume now that Condition \eqref{Hyp2} is satisfied. Because $S^{\sm}$ is dense in $S$, for any geometric point $s$ of $S$ we can find a morphism $f: \Delta=\Spec R\to S$, where $R$ is a DVR, such that the closed point of $\Delta$ is $s$ and the generic point of $\Delta$ is mapped to a point of $S^{\sm}$. In particular, the pull-back $X_{\Delta}/\Delta$ of the family $X/S$ along $f$ is a one-parameter smoothing of $X_s$. By pulling back $\ov \J_{X/S}^\chi$ along $f$, we get a fine compactified Jacobian stack $\ov \J_{X_\Delta/\Delta}^\chi$ for $X_{\Delta}/\Delta$ whose central fiber is $(\ov \J_{X/S}^{\chi})_s$. This implies that $(\ov \J_{X/S}^{\chi})_s$ is a fine compactified Jacobian stack that is weakly smoothable according to the terminology of \cite[Def. 2.21]{viviani2023new}. Hence, we can apply \cite[Thm. D]{viviani2023new} in order to deduce that there exists a unique V-stability condition $\s^s$ of characteristic $\chi$ on $X_s$ such that \eqref{E:fiberJXS} holds true also in this case.

By \eqref{E:fiberJXS} and by \cite[Thm. 8.8]{FPV1}, we deduce that $\s=\{\s^s\}$ is a V-stability condition on $X/S$ of characteristic $\chi$ such that
$$
\ov \J_{X/S}^{\chi}=\ov \J_{X/S}(\s), 
$$
and we are done.     
\end{proof}

\begin{remark}
 We think that the previous Theorem should hold true under the weaker assumption (which is implied by both Hypothesis \eqref{Hyp1} and Hypothesis \eqref{Hyp2}) that $S^{\sm}$ is dense in $S$. This would follow (using the same proof of the above Theorem) if one could prove that any weakly smoothable compactified Jacobian stack of a nodal curve is a V-compactified Jacobian stack (see the open questions in \cite{FPV2}), thus extending one of the equivalences of \cite[Thm. D]{viviani2023new} from fine compactified Jacobian stacks to arbitrary ones.
\end{remark}

\section{Compactified universal Jacobians}\label{Sec:cuJ}

Fix a pair of integers $(g,n)\in \NN^2$ such that $2g-2+n>0$. Denote by $\Mbargn$ the moduli stack of stable $n$-pointed nodal curves of genus $g$ and by $\pi:\Cbargn\to \Mbargn$ the universal family over it. We denote by $\Cgn/\Mgn$ the restriction of the universal family to the open substack of smooth curves. 

The aim of this section is to study all the compactified universal Jacobians for $\Cbargn/\Mbargn$, by which we mean the following.

\begin{definition}\label{D:cJ-Uni}
   A \textbf{compactified universal Jacobian stack} of characteristic $\chi$ over $\Mbargn$ (or of type $(g,n)$) is an open substack $\ov \J_{g,n}^{\chi}$ of the stack $\TF^{\chi}_{g,n}$ of relative rank-$1$ torsion-free sheaves of characteristic $\chi$ on $\Cbargn/\Mbargn$, admitting a relative proper good moduli space $F:\ov \J_{g,n}^\chi\xrightarrow{\Xi} \ov J_{g,n}^\chi\xrightarrow{f} \Mbargn$, called a \textbf{compactified universal Jacobian space}. The set of compactified universal Jacobians stacks forms a poset with respect to inclusion.

   A compactified universal Jacobian stack $\ov \J_{g,n}^\chi$ is called \emph{fine} if the good moduli morphism $\Xi:\ov \J_{g,n}^\chi\to \ov J_{g,n}^{\chi}$ is a $\Gm$-gerbe.
\end{definition}

Notice that the restriction of any compactified universal Jacobian stack $\ov \J_{g,n}^{\chi}$ (resp. space $\ov J_{g,n}^{\chi}$) over $\Mgn$ is the universal Jacobian stack $\J_{g,n}^{\chi}$ (resp. space $J_{g,n}^\chi$), parameterizing pairs $(X,L)$ consisting of an $n$-pointed smooth curve of genus $g$ together with a line bundle $L$ on $X$ of characteristic $\chi$.

It turns out that universal compactified Jacobian stacks/spaces are classified by universal V-stabilities conditions.

\begin{definition}\label{D:VStab-Uni}
The set of \emph{(universal) V-stability conditions} of type $(g,n)$ is the  set 
$$
\VStab_{g,n}=\varprojlim_{X\in \Mbargn} \VStab(X),
$$
where the limit is taken with respect to the pull-back maps induced by \'etale specializations of geometric points on $\Mbargn$. In other words, a universal V-stability condition of type $(g,n)$ is a collection of V-stability conditions 
 $$\s=\{\s(X)\in \VStab(X)\: : X\in \Mbargn(k=\ov k)\},$$
such that, for  any \'etale specialization $\xi: X_1\rightsquigarrow X_2$ of geometric points of $\Mbargn$, we have that $\xi^*(\s(X_2))=\s(X_1)$. 

A universal V-stability condition $\s\in \VStab_{g,n}$ is called \emph{general} if  $\s(X)$ is general for every $X\in \Mbargn(k=\ov k)$.
\end{definition}

Note that the V-stability conditions of type $(g,n)$ are exactly the relative V-stability conditions (see Definition \ref{D:VStabX}\eqref{D:VStabX2}) for the universal family $\Cbargn/\Mbargn$.

\begin{remark}\label{R:VStab-Uni}
From the above Definition \ref{D:VStab-Uni}, we deduce that, for a fixed $\s\in \VStab_{g,n}$, we have:
\begin{enumerate}[(i)]
\item \label{R:VStab-Uni1} If $X_1,X_2\in \Mbargn(k=\ov k)$ have the same dual graph $G$ (i.e. if they belong to the same stratum of $\Mbargn$), then we have that $\s(G):=\s(X_1)=\s(X_2)$ after having identified $\BCon(X_i)$ with the biconnected subsets of $V(G)$, i.e. the subsets $W\subset V(G)$ such that the induced subgraphs $G[W]$ and $G[W^c]$ are connected. Moreover, we have that $\s(G)$ is $\Aut(G)$-invariant. 
\item \label{R:VStab-Uni2} The characteristic $|\s(X)|$ is independent of the chosen $X\in \Mbargn$, and it will be called the characteristic of $\s$ and denoted by $|\s|$. Hence, we get a  decomposition according to the characteristic of the V-stability conditions
$$\VStab_{g,n}=\coprod_{\chi\in \ZZ}\VStab^{\chi}_{g,n} \text{ where }
\VStab_{g,n}^{\chi}:=\varprojlim_{X\in \Mbargn} \VStab^{\chi}(X).$$
\end{enumerate}
\end{remark}

The set $\VStab_{g,n}$ becomes a poset under the following order relation.

\begin{definition}\label{D:pos-VStab}
Let $\s_1,\s_2\in \VStab_{g,n}$. We say that $\s_1\geq \s_2$ if $|\s_1|=|\s_2|$ and, for any $X\in \Mbargn$, we have that 
$$\s_1(X)\geq \s_2(X) \text{, i.e. } \s_1(X)_Y\geq \s_2(X)_Y \text{ for any } Y\in \BCon(X).$$
\end{definition}

We can now state our main classification result for the poset of compactified universal Jacobians over $\Mbargn$.

\begin{theorem}\label{T:cJUniv}
There is an anti-isomorphism of posets
$$
\begin{aligned}
\VStab_{g,n} & \xrightarrow{\cong} \left\{\text{Compactified universal Jacobian stacks over $\Mbargn$}\right\},\\
\s &\mapsto \ov \J_{g,n}(\s):=\left\{ I \in  \TF_{g,n}: \: I_{|X} \text{ is $\s(X)$-semistable for every } X\in \Mbargn \right\}.
\end{aligned}
$$

Moreover, $\s$ is general if and only if $\ov \J_{g,n}(\s)$ is fine. 
\end{theorem}
We will denote by $\ov J_{g,n}(\s)\to \Mbargn$ the relative good moduli space of $\ov \J_{g,n}(\s)$. Note that if $\s\in \VStab^{\chi}_{g,n}$ then $\ov \J_{g,n}(\s)\subset \TF_{g,n}^\chi$.  
\begin{proof}
 The fact that $\ov \J_{g,n}(\s)$ is a compactified universal Jacobian stack, as well as the fact that it is fine if and only if $\s$ is general, follows from Fact \ref{F:VcJ}. 
The bijectivity of the map follows from Theorem \ref{T:cJ-nodal}, using that $\Cbargn/ \Mbargn$ is versal at any point. 
Finally, it remains to observe that
$$
\s_1\geq \s_2\Leftrightarrow \ov \J_{g,n}(\s_1)\subseteq \ov \J_{g,n}(\s_2).$$
Indeed, the implication $\Rightarrow$ follows from the definition of $\ov \J_{g,n}(\s_i)$. The implication $\Leftarrow$ follows from the above mentioned definition and \cite[Lemma 8.10]{FPV1}.
\end{proof}

We want now to describe more explicitly the poset of universal V-stability conditions on $\Mbargn$, extending \cite{fava2024} from the case of general universal V-stability conditions to arbitrary ones.  

Consider the \emph{stability domain} of type $(g,n)$ (see \cite[Def. 2.2]{fava2024})
\begin{equation}\label{E:Dgn}
\Dgn:=\{(e;h,A)\: : 
e\in \NN_{>0}, 0\leq h \leq g-e+1, A\subseteq [n],2h-2+e+|A|> 0, 2g-2h-e+|A^c|>0\},
\end{equation}
where $[n]:=\{1,\ldots, n\}$. We define two structures on $\Dgn$:
\begin{itemize}
\item the complement of an element of $\Dgn$ is defined as
$$
(e;h,A)^c:=(e; g-h-e+1,A^c) \text{ where } A^c:=[n]- A.
$$
This defines an involution on $\Dgn$ whose fixed points occur when $2h+e=g$ and $n=0$.
\item a triangle in $\Dgn$ is a $3$-elements multiset (that is, repetitions are allowed) $$\Delta=[(e_1;h_1,A_1),(e_2;h_2,A_2), (e_3;h_3,A_3)]$$  of $\Dgn$ such that 
$$
\begin{sis}
& 2\vert (e_1+e_2+e_3),\\
& e_i+e_j\geq e_k+2 \text{ for any } \{i,j,k\}=\{1,2,3\}, \\
& A_1\sqcup A_2\sqcup A_3=[n],\\
& g=h_1+h_2+h_3+\frac{e_1+e_2+e_3}{2}-2.
\end{sis}
$$
\end{itemize}
We define the \emph{log-canonical degree} of $(e;h,A) \in \Dgn$ as
 \begin{equation}\label{E:lcdeg}
\delta(e;h,A):= 2h-2+e+|A|\in (0,2g-2+n).
\end{equation}
It is easy to check that the log-canonical degree satisfies the following additivity properties:
\begin{equation}\label{E:for-lcdeg}
 \begin{sis}
& \delta(x)+\delta(x^c)=2g-2+n \text{ for any } x\in \Dgn,\\
& \delta(x_1)+\delta(x_2)+\delta(x_3)=2g-2+n \text{ for any triangle } \Delta=[x_1,x_2,x_3]  \text{ in } \Dgn.
 \end{sis}   
\end{equation}

%\begin{lemma}[Additivity of $\delta$ in triangles]\label{lem:delta-additivity-triangle}
%Let $\Delta=\{x_1,x_2,x_3\}$ be a triangle in $\mathbb{D}_{g,n}$, with
%$x_i=(e_i;h_i,A_i)$ for $i=1,2,3$. Then
%\begin{equation} \label{eq:delta-add}
%\delta(x_1^c)=\delta(x_2)+\delta(x_3),\qquad
%\delta(x_2^c)=\delta(x_1)+\delta(x_3),\qquad
%\delta(x_3^c)=\delta(x_1)+\delta(x_2).
%\end{equation}
%\end{lemma}

%\begin{proof}
%Since $\Delta$ is a triangle, we have $A_1\sqcup A_2\sqcup A_3=[n]$, hence
%\[
%|A_1^c|=|A_2|+|A_3|,\qquad |A_2^c|=|A_1|+|A_3|,\qquad |A_3^c|=|A_1|+|A_2|.
%\]
%Moreover, the genus relation for a triangle gives
%\[
%g=h_1+h_2+h_3+\frac{e_1+e_2+e_3}{2}-2.
%\]
%By definition of complement,
%\[
%\delta(x_i^c)=2(g-h_i-e_i+1)-2+(g-h_i-e_i+1)+|A_i^c|
%=3(g-h_i-e_i+1)-2+|A_i^c|.
%\]
%Substituting the two displayed triangle identities and simplifying yields
%\[
%\delta(x_1^c)=\delta(x_2)+\delta(x_3),\qquad
%\delta(x_2^c)=\delta(x_1)+\delta(x_3),\qquad
%\delta(x_3^c)=\delta(x_1)+\delta(x_2),
%\]
%as claimed.
%\end{proof}

\begin{remark}\label{R:vine-trian}
Note that there are bijections:
\begin{itemize}
    \item 
$\{\text{Pairs of complementary elements of } \Dgn\} \xrightarrow{\cong} \{\text{Vine strata of } \Mbargn\},$ 

which sends a pair  $\{(e;h,A), (e;h,A)^c\}$ to the vine stratum of $\Mbargn$ whose generic element is a $n$-pointed vine curve $X=C_1\cup C_2$ consisting of two smooth curves $C_1$ and $C_2$ of genera, respectively, $h$ and $g-h-e+1$, meeting in $e$ nodes and such that $A\subset C_1$ and $A^c\subset C_2$.  The pair consists of one element exactly when it is a fixed point of the involution $(e,h,A) \to (e,h,A)^c$ defined above, which happens precisely when the generic element of the corresponding stratum admits an involution that swaps the two irreducible components.
\item $\{\text{Triangles in } \Dgn\} \xrightarrow{\cong} \{\text{Triangular strata of } \Mbargn\},$

which sends a triangle  $[(e_1;h_1,A_1), (e_2;h_2,A_2), (e_3;h_3,A_3)]$ to the stratum of $\Mbargn$ whose generic element is a $n$-pointed curve $X=C_1\cup C_2\cup C_3$ consisting of three smooth curves $C_1$, $C_2$ and $C_3$ such that $g(C_i)=h_i$, $A_i=\{h\in [n]\: : p_h\in C_i\}$, and such that $|C_i\cap C_j|:=\frac{e_i+e_j-e_k}{2}$ for any $\{i,j,k\}=\{1,2,3\}$. 
\end{itemize}
\end{remark}

The set $\Dgn$ describes the combinatorial type (or simply the type) of the biconnected subcurves of any $X\in \Mbargn$ in the following sense: for any $(X,p_i)\in \Mbargn$ there is a function
\begin{equation}\label{E:type}
    \begin{aligned}
      \type=\type_X: \BCon(X) & \longrightarrow \Dgn\\
      Y & \mapsto \type(Y):=(|Y\cap Y^c|; g(Y),\{i\in [n]\: : p_i\in Y\}).
    \end{aligned}
\end{equation}
Note that: 
\begin{itemize}
\item $\type_X(Y^c)=\type_X(Y)^c$ for any $Y\in \BCon(X)$.
\item If $X=Y_1\cup Y_2\cup Y_3$ with $Y_i\in \BCon(X)$ and without pairwise common irreducible components, then $[\type_X(Y_1),\type_X(Y_2),\type_X(Y_3)]$ is a triangle in $\Dgn$.
\end{itemize}

\begin{definition}\label{D:Sigmagn}
   Denote by $\Sigma^\chi_{g,n}$ the set of all functions (called \emph{V-functions} of type $(g,n)$)
    \begin{align*}
        \sigma:\Dgn&\to \ZZ\\
        (e;h,A)&\mapsto \sigma(e;h,A)
    \end{align*}
    satisfying the following properties:
\begin{enumerate}
\item \label{E:condUni1} for any $(e;h,A)\in \Dgn$, we have 
\begin{equation}\label{E:sumUni}
\sigma(e;h,A)+\sigma((e;h,A)^c)-\chi
\in \{0,1\}.
\end{equation}

An element $(e;h,A)\in \Dgn$ is said to be \emph{$\sigma$-degenerate} if $\sigma(e;h,A)+\sigma((e;h,A)^c)=\chi$,
and \emph{$\sigma$-nondegenerate} otherwise.

\item  \label{E:condUni2} for any triangle $\Delta=[(e_1;h_1,A_1), (e_2;h_2,A_2), (e_3;h_3,A_3)]$ of $\Dgn$, we have that:
\begin{enumerate}
 \item if two among the elements of $\Delta$ are $\sigma$-degenerate, then so is  the third. 
            \item the following holds
            \begin{equation}\label{E:triaUni}
            \sum_{i=1}^{3}\sigma(e_i; h_i,A_i)-\chi
            \in \begin{cases}
                \{1,2\} \textup{ if $(e_i;h_i,A_i)$ is $\sigma$-nondegenerate for all $i$};\\
                \{1\} \textup{ if there exists a unique $i$ such that } \\
                \hspace{2cm} \textup{ $(e_i;h_i,A_i)$ is $\sigma$-degenerate};\\
                \{0\} \textup{ if $(e_i;h_i,A_i)$ is $\sigma$-degenerate for all $i$}.
            \end{cases}
        \end{equation}
\end{enumerate}
We say that $\sigma$ has Euler characteristic $\chi$ and we write $\chi=|\sigma|$. We set 
$$\Sigma_{g,n}:=\coprod_{\chi \in \ZZ}\Sigma_{g,n}^\chi.$$
\end{enumerate}
\end{definition}
 The \emph{degeneracy subset} of $\sigma$ is the collection
\begin{equation}\label{E:degsigma}
\D(\sigma):=\{(e;h,A)\in \Dgn: (e;h,A) \text{ is $\sigma$-degenerate}\}.
\end{equation}
We say that $\sigma$ is \emph{general}  if $\D(\sigma)=\emptyset$.

Note that if $\sigma\in \Sigma_{g,n}^\chi$ is general, then the function  $\m_{(e;h,A)}:=\sigma(e;h,A)+h-1$ defines a degree $d:=\chi+g-1$ universal stability condition of type $(g,n)$ in the sense of \cite[Def. 2.4]{fava2024}.

\begin{remark}\label{R:triang}
   It follows from the  Conditions \eqref{E:condUni1} and \eqref{E:condUni2} of Definition \ref{D:Sigmagn} that, for any triangle  $\Delta=[(e_1;h_1,A_1), (e_2;h_2,A_2), (e_3;h_3,A_3)]$ of $\Dgn$, we have that
$$
\sigma((e_3;h_3,A_3)^c)-\sigma(e_1;h_1,A_1)-\sigma(e_2;h_2,A_2)=
\begin{sis}
0 & \text{ if either } (e_1;h_1,A_1)\in \D(\sigma) \\
&  \text{ or }      (e_2;h_2,A_2)\in \D(\sigma), \\
-1& \text{ if } (e_1;h_1,A_1), (e_2;h_2,A_2)\not \in \D(\sigma)\\
& \text{ and }  (e_3;h_3,A_3) \in \D(\sigma),\\
\{0,-1\} & \text{ if }(e_1;h_1,A_1), (e_2;h_2,A_2), (e_3;h_3,A_3) \not \in \D(\sigma).
\end{sis}
$$
Note that we have the convenient additivity $\delta((e_3;h_3,A_3)^c)=\delta(e_1;h_1,A_1)+\delta(e_2;h_2,A_2)$ for the log-canonical degrees introduced in \eqref{E:lcdeg}.
\end{remark}

We now introduce a poset structure on the set of V-functions of type $(g,n)$. 

\begin{definition}\label{D:Sigma-pos}
The space of V-functions $\Sigma_{g,n}$ of type $(g,n)$ comes with the following order relation 
  $$
  \sigma_1\geq \sigma_2 \Longleftrightarrow 
  \begin{sis}
  &|\sigma_1|=|\sigma_2|,\\
  & \sigma_1(e;h,A)\geq \sigma_2(e;h,A) \text{ for any } (e;h,A)\in \Dgn.\\
  \end{sis}
  $$
\end{definition}
Note that each $\Sigma_{g,n}^{\chi}$ is a union of connected components of the poset $\Sigma_{g,n}$.

\begin{remark}\label{R:Sigma-pos}
    Let $\sigma_1,\sigma_2\in \Sigma_{g,n}$. 
%VECCHIA PRIMA PARTE
%\item \label{R:Sigma-pos1} It follows from Proposition \ref{P:Stabgn} that $$\sigma_1\geq \sigma_2\Leftrightarrow \s^{\sigma_1}(X)\geq \s^{\sigma_2}(X) \text{ for any } X\in \Mbargn,$$ 
%        where the latter order relation is the one of \cite[Def. 4.11]{FPV1}. 
     It follows from Definition \ref{D:Sigmagn} that if $\sigma_1\geq \sigma_2$ then $\D(\sigma_1)\subseteq \D(\sigma_2)$ and for any $(e;h,A)\in \Dgn$ we have that:
        $$
(\sigma_1(e;h,A),\sigma_1((e;h,A)^c))=
   \begin{cases}
   (\sigma_2(e;h,A),\sigma_2((e;h,A)^c))  \\
   \hspace{3cm} \text{ if either } (e;h,A)\in \D(\sigma_1) \text{ or } (e;h,A)\not \in \D(\sigma_2), \\
    (\sigma_2(e;h,A)+1,\sigma_2((e;h,A)^c)) \text{ or }  (\sigma_2(e;h,A),\sigma_2((e;h,A)^c)+1) \\
    \hspace{3cm} \text{ if } (e;h,A)\in \D(\sigma_2)- \D(\sigma_1).
   \end{cases}
$$
%\item \label{R:Sigma-pos3} We have that 
%$$
%\sigma_1\geq \sigma_2\Leftrightarrow \ov \J_{g,n}(\sigma_1)\subseteq \ov \J_{g,n}(\sigma_2).$$
%Indeed, the implication $\Rightarrow$ follows from the definition of $\ov \J_{g,n}(\sigma_i)$ (see Corollary \ref{C:cJUniv}). The implication $\Leftarrow$ follows from the above mentioned definition and \cite[Lemma 8.10]{FPV1}.
\end{remark}

\begin{proposition}\label{P:Stabgn}
There is an isomorphism of posets 
\begin{equation*}
   \begin{aligned}
      \Sigma^\chi_{g,n} & \xrightarrow{\cong} \VStab^\chi_{g,n}\\
      \sigma & \mapsto \s^{\sigma}=\{\s^{\sigma}(X)\}_{X\in \Mbargn}
   \end{aligned} 
\end{equation*}
defined by 
$$
\s^{\sigma}(X)_Y:=\sigma(\type_X(Y)) \: \text{ for any } Y\in \BCon(X).
$$
\end{proposition}
Observe that $\sigma$ is general if and only if $\s^\sigma$ is general. 
\begin{proof}
Let us first show that the map is well-defined, i.e. that $\s^{\sigma}\in \VStab^{\chi}_{g,n}$. The function $\s^\sigma(X)$ is a V-stability on $X$ of characteristic $\chi$, i.e. $\s^\sigma(X)$ satisfies the two properties of \cite[Def. 4.1]{FPV1}, as it follows from the fact that $\sigma$ satisfies the two properties \eqref{E:condUni1} and \eqref{E:condUni2} of Definition \ref{D:Sigmagn} and by using that $\type(Y^c)=\type(Y)^c$ and that, if $X=Y_1\cup Y_2\cup Y_3$ with $Y_1, Y_2, Y_3 \in \BCon(X)$ have no pairwise common irreducible components, then $[\type(Y_1, \type(Y_2), \type(Y_3))]$ is a triangle in $\Dgn$.  Moreover, the collection $\s^{\sigma}=\{\s^\sigma(X)\}_{X\in \Mbargn}$ is a V-stability of type $(g,n)$ since, for  any \'etale specialization $\xi: X_1\rightsquigarrow X_2$ of geometric points of $\Mbargn$ and any $Y\in \BCon(X_1)$, we have that $\type(Y)=\type(\xi_*(Y))$, which then implies that $\xi^*(\s(X_2))=\s(X_1)$.

We will now define a map in the other direction.  Start with a V-stability $\s=\{\s(X)\}_{X\in \Mbargn}$ of type $(g,n)$. For any pair $(e;h,A)\in \Dgn$, pick a vine curve $V=C_1\cup C_2\in \Mbargn$ such that $\type(C_1)=(e;h,A)$. Set $\sigma^\s(e;h,A):=\s(V)_{C_1}$ which is well-defined, i.e. it is independent of the chosen $V$, by Remark \ref{R:VStab-Uni}\eqref{R:VStab-Uni1}. Property (1) of \cite[Def. 4.1]{FPV1} for $\s(V)$ applied to $C_1\in \BCon(V)$ and $C_2=C_1^c$ gives Property \eqref{E:condUni1} of Definition \ref{D:Sigmagn} for $\sigma^\s$.
Let us now check that $\sigma^\s$ also satisfies  Property \eqref{E:condUni2} of Definition \ref{D:Sigmagn}. Fix a triangle $\Delta=[(e_i;h_i,A_i); i=1,2,3]$ in $\Dgn$ and consider a triangular curve $(T,p_k)\in \Mbargn$ in the triangular stratum corresponding to $\Delta$ as in Remark \ref{R:vine-trian}, i.e. $T=C_1\cup C_2\cup C_3$ with $C_i$ smooth of genus $h_i$ with the property that $A_i=\{k\in [n]\: : p_k\in C_i\}$ and $t_{i}:=|C_i\cap C_i^c|$. In particular, $\type(C_i)=(e_i;h_i,A_i)$.  For any $1\leq i\leq 3$, denote by $\xi_i:V_i\rightsquigarrow T$ the geometric specialization that corresponds to the smoothing of the nodes of $C_j\cap C_k$ where $\{1,2,3\}=\{i,j,k\}$. Then $V_i$ is a vine curve with a component $\wt C_i$ such that $(\xi_i)_*(\wt C_i)=C_i$. By the compatibility of the V-stability conditions of $\s$ and the definition of $\sigma^\s$, we get that 
$$
\s(T)_{C_i}=\s(T)_{(\xi_i)_*(\wt C_i)}=(\xi_i^*\s(T))_{\wt C_i}=\s(V_i)_{\wt C_i}=\sigma^\s(e_i; h_i,A_i). 
$$
Therefore, Property~(2) of \cite[Def. 4.1]{FPV1} for $\s(T)$ applied to $T=C_1\cup C_2\cup C_3$ corresponds to Property~\eqref{E:condUni2}  of Definition \ref{D:Sigmagn} for $\sigma^\s$. 

We will now show that the two maps $\sigma\mapsto \s^\sigma$ and $\s\mapsto \sigma^\s$ are inverses one of the other. The fact that $\sigma^{\s^{\sigma}}=\sigma$ follows immediately from the definitions. In order to prove that $\s^{\sigma^\s}=\s$, consider a curve $X\in \Mbargn$ and any $Y\in \BCon(X)$. Denote by $\xi:\wt X\rightsquigarrow X$ the geometric specialization that corresponds to the smoothing of all the nodes of $X$ except those belonging to $Y\cap Y^c$. Then $\wt X$ is a vine curve of $\Mbargn$ with an irreducible component $\wt Y$ such that $\xi_*(\wt Y)=Y$. Since $\xi^*\s(X)=\s(\wt X)$ by the assumption on $\s$, we get that 
\begin{equation}\label{E:calco1}
    \s(X)_Y=\s(X)_{\xi_*(\wt Y)}=\xi^*\s(X)_{\wt Y}=\s(\wt X)_{\wt Y}.
\end{equation}
On the other hand, by the definition of the  two maps $\sigma\mapsto \s^\sigma$ and $\s\mapsto \sigma^\s$, we have that 
\begin{equation}\label{E:calco2}
    \s^{\sigma^\s}(X)_Y=\sigma^\s(\type_X(Y))=\sigma^\s(\type_{\wt X}(\wt Y))=\s(\wt X)_{\wt Y}.
\end{equation}
By comparing \eqref{E:calco1} and \eqref{E:calco2}, we get that $\s^{\sigma^s}=\s$, as required.

Finally, it is clear from the definition of the two maps $\sigma\mapsto \s^\sigma$ and $\s\mapsto \sigma^\s$ that we have
$$\sigma_1\geq \sigma_2\Leftrightarrow \s^{\sigma_1}(X)\geq \s^{\sigma_2}(X) \text{ for any } X\in \Mbargn.$$ 
\end{proof}

We conclude this section by proving Corollary~\ref{C:stesso-HD} from the introduction. 

\begin{proof}[Proof of Corollary \ref{C:stesso-HD}]
Part \eqref{C:isoH} follows from the fact that $\ov J_{g,n}(\sigma_i)$ is regular (since $\TF_{g,n}$ is regular and $\ov J_{g,n}(\sigma_i)=\ov \J_{g,n}(\sigma_i)\fatslash \Gm$ because $\sigma_i$ is general) and that the morphism $\ov \J_{g,n}(\sigma_i)_{\CC}\to (\Mbargn)_{\CC}$ has full support, which can be deduced arguing as in \cite[Thm. 5.12]{Migliorini_2021}. Indeed, the result of loc. cit. implies that any classical (in the sense of \cite[Ex. 6.7]{FPV1}) relative fine compactified Jacobian over a versal family of nodal curves has full support. However, the fact that the relative compactified Jacobian is classical is only used in the computation of the class in the Grothendieck group of varieties of a classical fine compactified Jacobian of a nodal curve over $k=\ov k$, see \cite[Prop. B.2]{Migliorini_2021}. Since the same formula holds true for fine V-compactified Jacobians of nodal curves by \cite[Cor. 2.6]{viviani2023new} and the geometric fibers of $\ov J_{g,n}(\sigma_i)\to \Mbargn$ are fine V-compactified Jacobians  by Theorem \ref{T:thmA}, the  proof of \cite[Thm. 5.12]{Migliorini_2021} extends to this case.

Part \eqref{C:isoD}: since $\ov J_{g,n}^\chi(\sigma_i)$ are fine compactified universal Jacobians (for $i=1,2$), then we can pick a tautological sheaf $\I(\sigma_i):=\I^{\taut}_{g,n}(\sigma_i)$ on the universal family $\ov J_{g,n}^\chi(\sigma_i)\times_{\Mbargn} \Cbargn\to \ov J_{g,n}^\chi(\sigma_i)$ (a fact that it is obvious for $n\geq 1$ because the universal family has sections and that is proved in Corollary \ref{C:fineMg} for $n=0$). 
Consider the sheaf 
$p_{13}^*(\I(\sigma_1))\otimes p_{23}^*(\I(\sigma_2))$ on $\ov J_{g,n}^\chi(\sigma_1)\times_{\Mbargn}\ov J_{g,n}^\chi(\sigma_2)\times_{\Mbargn} \Cbargn\xrightarrow{\pi} \ov J_{g,n}^\chi(\sigma_1)\times_{\Mbargn}\ov J_{g,n}^\chi(\sigma_2)$, where $p_{i3}$ is the projection onto the $i$-th and third factor. Over the big open substack 
$$\iota: (\ov J_{g,n}^\chi(\sigma_1)\times_{\Mbargn}\ov J_{g,n}^\chi(\sigma_2))^\sharp\hookrightarrow \ov J_{g,n}^\chi(\sigma_1)\times_{\Mbargn}\ov J_{g,n}^\chi(\sigma_2)$$ 
whose geometric points consist of pairs of sheaves $(I_1,I_2)$ on $C\in \Mbargn(k=\ov k)$ such that for any node $p\in C$ either $I_1$ or $I_2$ is free, the sheaf  $p_{13}^*(\I(\sigma_1))\otimes p_{23}^*(\I(\sigma_2))$ is flat and hence we can consider the Poincar\'e line bundle $$\P:=d_{\pi}(p_{13}^*(\I(\sigma_1))\otimes p_{23}^*(\I(\sigma_2)))^{-1}\otimes d_{\pi}(p_{13}^*(\I(\sigma_1)))\otimes d_{\pi}(p_{23}^*(\I(\sigma_2)))\in \Pic((\ov J_{g,n}^\chi(\sigma_1)\times_{\Mbargn}\ov J_{g,n}^\chi(\sigma_2))^\sharp).$$ The pushforward $\ov \P:=\iota_*(\P)$ (called the Poincar\'e or Arinkin sheaf) is a maximal Cohen-Macaulay sheaf that is flat over each factor (see \cite[Thm.~A]{Ari}, \cite[Thm. 4.6]{MRV2}, \cite[Thm.~1.2]{MolchoFourier}). The sheaf $\ov \P$ satisfies the following identity in the derived category of $\ov J_{g,n}^\chi(\sigma_1)\times_{\Mbargn}\ov J_{g,n}^\chi(\sigma_1)$
\begin{equation}\label{E:convP}
Rp_{\pi_{13*}}(\pi_{12}^*((\ov \P)^\vee)\otimes \pi_{23}^*(\ov \P))\cong \O_{\Delta}[-g],
\end{equation}
where $\pi_{ij}$ denotes the projection onto the $i$-th and $j$-th term of the product $\ov J_{g,n}^\chi(\sigma_1)\times_{\Mbargn}\ov J_{g,n}^\chi(\sigma_2)\times_{\Mbargn}\ov J_{g,n}^\chi(\sigma_1)$. The identity \eqref{E:convP} can be checked \'etale-locally on the base and hence it follows from \cite[Thm. 6.2]{MRV2} where it is proved for two relative classical compactified Jacobians over the semiuniversal family of a reduced curve with planar singularity; however, an inspection of the proof reveals that the same arguments work for two relative V-compactified Jacobians. The identity \eqref{E:convP} now implies that the Fourier-Mukai transform with kernel $\ov \P$ gives the desired equivalence of triangulated categories. 
\end{proof}

%VECCHIO COROLLARIO (ora inglobato nel Teorema A)
%By combining Theorem \ref{T:cJUniv} and Proposition \ref{P:Stabgn}, we deduce the following result which was shown in \cite[Thm. A]{fava2024} for fine compactified universal Jacobians. 

%\begin{corollary}\label{C:cJUniv}
%We have a bijection 
%$$
%\begin{aligned}
%\Sigma_{g,n} & \xrightarrow{\cong} \left\{\text{Compactified universal Jacobian stacks over $\Mbargn$}\right\},\\
%\sigma &\mapsto \ov \J_{g,n}(\sigma):=\left\{
%\begin{aligned} 
%& I \in  \TF_{g,n}^{|\sigma|}: \: \chi((I_{|X})_Y)\geq \sigma(\type_X(Y)) \\ 
%& \text{ for any } X\in \Mbargn \text{ and any } Y\in \BCon(X) 
%\end{aligned} 
%\right\}.
%\end{aligned}
%$$
%Moreover, $\sigma$ is general if and only if $\ov \J_{g,n}(\sigma)$ is fine. 
%\end{corollary}

\section{Classical compactified universal Jacobians}\label{Sec:classcJ}

The aim of this section is to describe the classical compactified universal Jacobians over $\Mbargn$, extending the work of \cite{Kass_2019} from the fine case to the general case. 

\begin{definition}\label{D:Pol-Uni}(\cite[Def. 3.2]{Kass_2019})
The set of \emph{(universal) numerical polarizations} of type $(g,n)$ is the following set 
$$
\Pol_{g,n}=\varprojlim_{X\in \Mbargn} \Pol(X),
$$
with respect to the pull-back maps induced by \'etale specializations of geometric points on $\Mbargn$. 

In other words, a universal numerical polarization of type $(g,n)$ is a collection of numerical polarizations 
 $$\psi=\{\psi(X)\in \Pol(X)\: : X\in \Mbargn(k=\ov k)\},$$
such that, for  any \'etale specialization $\xi: X_1\rightsquigarrow X_2$ of geometric points of $\Mbargn$, we have that $\xi^*(\psi(X_2))=\psi(X_1)$. 
\end{definition}
Note that we have a decomposition according to the Euler characteristic 
$$\Pol_{g,n}=\coprod_{\chi\in \ZZ}\Pol^{\chi}_{g,n} \text{ where }
\Pol_{g,n}^{\chi}:=\varprojlim_{X\in \Mbargn} \Pol^{\chi}(X).$$
Each $\Pol_{g,n}^\chi$ is a real affine space with underlying real vector space $\Pol^0_{g,n}$. 

It follows from \cite[Lemma-Definition 4.26]{FPV1} that there is a map 
\begin{equation}\label{E:map-ceil}
\begin{aligned}
\lceil - \rceil: \Pol_{g,n}=\coprod_{\chi\in \ZZ}\Pol^{\chi}_{g,n} & \longrightarrow \coprod_{\chi\in \ZZ}\VStab^{\chi}_{g,n}=\VStab_{g,n} \\
\psi=\{\psi(X)\} & \mapsto \s_{\psi}=\{\s_{\psi}(X): Y \mapsto \lceil \psi(X)_Y \rceil\}
\end{aligned}
\end{equation}

The affine spaces $\Pol_{g,n}^\chi$ have been described by Kass-Pagani in \cite[Sec. 3]{Kass_2019} in terms of the real space of relative line bundles for $\Cbargn/\Mbargn$, as we now recall. We set 
$$
\begin{sis}
&  \PicRel_{g,n}(\ZZ):=\Pic(\Cbargn)/\pi^*\Pic(\Mbargn),\\
&  \PicRel_{g,n}(\RR):=\PicRel_{g,n}(\ZZ)\otimes_{\ZZ}\RR,\\
& \PicRel_{g,n}^\chi(\RR):=\{L\in \PicRel_{g,n}(\RR)\: : \deg_{\pi}(L)=\chi\} \text{ for any } \chi\in \ZZ,\\
& \PicRel_{g,n}^{\ZZ}(\RR):=\coprod_{\chi \in \ZZ} \PicRel_{g,n}^\chi(\RR).
\end{sis}
$$
where $\deg_{\pi}(L)$ is the $\pi$-relative degree of $L$.

It follows from the results of Arbarello-Cornalba (see \cite[Fact 1]{Kass_2019} and the references therein) that $\PicRel_{g,n}(\ZZ)$ is the abelian group generated by:
\begin{itemize}
    \item the relative dualizing line bundle $\omega_\pi$; 
    \item the image $\Sigma_i$ of the $i$-th section of $\Cbargn/\Mbargn$ (for $1\leq i \leq n$);
    \item the boundary line bundles $\{\O(C_{(h,A)})\}_{(h,A)\in \Bgn}$ on $\Cbargn$, where  \begin{equation}\label{E:Bgn}
        \Bgn:=\{(h,A)\: :  0\leq h\leq g, A\subseteq [n], 2h-2+|A|>0, 2g-2h+|A^c|>0\}, 
        \end{equation}
    and  $C_{(h,A)}$ is the divisor of $\Cbargn$ whose generic point is a curve made of two smooth irreducible components $C_1$ and $C_2$ of genera, respectively, $h$ and $g-h$, meeting at a node, and containing the marked points $p_i$ such that, respectively, $i\in A$ or $i\in A^c$, and in such a way that if $(h,n)\neq (\frac{g}{2},0)$ then the tautological point lies on $C_1$;
    %$\{\O(C_{i,S}^+), \O(C_{i,S}^-)\}$, where $C_{i,S}^+$ and $C_{i,S}^-$ are the two divisors (with the identification $C_{g/2,\emptyset}^+=C_{g/2,\emptyset}^-$ if $n=0$ and $g$ is even) on $\Cbargn$ that lie over the boundary divisor $\Delta_{i,S}$ of $\Mbargn$;
\end{itemize}
subject to the following relations:
\begin{itemize}
    \item $\O(C_{(h,A)})+\O(C_{(h,A)^c})=0$ where $(h,A)^c:=(g-h,A^c)$; 
    \item if $g=1$ then $\omega_\pi=0$;
    \item if $g=0$ then $\Sigma_1=\ldots=\Sigma_n$ and $\omega_\pi=-2\Sigma_1$.
\end{itemize}
In particular, $\PicRel_{g,n}(\ZZ)$ is torsion-free unless $n=0$ and $g$ is even, in which case $\O(C_{(g/2,\emptyset)})$ is a $2$-torsion element that generates the torsion subgroup of $\PicRel_{g,n}(\ZZ)$. 

From the above description of $\PicRel_{g,n}(\ZZ)$, it follows that  
\begin{equation}\label{E:PicRelchi}
    \PicRel_{g,n}^\chi(\RR)=\left\{\beta\omega_{\pi}+\sum_{i=1}^n \alpha_i\Sigma_i+\sum_{(h,A)\in \Bgn}\gamma_{(h,A)}\O(C_{(h,A)})\: : (2g-2)\beta+\sum_{i=1}^n \alpha_i=\chi\right\}.
\end{equation}

\begin{fact} (\cite[Thm. 1]{Kass_2019})\label{F:RelPic}
For any $\chi\in \ZZ$, we have an isomorphism 
$$
\begin{aligned}
\un \deg:\PicRel^{\chi}_{g,n}(\RR)& \xrightarrow{\cong} \Pol^{\chi}_{g,n},\\
L & \mapsto \un \deg(L)=\{\un \deg(L)(X)\}
\end{aligned}
$$
where $\un \deg(L)(X)$ is the multidegree function of $L_{|X}$, i.e. it is the additive function on the subcurves of $X$ that sends a subcurve $Y\subseteq X$ to the degree $\deg_Y(L_{|X})$ on $Y$ of the line bundle $L_{|X}$ on $X$. 
\end{fact}

By composing the map in \eqref{E:map-ceil} with the isomorphism of Fact \ref{F:RelPic}, we get a map 
 \begin{equation}\label{E:map-s-Uni}
 \begin{aligned}
     \s_{-}: \PicRel^{\ZZ}_{g,n}(\RR)=\coprod_{\chi\in \ZZ}\PicRel^{\chi}_{g,n}(\RR) & \longrightarrow \coprod_{\chi \in \ZZ}\VStab^{\chi}_{g,n}=\VStab_{g,n}\\
   L & \mapsto \s_L:=\{\s_L(X)\},\\
 \end{aligned}
 \end{equation}
 %$$\sigma(\beta\omega_{\pi}+\sum_{i=1}^n \alpha_i\Sigma_i+\sum_{(h,A)\in \Bgn}\gamma_{(h,A)}\O(C_{(h,A)}))(e;h,A):=
%   \begin{cases}
 %      \beta(2h-2+1)+\sum_{i\in A} \alpha_i- \gamma_{(h,A)}+\gamma_{(h,A)^c} & \text{ if } e=1,\\
 %\beta(2h-2+e)+\sum_{i\in A} \alpha_i& \text{ if } e\geq 2.
  % \end{cases}$$
 where 
 $$
 \begin{aligned}
     \s_L(X): \BCon(X) & \rightarrow \ZZ,\\
     Y & \mapsto \s_L(X)_Y:=\lceil \deg_Y(L_{|X})\rceil.
 \end{aligned}
 $$
By composing the above map \eqref{E:map-s-Uni} with the isomorphism of Proposition \ref{P:Stabgn}, we obtain a map 
 \begin{equation}\label{E:map-sigma}
     \sigma_{-}: \PicRel^{\ZZ}_{g,n}(\RR)=\coprod_{\chi\in \ZZ}\PicRel^{\chi}_{g,n}(\RR)  \longrightarrow \coprod_{\chi \in \ZZ}\Sigma^{\chi}_{g,n}=\Sigma_{g,n},
 \end{equation}
which sends $L=\beta\omega_{\pi}+\sum_{i=1}^n \alpha_i\Sigma_i+\sum_{(h,A)}\gamma_{(h,A)}\O(C_{(h,A)})$ to 
 \begin{equation}\label{E:for-deg}
 \sigma_{L}(e;h,A):=
   \begin{cases}
       \lceil\beta(2h-2+1)+\sum_{i\in A} \alpha_i- \gamma_{(h,A)}+\gamma_{(h,A)^c}\rceil & \text{ if } e=1,\\
 \lceil\beta(2h-2+e)+\sum_{i\in A} \alpha_i\rceil& \text{ if } e\geq 2.
   \end{cases}
   \end{equation}

% $$\sigma\left(\beta\omega_{\pi}+\sum_{i=1}^n \alpha_i\Sigma_i+\sum_{(h,A)}\gamma_{(h,A)}\O(C_{(h,A)})\right)(e;h,A):=
%   \begin{cases}
%       \beta(2h-2+1)+\sum_{i\in A} \alpha_i- \gamma_{(h,A)}+\gamma_{(h,A)^c} & \text{ if } e=1,\\
% \beta(2h-2+e)+\sum_{i\in A} \alpha_i& \text{ if } e\geq 2.
%   \end{cases}$$

 In order to describe the fibers of the above map $\sigma_-$, consider the following arrangement of hyperplanes in $\PicRel^{\ZZ}_{g,n}(\RR)$:
 \begin{equation}\label{E:Agn}
\begin{aligned} 
\A_{g,n}:=& \bigcup_{\substack{(1;h,A)\in \Dgn \\ k\in \ZZ}}\Bigg\{(2h-2+1)\omega_{\pi}^\vee+\sum_{i\in A}\Sigma_i^\vee+\O(C_{(h,A)})^\vee=k\Bigg\}\\
& \bigcup_{\substack{(e;h,A)\in \Dgn \text{ with } e\geq 2\\ m\in \ZZ}}\Bigg\{(2h-2+e)\omega_{\pi}^\vee+\sum_{i\in A}\Sigma_i^\vee=m\Bigg\},
\end{aligned} 
\end{equation}
where $(-)^\vee\in \PicRel_{g,n}^\Z(\R)^\vee$ denotes the functional dual to a certain element.
And, for any $\chi \in \ZZ$, denote by $\A_{g,n}^{\chi}$ the restriction of $\A_{g,n}$ to $\PicRel^{\chi}_{g,n}(\R)$.

 \begin{lemma}\label{L:arr-Uni} 
 Let $L,L'\in \PicRel^{\ZZ}_{g,n}(\RR)$. Then $\sigma_L=\sigma_{L'}$ if and only if $L$ and $L'$ belong to the same region of $\PicRel^{\ZZ}_{g,n}(\RR)$ with respect to the hyperplane arrangement $\A_{g,n}$. 
 Hence, the map $\sigma_{-}$ induces an order-preserving embedding 
 $$
 \sigma_{-}: \left\{
   \begin{aligned}
   & \text{Regions of } \PicRel_{g,n}^{\ZZ}(\RR)\\
   &\text{ with respect to } \A_{g,n}
   \end{aligned}\right\} \hookrightarrow \Sigma_{g,n}.
 $$
 Moreover, $\sigma_L$ is general if and only if $L$ belongs to a chamber, i.e. a maximal dimensional region.
 \end{lemma}
This is a restatement of \cite[Thm. 2]{Kass_2019} in our language. The V-functions of type $(g,n)$ (resp. the universal V-stabilities of type $(g,n)$) that are in the image of the map $\sigma_-$ (resp. of the map $\s_-$) are called \emph{classical}.
\begin{proof}
This is immediate from Formula~\eqref{E:for-deg}.
\end{proof}

We can now state the classification result of classical compactified Jacobians over $\Mbargn$.

\begin{theorem}\label{T:cla-cJUni}
\noindent 
\begin{enumerate}
\item \label{T:cla-cJUni1}  We have an order-reversing injection of posets  
  $$
  \begin{aligned}
   \left\{
   \begin{aligned}
   & \text{Regions of } \PicRel_{g,n}^{\ZZ}(\RR)\\
   &\text{ with respect to } \A_{g,n}
   \end{aligned}\right\}  & \hookrightarrow
   \left\{\text{Compactified universal Jacobian stacks over } \Mbargn \right\}\\
   [L] &\mapsto \ov \J_{g,n}([L]):=
   \left\{
\begin{aligned} 
& I \in  \TF_{g,n}^{\deg_{\pi}(L)}: \: \chi((I_{|X})_Y)\geq \deg_Y(L_{|X}) \\ 
& \text{ for any } X\in \Mbargn \text{ and any } Y\in \BCon(X)
\end{aligned} 
\right\}.
  \end{aligned}
  $$
  The compactified universal Jacobian stacks of the form $\ov \J_{g,n}([L])$ are called \textbf{classical compactified universal Jacobian stacks} of type $(g,n)$, and their associated good moduli spaces, denoted by $\ov J_{g,n}([L])$, are called \textbf{classical compactified universal Jacobian spaces}. 
 \item \label{T:cla-cJUni2}  Any classical compactified universal Jacobian space $\ov J_{g,n}([L])$ is  locally projective over $\Mbargn$. 
  \end{enumerate}
\end{theorem}
Part \eqref{T:cla-cJUni1} of the above Theorem follows from the results of \cite[Sec. 4, 5]{Kass_2019} in the case of fine classical compactified Jacobians, which correspond to the chambers of $\PicRel_{g,n}^{\ZZ}(\RR)$  with respect to $\A_{g,n}$. See also \cite{melo2019} for another construction of the classical fine compactified universal Jacobians over $\Mbargn$. Part~\eqref{T:cla-cJUni2} was claimed\footnote{However, Sam Molcho pointed out to us that the proof of \cite[Cor.~4.8]{melo2019} contains a flaw (since \cite[Prop.~6.8]{Kol} does not work for families of non-irreducibile curves) and, even more, that the line bundle $\det \mathcal Q$ in \cite[p. 19]{melo2019} cannot, in general, be relatively ample on $\ov J_{g,n}([L])/\Mbargn$.} in \cite[Thm. C(iii)]{melo2019} for fine classical compactified Jacobian spaces. 
\begin{proof}
Part \eqref{T:cla-cJUni1} follows by combining Lemma \ref{L:arr-Uni} and Theorem \ref{T:cJUniv}.

Let us prove Part \eqref{T:cla-cJUni2}. First of all, since the arrangement of hyperplanes $\A_{g,n}^\chi$ is rational, we can perturb $L\in \PicRel^{\chi}(\RR)$ to a rational relative line bundle in such a way that the perturbed line bundle lies inside the same region. In other words, we can assume that $L$ is a rational relative line bundle on $\Cbargn/\Mbargn$. Then we can write
$$
L=\frac{M}{r} \text{ for some } r\in \NN_{>0} \text{ and } M\in \PicRel(\ZZ) \text{ with } \deg_{\pi}(M)=r\chi. 
$$
Consider now the following vector bundle on $\Cbargn/\Mbargn$
$$
E=\O_{\Cbargn}^{\oplus(r-1)}\oplus M^{-1},
$$
which has relative slope equal to $-\chi$.
By construction, we have that $\un \deg(L)=\psi(E)\in \Pol_{g,n}^\chi$, where $\psi(E)$ is the Esteves \cite{esteves} relative numerical polarization on $\Cbargn/\Mbargn$ (see \cite[Example 4.27(2)]{FPV1}. This implies (using the notation of \cite[Example 6.7(2)]{FPV1}) that 
$$\ov \J_{g,n}([L])= \ov \J_{\Cbargn/\Mbargn}(E) \text{ and } \ov J_{g,n}([L])= \ov J_{\Cbargn/\Mbargn}(E).$$
Then we conclude by applying \cite[Thm. 7.1]{FPV1} to an \'etale cover of $\Mbargn$ by an algebraic space.
\end{proof}

Regarding the image of the map in \eqref{E:map-sigma}, we have the following

\begin{fact}(\cite[Thm. 3.9]{fava2024})
 The map $\sigma_-$ of \eqref{E:map-sigma} is not surjective on general universal stabilities precisely in the following cases:
 \begin{enumerate}
     \item $g\geq 4$ and $n\geq 1$;
     \item $g=3$ and $n\geq 2$;
     \item $g=2$ and $n\geq 4$;
     \item $g=1$ and $n\geq 6$.
 \end{enumerate}
 In particular, in the above ranges there are non classical fine compactified Jacobians. 
\end{fact}

Indeed, it turns out that the map $\sigma_-$ is surjective (even on non-general universal stabilities) when we are not in the above range for $(g,n)$: the case $n=0$ follows from Proposition~\ref{P:Dsep}; the case $g=0$ follows from Proposition~\ref{P:Dsep} and Theorem~\ref{T:class-n0}; the case $g=1$ and $1\leq n\leq 5$ follows from Proposition~\ref{P:Dsep}, Remark~\ref{R:g1-Kn} and the easy verification that on any nodal curve whose dual graph is the complete graph $K_n$ then any V-stability condition is classical if and only if $n\leq 5$; the case $g=2$ and $1\leq n\leq 3$ can be handled as in \cite[Prop.~3.13]{fava2024}; the case $g=3$ and $n=1$ can be handled as in \cite[Prop.~3.15]{fava2024}.

\section{Isomorphisms among compactified universal Jacobians}\label{Sec:equiv-cUJ}

The aim of this section is to study when two compactified universal Jacobian stacks/spaces are isomorphic over $\Mbargn$, extending the results of Kass-Pagani \cite[Sec. 6.2]{Kass_2019} from fine classical compactified universal Jacobians to arbitrary compactified universal Jacobians. 

An important role is played by the following group (see \cite[Def. 6.11]{Kass_2019})
\begin{equation}\label{E:PR}
    \wPR_{g,n}:=\PicRel_{g,n}(\ZZ)\rtimes (\ZZ/2\ZZ),
\end{equation}
where $\PicRel_{g,n}(\ZZ)$ is the integral relative Picard group of $\Cbargn/\Mbargn$ and the action of $ \ZZ/2\ZZ$ on $\PicRel_{g,n}(\ZZ)$ is via the inverse map. 

We now define an action of $\wPR_{g,n}$ on the space of universal V-stabilities of type $(g,n)$.

\begin{lemma-definition}\label{LD:PR-VStab}
The group $\wPR_{g,n}$ acts on $\VStab_{g,n}$ as follows:
\begin{enumerate}[(i)]
    \item  $L\in \PicRel_{g,n}(\ZZ)$ acts by sending $\s\in \VStab^{\chi}_{g,n}$ into 
    $$
    L\cdot \s:=\{(L\cdot \s)(X): \BCon(X)\ni Y\mapsto \s(X)_Y+\deg_Y(L_{|X})\}_{X\in \Mbargn} \in \VStab_{g,n}^{\chi+\deg_{\pi}(L)}.
    $$
    \item The generator $\iota$ of $\ZZ/2\ZZ$ acts by sending $\s\in \VStab^{\chi}_{g,n}$ into 
   $$
    \iota\cdot \s=\left\{(\iota\cdot \s)(X): \BCon(X)\ni Y\mapsto \begin{cases}
    -\s(X)_Y& \text{ if } Y\in \D(\s(X))\\
    -\s(X)_Y+1 & \text{ if } Y\not \in \D(\s(X))
    \end{cases}
    \right\}_{X\in \Mbargn} \in \VStab_{g,n}^{-\chi}.
    $$ 
\end{enumerate}
Moreover, the action preserves the degeneracy subset and it is compatible with the poset structure on $\VStab_{g,n}$ (see Remark \ref{R:Sigma-pos}).
\end{lemma-definition}
\begin{proof}
    This is straightforward using Definition \ref{D:VStab-Uni}.
\end{proof}

\begin{remark}\label{R:PR-Pol}
\noindent
\begin{enumerate}[(i)]
\item In terms of the isomorphism of Proposition \ref{P:Stabgn}, the action of $\wPR_{g,n}$ on $\Sigma_{g,n}$ is given by sending $\sigma \in \Sigma_{g,n}$ into 
 $$
    \begin{sis}
    & (L\cdot \sigma)(e;h,A):=\sigma(e;h,A)+
    \begin{cases}
 \beta(2h-2+1)+\sum_{i\in A} \alpha_i- \gamma_{(h,A)}+\gamma_{(h,A)^c} & \text{ if } e=1,\\
 \beta(2h-2+e)+\sum_{i\in A} \alpha_i& \text{ if } e\geq 2.
\end{cases}\\
%&  \text{ for any } L=\beta\omega_{\pi}+\sum_{i=1}^n \alpha_i\Sigma_i+\sum_{(h,A)\in \Bgn}\gamma_{(h,A)}\O(C(h,A))\in \PicRel_{g,n}(\ZZ),\\
     & (\iota \cdot \sigma)(e;h,A):=
     \begin{cases}
      -\sigma(e;h,A) & \text{ if } (e;h,A)\in \D(\sigma), \\
-\sigma(e;h,A)+1 & \text{ if } (e;h,A)\not \in \D(\sigma),    
     \end{cases}\\
    \end{sis}
    $$
where $L=\beta\omega_{\pi}+\sum_{i=1}^n \alpha_i\Sigma_i+\sum_{(h,A)\in \Bgn}\gamma_{(h,A)}\O(C(h,A))\in \PicRel_{g,n}(\ZZ)$. 
\item The group $\wPR_{g,n}$ acts (see \cite[Eq. (43)]{Kass_2019}) on the space $\Pol_{g,n}$ of universal numerical polarizations of type $(g,n)$ by sending $\psi\in \Pol^{\chi}_{g,n}$ into
    $$
    \begin{sis}
    & L\cdot \psi:=\{(L\cdot \psi)(X): Y\mapsto \psi(X)_Y+\deg_Y(L_{|X})  \}_{X\in \Mbargn}\in \Pol_{g,n}^{\chi+\deg_{\pi}(L)}\\
     & \iota \cdot \psi:=\{(\iota\cdot \psi)(X): Y\mapsto -\psi(X)_Y  \}_{X\in \Mbargn}\in \Pol_{g,n}^{-\chi}\\
    \end{sis}
    $$
    In terms of the isomorphism $\un{\deg}$ of Fact \ref{F:RelPic}, the action of $\wPR_{g,n}$ is such that the elements of $\PicRel_{g,n}(\ZZ)$ act on $\PicRel^{\ZZ}_{g,n}(\RR)$ via translation, and $\iota$ acts as the inverse.
 \item    
    The map $\lceil -\rceil$ of \eqref{E:map-ceil}, or equivalently the map $\s_-$ of \eqref{E:map-s-Uni},  is $\wPR$-equivariant.  
    \end{enumerate}
\end{remark}

The group acts on the stack $\TF_{g,n}$ of relative rank-$1$ torsion-free sheaves on $\Cbargn/\Mbargn$ by permuting compactified universal Jacobian stacks.

\begin{proposition}\label{P:PR-TF}
  The group $\wPR_{g,n}$ acts on the stack $\TF_{g,n}$ by sending a relative torsion-free rank-one sheaf $\I$ into
  $$
  \begin{aligned}
   & L\cdot \I:=\I\otimes L,\\
   & \iota \cdot \I:=\I^*:={\mathcal Hom}(\I,\omega_{\Cbargn/\Mbargn}).
   %\I^{\vee}\otimes \omega_{\Cbargn/\Mbargn}:={\mathcal Hom}(\I,\O)\otimes \omega_{\Cbargn/\Mbargn}.
  \end{aligned}
  $$
  Moreover, the action of $\wPR_{g,n}$ permutes the compactified universal Jacobian stacks in such a way that the bijection of Theorem \ref{T:cJUniv} is $\wPR_{g,n}$-equivariant. 
\end{proposition}
\begin{proof}
The action of $\wPR_{g,n}$ on $\TF_{g,n}$ is well-defined:  $L\cdot \I$ is well-defined since $L$ is a relative line bundle on $\Cbargn/\Mbargn$, $\iota\cdot \I$ is well-defined as observed in \cite[Lemma 7.4]{Kass_2019} and clearly we have that $\iota\cdot (L\cdot \I)=L^{-1}\cdot (\iota\cdot \I)$.

In order to prove the second statement, we have to prove that for any $\s\in \VStab_{g,n}$ we have that 
$$
L\cdot \ov \J_{g,n}(\s)\subseteq \ov \J_{g,n}(L\cdot \s) \text{ and } \iota\cdot \ov \J_{g,n}(\s)\subseteq \ov \J_{g,n}(\iota\cdot \s),
$$
for any $L\in \PicRel_{g,n}(\ZZ)$. This amounts to proving that, for any given $X\in \Mbargn$ and any $I\in \TF_X$ which is $\s(X)$-semistable, we have that
\begin{enumerate}[(i)]
    \item \label{E:ss-L} $I\otimes L_{|X}$ is $(L\cdot \s)(X)$-semistable, for any $L\in \PicRel_{g,n}(\ZZ)$;
    \item \label{E:ss-iota} $I^*={\mathcal Hom}(I,\omega_X)$ is $(\iota \cdot \s)(X)$-semistable. 
\end{enumerate}
Property \eqref{E:ss-L} follows from the definition of $(L\cdot \s)(X)$ (see Lemma-Definition \ref{LD:PR-VStab}) together with the fact that 
$$
\chi((I\otimes L_{|X})_Y)=\chi(I_Y)+\deg_Y(L_{|X}) \text{ for any } Y\in \BCon(X).$$

Property \eqref{E:ss-iota} follows from the following computation for any $Y\in \BCon(X)$:
$$
\begin{aligned}
&     \chi((I^*)_Y)=\chi((I_Y)^*)+|Y\cap Y^c\cap \NF(I)^c|  & \text{ by Lemma \ref{L:duality}}\\
&     =-\chi(I_Y)+|Y\cap Y^c\cap \NF(I)^c|  & \text{ by Serre duality,}\\
& = \chi(I_{Y^c})-\chi(I) & \text{ by \eqref{E:2restY} and \eqref{E:add-chi},}\\
& \geq \s(X)_{Y^c}-|\s(X)| & \text{ since $I$ is $\s(X)$-semistable,}\\
&=\begin{cases}
   -\s(X)_Y & \text{ if } Y\in \D(\s(X)),\\
   -s(X)_Y+1 & \text{ if } Y\not \in \D(\s(X)).
   \end{cases}
   & \text{ by Definition \ref{D:VStabX},}\\
 & =(\iota\cdot \s(X))_Y & \text{ by Lemma-Definition \ref{LD:PR-VStab}.}  
\end{aligned}
$$
\end{proof}

\begin{lemma}\label{L:duality}
Let $X$ be a nodal curve.  Let $I$ be a rank-$1$ torsion-free sheaf on $X$ and consider the sheaf $I^*:={\mathcal Hom}(I,\omega_X)$. Then, for any subcurve $Y\subset X$, we have that 
\begin{equation}\label{E:dual-res}
(I_Y)^*=(I^*)_Y(-Y\cap Y^c\cap \NF(I)^c)=\leftindex_{Y}(I^*),
\end{equation}
where $(I_Y)^*={\mathcal Hom}(I_Y,\omega_Y)$ and $\NF(I)$ is the set of nodes of $X$ at which $I$ is not locally free. 

In particular, we have that 
\begin{equation}\label{E:chi-dualY}
    \chi((I^*)_Y)=\chi((I_Y)^*)+|Y\cap Y^c\cap \NF(I)^c|.
\end{equation}
\end{lemma}
\begin{proof}
Equality \eqref{E:chi-dualY} follows from \eqref{E:dual-res} by taking the Euler characteristic. 

In order to prove \eqref{E:dual-res}, consider the two exact sequences
\begin{equation}\label{E:2seqex}
   \begin{aligned}
     & 0 \to \omega_Y\xrightarrow{a} \omega_X\xrightarrow{b} \omega_{Y^c}(Y\cap Y^c)\to 0,\\
     & 0 \to I_{Y^c}(-Y\cap Y^c\cap \NF(I)^c)\xrightarrow{\beta} I \xrightarrow{\alpha} I_Y\to 0,
   \end{aligned} 
\end{equation}
where the first exact sequence follows by adjunction and the second exact sequence follows from \eqref{E:2restY}.  From a straightforward  diagram chasing with the two exact sequences in \eqref{E:2seqex}, we get  following  exact sequence
\begin{equation}\label{E:leftex}
    \begin{aligned}
       0 \to \cHom(I_Y,\omega_Y) \rightarrow &  \cHom(I,\omega_X) \xrightarrow{}  \cHom(I_{Y^c}(-Y\cap Y^c\cap \NF(I)^c), \omega_X)\oplus \cHom(I,\omega_{Y^c}(Y\cap Y_c))\\
      f \mapsto &   a\circ f\circ \alpha & \\
         & \hspace{1.8cm} g   \mapsto  (g\circ \beta,b\circ g)\\
    \end{aligned}
\end{equation}
This shows that $(I_Y)^*$ is a subsheaf of $I^*$ (supported on $Y$) such that the quotient $I^*/(I_Y)^*$ is a torsion-free sheaf supported on $Y^c$, and hence that it is the largest subsheaf of $I^*$ supported on $Y$, which implies that $(I_Y)^*=\leftindex_{Y}{(I^*)}$. Combining this with \eqref{E:2restY}, we deduce that 
$$
(I_Y)^*=\leftindex_{Y}{(I^*)}=(I^*)_Y(-Y\cap Y^c\cap \NF(I)^c),
$$
as required.   
\end{proof}

\begin{remark}\label{R:PR-class}
 It follows from Proposition \ref{P:PR-TF} and Remark \ref{R:PR-Pol} that the action of $\wPR_{g,n}$ on the set of compactified universal Jacobians preserves the classical compactified universal Jacobians. More specifically, we have that (see \cite[Lemma 6.12]{Kass_2019})
 $$
L\cdot \ov \J_{g,n}(\lceil \psi\rceil)=\ov \J_{g,n}(\lceil L\cdot \psi\rceil) \text{ for any } L\in \PicRel_{g,n}(\ZZ) \text{ and } \iota\cdot \ov \J_{g,n}(\lceil \psi\rceil)= \ov \J_{g,n}(\lceil \iota\cdot \psi\rceil).
$$
\end{remark}

The following Theorem characterizes the compactified universal Jacobian stacks that are isomorphic over $\Mbargn$, or more generally those having a birational morphism over $\Mbargn$.

\begin{theorem}\label{T:bir-UniSt}
 Let $\s_1,\s_2\in \VStab_{g,n}$.  We have that 
$$
\begin{aligned}
& \text{There exists a birational morphism }\\ 
& \phi:\ov \J_{g,n}(\s_1)\to \ov \J_{g,n}(\s_2) \text{ over } \Mbargn
\end{aligned} \Longleftrightarrow 
\begin{aligned}
& \text{There exists } g\in \wt \PR_{g,n} \text{ such that }\\ 
& \s_1\geq g\cdot \s_2. 
\end{aligned}
$$
\end{theorem}
%This follows from \cite[Lemma 6.13]{Kass_2019} in the case of classical fine compactified universal Jacobians and we use ideas from the proof of loc. cit. to treat the more general setting. 
\begin{proof}
% This was proved in \cite[Lemma 6.13]{Kass_2019} for classical fine compactified universal Jacobians, but the same proof carries over to our more general setting as we now sketch. 
The implication $\Longleftarrow$ follows from Proposition~\ref{P:PR-TF} and Theorem \ref{T:cJUniv}.

In order to show the other implication, fix a birational morphism $\phi:\ov \J_{g,n}(\s_1)\rightarrow\ov \J_{g,n}(\s_2)$ over $\Mbargn$. The proof consists of two steps.

 \un{Step I:} There exists an element $g\in \wt \PR_{g,n}$ such that the composition 
 $$
 \psi: \ov \J_{g,n}(\s_1)\xrightarrow{\phi}\ov \J_{g,n}(\s_2)\xrightarrow[\cong]{\cdot g}\ov \J_{g,n}(g\cdot\s_2)
 $$
 restricts to an inclusion $\ov \J_{g,n}(\s_1)^{\leq 1}\subseteq \ov \J_{g,n}(g\cdot \s_2)^{\leq 1}$ over the open substack $\Mbargn^{\leq 1}$ parametrizing stable curves with at most one node, where we set $\ov \J_{g,n}(\s)^{\leq 1}:=\ov \J_{g,n}(\s)_{|\Mbargn^{\leq 1}}$

 Indeed, by post-composing with the multiplication by an element of the form $(\beta\omega_{\pi}+\sum_{i=1}^n \alpha_i\Sigma_i,\iota^{e})\in \wt \PR_{g,n}$ (with $e=0,1$), we can assume that the morphism becomes the identity on the universal Jacobian over smooth curves $\J_{g,n}$ (using that any automorphism of $\J_{g,n}$ over $\M_{g,n}$ is the multiplication by one such element, see \cite[Cor. 6.1]{Kass_2019}). Then, by further post-componing with an element of the form $\sum_{(h,A)\in \Bgn}\gamma_{(h,A)}\O(C(h,A))\in \PicRel_{g,n}(\ZZ)$, we can assume that the morphism becomes an inclusion over the generic point of each boundary divisor of $\Mbargn$ (using the structure of compactified Jacobian stacks of nodal curves with at most one node, see e.g. \cite[Prop. 5.8]{FPV2}).   

\un{Step II:} The morphism $\psi$ is the inclusion $\ov \J_{g,n}(\s_1)\subseteq \ov \J_{g,n}(g\cdot \s_2)$ (and hence $\s_1\geq g\cdot \s_2$ by Theorem~\ref{T:cJUniv}). 
 
 Indeed, by Proposition~\ref{P:PR-TF} and \cite[Remark~6.5, Part~(i)]{FPV1}  it is enough to show that $\psi$ commutes with the open embeddings of the domain and of the codomain in $\TF_{g,n}$. In order to show this, consider the universal sheaf $\I_1$ (resp. $\I_2$) on $\ov \J_{g,n}(\s_1)\times_{\Mbargn} \Cbargn$ (resp. on $\ov \J_{g,n}(g\cdot \s_2)\times_{\Mbargn} \Cbargn$), which is the restriction from $\TF_{g,n}\times_{\Mbargn}\Cbargn$ of the universal sheaf. The fact that $\psi$ commutes with the open embeddings of the domain and of the codomain in $\TF_{g,n}$ is equivalent to the equality of sheaves 
\begin{equation}\label{E:eq-shea}
    (\psi\times \id)^*(\I_2)=\I_1.
\end{equation}
By Step I, we have that 
\begin{equation}\label{E:eq-open}
    (\psi\times \id)^*(\I_2)_{|(\ov \J_{g,n}(\s_1)\times_{\Mbargn} \Cbargn)^{\leq 1}}=(\I_1)_{|(\ov \J_{g,n}(\s_1)\times_{\Mbargn} \Cbargn)^{\leq 1}},
\end{equation}
 where $(\ov \J_{g,n}(\s_1)\times_{\Mbargn} \Cbargn)^{\leq 1}=\ov \J_{g,n}(\s_1)^{\leq 1}\times_{\Mbargn^{\leq 1}} \Cbargn^{\leq 1}$ is the open substack obtained by restricting to $\Mbargn^{\leq 1}$. 
Now, the equality \eqref{E:eq-open} implies the equality \eqref{E:eq-shea} by \cite[Cor. 7.2]{Kass_2019}, using that $\ov \J_{g,n}(\s_1)\subset \TF_{g,n}$ is regular, that $\ov \J_{g,n}(\s_1)\times_{\Mbargn} \Cbargn\subset \TF_{g,n}\times_{\Mbargn}\Cbargn$ is Cohen-Macaulay and regular in codimension one (by the same proof of \cite[Lemma 7.3]{Kass_2019}) and that $\ov \J_{g,n}(\s_1)^{\leq 1}\subset \ov \J_{g,n}$ is a big open subset (i.e. its complement has codimension at least two).
\end{proof}

\begin{corollary}\label{C:iso-UniSt}
 Let $\sigma_1,\sigma_2\in \Sigma_{g,n}$. Then the following conditions are equivalent:
 \begin{enumerate}[(i)]
 \item \label{C:iso-UniSt1} $\sigma_1$ and $\sigma_2$ lie in the same orbit for the action of $\wPR_{g,n}$ on $\Sigma_{g,n}$.
     \item \label{C:iso-UniSt2} $\ov \J_{g,n}(\sigma_1)$ and $\ov \J_{g,n}(\sigma_2)$ are isomorphic over $\Mbargn$.
     \item \label{C:iso-UniSt3} There exists a birational morphism $\phi:\ov \J_{g,n}(\sigma_1)\to \ov \J_{g,n}(\sigma_2)$ over $\Mbargn$ and $\D(\sigma_1)=\D(\sigma_2)$.
 \end{enumerate}

\end{corollary}
This was proved in \cite[Lemma 6.13]{Kass_2019} for classical fine compactified universal Jacobians. 
\begin{proof}
We will use throughout this proof the $\wt \PR_{g,n}$-equivariant isomorphism of poset $\Sigma_{g,n}\cong \VStab_{g,n}$ of Proposition \ref{P:Stabgn}. We will prove a chain of implications. 

$\eqref{C:iso-UniSt1}\Rightarrow \eqref{C:iso-UniSt2}$ follows from Proposition \ref{P:PR-TF}.

$\eqref{C:iso-UniSt2}\Rightarrow \eqref{C:iso-UniSt3}$: we have only to show that if $\ov \J_{g,n}(\sigma_1)$ and $\ov \J_{g,n}(\sigma_2)$ are isomorphic over $\Mbargn$ then $\D(\sigma_1)=\D(\sigma_2)$. This follows from the fact that $(e;h,A)\in \D(\sigma_i)$ if and only if the fiber of $\ov \J_{g,n}(\sigma_i)$ over the geometric generic point of the vine stratum corresponding to $(e;h,A)$ (see Remark \ref{R:vine-trian}) is non-fine, which indeed occurs if and only if this fiber has $(e+1)$-irreducible components (see \cite[Sec. 8]{Oda1979CompactificationsOT}). 

$\eqref{C:iso-UniSt3}\Rightarrow \eqref{C:iso-UniSt1}$: the existence of the birational morphism $\phi$ as in \eqref{C:iso-UniSt3} implies, by Theorem \ref{T:bir-UniSt}, that there exists $g\in \wt \PR_{g,n}$ such that $\sigma_1\geq g\cdot \sigma_2$. Since we have that $\D(\sigma_1)=\D(\sigma_2)=\D(g\cdot \sigma_2)$ by assumption and Proposition \ref{P:Deg-map}\eqref{P:Deg-map1}, Remark \ref{R:Sigma-pos} implies that $\sigma_1=g\cdot \sigma_2$. 
\end{proof}

We now want to investigate when two compactified universal Jacobian spaces are isomorphic over $\Mbargn$. With this aim, we will decompose the poset $\Sigma_{g,n}$ of V-functions of type $(g,n)$ into a separating and a non-separating part. 

First of all, we can partition the stability domain \eqref{E:Dgn} into a \emph{separating} and a \emph{non-separating} domain 
\begin{equation}\label{E:Dgn-ns}
  \begin{aligned}
& \Dgns:=\{(1;h,A): (1; h,A)\in \Dgn\}=\{(1;h,A): (h,A)\in \Bgn\}, \\
&  \Dgnns:= \{(e;h,A): (e; h,A)\in \Dgn \text{ and } e\geq 2\}.
  \end{aligned}  
\end{equation}
Note that we have a partition
\begin{equation}\label{E:Dgn-s-ns}
\Dgn=\Dgns\bigsqcup \Dgnns,
\end{equation}
which is stable under the complement operation $(-)\mapsto (-)^c$. Moreover, each triangle in $\Dgn$ is  contained in $\Dgnns$.

We can now define analogues of the V-functions for $\Dgns$ and $\Dgnns$.

\begin{definition}\label{D:Sigmagn2}
\noindent 
\begin{enumerate}
    \item 
   Denote by $\Sigmas_{g,n}^\chi$ the set of all functions (called \emph{separating V-functions} of type $(g,n)$)
    \begin{align*}
        \sigma:\Dgns&\to \ZZ\\
        (1;h,A)&\mapsto \sigma(1;h,A)
    \end{align*}
    satisfying Property \eqref{E:condUni1} of Definition \ref{D:Sigmagn}, endowed with the following poset structure 
    $$
    \sigma_1\geq \sigma_2 \Leftrightarrow  \sigma_1(1;h,A)\geq \sigma_2(1;h,A) \text{ for any } (1;h,A)\in \Dgns.
    $$
    The (separating) degeneracy subset of $\sigma\in \Sigmas_{g,n}^{\chi}$ is defined by 
    $$
    \Ds(\sigma):=\{(1;h,A)\: : \sigma(1;h,A)+\sigma((1;h,A)^c)=\chi\}\subseteq \Dgns.
    $$
   \item  
   Denote by $\Sigmans^\chi_{g,n}$ the set of all functions (called \emph{non-separating V-functions} of type $(g,n)$)
    \begin{align*}
        \sigma:\Dgnns&\to \ZZ\\
        (e;h,A)&\mapsto \sigma(e;h,A)
    \end{align*} 
     satisfying Properties \eqref{E:condUni1} and \eqref{E:condUni2} of Definition \ref{D:Sigmagn}, endowed with the following poset structure 
    $$
    \sigma_1\geq \sigma_2 \Leftrightarrow  \sigma_1(e;h,A)\geq \sigma_2(e;h,A) \text{ for any } (e;h,A)\in \Dgnns.
    $$ 
    The (non-separating) degeneracy subset of $\sigma\in \Sigmans_{g,n}^{\chi}$ is defined by 
    $$
    \Dns(\sigma):=\{(e;h,A)\in \Dgnns\: : \sigma(1;h,A)+\sigma((1;h,A)^c)=\chi\}\subseteq \Dgnns.
    $$
\end{enumerate}
%In both cases, we can define the degeneracy subset $\D(\sigma)$ as in \eqref{E:degsigma}. 

Moreover, in both cases, we say that $\sigma$ has Euler characteristic $\chi:=|\sigma|$ and we set 
$$\Sigmas_{g,n}:=\coprod_{\chi \in \ZZ}{} \Sigmas_{g,n}^\chi \quad \text{ and } \quad \Sigmans_{g,n}:=\coprod_{\chi \in \ZZ}{} \Sigmans_{g,n}^\chi.$$
\end{definition}

\begin{lemma}\label{L:Sigma-s-ns}
For any $\chi \in \ZZ$, we have an isomorphism of posets
$$
\begin{aligned}
    \Sigma_{g,n}^\chi & \xrightarrow{\cong}  {}\Sigmas_{g,n}^\chi\times \Sigmans_{g,n}^\chi\\
    \sigma&\mapsto (\sigma^s,\sigma^{ns}),
\end{aligned}
$$
where $\sigma^s:=\sigma_{|\Dgns}$  and $\sigma^{ns}:=\sigma_{|\Dgnns}$.
Moreover, $\D(\sigma)=\Ds(\sigma^s)\bigsqcup \Dns(\sigma^{ns})$.
\end{lemma}
We will call $\sigma^s$ (resp. $\sigma^{ns}$) the separating (resp. non-separating) component of $\sigma\in \Sigma_{g,n}$.
\begin{proof}
This follows immediately from the fact that all the triangles of $\Dgn$ are contained in $\Dgnns$.    
\end{proof}

\begin{remark}\label{R:class-s-ns}
Recall that an element  $\sigma\in \Sigma_{g,n}^\chi$ is classical if $\sigma=\sigma_L$ for some $L\in \PicRel^\chi_{g,n}(\R)$, where we have used the map in \eqref{E:map-sigma}. 
Similarly, we say that an element of $\sigma\in \Sigmans^{\chi}_{g,n}$ (resp. $\Sigmas^{\chi}_{g,n}$) is classical if it is the restriction of an element of $\Sigma^{\chi}_{g,n}$ that is classical.  From the explicit formula \eqref{E:for-deg}, it follows that 
 \begin{enumerate}[(i)]
     \item Every element of $\Sigmas^\chi_{g,n}$ is classical.
     \item The classical elements of $\Sigmans^\chi_{g,n}$ are those given by 
     \begin{equation}\label{E:sigma-cl}
         \sigma^\chi_g[\un\alpha](e;h,A):=
         \begin{sis}
        \left\lceil\Big(\chi-\sum_{i=1}^n\alpha_i\Big)\frac{2h-2+e}{2g-2}+\sum_{i\in A}\alpha_i\right\rceil & \quad \text{ if } g\geq 2,\\
         \left\lceil\sum_{i\in A}\alpha_i\right\rceil & \quad \text{ if } g=1.\\
        \end{sis}
     \end{equation}
     for some $\un \alpha:=(\alpha_1,\ldots,\alpha_n)\in \R^n$ (resp. such that $\sum_i \alpha_i=\chi$ if $g=1$). 

     Moreover, we have that $(e;h,A) \in \D(\sigma_g^\chi[\un\alpha])$ if and only if $e \geq 2$ and the function of $(e;h,A) $ inside the ceiling function $\lceil \cdot \rceil$ on the RHS is an integer.
 \end{enumerate}
\end{remark}

We now clarify the geometric meaning of the above decomposition. 
Note that we have a bijection 
$$
\begin{aligned}
\{\text{Pairs of complementary elements of $\Bgn$}\} & \xrightarrow{\cong} \{\text{Non-irreducible boundary divisors of $\Mbargn$}\}\\
\{(h,A), (h,A)^c\} & \mapsto \Delta(h,A)=\Delta((h,A)^c). 
\end{aligned}
$$
Consider the  open substack of $\Mbargn$
$$
\Mbargn^{ns}:=\Mbargn - \bigcup_{(h,A)\in \Bgn} \Delta(h,A)
$$
parameterizing stable $n$-marked curves of genus $g$ all of whose nodes are non-separating. 

\begin{lemma}\label{L:ugua-ns}
    Let $\sigma_1,\sigma_2\in \Sigma_{g,n}$. Then 
    $$
    \ov \J_{g,n}(\sigma_1)_{|\Mbargn^{ns}}=\ov \J_{g,n}(\sigma_2)_{|\Mbargn^{ns}}\Leftrightarrow \sigma_1^{ns}=\sigma_2^{ns}.
    $$
\end{lemma}
\begin{proof}
 This follows from the fact that 
 $$X\in \Mbargn^{ns}\Leftrightarrow \type_X(Y)\in \Dgnns \text{ for any } Y\in \BCon(X),$$
 together with \cite[Lemma 8.10]{FPV1}. 
\end{proof}

The action of $\wPR_{g,n}$ also simplifies according to the decomposition of Lemma \ref{L:Sigma-s-ns}. Consider the normal subgroup
$$
W_{g,n}:=\langle \O(C_{h,A}) : (h,A)\in \Bgn\rangle\unlhd \wPR_{g,n},
$$
and the quotient 
$$
\PR_{g,n}:=\wPR_{g,n}/W_{g,n}.
$$
Explicitly, the quotient $\PR_{g,n}$ is equal to 
\begin{equation}\label{E:PRbis}
    \PR_{g,n}\cong \PicRel^{op}_{g,n}(\ZZ)\rtimes (\ZZ/2\ZZ),
\end{equation}
where $\PicRel^{op}_{g,n}(\ZZ)=\PicRel_{g,n}(\ZZ)/W_{g,n}$ is the integral relative Picard group of the universal curve $\C_{g,n}/\M_{g,n}$, and the action of $ \ZZ/2\ZZ$ on $\PicRel^{op}_{g,n}(\ZZ)$ is via the inverse map. The group $\PicRel^{op}_{g,n}(\ZZ)$ is the abelian group  generated by:
\begin{itemize}
    \item the relative dualizing line bundle $\omega_\pi$; 
    \item the image $\Sigma_i$ of the $i$-th section of $\C_{g,n}/\M_{g,n}$ (for $1\leq i \leq n$);
\end{itemize}
subject to the following relations:
\begin{itemize}
    \item if $g=1$ then $\omega_\pi=0$;
    \item if $g=0$ then $\Sigma_1=\ldots=\Sigma_n$ and $\omega_\pi=-2\Sigma_1$.
\end{itemize}

\begin{lemma}\label{L:Sigma-ac}
\noindent 
\begin{enumerate}
\item  \label{L:Sigma-ac1} For any $\chi\in \ZZ$, the action of $W_{g,n}$ on $\Sigma_{g,n}^\chi$ is trivial on $\Sigmans_{g,n}^\chi$ and two elements $\sigma_1,\sigma_2\in \Sigmas_{g,n}^{\chi}$ are in the same orbit of $W_{g,n}$ if and only if $\Ds(\sigma_1)=\Ds(\sigma_2)$.
\item \label{L:Sigma-ac2} The group $\PR_{g,n}$ acts on $\Sigmans_{g,n}$ by sending 
$\sigma \in \Sigmans_{g,n}$ into 
 $$
    \begin{sis}
    & (L\cdot \sigma)(e;h,A):=\sigma(e;h,A)+ \beta(2h-2+e)+\sum_{i\in A} \alpha_i\\
&  \text{ for any } L=\beta\omega_{\pi}+\sum_{i=1}^n \alpha_i\Sigma_i\in \PicRel^{op}_{g,n}(\ZZ), \text{and}\\
     & (\iota \cdot \sigma)(e;h,A):=
     \begin{cases}
      -\sigma(e;h,A) & \text{ if } (e;h,A)\in \D(\sigma), \\
-\sigma(e;h,A)+1 & \text{ if } (e;h,A)\not \in \D(\sigma).    
     \end{cases}\\
    \end{sis}
    $$
\end{enumerate}  
\end{lemma}
\begin{proof}
The proof is straightforward from Remark \ref{R:PR-Pol}.
\end{proof}

We now prove that there are a finite number of compactified universal Jacobians over $\Mbargn$ up to the action of the group $\wPR_{g,n}$.

\begin{proposition}\label{P:finite-orb}
 The action of $\wPR_{g,n}$ on $\Sigma_{g,n}$ has finitely many orbits.

 In particular, there are finitely many isomorphism classes of compactified universal Jacobian stacks (and spaces) over $\Mbargn$.
\end{proposition}
This was proved for general universal V-stabilities in \cite[Thm. B(2)]{fava2024}, generalizing the case of general classical universal V-stabilities treated in \cite[Cor. 6.16]{Kass_2019}.
\begin{proof}
Using Lemma \ref{L:Sigma-ac}\eqref{L:Sigma-ac1}, it is enough to prove that the action of $\PR_{g,n}$ on $\Sigmans_{g,n}$ has finitely many orbits. 
Moreover, Lemma \ref{L:Sigma-ac}\eqref{L:Sigma-ac2} implies that, by translating by an element of $\PicRel^{op}_{g,n}(\ZZ)$, any element of $\Sigmans_{g,n}$ lies in the same orbit of a unique \emph{normalized} element $\sigma\in \Sigmans_{g,n}$, i.e. an element $\sigma\in \Sigmans_{g,n}$ such that 
\begin{equation}\label{E:norm-sigma}
\begin{sis}
&  \sigma(3;0,\emptyset)=0, \\  
&  \sigma(2;0,\{i\})=0 \text{ for any } i\in [n]. \\  
\end{sis}
\end{equation}
We now show that there are a finite number of normalized elements. 

By applying Remark \ref{R:triang} to the triangle $\Delta=[(3;0,\emptyset),(e+1;h-1,A), (e; g-h-e+1,A^c)]$ of $\Dgnns$ for any $h\geq 1$ and using the first vanishing in \eqref{E:norm-sigma}, we get that 
\begin{equation}\label{E:condiz1}
    -1\leq \sigma(e;h,A)-\sigma(e+1;h-1,A)\leq 0 \text{ for any } (e;h,A)\in \Dgnns \text{ such that } h\geq 1. 
\end{equation}

By applying Remark \ref{R:triang} to the triangle $\Delta=[(3;0,\emptyset),(e-1;h,A), (e; g-h-e+1,A^c)]$ of $\Dgnns$ for any $e\geq 3$ and using the first vanishing in \eqref{E:norm-sigma}, we get that 
\begin{equation}\label{E:condiz2}
    -1\leq \sigma(e;h,A)-\sigma(e-1;h,A)\leq 0 \text{ for any } (e;h,A)\in \Dgnns \text{ such that } e\geq 3.
\end{equation}

By applying Remark \ref{R:triang} to the triangle $\Delta=[(2;0,\{i\}),(e;h,A-\{i\}), (e; g-h-e+1,A^c)]$ of $\Dgnns$ for any $i\in A$ and using the second vanishing in \eqref{E:norm-sigma}, we get that 
\begin{equation}\label{E:condiz3}
    -1\leq \sigma(e;h,A)-\sigma(e;h,A-\{i\})\leq 0 \text{ for any } (e;h,A)\in \Dgnns \text{ such that } i\in A.
\end{equation}

Now, by iterating \eqref{E:condiz1} $h$-times, we get that 
\begin{equation}\label{E:condiz1bis}
    -h\leq \sigma(e;h,A)-\sigma(e+h;0,A)\leq 0 \text{ for any } (e;h,A)\in \Dgnns. 
\end{equation}
By iterating \eqref{E:condiz2} and using the first vanishing in \eqref{E:norm-sigma}, we get that
\begin{equation}\label{E:condiz2bis}
\begin{sis}
&     -(e-2)\leq \sigma(e;0,A)-\sigma(2;0,A)\leq 0 & \text{ for any } e\geq 2 \text{ and any } |A|>0;\\
&     -(e-3)\leq \sigma(e;0,\emptyset)-\sigma(3;0,\emptyset)=\sigma(e;0,\emptyset)\leq 0 & \text{ for any } e\geq 3.\\
\end{sis}
\end{equation}
By iterating \eqref{E:condiz3} and using the second vanishing in \eqref{E:norm-sigma}, we get that 
\begin{equation}\label{E:condiz3bis}
    -(|A|-1)\leq \sigma(2;0,A)-\sigma(2;0,\{i\})=\sigma(2;0,A)\leq 0  \: \text{ if } i\in A.
\end{equation}

By putting together \eqref{E:condiz1bis}, \eqref{E:condiz2bis} and \eqref{E:condiz3bis}, we get that 
\begin{equation}\label{E:final-s}
 -(2h+e+|A|-3)\leq \sigma(e;h,A)\leq 0 \text{ for any } (e;h,A)\in \Dgnns,
\end{equation}
which show that the set of normalized elements of $\Dgnns$ is finite, as we wanted to show. 
\end{proof}

Finally, we can describe when two compactified universal Jacobian spaces are isomorphic over $\Mbargn$.

\begin{theorem}\label{T:iso-UniSp}
     Let $\sigma_1,\sigma_2\in \Sigma_{g,n}$. 
     The following conditions are equivalent:
     \begin{enumerate}
         \item \label{T:iso-UniSp1}
         $\sigma_1^{ns}$ and $\sigma_2^{ns}$ lie in the same orbit for the action of $\PR_{g,n}$ on $\Sigmans_{g,n}$.
         \item \label{T:iso-UniSp2} $\ov \J_{g,n}(\sigma_1)$ and $\ov \J_{g,n}(\sigma_2)$ are isomorphic over $\Mbargn^{ns}$.
         \item \label{T:iso-UniSp3} $\ov J_{g,n}(\sigma_1)$ and $\ov J_{g,n}(\sigma_2)$ are isomorphic over $\Mbargn$.
          \item \label{T:iso-UniSp4} $\ov J_{g,n}(\sigma_1)$ and $\ov J_{g,n}(\sigma_2)$ are isomorphic over $\Mbargn^{ns}$.
     \end{enumerate}
\end{theorem}
\begin{proof}
 The implications $\eqref{T:iso-UniSp2}\Rightarrow \eqref{T:iso-UniSp4}$ and   $\eqref{T:iso-UniSp3}\Rightarrow \eqref{T:iso-UniSp4}$ are obvious. 

The implication $\eqref{T:iso-UniSp1}\Rightarrow  \eqref{T:iso-UniSp2}$ follows by combining Proposition \ref{P:PR-TF} and Lemmas \ref{L:ugua-ns} and \ref{L:Sigma-ac}\eqref{L:Sigma-ac1}.

We now prove the implication $\eqref{T:iso-UniSp1}\Rightarrow \eqref{T:iso-UniSp3}$. Up to replacing $\sigma_2$ with a conjugate for the action of $\wPR_{g,n}$, we can assume that $\sigma_1^{ns}=\sigma_2^{ns}$. In particular, we have that $|\sigma_1|=|\sigma_2|:=\chi$. 
By Lemma~\ref{L:Sigma-s-ns} and Proposition~\ref{P:Dsep}\eqref{P:Dsep4}, we can choose $\ov\sigma_i\in \Sigma_{g,n}^\chi$ (for $i=1,2$) such that 
$$
\ov \sigma_i\leq \sigma_i, \quad (\ov \sigma_i)^{ns}= \sigma_i^{ns} \text{ and } \Ds(\ov \sigma_i^s) \text{ is the minimum element of } \Im \Ds^{\chi}.
$$
By Lemma \ref{L:Sigma-ac}\eqref{L:Sigma-ac1} and the assumption $\sigma_1^{ns}=\sigma_2^{ns}$, we deduce that $\ov \sigma_1$ and $\ov \sigma_2$ are in the same orbit for the action of $\wPR_{g,n}$ on $\Sigma_{g,n}$. Hence, by applying Proposition \ref{P:PR-TF}, we deduce that the universal compactified Jacobian stacks $\ov \J_{g,n}(\ov \sigma_1)$ and $\ov \J_{g,n}(\ov \sigma_2)$ are isomorphic over $\Mbargn$, which implies, by passing to their relative good moduli spaces, that 
$$
\ov J_{g,n}(\ov \sigma_1) \text{ and } \ov J_{g,n}(\ov \sigma_2) \text{ are isomorphic over } \Mbargn. 
$$
In order to conclude the proof of this implication, it remains to prove that (for $i=1,2$) the inclusion $\ov \J_{g,n}(\sigma_i)\subseteq \ov \J_{g,n}(\ov \sigma_i)$ (see Theorem \ref{T:thmA}) induces, by passing to their relative good moduli spaces, an isomorphism 
$$\Phi: \ov J_{g,n}(\sigma_i)\xrightarrow{\cong} \ov J_{g,n}(\ov \sigma_i).$$
This can be proved as follows. First, the fact that $\Phi$ is an isomorphism can be checked on geometric fibers by \cite[17.9.5]{EGAIV4} and \cite[2.7.1]{EGAIV2}, using that $\ov J_{g,n}(\ov \sigma_i)$ (as well as 
$\ov J_{g,n}(\sigma_i)$) is flat over $\Mbargn$ by \cite[Thm. C]{FPV1}.
We then conclude using the assumptions $\ov \sigma_i\leq \sigma_i$ and $(\ov \sigma_i)^{ns}= \sigma_i^{ns}$ and Lemma \ref{L:iso-sepnod} below.

We finally prove the implication $\eqref{T:iso-UniSp4}\Rightarrow \eqref{T:iso-UniSp1}$, which will conclude the proof. By assumption, we have an isomorphism over $\Mbargn^{ns}$ 
\begin{equation}\label{E:iso-Psi}
\Psi:\ov J_{g,n}(\sigma_1)_{|\Mbargn^{ns}}\xrightarrow{\cong}\ov J_{g,n}(\sigma_2)_{|\Mbargn^{ns}}.
\end{equation}
By \cite[Cor. 6.1]{Kass_2019}, the restriction of $\Psi$ over $\Mgn$ is given by the action of a uniquely defined element of $g\in \PR_{g,n}$. More precisely, there exists a unique element $g\in \PR_{g,n}$ such that, for any lift $\wt g\in \wPR_{g,n}$ of $g$, we have that 
$$\Psi_{|\Mgn}:\ov J_{g,n}(\sigma_1)_{|\Mgn}=\TF_{\Cgn/\Mgn}^\chi\fatslash \Gm\xrightarrow[\cong]{\cdot \wt g}\TF_{\Cgn/\Mgn}^\chi\fatslash \Gm=\ov J_{g,n}(\wt g\cdot \sigma_1)_{|\Mgn}$$
where the action of $\wt g$ is the one in Proposition \ref{P:PR-TF}. Therefore, up to replacing $\sigma_1$ with $\wt g\cdot \sigma_1$ (and hence $\sigma_1^{ns}$ with $g\cdot \sigma_1^{ns}$) and $\Psi$ with $\Psi\circ (\cdot \wt g)^{-1}$, we can (and will) assume that 
\begin{equation}\label{E:Psi-ass}
\Psi_{|\Mgn}=\id.
\end{equation}
We now show that this implies the equality $\sigma_1^{ns}=\sigma_2^{ns}$, which will conclude the proof. Fix $(e;h,A)\in \Dgnns$ (so that $e\geq 2$) and pick a vine curve $X=C_1\cup C_2\in \Mbargn$ such that $\type_X(C_1)=(e;h,A)$ and $\type_X(C_2)=(e;h,A)^c$. For $i=1,2$, consider the V-stability $\s_i:=\sigma^{\sigma_i}(X)$ on $X$ as in Proposition \ref{P:Stabgn}, which satisfies $(\s_i)_{C_1}=\sigma_i(e;h,A)$ and $(\s_i)_{C_2}=\sigma_i((e;h,A)^c)$. It is then enough to show that $\s_1=\s_2$. 

For that purpose, consider the effective semiuniversal deformation space $\Phi:\fX\to \Spec R_X$ of $X$, as in \cite[Sec. 8]{FPV1}. Since $R_X$ is the complete local ring of $\Mbargn^{ns}\subset \Mbargn$ at $X$, there is a morphism  $\eta:\Spec R_X\to \Mbargn^{ns}$ sending $o$ into $X$. By pulling back the isomorphism $\Psi$ of \eqref{E:iso-Psi} along $\eta$ and using that good moduli spaces are functorial, we get an isomorphism 
$$
\Psi_X:\ov J_{\fX/R_X}(\s_1)\xrightarrow{\cong} \ov J_{\fX/R_X}(\s_1).
$$
Moreover, since the open subset $U_X$ of $\Spec R_X$  parametrizing smooth fibers of $\fX\to \Spec R_X$ is the pull-back of $\Mgn$ via $\eta$, the assumption from \eqref{E:Psi-ass} translates into the fact that 
$$
(\Psi_X)_{|U_X}=\id.
$$
Now we can apply Lemma \ref{L:vine-cla} below in order to conclude that $\s_1=\s_2$, and we are done. 
\end{proof}

\begin{lemma}\label{L:iso-sepnod}
 Let $X$ be a nodal curve over $k=\ov k$. Consider two V-stability conditions $\s,\ov \s\in \VStab^\chi(X)$ such that 
 $$
 \ov \s\leq \s \text{ and } \ov \s_Y=\s_Y \text{ for any } Y\in \BCon(X) \text{ such that } |Y\cap Y^c|\geq 2.
 $$
 Then the inclusion $\ov \J_X(\s)\subseteq \ov \J_X(\ov \s)$ induces, by passing to the good moduli space, an isomorphism 
 $$
 \Upsilon: \ov J_X(\s) \xrightarrow{\cong} \ov J_X(\ov \s).
 $$
\end{lemma}
For nodal curves of compact type, the above result could be easily deduced from \cite[Prop. 5.8]{FPV2}. The proof in the general case is an easy adaptation of the proof of loc. cit. 
\begin{proof}
Since the biconnected subcurves $Y\subset X$ such that $|Y\cap Y^c|=1$ do not appear in Condition \eqref{D:VStabX2} of Definition \ref{D:VStabX}, we can pick a V-stability condition $\wh \s\in \VStab^\chi(X)$ such that 
\begin{equation}\label{E:whs} 
   \begin{sis}
    &  \wh \s\leq \ov \s\leq \s,\\   
    & \wh \s_Y=\ov \s_Y=\s_Y \text{ for any } Y\in \BCon(X) \text{ such that } |Y\cap Y^c|\geq 2, \\
    & \D(\wh \s)\supseteq \{Y\in \BCon(X): \: |Y\cap Y^c|=1\}.
   \end{sis} 
\end{equation}
Then we have inclusions $\ov \J_X(\wh \s)\subseteq \ov \J_X(\s)\subseteq \ov \J_X(\ov \s)$, that induce the following maps by passing to the good moduli spaces
$$\Upsilon_1: \ov \J_X(\s)\xrightarrow{\Upsilon} \ov \J_X(\ov \s)\xrightarrow{\Upsilon_2} \ov \J_X(\wh \s).$$
In order to show that $\Upsilon$ is an isomorphism, it is enough to show that $\Upsilon_1$ and $\Upsilon_2$ are isomorphisms. Let us show that $\Upsilon_1$ is an isomorphism, the proof for $\Upsilon_2$ being the same. 

Denote by $\{Z_i\}$ the closure of the connected components of $X-X_{\rm sep}$, where $X_{\rm sep}$ is the set of separating nodes of $X$. Then we get a decomposition into connected subcurves  
$$
X=\bigcup_i Z_i \text{ with } Z_i\wedge Z_j=\emptyset \text{ and } Z_i\cap Z_j\subset X_{\rm sep} \text{ for any } i\neq j.
$$
Because of the third property in \eqref{E:whs}, we have that $Z_i \in \wh \D(\wh \s)$. Consider the induced V-stabilities $\wh \s(Z_i)$ as in \cite[Lemma-Definition 4.8]{FPV1}. By \cite[Prop. 5.8]{FPV1}, we have a closed embedding 
$$
\begin{aligned}
\Xi: \prod_i \ov \J_{Z_i}(\wh s(Z_i))& \hookrightarrow \ov \J_X(\wh \s)\\
(I_i) & \mapsto \bigoplus_i I_i.
\end{aligned}
$$
%The associated map of good moduli spaces
%$$
%!\Xi!: \prod_i \ov J_{Z_i}(\wh s(Z_i))& \hookrightarrow \ov J_X(\wh \s)
%$$
%is again a closed embedding by \cite[]{Alp}
Since the subcurves $Z_i$ meet only at separating nodes of $X$, by \cite[Prop. 4.7]{FPV2} we deduce that any sheaf of $\ov \J_X(\wh \s)$ (and hence of $\ov \J_X(\s)$) isotrivially specializes to a unique sheaf of $\prod_i \ov \J_{Z_i}(\wh s(Z_i))$. This implies, using the properties \cite[Thm. 4.16]{Alp} of good moduli spaces, that the map $\Xi$ induces, by passing to the good moduli spaces, two isomorphisms
$$
\begin{sis}
& |\Xi|:\prod_i \ov J_{Z_i}(\wh \s(Z_i)) \xrightarrow{\cong} \ov J_X(\wh \s),\\
& \ov J_X(\s) \xrightarrow{\Upsilon_1} \ov J_X(\wh \s) \xrightarrow[\cong]{|\Xi|^{-1}} \prod_i \ov J_{Z_i}(\wh \s(Z_i)).
\end{sis}
$$
We conclude that $\Upsilon_1$ is an isomorphism, as required.
\end{proof}

\begin{lemma}\label{L:vine-cla}
Let $X=C_1\cup C_2$ be a vine curve with $e\geq 2$ nodes.   Let $\s_1,\s_2\in \VStab^{\chi}(X)$ and consider the associated relative compactified Jacobian stacks/spaces
$$F_{\fX/R_X}(\s_i):\ov \J_{\fX/R_X}(\ov \s_i)\xrightarrow{\Xi_{\fX/R_X}(\ov \s_i)}  \ov J_{\fX/R_X}(\ov \s_i)\xrightarrow{f_{\fX/R_X}(\ov \s_i)} \Spec R_X$$ 
over the effective semuniversal deformation space $\Phi:\fX\to \Spec R_X$ of $X$. Denote by $U_X$ the open subset of $\Spec R_X$ parametrizing smooth fibers of $\fX\to \Spec R_X$ and denote by 
$$\ov{\id}_{U_X}:\ov J_{\fX/R_X}(\ov \s_1)_{|U_X}\xrightarrow{\cong} \ov J_{\fX/R_X}(\ov \s_2)_{|U_X}$$
the isomorphism induced by the identification 
$\id_{U_X}:\ov \J_{\fX/R_X}(\ov \s_1)_{|U_X}=(\TF_{\fX/R_X}^\chi)_{|U_X}= \ov \J_{\fX/R_X}(\ov \s_2)_{|U_X}$.

If there exists a $\Spec R_X$-isomorphism 
$$
\Psi: \ov J_{\fX/R_X}(\ov \s_1)\xrightarrow{\cong} \ov J_{\fX/R_X}(\ov \s_2)
$$
such that $\Psi_{|U_X}=\ov{\id}_{|U_X}$, then $\s_1=\s_2$.
\end{lemma}
\begin{proof}
Denote by $\Delta_X$ the closed subset of $\Spec R_X$ parametrizing locally trivial deformations of $X$. Observe that $\Delta_X$ has codimension $e\geq 2$ in $\Spec R_X$ and its complement $V_X:=\Spec R_X- \Delta_X$ is the open subset of $\Spec R_X$ parametrizing irreducible fibers of  $\fX\to \Spec R_X$. Therefore, we also have an isomorphism 
$$\ov{\id}_{V_X}:\ov J_{\fX/R_X}(\ov \s_1)_{|V_X}\xrightarrow{\cong} \ov J_{\fX/R_X}(\ov \s_2)_{|V_X}$$
induced by the identification 
$\id_{V_X}:\ov \J_{\fX/R_X}(\ov \s_1)_{|V_X}=(\TF_{\fX/R_X}^\chi)_{|V_X}= \ov \J_{\fX/R_X}(\ov \s_2)_{|V_X}$.
Since $\Psi_{|U_X}=\id_{U_X}$ by assumption and $U_X$ is dense in $V_X$, we must have that $\Psi_{|V_X}=\ov{\id}_{V_X}$ using that $\ov J_{\fX/R_X}(\ov \s_i)\to \Spec R_X$ (for $i=1,2$) is proper with geometrically reduced fibers by \cite[Thm. 8.7]{FPV1}.

For $i=1,2$, denote by $\ov \J_{\fX/R_X}(\ov \s_i)^{\st}$ the open subset of $\ov \J_{\fX/R_X}(\ov \s_i)$ parametrizing families of $\ov \s_i$-stable sheaves, which clearly contains $\ov \J_{\fX/R_X}(\ov \s_i)_{|V_X}$.  The restriction of $\Xi_{\fX/R_X}(\ov \s_i)$ to $\ov \J_{\fX/R_X}(\ov \s_i)^{\st}$ is the $\Gm$-rigidification onto its image 
$$
\ov J_{\fX/R_X}(\ov \s_i)^{\st}:=\Xi_{\fX/R_X}(\ov s_i)(\ov \J_{\fX/R_X}(\ov \s_i)^{\st})\subseteq \ov J_{\fX/R_X}(\ov \s_i).
$$
Moreover, since $\fX\to \Spec R_X$ admits sections, the $\Gm$-gerbe $\ov \J_{\fX/R_X}(\ov \s_i)^{\st}\to \ov J_{\fX/R_X}(\ov \s_i)^{\st}$ is trivial, i.e. we have a (non-canonical) isomorphism 
\begin{equation}\label{E:st-gerbe}
    \ov \J_{\fX/R_X}(\ov \s_i)^{\st}\cong \ov J_{\fX/R_X}(\ov \s_i)^{\st}\times B \Gm.
\end{equation}

From the explicit structure of compactified Jacobian spaces of vine curves (see \cite[Sec. 8 and p. 84-85]{Oda1979CompactificationsOT}), it follows that the complement of the stable locus $\ov J_X(\s_i)-\ov J_X(\s_i)^{\st}$ (for $i=1,2$) is non-empty if and only if the intersection of all irreducible components of $\ov J_X(\s_i)$ is irreducible for $e\geq 3$ (resp. if and only if the singular locus of $\ov J_X(\s_i)$ is irreducible for $e=2$), in which case this irreducible locus is exactly $\ov J_X(\s_i)-\ov J_X(\s_i)^{\st}$. Therefore, the isomorphism $\Phi$ restricts to an isomorphism 
\begin{equation}\label{E:Psi-st}
  \Psi^{\st}:   \ov J_{\fX/R_X}(\ov \s_1)^{\st}\xrightarrow{\cong} \ov J_{\fX/R_X}(\ov \s_2)^{\st}.
  \end{equation}
Using the isomorphism \eqref{E:st-gerbe} and the assumption that $\Psi^{\st}_{|V_X}=\Psi_{|V_X}=\id_{V_X}$, we can lift $\Psi^{\st}$ to an isomorphism 
\begin{equation}\label{E:Psi-wt}
     \wt \Psi:   \ov \J_{\fX/R_X}(\ov \s_1)^{\st}\xrightarrow{\cong} \ov \J_{\fX/R_X}(\ov \s_2)^{\st} \text{ such that } \wt \Psi_{|V_X}=\id_{V_X}. 
\end{equation}
Denote by $\I_i$ the universal sheaf on $\ov \J_{\fX/R_X}(\ov \s_i)^{\st}\times_{\Spec R_X} \fX$ for $i=1,2$, and consider the two sheaves $\I_1$ and $\wt \I:=(\wt \Psi\times \id_{\fX})^*(\I_2)$ on $\ov \J_{\fX/R_X}(\ov \s_1)^{\st}\times_{\Spec R_X} \fX$.  Since $\wt \Psi_{|V_X}=\id_{V_X}$, the two sheaves $\I_1$ and $(\wt \Psi\times \id_{\fX})^*(\I_2)$ coincide on the open subset $\ov (\J_{\fX/R_X}(\ov \s_1)^{\st}\times_{\Spec R_X} \fX)_{|V_X}$ whose complement has codimension $e\geq 2$. Hence,  using that $\ov \J_{\fX/R_X}(\ov \s_1)^{\st}\times_{\Spec R_X} \fX$ is Cohen-Macaualy and regular in codimension one (which is proved as in \cite[Lemma 7.3]{Kass_2019}), we can apply \cite[Cor. 7.2]{Kass_2019} in order to conclude that $\I_1= (\wt \Psi\times \id_{\fX})^*(\I_2)$.
This implies that $\wt \Psi=\id$, which then gives  
$$\ov \J_{\fX/R_X}(\ov \s_1)^{\st}=\ov \J_{\fX/R_X}(\ov \s_2)^{\st}\subset \TF_{\fX/R_X}^\chi. $$
 Taking the central fiber, we get that 
$$\ov \J_{X}(\s_1)^{\st}=\ov \J_{X}(\s_2)^{\st}\subset \TF_X^{\chi},$$
which then implies that $\s_1=\s_2$, by using that a line bundle $L$ on $X$ belongs to $\ov \J_{X}(\s_i)^{\st}$ if and only if (for $1\leq i,j\leq 2$)
$$
\chi(L_{C_j})\in 
\begin{sis}
[(\s_i)_{C_j},(\s_i)_{C_j}+e-1] & \text{ if } \s_i \text{ is general,}\\
[(\s_i)_{C_j}+1,(\s_i)_{C_j}+e-1]      & \text{ if } \s_i \text{ is not general.}
\end{sis}
$$  
\end{proof}

\begin{remark}\label{R:Kod-max}
Arguing as in \cite[Lemma 6.18]{Kass_2019} (which generalizes \cite[Thm. 7.13]{BFV}), Theorem~\ref{T:iso-UniSp} can be improved if the coarse moduli space $\ov M_{g,n}$ of $\Mbargn$ is of general type (i.e. if $g\geq 22$): under this assumption, the conditions of the Theorem become equivalent to the fact that $\ov J_{g,n}(\sigma_1)$ and $\ov J_{g,n}(\sigma_2)$ are abstractly (and not only over $\Mbargn$) isomorphic . 
\end{remark}

\section{Universal families}\label{Sec:univ-fam}

The aim of this section is to describe the universal family of a compactified Jacobian stack over $\Mbargn$ in terms of a compactified Jacobian stack over $\ov\M_{g,n+1}$. 

\subsection{The universal family over $\TF_{g,n}$}

We first begin by describing the relationship between $\TF_{g,n+1}$ and the universal family over $\TF_{g,n}$. 

Recall that there is a canonical isomorphism between the universal family $\pi:\Cbargn\to \Mbargn$ and $\Mbargnbis$:
\begin{equation}\label{E:univ-Mg}
\begin{tikzcd}
  \Mbargnbis \arrow["\Phi", rd]  \arrow[rr, "\Upsilon", "\cong"']& & \Cbargn \arrow[ld, "\pi"']   \\
     & \Mbargn  \arrow[ur, bend right=30, "\sigma_i"']\arrow[ul, bend left=30, "\sigma_i'"]&  
\end{tikzcd}
\end{equation}
where the isomorphism $\Upsilon$ and the morphism $\Phi$ are defined on geometric points by
$$
\Upsilon(C):=(C^{\st}, \st(p_{n+1})) \quad \text{ and } \quad \Phi(C)=C^{\st},$$
with $\st_C=\st: C=(C,p_1,\ldots,p_{n+1})\to C^{\st}=(C,p_1,\ldots, p_n)^{\st}$ being the \emph{stabilization} morphism that forgets the last marked point $p_{n+1}$ and then it stabilizes the resulting $n$-pointed curve. Concretely, either $\st$ is an isomorphism or it is the contraction of a smooth rational curve $E$ (called an \emph{exceptional component}), which either is attached to the complementary curve $E^c$ at one node and it contains exactly two marked points $\{p_{j}, p_{n+1}\}$ (for some $1\leq j \leq n$), or it is attached to the complementary subcurve $E^c$ at two nodes and it contains only the marked point $p_{n+1}$. The universal curve $\Cbargn/\Mbargn$ is endowed with $n$ canonical sections defined on geometric points by  (for any $1\leq i \leq n$)
$$
\sigma_i(C)=(C,p_i).
$$
Via the isomorphism $\Upsilon^{-1}$, the section $\sigma_i$ is sent to the section $\sigma_i':=\Upsilon^{-1}\circ \sigma_i$ given on geometric points by 
$$
\sigma_i'(C)=B_{p_i}(C),
$$
where $B_{p_i}(C)$ is the bubbling of $C$ at $p_i$, i.e. the stable $n+1$-pointed curve obtained by gluing a smooth rational curve $E$ with the $n$-pointed curve $C$ at the old marked point $p_i$, and then putting the new $i$-th and the $(n+1)$-th marked point on $E$. 

The goal of this subsection is to lift the diagram in \eqref{E:univ-Mg} to the stack $\TF_{g,n}^\chi$. Observe that the universal family over $\TF_{g,n}^\chi$, together with its $n$ canonical sections $\wh{\sigma_i}$
\begin{equation}
\begin{tikzcd}
  \Cbargn\times_{\Mbargn} \TF_{g,n}^\chi \arrow["\wh{\pi}", r]  & \TF_{g,n}^\chi \arrow[l, bend left=30, "\wh{\sigma_i}"]
\end{tikzcd}
\end{equation}
is given by pulling back the universal family $\Cbargn/\Mbargn$, together with its $n$ canonical sections, along the forgetful morphism $\TF_{g,n}\to \Mbargn$. We will denote by $\I_{g,n}$ the universal sheaf on $\Cbargn\times_{\Mbargn} \TF_{g,n}^\chi$. 

We now want to lift the morphisms $\Phi$ and $\Upsilon$, together with the sections $\sigma_i'$, to $\TF_{g,n}^\chi$. It turns out that we cannot lift $\Phi$ and $\Upsilon$ to the entire stack $\TF_{g,n+1}^\chi$ but only to an open substack, as we now show.

\begin{theorem}\label{T:univ-TF}
\noindent 
\begin{enumerate}
    \item \label{T:univ-TF1} There exist open substacks $\TFp_{g,n+1}^\chi, \TFm_{g,n+1}^\chi\subseteq \TFo_{g,n+1}^\chi$ of $\TF_{g,n+1}^\chi$ whose geometric points are given by 
$$
    \begin{aligned}
     &\TFo_{g,n+1}^\chi(k):=\{(C,I)\in \TF_{g,n+1}^\chi(k)\: : \chi(I_E)\geq 0 \text{ and } \chi(I_{E^c})\geq \chi \text{ for any exceptional comp. } E \text{ of } C\},\\
    & \TFp_{g,n+1}^\chi(k):=\{(C,I)\in \TF_{g,n+1}^\chi(k)\: : \chi(I_E)\geq 1 \text{ and } \chi(I_{E^c})\geq \chi \text{ for any exceptional comp. } E \text{ of } C\},\\
    & \TFm_{g,n+1}^\chi(k):=\{(C,I)\in \TF_{g,n+1}^\chi(k)\: : \chi(I_E)\geq 0 \text{ and } \chi(I_{E^c})\geq \chi+1 \text{ for any exceptional comp. } E \text{ of } C\}.\\
    \end{aligned}
    $$
    \item \label{T:univ-TF2} There is a commutative diagram 
\begin{equation}\label{E:univ-TF}
\begin{tikzcd}
  \TFo_{g,n+1}^\chi \arrow["\wh{\Phi}", rd]  \arrow[rr, "\wh{\Upsilon}"]& & \Cbargn\times_{\Mbargn} \TF_{g,n}^\chi \arrow[ld, "\wh{\pi}"']   \\
     & \TF_{g,n}^\chi  \arrow[ur, bend right=30, "\wh{\sigma_i}"']\arrow[ul, bend left=30, "\wh{\sigma_i}'"]& 
\end{tikzcd}
\end{equation}
lying over the diagram in \eqref{E:univ-Mg}, whose morphisms are defined on geometric points by
$$
\begin{sis}
& \wh{\Phi}(C,I)=(C^{\st},\st_*(I)),\\   
& \wh{\Upsilon}(C,I)=(C^{\st},\st(p_{n+1}),\st_*(I)),\\
& \wh{\sigma_i}'(C,I)=(B_{p_i}(C),\st^*(I)).
\end{sis}
$$
Moreover, the morphisms $\wh{\sigma_i}'$ (for $1\leq i \leq n$) are sections of $\wh{\Phi}$. 
\item \label{T:univ-TF2bis} The morphism $\wh{\Upsilon}$ is a good moduli space morphism. In particular, $\wh{\Upsilon}$ and $\wh{\Phi}$ are universally closed, S-complete and $\Theta$-complete.
\item \label{T:univ-TF3} The restrictions $\wh{\Upsilon}^+=\wh{\Upsilon}_{|\TFp_{g,n+1}^\chi}$ and $\wh{\Upsilon}^-=\wh{\Upsilon}_{|\TFm_{g,n+1}^\chi}$ are representable and proper.
\item \label{T:univ-TF4} 
The restriction of $\wh{\Upsilon}$ to the open loci $\TFp_{g,n+1}^\chi$ and $\TFm_{g,n+1}^\chi$ induces the following isomorphisms over $\Cbargn\times_{\Mbargn} \TF_{g,n}^\chi$ 
$$
\TFp_{g,n+1}^\chi\cong \PP(\I_{g,n})   \text{ and } \TFm_{g,n+1}^\chi\cong \PP(\I_{g,n}^\vee),
$$
under which the tautological line bundle $\O(1)$ becomes isomorphic to, respectively,  
$$
\O_{\PP(\I_{g,n})}(1)=\wh{\sigma}_{n+1}^*(\I_{g,n+1})_{|\TFp_{g,n+1}^\chi} \text{ and } \O_{\PP(\I_{g,n}^\vee)}(1)=\wh{\sigma}_{n+1}^*(\I_{g,n+1}^\vee)_{|\TFm_{g,n+1}^\chi}
$$
%where $\I_{g,n}$ is the universal sheaf on $\Cbargn\times_{\Mbargn} \TF_{g,n}^\chi$.
\end{enumerate}    
\end{theorem}
Indeed, it will follow from the proof of Part~\eqref{T:univ-TF2} below that $\TFo_{g,n+1}^\chi$ is the largest open substack of $\TF_{g,n+1}^\chi$ where the morphisms $\wt{\Upsilon}$ and $\wt{\Phi}$ can be defined by the formulas in \eqref{T:univ-TF2}. The proof of Parts~\eqref{T:univ-TF1} and~\eqref{T:univ-TF2} are inspired by the proof of \cite[Thm. 3.1]{estevespacini}, where a similar result is proved for the push-forward of a line bundle under the stabilization morphism of a family of semistable curves. The proof of Part~\eqref{T:univ-TF4} is inspired by and generalizes \cite[Thm. 3.30]{APag}.
\begin{proof}
Part \eqref{T:univ-TF1}: it is enough to show that there exist open substacks $\U_{\geq 0}$, $\U_{>0}$, $\V_{\geq \chi}$ and $\V_{>\chi}$ of $\TF_{g,n+1}^\chi$ whose geometric points are given by 
$$
    \begin{sis}
     &\U_{\geq 0}(k):=\{(C,I)\in \TF_{g,n+1}^\chi(k)\: : \chi(I_E)\geq 0  \text{ for any exceptional component } E \text{ of } C\},\\
     &\U_{>0}(k):=\{(C,I)\in \TF_{g,n+1}^\chi(k)\: : \chi(I_E)> 0  \text{ for any exceptional component } E \text{ of } C\},\\
    & \V_{\geq \chi}(k):=\{(C,I)\in \TF_{g,n+1}^\chi(k)\: :  \chi(I_{E^c})\geq \chi \text{ for any exceptional component } E \text{ of } C\},\\
    & \V_{> \chi}(k):=\{(C,I)\in \TF_{g,n+1}^\chi(k)\: :  \chi(I_{E^c})> \chi \text{ for any exceptional component } E \text{ of } C\}.\\
    \end{sis}
    $$
In order to show this, consider the universal stabilization morphism
\begin{equation}\label{E:univ-stab}
\begin{tikzcd}
  \C_1:=\Cbargnbis \times_{\Mbargnbis} \TF_{g,n+1}^\chi   \arrow[rr, "\St"] \arrow[rd, "\psi_1"']& & \Phi^*\Cbargn\times_{\Mbargnbis} \TF_{g,n+1}^\chi=:\C_2  \arrow[ld, "\psi_2"]  \\
     & \TF_{g,n+1}^\chi  & 
\end{tikzcd}
\end{equation}
Fix a $\psi_2$-relative ample line bundle $\A$ on $\C_2$ and set $\wh{\A}:=\St^*(\A)$. On $\C_1$ there is a universal sheaf $\I_{g,n+1}$ whose restriction on a geometric fiber $\psi_1^{-1}(C,I)=C$ is equal to $I$.

 Let us first show that $\U_{\geq 0}$ is open. Pick a  geometric point $(C,I)$ of $\TF_{g,n+1}^\chi$ that belongs to $\U_{\geq 0}(k)$. We will show that $\U_{\geq 0}$ contains an open neighborhood of $(C,I)$. 

If $C$ does not have any exceptional component, i.e. if $\st_C:C\to C^{\st}$ is an isomorphism, then the same will be true on an open neighborhood of $C$, which then will belong to $\U_{\geq 0}$.  

Assume now that $C$ has an exceptional component $E$, and consider the exact sequence (see \eqref{E:resY})
\begin{equation}\label{E:I-resE}
  0\to \leftindex_{E^c}I \to I \to I_{E}\to 0.  
\end{equation}
Since the stabilization morphism $\st_C:C\to C^{\st}$ is finite on $E^c$, the line bundle $\wh{\A}_C:=\wh{\A}_{|\psi_1^{-1}(C,I)=C}$ is ample on $E^c$. Hence, there exists a constant $N_C\gg 0$ such that 
\begin{equation}\label{E:van-H1}
H^1(C,\leftindex_{E^c}I\otimes \wh{\A}_C^{m})=0 \text{ for any } m\geq N_C.
\end{equation}
By replacing $\TF_{g,n+1}^\chi$ with a quasi-compact open neighboorhood of $(C,I)$, we may assume that $N_C:=N$ is independent of $C$. By taking the long exact sequence in cohomology associated to \eqref{E:I-resE} and using \eqref{E:van-H1}, we deduce that 
\begin{equation}\label{E:equa-H1}
 H^1(C, I\otimes \wh{\A}_C^{m})\cong   H^1(C, I_E\otimes \wh{\A}_C^{m})=H^1(E, I_E) \text{ for any } m\geq N,
\end{equation}
where we have used that $\wh{\A}_C$ is trivial on $E$ since $E$ is contracted by $\st_C$. 
Therefore, we deduce that 
\begin{equation}\label{E:U-vanH1}
(C,I)\in \U_{\geq 0}(k) \Longleftrightarrow \chi(I_E)\geq 0 \Longleftrightarrow   H^1(C, I\otimes \wh{\A}_C^{m})=H^1((\I_{g,n+1}\otimes \wh{\A}^m)_{|\psi_1^{-1}(C,I)})=0 \text{ for any } m\geq N.
\end{equation}
By the semicontinuity of cohomology, the last condition of \eqref{E:U-vanH1} is open and therefore it is satisfied in an open neighboorhood of $(C,I)$. 

The proofs for the other three loci are similar, with the following replacements:

$\bullet$ For $\U_{>0}$, one replaces the universal sheaf $\I_{g,n+1}$ with the second reflexive hull $\I_{g,n+1}^{[2]}$ of $\I_{g,n+1}$, whose restriction to a geometric fiber $\psi_1^{-1}(C,I)$ is equal to $I^{[2]}$. Using that $(I^{[2]})_E=(I_E)^{2}$, Condition~\eqref{E:U-vanH1} becomes 
\begin{equation}\label{E:U>-vanH1}
(C,I)\in \U_{> 0}(k) \Longleftrightarrow \chi(I_E)> 0 \Longleftrightarrow   H^1(C, I^{[2]}\otimes \wh{\A}_C^{m})=H^1((\I_{g,n+1}^{[2]}\otimes \wh{\A}^m)_{|\psi_1^{-1}(C,I)})=0 \text{ for any } m\geq N.
\end{equation}

$\bullet$ For $\V_{>0}$, one replaces the universal sheaf $\I_{g,n+1}$ with its dual sheaf $\I_{g,n+1}^*=\cHom(\I_{g,n+1},\omega_{\psi_1})$, whose restriction to a geometric fiber $\psi_1^{-1}(C,I)$ is equal to $I^*=\cHom(I,\omega_C)$. Using that $(I^*)_E=(\leftindex_{E}I)^{*}=\cHom(\leftindex_{E}I,\omega_E)$ by Lemma \ref{L:duality}, Condition \eqref{E:U-vanH1} becomes 
\begin{equation}\label{E:V-vanH1}
\begin{aligned}
& (C,I)\in \V_{\geq \chi}(k) \Longleftrightarrow \chi(I_{E^c})\geq \chi \Longleftrightarrow 0\geq \chi(\leftindex_{E}I)=-\chi((\leftindex_{E}I)^*)=-\chi((I^*)_E) \Longleftrightarrow \\
& \Longleftrightarrow H^1(C, I^{*}\otimes \wh{\A}_C^{m})=H^1((\I_{g,n+1}^{*}\otimes \wh{\A}^m)_{|\psi_1^{-1}(C,I)})=0 \text{ for any } m\geq N.
\end{aligned}
\end{equation}

$\bullet$ For $\V_{>0}$, one replaces the universal sheaf $\I_{g,n+1}$ with the second reflexive hull $(\I_{g,n+1}^*)^{[2]}$ of the dual sheaf $\I_{g,n+1}^*$, whose restriction to a geometric fiber $\psi_1^{-1}(C,I)$ is equal to $(I^*)^{[2]}$. Using that $((I^*))^{[2]}_E=((I^*)_E)^2=((\leftindex_{E}I)^{*})^2$ by Lemma \ref{L:duality}, Condition \eqref{E:U-vanH1} becomes 
\begin{equation}\label{E:V>-vanH1}
\begin{aligned}
& (C,I)\in \V_{> \chi}(k) \Longleftrightarrow \chi(I_{E^c})> \chi \Longleftrightarrow 0> \chi(\leftindex_{E}I)=-\chi((\leftindex_{E}I)^*)=-\chi((I^*)_E) \Longleftrightarrow \\
& \Longleftrightarrow H^1(C, (I^{*})^{[2]}\otimes \wh{\A}_C^{m})=H^1(((\I_{g,n+1}^{*})^{[2]}\otimes \wh{\A}^m)_{|\psi_1^{-1}(C,I)})=0 \text{ for any } m\geq N.
\end{aligned}
\end{equation}

Part \eqref{T:univ-TF2}: let $(C,I)$ be a geometric point of $\TFo_{g,n+1}^\chi$ and consider the stabilization $\st:C\to C^{\st}$. Lemma \ref{L:push-st} implies that  $\st_*(I)\in \TF_{C^{\st}}^\chi$. 

In order to define the morphism $\wh{\Phi}$, consider the restriction of the diagram in \eqref{E:univ-stab} to the open substack $\TFo_{g,n+1}^\chi$:
\begin{equation}\label{E:univ-stabo}
\begin{tikzcd}
  \C_1^o:=\Cbargnbis \times_{\Mbargnbis} {\TFo_{g,n+1}^\chi}   \arrow[rr, "\St^o"] \arrow[rd, "\psi_1^o"']& & \Phi^*\Cbargn\times_{\Mbargnbis} {\TFo_{g,n+1}^\chi}=:\C_2^o \arrow[ld, "\psi^o_2"]  \\
     & \TFo_{g,n+1}^\chi  & 
\end{tikzcd}
\end{equation}
On $\C_1^o$ there is a universal sheaf $\I_{g,n+1}$ whose restriction to a geometric fiber $\psi_1^{-1}(C,I)=C$ is equal to $I$. The existence of the morphism $\wh{\Phi}$ lying over $\Phi$ and given on geometric points by $\wh \Phi(C,I)=(C^{\st}, \st_*(I))$ will follow from the following 

\un{Claim:} The sheaf $\St^o_*(\I_{g,n+1})$ is a relative rank-one torsion-free sheaf on $\C_2^o/\TFo_{g,n+1}^\chi$ such that $\St^o_*(\I_{g,n+1})_{|(\psi_2^o)^{-1}(C,I)}=\st_*(I)$. 

Indeed, fix a $\psi_2$-relative ample line bundle $\A$ on $\C_2^o$ and set $\wh{\A}:=\St^*(\A)$. The flatness of $\St^o_*(\I_{g,n+1})$ over $\TFo_{g,n+1}^\chi$ is equivalent to showing that, locally on $\TFo_{g,n+1}^\chi$, the sheaf 
\begin{equation}\label{E:sheaf-A}
(\psi_2^o)_*(\St^o_*(\I_{g,n+1})\otimes \A^m)=(\psi_1^o)_*(\I_{g,n+1}\otimes \wh{\A}^m)
\end{equation}
is locally free for any $m\gg 0$. This follows from the vanishing on the right hand side of \eqref{E:U-vanH1}, using that $\TFo_{g,n+1}^\chi\subseteq \U_{\geq 0}$. 
Moreover, the vanishing on the right hand side of \eqref{E:U-vanH1} implies also that the sheaf $(\psi_1^o)_*(\I_{g,n+1}\otimes \wh{\A}^m)$ commutes with base change (locally on $\TFo_{g,n+1}^\chi$) for any $m\gg 0$, which in turn implies that the sheaf $\St^o_*(\I_{g,n+1})$ satisfies base change. In particular, the fiber of $\St^o_*(\I_{g,n+1})$ over a geometric point $(C,I)$ of $\TFo_{g,n+1}^\chi$ is equal to $\st_*(I)$, which is a rank-one torsion-free sheaf on $C^{\st}$ of Euler characteristic $\chi$ by Lemma \ref{L:push-st}\eqref{L:push-st1}. This concludes the proof of Claim.

\vspace{0.1cm}

The morphism $\wh \Upsilon$ is defined as the fiber product of $\wh{\Phi}$ and of the morphism $\TFo_{g,n+1}^\chi\to \Mbargnbis\xrightarrow{\Upsilon} \Cbargn$, and therefore it lies over $\Upsilon$ and it is given on geometric points by $\wh{\Upsilon}(C,I)=(C^{\st},\st(p_{n+1}),\st_*(I))$.

In order to define the section $\wh{\sigma_i}'$ of $\wh{\Phi}$, consider the following diagram induced by the section $\sigma_i'$ of $\Phi$: 
\begin{equation}\label{E:univ-si}
\begin{tikzcd}
  \D_1:=(\sigma_i')^*\Cbargnbis \times_{\Mbargn} \TF_{g,n}^\chi   \arrow[rr, "\ov{\St}"] \arrow[rd, "\phi_1"']& & \Cbargn\times_{\Mbargn} \TF_{g,n}^\chi=:\D_2  \arrow[ld, "\phi_2"]  \\
     & \TF_{g,n}^\chi  .& 
\end{tikzcd}
\end{equation}
The fiber of $\ov{\St}$ over a geometric point $(C,I)\in \TF_{g,n}^\chi$ is equal to the stabilization morphism $\st=\st_{B_{p_i}(C)}:B_{p_i}(C)=\phi_1^{-1}(C,I)\to \phi_2^{-1}(C,I)=C$ of the bubbling of $C$ at the $i$-th marked point $p_i$. 

Consider now the sheaf $\ov{\St}^*(\I_{g,n})$ on $\D_1$, where $\I_{g,n}$ is the universal sheaf on $\D_2$. Since the sheaf $\I_{g,n}$ is locally free around the image of the section $\sigma_i$, it follows that $\ov{\St}^*(\I_{g,n})$ is a relative rank-one torsion-free sheaf on $\D_1/\TF_{g,n}^\chi$, whose fiber over a geometric point $(C,I)$ of $\TF_{g,n}^{\chi}$ is equal to $\st^*(I)$. Therefore, the sheaf $\ov{\St}^*(\I_{g,n})$ defines a morphism $\wh{\sigma_i}': \TF_{g,n}^\chi\to \TF_{g,n+1}^\chi$ lying over $\sigma_i'$ and which is given on geometric points by $\wh{\sigma_i}'(C,I)=(B_{p_i}(C),\st^*(I))$. Moreover, if we denote by $E$ the exceptional component of $B_{p_i}(C)$ (so that $\st$ identifies $E^c$ with $C$) then we have that 
$$
\begin{sis}
  & \st^*(I)_{E}=\O_E\Rightarrow \chi(\st^*(I)_E)=1,\\
  & \st^*(I)_{E^c}=I\Rightarrow \chi(\st^*(I)_{E^c})=\chi,
\end{sis}
$$
which implies that $\wh{\sigma_i}'$ factors through the open subset $\TFo_{g,n+1}^\chi$. Finally, the fact that $\wh{\sigma_i}'$ is a section of $\wh{\Phi}$ follows from the fact that $\ov{\St}_*(\ov{\St}^*(\I_{g,n}))=\I_{g,n}$. 

Part \eqref{T:univ-TF2bis}: the fact that $\wh{\Upsilon}$ is a good moduli space morphism follows from Proposition \ref{P:push-loc}, using that this property can be checked fpqc locally on the codomain by \cite[Rmk. 4.4, Prop. 4.7]{Alp}.
The last properties of $\wh{\Upsilon}$ and $\wh{\Phi}$ follows from the fact that a good moduli space morphism is universally closed, $\Theta$-complete and S-complete (see \cite{AHLH}) and that $\wh{\Phi}=\wh{\pi}\circ \wh{\Upsilon}$ with $\wh{\pi}$ being proper and representable.

%Therefore, it remains to prove that $\wh{\Upsilon}$ is a good moduli space morphism. This can be checked fpqc locally on the codomain by \cite[Rmk. 4.4, Prop. 4.7]{Alp}. Hence, we fix a geometric point $(D,q,J)\in \Cbargn\times_{\Mbargn} \TF_{g,n}^\chi$ and we aim to find a fpqc neighborhood of $(D,q,J)$ over which $\wh{\Upsilon}$ is a good moduli space morphism. We will distinguish several cases:
%\begin{enumerate}[(a)]
%    \item If $q$ is smooth non-marked point of $D$, then $\wh{\Upsilon}$ is an isomorphism locally above $(D,q,J)$
% and we are done. 
% \item If $q$ is marked point of $D$ then ....
% \item If $q$ is a node of $D$ and $J$ is free at $q$, then $\wh{\Upsilon}$ is an an isomorphism locally above $(D,q,J)$  by Lemma \ref{L:push-pm}\eqref{L:push-pm2} and we are done.
% \item If $q$ is a node of $D$ and $J$ is not free at $q$, then...
 % \end{enumerate}

Part \eqref{T:univ-TF3}: this follows from Proposition \ref{P:push-loc}, using that the properties in question can be checked fpqc locally on the codomain.

Part \eqref{T:univ-TF4}: we will  prove the statement for $\TFp_{g,n+1}^\chi$, and we will leave to the reader the task of figuring out the necessary (small) adjustments needed in order to prove the statement for  $\TFm_{g,n+1}^\chi$. Set $\wh{\Upsilon}^+:=\wh{\Upsilon}_{|\TFp_{g,n+1}^\chi}$ and $\wh{\Phi}^+:=\wh{\Phi}_{|\TFp_{g,n+1}^\chi}$
Using the universal property of $\PP(\I_{g,n})$ (see \cite[Prop.~5.4, 5.5]{estevespacini}), it is enough to check the following properties:
\begin{enumerate}[(a)]
    \item $\wh{\Phi}^+$ is a family of quasi-stable $n$-pointed curves and $\wh{\Upsilon}^+$ is its stabilization morphism;
    \item the line bundle $\wh{\sigma}_{n+1}^*(\I_{g,n+1})$ has degree $1$ on every exceptional component of $\wh{\Upsilon}^+$;
    \item $(\wh{\Upsilon}^+)_*(\wh{\sigma}_{n+1}^*(\I_{g,n+1}))=\I_{g,n}$. 
\end{enumerate}

Observe that the morphism $\wh{\Phi}^+$ is proper and representable, being the composition of $\wh{\Upsilon}^+$ (which is proper and representable by Part~\eqref{T:univ-TF3}) with the proper and representable morphism $\wh{\pi}$. Proposition \ref{P:push-loc} implies that the geometric fiber of $\wh{\Phi}^+$ over $(D,J)\in \TF_{g,n}^\chi(k=\ov k)$ is a quasi-stable $n$-pointed curve which is obtained by bubbling $D$ into the nodes on which $J$ is not free and, furthermore, that $\wt{\Upsilon}^+$ is, on geometric fibers, the stabilization morphism. And finally, since the domain and codomain of $\wh{\Phi}^+$ are regular (see \cite[Sec. 8]{CMKVbirational}), by miracle flatness we have that $\wh{\Phi}^+$ is flat because it is equidimensional of relative dimension one. This shows that $\wh{\Phi}^+$ is a family of quasi-stable $n$-pointed curves and part (a) is established.

\vspace{0.1cm}

Consider now an exceptional component of $\wh{\Upsilon}^+$, which, as noted above, is the preimage $(\wh{\Upsilon}^+)^{-1}(D,p,J)$, for $(D,p,J)\in \Cbargn\times_{\Mbargn} \TF_{g,n}^\chi$ such that $p$ is a node of $D$ and $J\in \TF_D^\chi$ is not free at $p$. Denote by $B_p(D)$ the bubbling of $D$ at $p$ and by $\st:B_p(D)\to D$ the stabilization morphism which contracts the exceptional component $E_p\cong \PP^1$ of the bubbling into $p$. From the proof of Proposition \ref{P:push-loc}\eqref{P:push-loc4} it follows that the exceptional component $(\wh{\Upsilon}^+)^{-1}(D,p,J)$ is the rational curve  $E_p=\PP^1\hookrightarrow \TF_{B_p(D)}^\chi$ corresponding to the relative rank-one torsion-free sheaf $\I^+:=\pi_1^*(I^+)\otimes \O(\Delta)$ of Euler characteristic $\chi$ on $B_p(D)\times E_p\to E_p$, where $\pi_1$ is the projection onto the first factor, $I^+\in \TF_{B_p(D)}^{\chi-1}$ is free at $E_p\cap E_p^c$ with $(\st_{|E_p^c})_*(I^+_{E_p^c})=J$ and $\Delta\subset E_p\times E_p\subset B_p(D)\times E_p$ is the diagonal. Therefore, the restriction of the line bundle $\wh{\sigma}_{n+1}^*(\I_{g,n+1})$ on $(\wh{\Upsilon}^+)^{-1}(D,p,J)\cong E_p=\PP^1$ is equal to $p_{n+1}^*(\I^+)=p_{n+1}^*(\pi_1^*(I^+)\otimes \O(\Delta))=p_{n+1}^*\O(\Delta)=\O_{\PP^1}(p_{n+1})$  using that $p_{n+1}\in E_p$ and $\pi_1^*(I^+)$ is trivial on $E_p$. Part (b) follows from this.  

\vspace{0.1cm}

It remains to prove Part (c). Let $(\TF_{g,n}^\chi)^o$ be the open substack of $\TF_{g,n}^\chi$ whose geometric points are $(C,L)$ where $C\in \Mbargn(k)$ and $L$ is a line bundle on $C$. Note that $(\TF_{g,n}^\chi)^o$ is a big open subset of  $\TF_{g,n}^\chi$, i.e. its complement has codimension at least two. By the above description of $\wh{\Upsilon}^+$, it follows that $\wh{\Upsilon}^+$ is an isomorphism over $\Cbargn\times_{\Mbargn} (\TF_{g,n}^\chi)^o$. Moreover, the composition $\wh{\sigma}_{n+1}\circ (\wh{\Upsilon}^+)^{-1}$ is given on geometric points by 
$$
\begin{aligned}
\wh{\sigma}_{n+1}\circ (\wh{\Upsilon}^+)^{-1}: \Cbargn\times_{\Mbargn} (\TF_{g,n}^\chi)^o& \to \Cbargnbis\times_{\Mbargnbis} \TFo_{g,n+1}^\chi\\
(C,p,L) & \mapsto (\Upsilon^{-1}(C,p),p_{n+1}, \st^*(L)) 
\end{aligned}
$$
where $\Upsilon^{-1}$ is the inverse of the isomorphism $\Upsilon$ of \eqref{E:univ-Mg} and $\st:\Upsilon^{-1}(C,p)\to C$ is the map that forgets the last marked point and stabilizes. From the explicit form of this map, it is clear that $(\wh{\sigma}_{n+1}\circ (\wh{\Upsilon}^+)^{-1})^*(\I_{g,n+1})=\I_{g,n}$. This is equivalent to saying that 
$$\wh{\Upsilon}^+_*(\wh{\sigma}_{n+1}^*(\I_{g,n+1}))_{|\Cbargn\times_{\Mbargn} (\TF_{g,n}^\chi)^o}=(\I_{g,n})_{|\Cbargn\times_{\Mbargn} (\TF_{g,n}^\chi)^o}.$$
We now conclude that the two sheaves agree everywhere by \cite[Cor. 7.2]{Kass_2019} using that: $\wh{\Upsilon}^+_*(\wh{\sigma}_{n+1}^*(\I_{g,n+1}))$ is a family of rank-one torsion-free sheaves on $\Cbargn\times_{\Mbargn}\TF_{g,n}^\chi$ by Parts (a) and (b) and by \cite[Prop. 5.4]{estevespacini}; $(\TF_{g,n}^\chi)^o\subset \TF_{g,n}^\chi$ is a big open substack, $\TF_{g,n}^\chi$ is regular (see \cite[Sec. 8]{CMKVbirational}) and $\Cbargn\times_{\Mbargn}\TF_{g,n}^\chi$ is $G_1$ and $S_2$ (see \cite[Lemma 7.3]{Kass_2019}).
\end{proof}

\begin{lemma}\label{L:push-st}
Let $C$ be a geometric point of $\Mbargnbis$ and assume that the stabilization morphism $\st:C\to C^{\st}$ contracts an exceptional component $E\cong \PP^1$  to the point $q=\st(E)\in C^{\st}$. For any sheaf $I$ belonging to 
$$\TFo_C^\chi:=\{I\in \TF_C^\chi\: : \chi(I_E)\geq 0 \text{ and } \chi(I_{E^c})\geq \chi\},$$ 
we have that:
\begin{enumerate}[(i)]
    \item \label{L:push-st1} $\st_*(I)\in \TF_{C^{\st}}^\chi$;
    \item \label{L:push-st2} $R^1\st_*(I)=0$;
    \item \label{L:push-st3} there are exact sequences
    $$
    \begin{sis}
        & 0 \to \st_*(\leftindex_{E^c}I)\to  \st_*(I) \to \O_q^{\oplus h^0(E,I_E)}\to 0,\\
        & 0\to \st_*(I)\to  \st_*(I_{E^c}) \to \O_q^{\oplus h^1(E,\leftindex_{E}I)}\to 0.
    \end{sis}
    $$
\end{enumerate}
\end{lemma}
\begin{proof}
%Observe that $\st_{|E^c}:E^c\to C^{\st}$ is a partial normalization. 
Consider the two exact sequences (see \eqref{E:resY})
\begin{equation}\label{E:2ex-seq}
\begin{sis}
&    0 \to \leftindex_{E^c} I \to I \to I_E\to 0, \\
&    0 \to \leftindex_{E} I \to I \to I_{E^c}\to 0. \\
\end{sis}
\end{equation}
By applying $\st_*$ to the first exact sequence of \eqref{E:2ex-seq} and using that $R^1\st_*(\leftindex_{E^c} I)=0$ since $\st_{|E^c}:E^c\to C^{\st}$ is a partial normalization (and hence a finite map), we deduce that 
\begin{equation}\label{E:equa1}
   R^1\st_*(I)\cong R^1\st_*(I_E)=\O_q^{\oplus h^1(E,I_E)}=0 \text{ since } \chi(I_E)\geq 0.
\end{equation}
By applying $\st_*$ to the second exact sequence of \eqref{E:2ex-seq} and using \eqref{E:equa1}, we get the exact sequence
\begin{equation}\label{E:equa2}
    0\to \st_*(\leftindex_{E}I)=\O_q^{\oplus h^0(E,\leftindex_{E}I)}\to \st_*(I)\to \st_*(I_{E^c})\to R^1\st_*(\leftindex_{E}I)=\O_q^{\oplus h^1(E,\leftindex_{E}I)}\to  0.
\end{equation}
Since $\st_*(I_{E^c})$ is a rank-one torsion-free sheaf on $C^{\st}$ with $\chi(\st_*(I_{E^c}))=\chi(I_{E^c})$ (because $\st_{|E^c}:E^c\to C$ is a partial normalization map), we deduce from \eqref{E:equa2} that $\st_*(I)$ is a rank-one sheaf on $C^{\st}$ of Euler characteristic 
$$
\chi(\st_*(I))=\chi(I_{E^c})+\chi(\leftindex_{E} I)=\chi(I).
$$
Moreover, $\st_*(I)$ is torsion-free since 
\begin{equation}\label{E:equa3}
 \O_q^{\oplus h^0(E,\leftindex_{E}I)}=0 \text{ because } \chi(\leftindex_{E}I)=\chi(I)-\chi(I_{E^c})\leq 0.
\end{equation}
Finally, the second exact sequence in Part \eqref{L:push-st3} follows from \eqref{E:equa2}, while the first exact sequence follows by applying $\st_*$ to the first sequence of \eqref{E:2ex-seq} and using that $R^1\st_*(\leftindex_{E^c}I)=0$ since $\st_{|E^c}$ is a finite map.
\end{proof}

\begin{proposition}\label{P:push-loc}
Let $(D,p,J)$ be a geometric point of $\Cbargn\times_{\Mbargn}\TF_{g,n}$ over $k=\ov k$. Then 
\begin{enumerate}
    \item \label{P:push-loc1} If $p$ is not a node or a marked point, then $\wh{\Upsilon}$ is an isomorphism above $(D,q,J)$.
    \item \label{P:push-loc2} If $p$ is a node of $D$ and $J$ is free at $p$, then $\wh{\Upsilon}$ is an isomorphism above $(D,p,J)$.
    \item \label{P:push-loc3} If $p$ is a marked point of $D$, then, \'etale locally on $(D,p,J)$, the morphisms $\wh{\Upsilon}$, $\wh{\Upsilon}^+$ and $\wh{\Upsilon}^-$ are pull-backs via a smooth morphism of the following maps $\mu$, $\mu^+$ and $\mu^-$:
    \begin{equation}\label{E:loc-mark}
\begin{tikzcd}
  \AA^1\cong \frac{V_1\oplus V_1^*-\{0\}\oplus V_1^*}{\Gm} \arrow[hook, rd]\arrow[ddr, "\mu^+" swap]  & & \frac{V_1\oplus V_1^*-V_1\oplus \{0\}}{\Gm}\cong \AA^1 \arrow[hook', ld]\arrow[ddl, "\mu^-"] \\
     & \left[ V_1\oplus V_1^*/\Gm \right]\arrow[d, "\mu"] & \\
     & V_1\oplus V_1^*/\!\!/\Gm \cong \AA^1 & 
\end{tikzcd}
\end{equation}
    where $V_1$ is a one dimensional vector space and $\Gm$ is the scalar multiplication on $V_1$. In particular, $\wh{\Upsilon}^+$ and $\wh{\Upsilon}^-$ are isomorphisms above $(D,p,J)$.
\item \label{P:push-loc4} If $p$ is a node of $D$ and $J$ is not free at $p$, then, \'etale locally on $(D,p,J)$, the morphisms $\wh{\Upsilon}$, $\wh{\Upsilon}^+$ and $\wh{\Upsilon}^-$ are pull-backs via a smooth morphism of the following maps $\nu$, $\nu^+$ and $\nu^-$:
    \begin{equation}\label{E:loc-node}
\begin{tikzcd} [column sep=0em, row sep=1.5em]
  \Tot(\O_{\PP(V_2)}(-1)^{\oplus 2})\cong \frac{V_2\oplus V_2^*-\{0\}\oplus V_2^*}{\Gm} \arrow[hook, rd]\arrow[ddr, "\nu^+" swap]  & & \frac{V_2\oplus V_2^*-V_2\oplus \{0\}}{\Gm} \cong \Tot(\O_{\PP(V_2^*)}(-1)^{\oplus 2}) \arrow[hook', ld]\arrow[ddl, "\nu^-"] \\
     & \left[ V_2\oplus V_2^*/\Gm \right]\arrow[d, "\nu"] & \\
     & V_2\oplus V_2^*/\!\!/\Gm \cong \Spec\frac{k[x,y,u,v]}{(xy-uv)} & 
\end{tikzcd}
\end{equation}
    where $V_2$ is a one dimensional vector space and $\Gm$ is the scalar multiplication on $V_2$. In particular, $\wh{\Upsilon}^+$ and $\wh{\Upsilon}^-$ are, locally on $(D,p,J)$, small proper resolutions of the codomain and their fiber over $(D,p,J)$ is $\PP^1$.
\end{enumerate}
\end{proposition}
Note that the diagrams in \eqref{E:loc-mark} and \eqref{E:loc-node} are examples of variation of GIT. The diagram in %\eqref{E:loc-mark} is sometimes called a Thaddeus flop while the diagram 
\eqref{E:loc-node} is the famous \emph{Atiyah flop}. 
\begin{proof}
Part \eqref{P:push-loc1} is obvious.

Part \eqref{P:push-loc2}: let us compute the fiber $\wh{\Upsilon}^{-1}(D,p,J)$. First of all, $\Upsilon^{-1}(D,p)$ is the bubbling $B_p(D)$ of $D$ at $p$, i.e. the stable $n+1$-pointed curve obtained by inserting a $E\cong \PP^1$ at $p$ and putting the marked point $p_{n+1}$ on the exceptional component $E$. Therefore, 
$$
\wh{\Upsilon}^{-1}(D,p,J)=\{(B_p(D),I)\: : I\in \TFo^\chi_{B_p(D)} \text{ with } \st_*(I)=J\}.
$$
Take now $I\in \TFo^\chi_{B_p(D)}$ such that $\st_*(I)=J$. Since $J$ is free at $p$, we must have by Lemma \ref{L:push-st}\eqref{L:push-st3} that 
$$
h^0(E,I_E)>0 \text{ and } h^1(E,\leftindex_{E}I)>0,
$$
which is only possible if $I$ is free at $E\cap E^c$ and $I_E=\O_E$. This implies that $I$ is the pull-back via $\st$ of a sheaf on $(B_p(D))^{\st}=D$, and then, by the projection formula, we must have that $I=\st^*(J)$.  Hence, we get that 
$$
\wh{\Upsilon}^{-1}(D,p,J)=\{(B_p(D),\st^*(J))\}.
$$
This also shows that $\wh{\Upsilon}$ is an isomorphism locally above $(D,p,J)$, since a local inverse is given by the morphism $\Cbargn\times_{\Mbargn} \TF_{g,n}\to \TFo_{g,n+1}^\chi$ induced by the pull-back of the universal sheaf $\I_{g,n}$ on the family 
$$(\Upsilon^{-1})^*(\Cbargn)\times_{\Mbargn} \TF_{g,n}^\chi\to \Cbargn\times_{\Mbargn} \TF_{g,n}.$$

Part \eqref{P:push-loc3}: observe that $\Upsilon^{-1}(D,p)$ is the bubbling $C:=B_p(D)$ of $D$ at the point $p$, which is a marked point of $D$. Hence $C=E\cup E^c$, where $E\cong \PP^1$ is is the exceptional component, $E^c$ is canonically isomorphic to $D$ via the restriction of the stabilization morphism $\st:C\to D$ and $E$ and $E^c$ meet in a unique point $q$ which is a node of $C$. The stack $\TFo_C^\chi$ admits the following stratification
\begin{equation}\label{E:TF-strat1}
\TFo_C^\chi=\TFp_C^\chi \sqcup \TFm_C^\chi \sqcup (\TFo_C^\chi)^{cl},
\end{equation}
where:

$\bullet$ $\TFp_C^\chi$ and $\TFm_C^\chi$ are the open substacks given by 
$$
\begin{aligned}
    & \TFp_C^\chi:=\{I\in \TF_C^\chi\: : \chi(I_E)\geq 1 \text{ and } \chi(I_{E^c})\geq \chi\}, \\
    & \TFm_C^\chi:=\{I\in \TF_C^\chi\: : \chi(I_E)\geq 0 \text{ and } \chi(I_{E^c})\geq \chi+1\};
\end{aligned}
$$

$\bullet$ $(\TFo_C^\chi)^{cl}$ is the closed substack parametrizing sheaves that are not free at $q$ and hence it is given by 
$$
(\TFo_C^\chi)^{cl}:=\{K\oplus \O_E(-1)\: : K\in \TF_D^\chi\}.
$$
Using Lemma \ref{L:push-st}, the map $\st_*:\TFo_C^\chi\to \TF_D^\chi$ is given on the above three strata by 
\begin{equation}\label{E:st-strat1}
\begin{aligned}
    \st_*:\TFp_C^\chi& \longrightarrow \TF_D^\chi \\
    I &\mapsto \st_*(I_{E^c})
\end{aligned}
\quad \text{ and } \quad
\begin{aligned}
    \st_*:\TFm_C^\chi& \longrightarrow \TF_D^\chi\\
    I &\mapsto \st_*(\leftindex_{E^c}I)
\end{aligned}
\quad \text{ and } \quad
\begin{aligned}
    \st_*:(\TFo_C^\chi)^{cl}& \longrightarrow \TF_D^\chi\\
    K\oplus \O_E(-1) &\mapsto K.
\end{aligned}
\end{equation}
By putting together \eqref{E:TF-strat1} and \eqref{E:st-strat1}, we deduce that 
\begin{equation}\label{E:inv-Ups1}
 \wh{\Upsilon}^{-1}(D,p,J)=\{(C,I)\: : J\oplus \O_E(-1)\in \ov{\{I\}}\}.   
\end{equation}

Hence, the map $\wh{\Upsilon}$, \'etale locally on $(D,p,J)$ and up to smooth factors and $\Gm$-rigidification, is modeled on the map of deformation functors
(by \cite[Theorem~4.19]{AHR}) \begin{equation}\label{E:Def-Ups1}
\frac{\Def_q(C,J\oplus \O_E(-1))}{\Gm}\xrightarrow{\wh{\Upsilon}_*} \Def_p(D,p,J),
\end{equation}
where $\Def_q(C,J\oplus \O_E(-1))$ is the local deformation functor of the pair $(C,J\oplus \O_E(-1))$ at $q$, the action of $\Gm$ is induced by the scalar multiplication of $\Gm$ on $E$, and $\Def_p(D,p,J)$ is the local deformation functor of the triple $(D,p,J)$ at $p$. Moreover, since the morphism $\wh{\Upsilon}$ sits above the morphism $\Upsilon$, the map in \eqref{E:Def-Ups1} fits into the following commutative diagram of local deformation functors 
\begin{equation}\label{E:Def-diag1}
\begin{tikzcd}
\displaystyle \frac{\Def_q(C,J\oplus \O_E(-1))}{\Gm} \arrow["\wh{\Upsilon}_*", r] \arrow[d] &   \Def_p(D,p,J) \arrow[d, "\cong"]   \\
    \Def_q(C) \arrow["\Upsilon_*", "\cong" swap, r]  &   \Def_p(D,p)  
\end{tikzcd}
\end{equation}
where the vertical maps are forgetful morphism, the bottom horizontal map is an isomorphism since $\Upsilon$ is an isomorphism and the vertical right map is an isomorphism since $J$ is free at $p$ (because $p$ is a smooth point of $D$) and hence its local deformations are trivial. Using the above diagram and the description of the local deformations of a pair consisting of a nodal curve and a rank-one torsion-free sheaf in \cite[Lemmas 3.14, 5.6]{CMKVlocal}, we deduce that the miniversal deformation spaces for the map in \eqref{E:Def-Ups1} are the completion at the origin of the following map 
\begin{equation}\label{E:model-Ups1}
\left[\frac{\Spec k[w,\ov w]}{\Gm}\right]\to \Spec k[t],
\end{equation}
where $\Gm$ acts via $\lambda\cdot (w,\ov w):=(\lambda w, \lambda^{-1}\ov w)$ and the map is given by identifying $t$ with 
$w\ov w$. If we set $V_1:=\langle w\rangle$, then the map in \eqref{E:model-Ups1} is exactly the map $\mu$ in the statement of Part \eqref{P:push-loc3}. Moreover, it follows again from loc. cit. that the local model for morphism $\wh{\Upsilon}^+$ (resp. $\wh{\Upsilon}^-$) is the restriction of \eqref{E:model-Ups1} to the open substack $\{w\neq 0\}$ (resp.  $\{\ov w\neq 0\}$), which is exactly the map $\mu^+$ (resp. $\mu^-$) in the statement of Part~\eqref{P:push-loc3}. The proof of Part~\eqref{P:push-loc3} is now complete.

Part \eqref{P:push-loc4}: observe that $\Upsilon^{-1}(D,p)$ is the bubbling $C:=B_p(D)$ of $D$ at the node $p$ of $D$.
 Hence $C=E\cup E^c$, where $E\cong \PP^1$ is is the exceptional component, $E^c$ is canonically isomorphic to the partial normalization of $D$ at $p$ via the restriction of the stabilization morphism $\st:C\to D$ and $E$ and $E^c$ meet in two points $\{q_1,q_2\}$ which are nodes of $C$.  The stack $\TFo_C^\chi$ admits the following stratification
\begin{equation}\label{E:TF-strat2}
\TFo_C^\chi=(\TFo_C^\chi)^{op}\sqcup (\TFp_C^\chi)^{cl} \sqcup (\TFm_C^\chi)^{cl} \sqcup (\TFo_C^\chi)^{cl},
\end{equation}
where:

$\bullet$ $(\TFo_C^\chi)^{op}$ is the open substack consisting of sheaves that are free at $q_1$ and $q_2$ and they are trivial on $E$, or equivalently
$$
(\TFo_C^\chi)^{op}=\TFp_C^\chi\cap \TFm_C^\chi=\{I\in \TF_{C}^\chi : \chi(I_{E})= 1 \text{ and } \chi(I_{E^c})= \chi+1\};
$$

$\bullet$ $(\TFp_C^\chi)^{cl}$ and $(\TFm_C^\chi)^{cl}$ are the locally closed substacks given by 
$$
\begin{sis}
    & (\TFp_C^\chi)^{cl}:=\TFp_C^\chi- (\TFo_C^\chi)^{op}=\{I\in \TF_{C}^\chi : \chi(I_{E})\geq 1 \text{ and } \chi(I_{E^c})= \chi\}, \\
    & (\TFm_{C}^\chi)^{cl}:=\TFm_C^\chi- (\TFo_C^\chi)^{op}=\{I\in \TF_{C}^\chi : \chi(I_{E})=0 \text{ and } \chi(I_{E^c})\geq \chi+1\}; \\
\end{sis}
$$

$\bullet$ $(\TFo_C^\chi)^{cl}$ is the closed substack parametrizing sheaves that are not free at $\{q_1,q_2\}$ and hence it is given by 
$$
(\TFo_C^\chi)^{cl}:=\{K\oplus \O_E(-1)\: : K\in \TF_{E^c}^\chi\}.
$$
In order to describe the morphisms $\st_*$, we partition the codomain into an open and closed substack as follows
$$
\TF_{D}^\chi=(\TF_D^\chi)^{op}\sqcup (\TF_D^\chi)^{cl},
$$
given by, respectively
$$
\begin{sis}
  &(\TF_D^\chi)^{op}:=\{J\in \TF_D^\chi : J \text{ is free at } p\}, \\
  &(\TF_D^\chi)^{cl}:=\{J\in \TF_D^\chi : J \text{ is not free at } p\}.
\end{sis}
$$
Since $(\TFo_C^\chi)^{op}$ consists of sheaves that are free at $q_1$ and $q_2$ and are trivial on $E$, the push-forward $\st_*$ induces an isomorphism
\begin{equation}\label{E:st-open}
   \st_*:  (\TF_C^\chi)^{op} \xrightarrow{\cong} (\TF_D^\chi)^{op},
\end{equation}
with inverse given by the pull-back map $\st^*$.
Using Lemma \ref{L:push-st}, the map $\st_*$ on the other three strata is given by  
\begin{equation}\label{E:st-strat2}
\begin{aligned}
    \st_*:(\TFp_C^\chi)^{cl}& \longrightarrow (\TF_D^\chi)^{cl} \\
    I &\mapsto \st_*(I_{E^c})
\end{aligned}
\quad \text{ and } \quad
\begin{aligned}
    \st_*:(\TFm_C^\chi)^{cl}& \longrightarrow (\TF_D^\chi)^{cl}\\
    I &\mapsto \st_*(\leftindex_{E^c}I)
\end{aligned}
\quad \text{ and } \quad
\begin{aligned}
    \st_*:(\TFo_C^\chi)^{cl}& \longrightarrow (\TF_D^\chi)^{cl}\\
    K\oplus \O_E(-1) &\mapsto \st_*(K).
\end{aligned}
\end{equation}
Now observe that, since $\st_{|E^c}:E^c\to D$ is the partial normalization of $D$ at the node $p$, we have an isomorphism 
\begin{equation}\label{E:iso-clst}
 (\st_{|E^c})_*:\TF^\chi_{E^c} \xrightarrow{\cong} (\TF_D^\chi)^{cl}.
\end{equation}
Hence, since $J\in (\TF_D^\chi)^{cl}$ by assumption, there exists a unique $\wt J\in \TF_{E^c}^\chi$ such that $\st_*(\wt J)=J$.

By combining \eqref{E:st-open} and \eqref{E:st-strat2}, we deduce that 
\begin{equation}\label{E:inv-Ups2}
 \wh{\Upsilon}^{-1}(D,p,J)=\{(C,I)\: : \wt J\oplus \O_E(-1)\in \ov{\{I\}}\}.   
\end{equation}

Hence, the map $\wh{\Upsilon}$, \'etale locally on $(D,p,J)$ and up to smooth factors and $\Gm$-rigidification, is modeled on the map of deformation functors (by \cite[Theorem~4.19]{AHR})
\begin{equation}\label{E:Def-Ups2}
\frac{\Def_{\{q_1,q_2\}}(C,\wt J\oplus \O_E(-1))}{\Gm}\xrightarrow{\wh{\Upsilon}_*} \Def_p(D,p,J)=\Def_p(D,p)\times_{\Def_p(D)}\Def_p(D,J),
\end{equation}
where $\Def_{\{q_1,q_2\}}(C,\wt J\oplus \O_E(-1))$ is the local deformation functor of the pair $(C,J\oplus \O_E(-1))$ at the two points $\{q_1,q_2\}$, the action of $\Gm$ is induced by the scalar multiplication of $\Gm$ on $E$, and $\Def_p(D,p,J)$ is the local deformation functor of the triple $(D,p,J)$ at $p$ which is the fiber product of the local deformation functor $\Def_p(D,p)$ of the pair $(D,p)$ at $p$ with the local deformation functor of $(D,J)$ at $p$ over the local deformation functor of $\Def_p(D)$ of $D$ at $p$.  Moreover, since the morphism $\wh{\Upsilon}$ lies above the morphism $\Upsilon$, the map in \eqref{E:Def-Ups2} fits into the following commutative diagram of local deformation functors 
\begin{equation}\label{E:Def-diag2}
\begin{tikzcd}
\displaystyle \frac{\Def_{\{q_1,q_2\}}(C,\wt J\oplus \O_E(-1))}{\Gm} \arrow["\wh{\Upsilon}_*", r] \arrow[d] &   \Def_p(D,p,J)=\Def_p(D,p)\times_{\Def_p(D)}\Def_p(D,J) \arrow[d]   \\
    \Def_{\{q_1,q_2\}}(C) \arrow["\Upsilon_*", "\cong" swap, r]  &   \Def_p(D,p)  
\end{tikzcd}
\end{equation}
where the vertical maps are forgetful morphism, the bottom horizontal map is an isomorphism since $\Upsilon$ is an isomorphism. Using the above diagram and the description of the local deformations of a pair consisting of a nodal curve and a rank-one torsion-free sheaf in \cite[Lemmas 3.14, 5.6]{CMKVlocal}, we deduce that the miniversal deformation spaces for the diagram  in \eqref{E:Def-diag2} are the completions at the origin of the following diagram 
\begin{equation}\label{E:model-Ups2}
\begin{tikzcd}
\displaystyle \left[\frac{\Spec k[w_1,\ov w_1, w_2, \ov w_2]}{\Gm}\right]  \arrow[rr] \arrow[dr]&& \displaystyle \Spec\frac{k[x,y,u,v]}{(xy-uv)} \arrow[ld]\\
& \Spec k[x,y] & 
\end{tikzcd}
\end{equation}
where $\Gm$ acts via $\lambda\cdot (w_1,\ov w_1,w_2, \ov w_2):=(\lambda w_1, \lambda^{-1}\ov w_1, \lambda w_2, \lambda^{-1}\ov w_2)$, and the maps are given by the rules 
$$x=w_1\ov w_1, \ y=w_2\ov w_2, \ u=w_1\ov w_2,\ v=w_2\ov w_1.$$
If we set $V_2:=\langle w_1,w_2\rangle$, then the map in \eqref{E:model-Ups2} is exactly the map $\nu$ in the statement of Part \eqref{P:push-loc4}. Moreover, it follows again from loc. cit. that the local model for morphism $\wh{\Upsilon}^+$ (resp. $\wh{\Upsilon}^-$) is the restriction of \eqref{E:model-Ups1} to the open substack $\{(w_1,w_2)\neq (0,0)\}$ (resp.  $\{(\ov w_1,\ov w_2)\neq (0,0)\}$), which is exactly the map $\nu^+$ (resp. $\nu^-$) in the statement of Part \eqref{P:push-loc4}. The proof of Part \eqref{P:push-loc4} is now complete. 
\end{proof}

\subsection{The case of compactified universal Jacobians}

We start by constructing some natural maps between the stability domain $\Dgn$ and the stability domain $\Dgnbis$. 
For that purpose, we need to introduce the \emph{extended stability domain} $\Dgnwh$ of type $(g,n)$ as
\begin{equation}\label{E:Dgn-ext}
\Dgnwh:=\{(e;h,A)\: : 
e\in \NN_{>0}, 0\leq h \leq g-e+1, A\subseteq [n],2h-2+e+|A|\geq 0, 2g-2h-e+|A^c|\geq 0\},
\end{equation}
where $[n]:=\{1,\ldots, n\}$. The set $\Dgnwh$ is endowed with a complementary involution $(-)^c$ and with triangles, which are defined exactly as for $\Dgn$, see \eqref{E:Dgn} and what follows. 

Note that
$$
\Dgnwh=\Dgn\sqcup \{(2;0,\emptyset),(2;0,\emptyset)^c\} \sqcup \bigsqcup_{i=1}^n \{(1;0,\{i\}), (1;0,\{i\})^c\}.
$$
Moreover, the triangles of $\Dgnwh$ that are not contained in $\Dgn$ are the ones of the form
$$
\Delta=[(2;0,\emptyset),(e;h,A),(e;h,A)^c].
$$
for some $(e;h,a) \in \Dgn$.
\begin{remark}\label{R:vine-trian2}
As in Remark \ref{R:vine-trian}, pairs of complementary elements of $\Dgnwh$ correspond to vine strata of the moduli stack $\Mbargn^{ss}$ parameterizing semistable $n$-pointed curves of genus $g$, and triangles in $\Dgnwh$ correspond to triangular strata in $\Mbargn^{ss}$.
\end{remark}

\begin{definition}\label{E:Dg-n+1}
 \noindent 
 \begin{enumerate}
     \item For any $1\leq i \leq n$, we define the following map
\begin{equation}\label{E:xi-i}
\begin{aligned}
\xi_i:\Dgn & \longrightarrow \Dgnbis\\
(e;h,A)& \mapsto 
\begin{cases}
(e;h,A\sqcup \{n+1\}) & \text{ if } i\in A,\\
(e;h,A) & \text{ if } i\not\in A.\\
\end{cases}
\end{aligned}
\end{equation}
\item We define the map
\begin{equation}\label{E:omega}
\begin{aligned}
\varpi:\Dgnbis & \longrightarrow \Dgnwh\\
(e;h,A)& \mapsto 
\begin{cases}
(e;h,A- \{n+1\}) & \text{ if } n+1\in A,\\
(e;h,A) & \text{ if } n+1\not\in A.\\
\end{cases}
\end{aligned}
\end{equation}
 \end{enumerate}
\end{definition}
The maps $\xi_i$ and $\varpi$ are compatible with the complementary involution and they send triangles into triangles. 
Moreover, they satisfy the relation $\varpi\circ \xi_i=\id$ for any $1\leq i \leq n$. Geometrically, using the bijections of Remarks \ref{R:vine-trian} and \ref{R:vine-trian2}, the map $\xi_i$ is induced by the $i$-section map $\sigma_i:\Mbargn\to \Cbargn\cong \Mbargnbis$ while the map $\varpi$ is induced by the  morphism $\Mbargnbis\to \Mbargn^{ss}$ that forgets the last marked point.

Using the above maps, we can relate the V-functions of type $(g,n)$ (see Definition~\ref{D:Sigmagn}) with those of type $(g,n+1)$.

\begin{lemma-definition}\label{E:Vfun-n+1}
Fix $\chi \in \Z$.
  We have the following well-defined maps (for any $1\leq i \leq n$):
\begin{equation}\label{E:Xi-i}
\begin{aligned}
\Xi_i:\Sigma_{g,n+1}^\chi & \longrightarrow \Sigma_{g,n}^\chi\\
\tau & \mapsto \Xi_i(\tau):=\tau\circ \xi_i
\end{aligned}
\end{equation}
\begin{equation}\label{E:Omega}
\begin{aligned}
\Omega:\Sigma_{g,n}^{\chi} & \longrightarrow \Sigma_{g,n+1}^\chi\\
\sigma & \mapsto \Omega(\sigma):=\wh \sigma\circ \varpi,
\end{aligned}
\end{equation}
where $\wh \sigma:\Dgnwh\to \Z$ is defined by 
$$
\begin{sis}
&   \wh\sigma_{|\Dgn}=\sigma,\\  
& \wh \sigma((2;0,\emptyset))=\wh \sigma((1;0,\{j\}))=0 \text{ for any } 1\leq j \leq n,\\
& \wh \sigma((2;0,\emptyset)^c)=\wh \sigma((1;0,\{j\})^c)=\chi \text{ for any } 1\leq j \leq n.
\end{sis}
$$
The above maps satisfy the following properties
\begin{enumerate}[(a)]
    \item \label{E:Vfun-n+1a} $\Xi_i\circ \Omega=\id$;
    \item \label{E:Vfun-n+1b} $\D(\Xi_i(\tau))=\xi_i^{-1}(\D(\tau))$;
    \item \label{E:Vfun-n+1c} $\D(\Omega(\sigma))=\varpi^{-1}\left(\D(\sigma)\cup (\Dgnwh- \Dgn)\right)$.
\end{enumerate}
\end{lemma-definition}
\begin{proof}
    The  map $\Xi_i$ is well-defined and it satisfies \eqref{E:Vfun-n+1b} since $\xi_i$ is compatible with the complementary involution and it sends triangles of $\Dgn$ into triangles of $\Dgnbis$.

     The  map $\Omega$ is well-defined and it satisfies \eqref{E:Vfun-n+1c} since:
     \begin{itemize}
     \item $\wh \sigma$ satisfies the analogous of the properties of Definition \ref{D:Sigmagn} on $\Dgnwh$ with degeneracy subset $\D(\wh \sigma)=\D(\sigma)\cup (\Dgnbis- \Dgn)$, as it follows from the definition of $\wh \sigma$ and the fact that the unique triangles of $\Dgnwh$ not contained in $\Dgn$ are the ones of the form $\Delta=[(2;0,\emptyset),(e;h,A),(e;h,A)^c]$, for which Condition \eqref{E:triaUni} reduces to Condition \eqref{E:sumUni} for the complementary pair $\{(e;h,A),(e;h,A)^c\}$.
    \item $\varpi$ is compatible with the complementary involution and it sends triangles of $\Dgnbis$ into triangles of $\Dgnwh$.
     \end{itemize}
     
    Finally, Property \eqref{E:Vfun-n+1a} follows from  the definitions of $\Xi_i$ and $\Omega$, together with the fact that $\varpi\circ \xi_i=\id$. 
\end{proof}

\begin{remark}\label{R:Pol-n+1}
The maps $\Xi_i$ and $\Omega$ preserve classical V-functions. Indeed, consider the following maps (which, by slight abuse of notation, we denote with the same letters):
\begin{equation}\label{E:Xi-i-Pol}
\begin{aligned}
\Xi_i:\PicRel_{g,n+1}^\chi(\RR)  & \longrightarrow \PicRel_{g,n}^\chi(\RR)\\
\beta\omega_{\pi}+\sum_{j=1}^{n+1} \alpha_j\Sigma_j+\sum_{(h,A)\in \Bgnbis}\gamma_{(h,A)}\O(C_{(h,A)}) & \mapsto \beta\omega_{\pi}+\sum_{j=1}^n \alpha_j\Sigma_j+\alpha_{n+1}\Sigma_{i}+\sum_{(h,A)\in \Bgn}\gamma_{(h,A)}\O(C_{(h,A)})
\end{aligned}
\end{equation}
\begin{equation}\label{E:Omega-Pol}
\begin{aligned}
\Omega:\PicRel_{g,n}^\chi(\RR)  & \longrightarrow \PicRel_{g,n+1}^\chi(\RR)\\
\beta\omega_{\pi}+\sum_{j=1}^{n} \alpha_j\Sigma_j+\sum_{(h,A)\in \Bgn}\gamma_{(h,A)}\O(C_{(h,A)}) & \mapsto \beta\omega_{\pi}+\sum_{j=1}^n \alpha_j\Sigma_j+\sum_{(h,A)\in \Bgn}\gamma_{(h,A)}\O(C_{(h,A)})+\\
& +\sum_{j=1}^n(\beta-\alpha_j)\O(C_{(0,\{j,n+1\})}),
\end{aligned}
\end{equation}
for any $1\leq i \leq n$; where we have used that 
$$
\Bgnbis=\Bgn \sqcup \bigsqcup_{j=1}^n \{(0,\{j,n+1\}), (0,\{j,n+1\})^c\}.
$$
Note that $\Xi_i\circ \Omega=\id$. Using Formula \eqref{E:for-deg}, it follows that:
$$\sigma(\Xi_i(L))=\Xi_i(\sigma_L) \quad \text{ and } \quad \sigma(\Omega(L))=\Omega(\sigma_L),$$
which shows that $\Xi_i$ and $\Omega$ send classical V-functions into classical V-functions.
\end{remark}

Using the map $\Omega$, we can compute the inverse image of a compactified Jacobian stack over $\Mbargn$ via the morphism $\wh{\Phi}$ of Theorem \ref{T:univ-TF}.

\begin{proposition}\label{P:inv-Phi}
 For any $\sigma\in \Sigma_{g,n}^\chi$, we have that 
 $$
 \wh{\Phi}^{-1}(\ov\J_{g,n}(\sigma))=\ov \J_{g,n+1}(\Omega(\sigma)).
 $$
\end{proposition}
\begin{proof}
Consider the V-stability conditions $\s^{\sigma}\in \VStab_{g,n}^\chi$ and $\s^{\Omega(\sigma)}\in \VStab_{g,n+1}^\chi$,  as in Proposition \ref{P:Stabgn}. 

First of all, we observe that $\ov \J_{g,n+1}(\Omega(\sigma))\subset \TFo_{g,n+1}^\chi$, since if $E$ is an exceptional component of $C\in \Mbargnbis(k)$ then $\type_C(E)=(2;0,\{n+1\})$ or $(1;0,\{j,n+1\})$ for some $1\leq j \leq n$, and hence  
$$
\s^{\Omega(\sigma)}(E)=\Omega(\sigma)(\type_C(E))=0 \text{ and } \s^{\Omega(\sigma)}(E^c)=\Omega(\sigma)(\type_C(E^c))=\chi,
$$
by definition of $\Omega(\sigma)$.

Therefore, it remains to show that for any geometric point $(C,I)$ of $\TFo_{g,n+1}^\chi$, we have that 
\begin{equation}\label{E:2-semistab}
  \begin{aligned}
      & I \text{ is } \s^{\Omega(\sigma)}(C)-\text{semistable, i.e.} \\
      & \chi(I_Z)\geq \Omega(\sigma)(\type_C(Z)) \text{ for any } Z\in \BCon(C)
      \end{aligned} 
      \Leftrightarrow 
      \begin{aligned}
      & \st_*(I) \text{ is } \s^{\sigma}(C^{\st})-\text{semistable, i.e. } \\
      & \chi(\st_*(I)_Y)\geq \sigma(\type_{C^{\st}}(Y)) \text{ for any } Y\in \BCon(C^{\st}).
      \end{aligned}
\end{equation}

 The biconnected subcurves of $C$ and of its stabilization $C^{\st}$, and their types, are related by the following commutative diagram with surjective horizontal arrows
\begin{equation}\label{E:2BCon}
\begin{tikzcd}
Z \dar[maps to, rr] && \un{\st}(Z):=\st(Z) \\
  \BCon(C)-\{E,E^c\}   \arrow[rr, twoheadrightarrow, "\un{\st}"] \arrow[d, "\type_C(-)"']&&   \BCon(C^{\st}) \arrow[d, "\type_{C^{\st}}(-)"]  \\
    \varpi^{-1}(\Dgn) \arrow[rr, twoheadrightarrow, "\varpi"]&&  \Dgn  
\end{tikzcd}
\end{equation}
where $E$ is the exceptional component of $C$ (if there is one) or $\emptyset$ (if $\st$ is an isomorphism). Moreover, the fibers of the map $\un{\st}$ are given as follows
$$
\un{\st}^{-1}(Y)=
\begin{cases}
  \{\st^{-1}(Y)\} & \text{ if either } E=\emptyset \text{  or } \st(E)\not\in Y\cap Y^c,\\
  \{\st^{-1}(Y), \st^{-1}(Y)-E\} & \text{ if }  \st(E)\in Y\cap Y^c.\\
\end{cases}
$$
Observe now that, since $(C,I)\in \TFo_{g,n+1}^\chi$ by assumption, the sheaf $I$ satisfies $\chi(I_E)\geq 0=\Omega(\sigma)(\type_C(E))$ and $\chi(I_{E^c})\geq \chi=\Omega(\sigma)(\type_C(E^c))$, so that the left hand side of \eqref{E:2-semistab} holds true for $E$ and $E^c$. Therefore, the equivalence  in \eqref{E:2-semistab} follows from the surjectivity of $\un{\st}$ and the following 

\un{Claim:} For any $Y\in \BCon(C^{\st})$, we have that 
$$
\chi(\st_*(I)_Y)=\min_{Z\in \un{\st}^{-1}(Y)}\{\chi(I_Z)\}.
$$

Indeed, the Claim is obvious if either $E=\emptyset$, i.e. $\st$ is an isomorphism, or $\st(E)\not \in Y\cap Y^c$. Suppose now that $\st(E)\in Y\cap Y^c$, so that 
$\un{\st}^{-1}(Y)=\{\st^{-1}(Y)-E:=\wt{Y},\st^{-1}(Y)=\wt{Y}\cup E\}$. Denote by $N$ the node in between $E$ and $\wt{Y}$ and set 
$$
\delta_N(I)=
\begin{cases}
0 &\text{ if } I \text{ is free at } N,\\
1 &\text{ if } I \text{ is not free at } N.\\
\end{cases}
$$
We will distinguish two exhaustive (but not mutually exclusive) cases:

$\bullet$ Case I: $\chi(I_E)+\delta_N(I)\geq 1$. 

In this case, we have that 
\begin{equation}\label{E:Iineq}
\chi(I_{\wt Y\cup E})=\chi(I_{\wt Y})+\chi(I_E)+\delta_N(I)-1\geq \chi(I_{\wt Y}).
\end{equation}
We will show that
\begin{equation}\label{E:push-res1}
    \st_*(I)_Y=\st_*(I_{\wt Y}),
\end{equation}
which, together with the fact that $R^1\st_*(I_{\wt Y})=0$ (since $R^1\st_*(I)=0$),  implies the desired equality 
$$
\chi(\st_*(I)_Y)=\chi(\st_*(I_{\wt Y}))=\chi(I_{\wt Y})=\min_{Z\in \un{\st}^{-1}(Y)}\{\chi(I_Z)\}.
$$
Consider the exact sequence (see \eqref{E:resY})
\begin{equation}\label{E:exa-se1}
0\to \leftindex_{(\wt Y)^c}I\to I \to I_{\wt Y}\to 0.
\end{equation}
Since $\leftindex_{(\wt Y)^c}I=I_{(\wt Y)^c}(-(\wt{Y}\cap \wt{Y}^c\cap \NF(I)^c)$ by \eqref{E:2restY}, we have that 
$$
(\leftindex_{(\wt Y)^c}I)_E=
\begin{cases}
I_E(-N) & \text{ if } I \text{ is free at } N,\\
I_E & \text{ if } I \text{ is not free at } N.\\
\end{cases}
$$
In any case, our assumption is equivalent to the fact $\chi((\leftindex_{(\wt Y)^c}I)_E)\geq 0$, which then implies that $R^1\st_*(\leftindex_{(\wt Y)^c}I)=0$ by \eqref{E:equa1}. By taking the push-forward via $\st$ of the exact sequence in \eqref{E:exa-se1} and using the above vanishing, we obtain the surjection 
$$
\st_*(I) \twoheadrightarrow \st_*(I_{\wt Y}). 
$$
Therefore, using that $\st_*(I_{\wt Y})$ is also torsion-free because $\st_{|\wt Y}$ is an isomorphism,
we deduce that \eqref{E:push-res1} holds.

$\bullet$ Case II: $\chi(I_E)+\delta_N(I)\leq 1$.

In this case, we have that 
\begin{equation}\label{E:IIineq}
\chi(I_{\wt Y\cup E})=\chi(I_{\wt Y})+\chi(I_E)+\delta_N(I)-1\leq \chi(I_{\wt Y}).
\end{equation}
We will show that
\begin{equation}\label{E:push-res2}
    \st_*(I)_Y=\st_*(I_{\wt Y\cup E}),
\end{equation}
which, together with the fact that $R^1\st_*(I_{\wt Y\cup E})=0$ (since $R^1\st_*(I)=0$), will imply the desired equality
$$
\chi(\st_*(I)_Y)=\chi(\st_*(I_{\wt Y\cup E}))=\chi(I_{\wt Y\cup E})=\min_{Z\in \un{\st}^{-1}(Y)}\{\chi(I_Z)\}.
$$
Consider the exact sequence (see \eqref{E:resY})
\begin{equation}\label{E:exa-se2}
0\to \leftindex_{(\wt Y\cup E)^c}I\to I \to I_{\wt Y\cup E}\to 0.
\end{equation}
Since the restriction of $\st$ to $(\wt Y\cup E)^c$ is an isomorphism, we have that $R^1\st_*(\leftindex_{(\wt Y\cup E)^c}I)=0$, which then implies that we have a surjection 
\begin{equation}\label{E:suj-push}
\st_*(I) \twoheadrightarrow \st_*(I_{\wt Y\cup E}). 
\end{equation}
We now apply our assumption in order to get  
$$\chi(I_{\wt Y\cup E})=\chi(I_{\wt Y})+\chi(I_E)+\delta_N(I)-1\leq \chi(I_{\wt Y}).$$
Then we apply \eqref{E:equa3} in order to deduce that 
$\st_*(I_{\wt Y\cup E})$ is torsion-free, which together with \eqref{E:suj-push} implies \eqref{E:push-res2}. 
\end{proof}

\begin{corollary}\label{C:univ-cJ}
Let $\sigma\in \Sigma_{g,n}^\chi$. 
\begin{enumerate}[(i)]
    \item \label{C:univ-cJ1} There is a commutative diagram 
\begin{equation}\label{E:univ-cJ}
\begin{tikzcd}
  \ov\J_{g,n+1}(\Omega(\sigma)) \arrow["\wh{\Phi}", rd]  \arrow[rr, "\wh{\Upsilon}"]& & \Cbargn\times_{\Mbargn} \ov\J_{g,n}(\sigma) \arrow[ld, "\wh{\pi}"']   \\
     & \ov\J_{g,n}(\sigma)  \arrow[ur, bend right=30, "\wh{\sigma_i}"']\arrow[ul, bend left=30, "\wh{\sigma_i}'"]& 
\end{tikzcd}
\end{equation}
lying over the diagram in \eqref{E:univ-Mg}, whose morphisms are defined on geometric points by
$$
\begin{sis}
& \wh{\Phi}(C,I)=(C^{\st},\st_*(I)),\\   
& \wh{\Upsilon}(C,I)=(C^{\st},\st(p_{n+1}),\st_*(I)),\\
& \wh{\sigma_i}'(C,I)=(B_{p_i}(C),\st^*(I)).
\end{sis}
$$
Moreover, the morphism $\wh{\Upsilon}$ is universally closed, $\Theta$-complete and S-complete and the morphisms $\wh{\sigma_i}'$ (for $1\leq i \leq n$) are sections of $\wh{\Phi}$. 
\item \label{C:univ-cJ2}
The morphism $\wh{\Upsilon}$ induces the following isomorphisms over $\Cbargn\times_{\Mbargn} \ov\J_{g,n}(\sigma)$: 
$$
\begin{sis}
& \PP(\I_{g,n})\cong \ov\J_{g,n+1}(\Omega(\sigma))^+:=\left\{\begin{aligned}
    & I\in \ov\J_{g,n+1}(\Omega(\sigma)): \chi(I_E)>0\\ 
    & \text{ for any exceptional component } E 
    \text{ of every } C\in \Mbargnbis
    \end{aligned}\right\},\\
    & \PP(\I_{g,n}^\vee) \cong \ov\J_{g,n+1}(\Omega(\sigma))^-:=\left\{\begin{aligned}
    & I\in \ov\J_{g,n+1}(\Omega(\sigma)): \chi(I_{E^c})>\chi\\ 
    & \text{ for any exceptional component } E 
    \text{ of every } C\in \Mbargnbis
    \end{aligned}\right\},\\
\end{sis}
$$
under which the tautological line bundle $\O(1)$ is isomorphic to, respectively,  
$$
\O_{\PP(\I_{g,n})}(1)\cong \wh{\sigma}_{n+1}^*(\I_{g,n+1})\text{ and } \O_{\PP(\I_{g,n}^\vee)}(1)\cong \wh{\sigma}_{n+1}^*(\I_{g,n+1}^\vee).
$$
\end{enumerate}
\end{corollary}
\begin{proof}
   It follows by combining Theorem \ref{T:univ-TF} with Proposition \ref{P:inv-Phi}.
\end{proof}

\begin{remark}\label{R:Omegapm}
Consider the following functions 
$$
\begin{aligned}
   \Omega(\sigma)^+:\DD_{g,n+1}& \longrightarrow \ZZ\\
     (e;h,A) & \mapsto 
     \begin{cases}
         \Omega(\sigma)(e;h,A) & \text{ if either } (e;h,A)\not\in \D(\Omega(\sigma)) \text{ or } n+1\not \in A, \\
         \Omega(\sigma)(e;h,A)+1 & \text{ if  } (e;h,A)\in \D(\Omega(\sigma)) \text{ and } n+1 \in A. \\
         \end{cases}\\
         \Omega(\sigma)^-:\DD_{g,n+1}& \longrightarrow \ZZ\\
     (e;h,A) & \mapsto 
     \begin{cases}
         \Omega(\sigma)(e;h,A) & \text{ if either } (e;h,A)\not\in \D(\Omega(\sigma)) \text{ or } n+1\in A, \\
         \Omega(\sigma)(e;h,A)+1 & \text{ if  } (e;h,A)\in \D(\Omega(\sigma)) \text{ and } n+1 \not \in A. \\
         \end{cases}\\
\end{aligned}
$$
Arguing as in the proof of Proposition~\ref{prop:empty-is-maximum-deg}, we deduce that $\Omega(\sigma)^+$ and $\Omega(\sigma)^-$ are well-defined. Moreover, they are general elements in $\Sigma_{g,n+1}^\chi$ such that $\Omega(\sigma)^+,\Omega(\sigma)^-\geq \Omega(\sigma)$ (recall Definition~\ref{D:Sigma-pos} for the order relation).
Furthermore, we have open embeddings
$$\ov\J_{g,n+1}(\Omega(\sigma)^+)\subseteq \ov\J_{g,n+1}(\Omega(\sigma))^+ \: \text{ and } \: \ov\J_{g,n+1}(\Omega(\sigma)^-)\subseteq \ov\J_{g,n+1}(\Omega(\sigma))^-,$$ which are equalities if (and only if) $\sigma$ is general. 

Assuming that $\sigma$ is general, the isomorphisms of Corollary~\ref{C:univ-cJ}\eqref{C:univ-cJ2} become the following isomorphisms over $\Cbargn\times_{\Mbargn} \ov\J_{g,n}(\sigma)$: 
\begin{equation}\label{E:Puniv-gen}
 \PP(\I_{g,n})\cong \ov\J_{g,n+1}(\Omega(\sigma)^+) \quad \text{ and }\quad 
 \PP(\I_{g,n}^\vee) \cong \ov\J_{g,n+1}(\Omega(\sigma)^-).
\end{equation}
The first of these isomorphisms was proved in \cite[Thm. 3.30]{APag} (note that in loc.cit. the authors use the dual convention for a projective space). It follows from the description in Proposition~\ref{P:push-loc} that the two compactified universal Jacobians of Equation~\eqref{E:Puniv-gen} are two small resolutions, related by an Atiyah flop, of the codimension $3$ singularity in the universal curve $ \Cbargn\times_{\Mbargn} \ov\J_{g,n}(\sigma)$.
\end{remark}

\section{Classification for $n=0$}\label{Sec:n0}

The aim of this section is to classify all compactified universal Jacobians over $\Mbarg$. 

For any $\chi \in \ZZ$, we consider the \emph{canonical universal polarization} of genus $g$ (in the notation of Fact~\ref{F:RelPic})
$$
\psi_{g}^\chi:=\un \deg\left(\frac{\chi\cdot \omega_{\pi}}{2g-2}\right).
$$
%and the associated \emph{canonical universal V-stability} of genus $g$ (see \eqref{E:map-s-Uni})
%$$
%\s_{g}^{\chi}:=\s(\psi_{g}^\chi).
%$$
The associated V-function $\displaystyle \sigma_{g}^{\chi}:=\sigma_{\frac{\chi\cdot \omega_{\pi}}{2g-2}}$ (see \eqref{E:map-sigma}), called the \emph{canonical V-function} of genus $g$,  is given by
\begin{equation}\label{E:sg-can}
\begin{aligned}
\sigma_{g}^{\chi}: \Dg& \longrightarrow \ZZ\\
(e;h,\emptyset)=:(e;h)& \mapsto \Big\lceil \frac{\chi}{2g-2}(2h-2+e) \Big\rceil.
\end{aligned}
\end{equation}

 The following is the main result of this section.

\begin{theorem}\label{T:class-n0}
For any $\chi \in \ZZ$, we have that $\Sigmans_g^{\chi}=\{(\sigma_g^\chi)^{ns}=:\sigma_g^{\chi, ns}\}$. 
\end{theorem}
Note that $\sigma_g^{\chi, ns}$ coincides with the element $\sigma^\chi_g[\emptyset]$ of \eqref{E:sigma-cl}. 

In the proof of the above Theorem, we will use the following stable graphs.

\begin{lemma}\label{L:graph}
 Let $(e;h)\in \Dg^{ns}$ such that $\delta:=2h-2+e\leq g-1$ and write $2g-2=q\delta+r$ with $0\leq r<\delta$ (notice that $q\geq 2$). Then there exists a Hamiltonian stable graph $G$ of genus $g$ (with no legs) with the following properties:
 \begin{enumerate}
     \item \label{L:graph1} If $\delta$ divides $2g-2$, then $G$ has $q$ vertices, each of them of type $(e;h)$.
     \item \label{L:graph2} If $\delta$ does not divide $2g-2$, then $G$ has $q$ vertices of type $(e;h)$ and $1$ vertex of type 
     $$
     \begin{cases} 
     \left(2;\frac{r}{2}\right) \text{ if } r  \text{ is even,}\\
     \left(3;\frac{r-1}{2}\right) \text{ if } r  \text{ is odd.}\\
     \end{cases} 
     $$
 \end{enumerate} 
\end{lemma}
\begin{proof}
In order to construct a graph $G$ as in Part~\eqref{L:graph1} (so $r=0$), we distinguish two cases:

\begin{enumerate}
    \item [(1A)]  If $e$ is even, then $G$ is the graph having vertices $\{v_1,\ldots, v_q\}$, each of genus $h$, and having $e/2$ edges in between $v_i$ and $v_{i+1}$ for every $i=1,\ldots,q$ (with the cyclic convention that $v_{q+1}=v_1$).

 \item[(1B)] If $e$ is odd, $q$ must be even. Then $G$ is the graph having vertices $\{v_1,\ldots, v_q\}$, each of genus $h$, and having $(e-1)/2$ edges in between $v_i$ and $v_{i+1}$ for every $i=1,\ldots,q$ (with the cyclic convention that $v_{q+1}=v_1$) and $1$ edge in between $v_i$ and $v_{i+q/2}$ for every $i=1,\ldots,q/2$.
\end{enumerate}

In order to construct a graph $G$ as  in Part~\eqref{L:graph2}, we distinguish three cases:

\begin{enumerate}

\item[(2A)] If $e$ is even (which implies that $\delta$ and $r$ are also even), then $G$ is the graph having vertices $\{w, v_1,\ldots, v_q\}$, with each $v_i$  of genus $h$, and $w$ of genus $r/2>0$, and having $e/2$ edges in between $v_i$ and $v_{i+1}$ for every $i=1,\ldots,q-1$, with $(e/2-1)$ edges in between $v_1$ and $v_q$, and with $w$ connected with $1$ edge to $v_1$ and with $1$ edge to $v_q$. 

\item[(2B)] If $e$ is odd and $r$ is even (which then implies that $q$ is even), then $G$ is the graph having vertices $\{w, v_1,\ldots, v_q\}$, with each $v_i$  of genus $h$ and $w$ of genus $r/2>0$, and having $(e-1)/2$ edges in between $v_i$ and $v_{i+1}$ for every $i=1,\ldots,q-1$, with $(e-3)/2$ edges in between $v_1$ and $v_q$, 
with $1$ edge in between $v_i$ and $v_{i+q/2}$ for every $i=1,\ldots,q/2$, 
and with $w$ connected with $1$ edge to $v_1$ and with $1$ edge to $v_q$.   

\item[(2C)] If $e$ is odd and $r$ is odd (which then implies that $q$ is also odd), then $G$ is the graph having vertices $\{w, v_1,\ldots, v_q\}$, with each $v_i$ of genus $h$ and $w$ of genus $(r-1)/2$, and having $(e-1)/2$ edges in between $v_i$ and $v_{i+1}$ for every $i=1,\ldots,q-1$, with $(e-3)/2$ edges in between $v_1$ and $v_q$, 
with $1$ edge in between $v_i$ and $v_{i+(q+1)/2}$ for every $i=1,\ldots,(q-1)/2$, and with $w$ connected with a single edge to $v_1$, $v_{(q+1)/2}$ and $v_q$.  
\end{enumerate}
\end{proof}

\begin{figure}[h]
  \centering

  % ---------- (a) e even ----------
  \begin{minipage}[b]{0.45\textwidth}
    \centering
    \begin{tikzpicture}[scale=1, every node/.style={circle,draw,inner sep=2pt}]
      % parameters
      \def\q{5}   % number of vertices
      \def\k{3}   % e/2 = number of parallel edges between neighbours
      \def\r{1.8} % radius

      % place vertices on a circle
      \foreach \i in {1,...,\q} {
        \node (v\i) at ({360/\q*(\i-1)}:\r) {$v_{\i}$};
      }

      % draw k parallel edges between consecutive vertices
      \foreach \i in {1,...,\q} {
        \pgfmathtruncatemacro{\next}{mod(\i,\q)+1}
        \pgfmathsetmacro{\maxbend}{18}
        \foreach \j in {1,...,\k} {
          \ifnum\k=1
            \draw (v\i) -- (v\next);
          \else
            \pgfmathsetmacro{\bend}{
              (\j-1) * (\maxbend/(\k-1)) - \maxbend/2
            }
            \draw[bend left=\bend] (v\i) to (v\next);
          \fi
        }
      }
    \end{tikzpicture}
    \caption*{(1A): Here $r=0$ and $e$ is even: $e/2$ edges between $v_i$ and $v_{i+1}$}
  \end{minipage}
  \hfill
  % ---------- (b) e odd ----------
  \begin{minipage}[b]{0.45\textwidth}
    \centering
    \begin{tikzpicture}[scale=1, every node/.style={circle,draw,inner sep=2pt}]
      % parameters
      \def\q{6}   % must be even
      \def\k{2}   % (e-1)/2 edges between neighbours
      \def\r{1.8} % radius

      % place vertices
      \foreach \i in {1,...,\q} {
        \node (v\i) at ({360/\q*(\i-1)}:\r) {$v_{\i}$};
      }

      % (e-1)/2 parallel edges between consecutive vertices
      \foreach \i in {1,...,\q} {
        \pgfmathtruncatemacro{\next}{mod(\i,\q)+1}
        \pgfmathsetmacro{\maxbend}{18}
        \foreach \j in {1,...,\k} {
          \ifnum\k=1
            \draw (v\i) -- (v\next);
          \else
            \pgfmathsetmacro{\bend}{
              (\j-1) * (\maxbend/(\k-1)) - \maxbend/2
            }
            \draw[bend left=\bend] (v\i) to (v\next);
          \fi
        }
      }

      % extra edges v_i -- v_{i+q/2} for i=1,...,q/2
      \pgfmathtruncatemacro{\half}{\q/2}
      \foreach \i in {1,...,\half} {
        \pgfmathtruncatemacro{\opp}{\i+\half}
        \draw[thick] (v\i) -- (v\opp);
      }
    \end{tikzpicture}
    \caption*{(1B): Here $r=0$ and $e$ is odd: $(e-1)/2$ edges between $v_i$ and $v_{i+1}$, and $1$  edge to the opposite vertex}
  \end{minipage}

\end{figure}

\begin{figure}[h]
\centering
\begin{tabular}{ccc}

% ===================== (2A) e even =====================
\begin{tikzpicture}[scale=1, every node/.style={circle,draw,inner sep=2pt}]
  % flat-top hexagon
  \node (a1) at (-0.65,  0.95) {$v_{1}$};  % top-left
  \node (a2) at (-1.30,  0.00) {$v_{2}$};  % left-middle
  \node (a3) at (-0.65, -0.95) {$v_{3}$};  % bottom-left
  \node (a4) at ( 0.65, -0.95) {$v_{4}$};  % bottom-right
  \node (a5) at ( 1.30,  0.00) {$v_{5}$};  % right-middle
  \node (a6) at ( 0.65,  0.95) {$v_{6}$};  % top-right

  % w on top
  \node (wA) at (0, 1.6) {$w$};

  \def\kA{3} % = e/2

  % parallel edges around the hexagon except top
  \foreach \u/\v in {1/2,2/3,3/4,4/5,5/6} {
    \pgfmathsetmacro{\maxbend}{14}
    \foreach \j in {1,...,\kA} {
      \ifnum\kA=1
        \draw (a\u) -- (a\v);
      \else
        \pgfmathsetmacro{\bend}{(\j-1)*(\maxbend/(\kA-1)) - \maxbend/2}
        \draw[bend left=\bend] (a\u) to (a\v);
      \fi
    }
  }

  % top edge v1--v6 has kA-1 edges
  \ifnum\kA>1
    \draw (a1) -- (a6);
    \ifnum\kA>2
      \draw[bend left=10] (a1) to (a6);
    \fi
  \fi

  % w to v1, v6
  \draw (wA) -- (a1);
  \draw (wA) -- (a6);
\end{tikzpicture}
&
% ===================== (2B) e odd, r even =====================
\begin{tikzpicture}[scale=1, every node/.style={circle,draw,inner sep=2pt}]
  % same flat-top hexagon
  \node (b1) at (-0.65,  0.95) {$v_{1}$};
  \node (b2) at (-1.30,  0.00) {$v_{2}$};
  \node (b3) at (-0.65, -0.95) {$v_{3}$};
  \node (b4) at ( 0.65, -0.95) {$v_{4}$};
  \node (b5) at ( 1.30,  0.00) {$v_{5}$};
  \node (b6) at ( 0.65,  0.95) {$v_{6}$};

  \node (wB) at (0, 1.6) {$w$};

  \def\kB{2} % = (e-1)/2

  % parallel edges around (except top)
  \foreach \u/\v in {1/2,2/3,3/4,4/5,5/6} {
    \pgfmathsetmacro{\maxbend}{12}
    \foreach \j in {1,...,\kB} {
      \ifnum\kB=1
        \draw (b\u) -- (b\v);
      \else
        \pgfmathsetmacro{\bend}{(\j-1)*(\maxbend/(\kB-1)) - \maxbend/2}
        \draw[bend left=\bend] (b\u) to (b\v);
      \fi
    }
  }

  % top edge v1--v6 : (kB - 1)
  \ifnum\kB>1
    \draw (b1) -- (b6);
  \fi

  % opposite chords (q=6): 1-4, 2-5, 3-6
  \draw[thick] (b1) -- (b4);
  \draw[thick] (b2) -- (b5);
  \draw[thick] (b3) -- (b6);

  % w to v1, v6
  \draw (wB) -- (b1);
  \draw (wB) -- (b6);
\end{tikzpicture}
&
% ===================== (2C) e odd, r odd =====================
\begin{tikzpicture}[scale=1, every node/.style={circle,draw,inner sep=2pt}]
  % flat-top pentagon
  \node (c1) at (-0.55, 0.95) {$v_{1}$};   % top-left
  \node (c2) at (-1.05,-0.05) {$v_{2}$};   % left
  \node (c3) at ( 0.00,-1.00) {$v_{3}$};   % bottom
  \node (c4) at ( 1.05,-0.05) {$v_{4}$};   % right
  \node (c5) at ( 0.55, 0.95) {$v_{5}$};   % top-right

  \node (wC) at (0, 1.6) {$w$};

  \def\kC{2}

  % around: v1-2, 2-3, 3-4, 4-5
  \foreach \u/\v in {1/2,2/3,3/4,4/5} {
    \pgfmathsetmacro{\maxbend}{12}
    \foreach \j in {1,...,\kC} {
      \ifnum\kC=1
        \draw (c\u) -- (c\v);
      \else
        \pgfmathsetmacro{\bend}{(\j-1)*(\maxbend/(\kC-1)) - \maxbend/2}
        \draw[bend left=\bend] (c\u) to (c\v);
      \fi
    }
  }

  % top edge v1--v5 : (kC - 1)
  \ifnum\kC>1
    \draw (c1) -- (c5);
  \fi

  % --- corrected chords ---
  % text-style wants v1 -- v4 and v2 -- v5
  \draw[thick] (c1) -- (c4);
  \draw[thick] (c2) -- (c5);

  % w to v1, v3, v5 (as in the figure: w joins the 2 "top" + the bottom)
  \draw (wC) -- (c1);
  \draw (wC) -- (c3);
  \draw (wC) -- (c5);
\end{tikzpicture}
\\
(2A) & (2B) & (2C)
\end{tabular}
\end{figure}

\begin{proof}[Proof of Theorem \ref{T:class-n0}]
It is enough to prove that $\Sigmans_g^\chi$ consists of one element, which then must be $\sigma_g^{\chi, ns}$. We will show that the value of any $\sigma \in \Sigmans_g^\chi$ on a given  $(e;h)\in \Dg^{ns}$ is uniquely determined by $\delta(e;h):=2h-2+e$. 

Using \eqref{E:sumUni}, it is enough to prove that the degeneracy set $\Dns(\sigma)$ is independent of $\sigma$ and that the value $\sigma(e;h)$ depends only on $\delta(e;h)$ for any element $(e;h)\in \Dg^{ns}$ such that $\delta(e;h)\leq g-1$. Write 
$$
2g-2=q\cdot \delta(e;h)+r \text{ with } 0\leq r <\delta(e;h).$$
Note that $q\geq 2$ since $\delta(e;h)\leq g-1$. 

We will now distinguish two cases:

\un{Case I}: $r=0$, i.e. $\delta(e;h)$ divides $2g-2$.

By applying Lemma \ref{L:graph}\eqref{L:graph1}, we find a stable Hamiltonian graph of genus $g$ with $q$ vertices of type $(e;h)$.
Now we apply iteratively $(q-2)$-times Remark \ref{R:triang} along the Hamiltonian cycle $\{v_1,\ldots, v_q\}$ of $G$, and then apply \eqref{E:sumUni}, to obtain
\begin{equation}\label{E:poly1}
q\sigma(e;h)-\chi\in \{0,\ldots, q-1\} \text{ with } 0 \text{ occurring if and only if } (e;h)\in \Dns(\sigma).
\end{equation}
This shows that the value $\sigma(e;h)$ depends only on $q$ (and hence on $\delta(e;h)$), and that $(e;h)\in \Dns(\sigma)$ if and only if $q$ divides $\chi$ (which shows that this property is independent of $\sigma$).  

\un{Case II:} $r>0$.

We will proceed by induction on $\delta(e;h)$, the base of the induction is the case $\delta(e;h)=1$, which is covered in Case I. 

By applying Lemma \ref{L:graph}\eqref{L:graph2}, we find a stable Hamiltonian graph of genus $g$ with $q$ vertices of type $(e;h)$ and one vertex of type 
$$(f;k):= 
\begin{cases}
    (2;r/2) & \text{ if } r \text{ is even,}\\
    (3;(r-1)/2)& \text{ if } r \text{ is odd.}\\
\end{cases}
$$

Now we apply iteratively $(q-2)$-times Remark \ref{R:triang} along the Hamiltonian cycle $\{v_1,\ldots, v_q, w\}$ of $G$, and then apply \eqref{E:sumUni}, to obtain
\begin{equation}\label{E:poly2}
q\sigma(e;h)+\sigma(f;k)-\chi\in 
\begin{cases}
\{0,\ldots,q-1\} & \text{ if } (f;k)\in \Dns(\sigma)\\
&\text{ with } 0 \text{ occurring if and only if } (e;h)\in \Dns(\sigma),\\
\{1,\ldots,q\} & \text{ if } (f;k)\not \in \Dns(\sigma).\\
\end{cases}
\end{equation}
Since $\delta(f;k)=r<\delta(e;h)$, we can apply either Case I (if $r$ divides $2g-2$) or the induction hypothesis (if $r$ does not divide $2g-2$) to deduce that $\sigma(f;k)$ depends only on $r$, and that the property that $(f;k)$ belongs to $ \Dns(\sigma)$ is independent of $\sigma$. 
Using this, Equation \ref{E:poly2} implies that  the value $\sigma(e;h)$ depends only on $\delta(e;h)$, and that the property that $(e;h)$ belongs to $\Dns(\sigma)$ is independent of $\sigma$.  
\end{proof}

The compactified universal Jacobian stack (resp. space) associated to the canonical universal V-stability of genus $g$  will be called the \emph{Caporaso's compactified universal Jacobian stack (resp. space)} and it will be denoted by 
\begin{equation}\label{E:Cap-cJ}
\ov \J_g^{\text{Cap}, \chi}:=\ov \J_g(\sigma_g^\chi)=\ov \J_g\left(\left[\frac{\chi}{2g-2}\omega_{\pi}\right]\right) \quad \left(\text{resp. } \ov J_g^{\text{Cap}, \chi}:=\ov J_g(\sigma_g^\chi)=\ov J_g\left(\left[\frac{\chi}{2g-2}\omega_{\pi}\right]\right)\right).
\end{equation}
The reason for choosing this name comes from the fact that  the absolute good moduli space of $\ov J_g^{\text{Cap}, \chi}$ is isomorphic to  Caporaso's \cite{Cap94} compactified universal Jacobian over the coarse moduli space $\ov M_g$ of $\ov \M_g$ (see \cite{pandharipande1995compactification}), while the stack $\ov \J_g^{\text{Cap}, \chi}$ is isomorphic to the stack studied by Caporaso \cite{Cap08} and Melo \cite{melo2009} (see \cite{estevespacini}).

From the above Theorem \ref{T:class-n0}, we deduce the following two classification results

\begin{corollary}\label{C:class-n0}
 Let $\ov \J_g^{\chi}$ be a compactified universal Jacobian stack of characteristic $\chi \in \Z$ over $\ov \M_g$ and let $\ov J_g^\chi$ be its associated compactified universal Jacobian space. Then we have that
 \begin{enumerate}
     \item \label{C:class-n01}  ${\ov \J_{g}^\chi}_{|\Mbarg^{ns}}={\ov \J_{g}^{Cap, \chi}}_{|\Mbarg^{ns}}$.
     \item \label{C:class-n02} $\ov J_{g}^\chi$ is isomorphic to $\ov J_{g}^{Cap, \chi}$ over $\ov \M_g$. 
 \end{enumerate}
\end{corollary}
%Part \eqref{C:iso-n01} of the above Corollary was stated in \cite[Cor. 9.8, Rmk. 9.9]{pagani2023stability} in the case of fine compactified universal Jacobians. The proof in loc.cit. contains an error, see Remark~\ref{rem:correcterror} for more details.
\begin{proof}
 This follows by combing Theorem \ref{T:class-n0} with Lemma \ref{L:ugua-ns} and Theorem \ref{T:iso-UniSp}.   
\end{proof}

%We can now describe when two compactified universal Jacobian stacks or spaces are isomorphic over $\ov \M_g$. 

\begin{corollary}\label{C:iso-n0}
   Let $\sigma_1\in \Sigma_g^{\chi_1}$ and $\sigma_2\in \Sigma_g^{\chi_2}$. Then we have that:
   \begin{enumerate}
       \item \label{C:iso-n01} $\ov \J_g(\sigma_1)$ is isomorphic to $\ov \J_g(\sigma_2)$ over $\ov \M_g$ if and only if $\D(\sigma_1^s)=\D(\sigma_2^s)$ and $\chi_1\equiv \pm \chi_2 \mod 2g-2$. 
        \item \label{C:iso-n02} $\ov J_g(\sigma_1)$ is isomorphic to $\ov J_g(\sigma_2)$ over $\ov \M_g$ if and only if $\chi_1\equiv \pm \chi_2 \mod 2g-2$. 
   \end{enumerate}
\end{corollary}
\begin{proof}
  Part \eqref{C:iso-n02}: Theorem \ref{T:iso-UniSp} implies that  $\ov J_g(\sigma_1)$ is isomorphic to $\ov J_g(\sigma_2)$ over $\ov \M_g$ if and only if $\sigma_1^{ns}$ and $\sigma_2^{ns}$ lie in the same orbit for the action of $\PR_g$. Theorem \ref{T:class-n0} implies that $\sigma_i^{ns}=\sigma_g^{\chi_i,ns}$ for $i=1,2$. We conclude by observing that the group $\PR_g$ is isomorphic to $\ZZ\langle \omega_{\pi}\rangle \rtimes (\ZZ/2\ZZ)\langle \iota\rangle$ (see \eqref{E:PRbis} and the discussion following it) and the action of $\PR_g$ on $\Sigma_g$ is such that (by Lemma \ref{L:Sigma-ac})
  $$
  \begin{sis}
   & \omega_{\pi}\cdot \sigma_g^\chi=\sigma_g^{\chi+2g-2},\\
   & \iota\cdot \sigma_{\pi}^\chi=\sigma_{\pi}^{-\chi}.
   \end{sis}
  $$

  Part \eqref{C:iso-n01}: Corollary \ref{C:iso-UniSt} implies that  $\ov \J_g(\sigma_1)$ is isomorphic to $\ov \J_g(\sigma_2)$ over $\ov \M_g$ if and only if $\sigma_1$ and $\sigma_2$ lie in the same orbit for the action of $\wt \PR_g$. In turn, Lemma \ref{L:Sigma-ac} implies that $\sigma_1$ and $\sigma_2$ lie in the same orbit for the action of $\wt \PR_g$ if and only if $\D(\sigma_1^s)=\D(\sigma_2^s)$ and $\sigma_1^{ns}$ and $\sigma_2^{ns}$ lie in the same orbit for the action of $\PR_g$. We conclude using what observed in the proof of Part \eqref{C:iso-n02}. 
\end{proof}

From the above results, we can classify the fine  compactified universal Jacobian spaces over $\Mbarg$ and prove that they admit a tautological sheaf. 

\begin{corollary}\label{C:fineMg}
\begin{enumerate}
    \item \label{C:fineMg1} There exist fine compactified universal Jacobian stacks of characteristic $\chi$ over $\Mbarg$ if and only if $\gcd(\chi, 2g-2)=1$,  and all of them are isomorphic over $\Mbarg$ to  Caporaso's compactified universal Jacobian stack $\ov \J_{g}^{Cap, \chi}$.
    \item \label{C:fineMg2} There exists a tautological sheaf  on 
    $\ov J_{g}^{Cap, \chi}\times_{\Mbarg} \Cbarg$, i.e. a relative rank-$1$ torsion-free sheaf $\I_{g,n}^{\taut}$ on  $\ov J_{g}^{Cap, \chi}\times_{\Mbarg} \Cbarg\to \ov J_{g}^{Cap, \chi}$ whose restriction to the fiber over a geometric point $(C,I)\in \ov J_{g}^{Cap, \chi}$ is isomorphic to $I$, if  and only if $\gcd(\chi, 2g-2)=1$. 
\end{enumerate}    
\end{corollary}
Part \eqref{C:fineMg1}  of the above Corollary was stated in \cite[Cor. 9.8, Rmk. 9.9]{pagani2023stability} in the case of fine compactified universal Jacobians. The proof in loc.cit. contains an error, see Remark~\ref{rem:correcterror} for more details.
\begin{proof}
    Part \eqref{C:fineMg1}: first of all, by \eqref{E:Cap-cJ} and \eqref{E:sg-can}, we get that
    \begin{equation}\label{E:Cap-nf}
        \ov \J_{g}^{Cap, \chi}\text{ is non-fine }
    \Longleftrightarrow \frac{\chi}{2g-2}(2h-2+e)\in \ZZ \text{ for some } (e;h)\in \DD_g.
    \end{equation}
    Since $0<2h-2+e<2g-2$ for any $(e;h)\in \DD_g$ and all the numbers in the interval $(0,2g-2)$ are of the form $2h-2+e$ for some $(e;h)\in \DD_g$, Condition \eqref{E:Cap-nf} implies that 
    \begin{equation}\label{E:Cap-nf2}
        \ov \J_{g}^{Cap, \chi}\text{ is fine }
    \Longleftrightarrow \gcd(\chi, 2g-2)= 1.
    \end{equation}
    This shows the if part of the first assertion. 
    
    Conversely, if $\gcd(\chi, 2g-2)=k\geq 2$ then $(\frac{2g-2}{k}+2; 0)\in \D(\sigma_g^{\chi})$ by \eqref{E:sg-can}, Hence, since $(\frac{2g-2}{k}+2; 0)\in \DD_g^{ns}$, Theorem \ref{C:class-n0}\eqref{C:class-n01} implies that  $(\frac{2g-2}{k}+2; 0)$ belongs to $\D(\sigma)$ for any $\sigma \in \Sigma_g^\chi$, which implies that there are no fine compactified Jacobians over $\Mbarg$ of characteristic $\chi$. This proves the only if part of the first assertion.

    The second assertion follows from Corollary \ref{C:iso-n0}\eqref{C:iso-n01}.

    Part \eqref{C:fineMg2}: the $\Gm$-gerbe $ \J_g^{\chi}\to J_g^{\chi}=\J_g^{\chi}\fatslash \Gm$ is trivial if and only if $\gcd(\chi, 2g-2)=1$ by \cite[Cor. 2.9]{MR}. Therefore, if $\gcd(\chi, 2g-2)\neq 1$ there cannot be any tautological sheaf on $J_g^\chi\times_{\Mg} \Cg\to J_g^\chi$ and hence neither on $\ov J_{g}^{Cap, \chi}\times_{\Mbarg} \Cbarg\to \ov J_{g}^{Cap, \chi}$.
    Conversely, if $\gcd(\chi, 2g-2)=1$ then $\ov J_g^{Cap,\chi}=\ov \J_g^{Cap,\chi}\fatslash \Gm$ by part \eqref{C:fineMg1} and  the $\Gm$-gerbe $\ov \J_g^{Cap,\chi}\to \ov \J_g^{Cap,\chi}\fatslash \Gm$ is trivial by what recalled above (because being trivial is Zariski locally on the base). 
     Hence, the universal sheaf $\I_{g,n}$ on $\ov \J_g^{Cap,\chi}\times_{\Mbarg}\Cbarg$ descends (non-uniquely) to a tautological sheaf $\I_{g,n}^{\taut}$ on $\ov J_g^{Cap,\chi}\times_{\Mbarg}\Cbarg$.
\end{proof}

\begin{remark}\label{rem:correcterror}
The classification of all fine compactified universal Jacobians over $\Mbarg$ was announced in \cite{pagani2023stability}, but the proof in loc.cit. is incorrect, as we now outline. 

The classification strategy of \cite{pagani2023stability} relied on the classification of all 'universal stability assignments', a notion introduced in loc.cit. for the universal family over $\Mbargn$, and equivalent to the notion of 'general, universal V-stability condition' discussed here (see \cite[Example~4.25]{FPV2}, \cite[Section~4.2]{pagani2023stability} and Definition~\ref{D:VStab-Uni}).

The classification of universal stability assignments for the $n=0$ case was stated in \cite[Theorem~1.2]{pagani2023stability} or equivalently in \cite[Corollary~9.8]{pagani2023stability}. These results follow from \cite[Theorem~9.7]{pagani2023stability}, which in turn relies on \cite[Lemma 9.4]{pagani2023stability} (a classification of universal stability assignments  over all vine curves with no separating nodes and with one component of genus $0$), on  \cite[Corollary~9.5]{pagani2023stability} (a classification over all remaining vine curves with no separating nodes) and on \cite[Proposition~9.6] {pagani2023stability} (a classification over the vine curves with separating nodes).

The error is in the proof of \cite[Corollary~9.5]{pagani2023stability}. The argument in loc.cit. aims to establish an explicit formula to determine the stability assignment on the vine graph $G_1$ with vertices of genera $i, j$ from the vine graph $G_2$ with vertices of genera $g+j-i-1,0$, for all $i,j \in \mathbb{Z}_{>0}$ with $i \geq j$.

The argument relies on a classification of stability assignments on the genus $1$ necklace graph $G'$ with $4$ vertices $w_1, w_2, w_3, w_4$, cyclically connected by $4$ edges, and with genera $j,j,0,0$, respectively. The authors argue, incorrectly, that for a given stability assignment on $G_2$ there exists a unique $\Aut(G')$-invariant stability assignment on $G'$ that is compatible via the contraction morphisms $G' \to G_2$. The argument is then completed by observing that there is a unique stability assignment on $G_1$ that is compatible via the contraction morphisms $G' \to G_1$.

While the statement of \cite[Corollary~9.5]{pagani2023stability} is correct as it can now be deduced from Theorem~\ref{T:class-n0}, an independent proof following the original line of reasoning outlined above does not seem to be straightforward.
%Thus, following the argument in \cite[Section~9]{pagani2023stability} that is not affected by the error, one can still classify fine compactified universal Jacobians over the smallest open subset 
\end{remark}

\begin{remark}
    With the tools developed in this section, we can give a  criterion for a compactified universal  Jacobian to be fine, and also invariant with respect to the action of some natural subgroups of the symmetric group $S_n$ acting on the markings.\footnote{This question was raised to us by Dan Petersen in a private correspondence.} 
    
    Let $n,k \geq 1$ and   $N_1 \sqcup \ldots \sqcup N_k = [n]$ be a partition of $[n]$, and define $H \subseteq S_n$ as the  subgroup  that fixes the partition (so if $n_i=|N_i|$ for all $i=1,\ldots,k$, then $H \cong S_{n_1} \times \ldots S_{n_k}$ as abstract groups).  
    
    We have two natural actions of $H$. One is obtained on $\overline{\mathcal{M}}_{g,n}$ by permuting the marked points. Another one is on the collection $\Sigma_{g,n}$ of $V$-functions $\sigma$ of type $(g,n)$, defined, for all $t \in H$, by $t(\sigma)(e;h,A):= \sigma(e;h,t(A))$ for all $(e;h,A) \in \Dgn$.

    Then the same argument given to prove Theorem~\ref{T:thmA} shows that the anti-isomorphism of posets in loc.cit. induces an anti-isomorphism of posets between the $H$-invariant $V$-functions of type $(g,n)$ and $H$-equivariant compactified universal Jacobians over $\Mbargn$, or equivalently  compactified universal Jacobians over the quotient $[\Mbargn/H]$ (where fine compactified universal Jacobians correspond to the general $V$-functions).
    
     We claim that, for $g \geq 1$, there exists a fine compactified universal Jacobian of characteristic $\chi$ over the quotient $[\Mbargn/H]$ (equivalently, a general $H$-invariant $V$-function in $\Sigma_{g,n}^\chi$) if and only if $\gcd(\chi,2g-2,n_1,\ldots,n_k)=1$. (Note that a $S_n$-invariant $V$-function in $\Sigma_{0,n}^\chi$ exists for every $\chi\in \mathbb{Z}$).

 Indeed, consider the element of $\PicRel^\chi_{g,n}$ defined by
  $$
  L= \begin{cases} \frac{\chi - \sum_{j=1}^k q_j n_j }{2g-2} \cdot \omega_\pi + \sum_{r=1}^k q_r \left(\sum_{j \in N_r} \Sigma_j\right) & \textrm{if } g \geq 2 \\
  \frac{\chi- \sum_{j=1}^k q_j n_j}{n} (\Sigma_1 + \ldots + \Sigma_n) + \sum_{r=1}^k q_r \left(\sum_{j \in N_r} \Sigma_j\right) & \textrm{if } g =1,
  \end{cases}
  $$
     where $q_1, \ldots, q_k \in \mathbb{R}$ are 'generic' choices, by which we mean that they satisfy $ \sum_{j=1}^k a_j q_j \notin \mathbb{Z}$ for any integers $1 \leq a_j \leq n_j$ for all $j=1,\ldots,k$. Then $L$ is manifestly $H$-invariant, so the corresponding $V$-function $\sigma_L$, defined via \eqref{E:map-sigma} and \eqref{E:for-deg}, is also $H$-invariant. By the assumption $\gcd(\chi,2g-2,n_1,\ldots,n_k)=1$, we deduce that $\sigma_L$ is general.

Conversely, if $\sigma\in\Sigma_{g,n}^\chi$ is a $H$-invariant $V$-function, and there exists $1<\ell|\gcd(\chi,2g-2,n_1,\ldots,n_k)$, arguing as in Case~$1$ in the proof of Theorem~\ref{T:class-n0}, we see that 
    $$
    \{(e;h,A) \in \Dgn: \ \ell(2h-2+e)=2g-2 ,n_1= \ell\cdot  |A \cap N_1| ,\ldots, n_k=\ell \cdot |A \cap N_k|\}\subseteq\D(\sigma),
    $$
    and, in particular, $\sigma$ is not general.
\end{remark}

\section{On the poset of compactified universal Jacobians} \label{Sec:poset}

The aim of this section is to study the poset of compactified universal  Jacobian stacks over $\Mbargn$, with respect to the order relation. By Theorem \ref{T:thmA}, this poset is anti-isomorphic to the poset $\Sigma_{g,n}$ of V-functions of type $(g,n)$ (see Definition \ref{D:Sigma-pos}). The main result here is a proof of Theorem~\ref{T:thmF} from the introduction.

We will often assume that $n\geq 1$ throughout this section, since the case $n=0$ has been already dealt with in Section \ref{Sec:n0}.

As a tool to study the poset $\Sigma_{g,n}$, we will   use the degeneracy map introduced in \eqref{E:degsigma},  which we will upgrade to a morphism of posets $\D \colon \Sigma_{g,n} \to \Deg_{g,n}$, where the latter is introduced in Definition~\ref{D:degsub}.  

After establishing the basic order-theoretic properties of $\D$ and its separating/non-separating factorization, we treat the case $n=1$ (Subsection~\ref{Sub:n1}), where the poset $\Sigma_{g,1}$ and its image via the degeneracy map are  described explicitly.
We then return to general $n\ge 1$ (Subsection~\ref{sub:max-wall}) and classify maximal and submaximal elements of $\Sigma_{g,n}$ in terms of $\Deg_{g,n}$.
Finally, Subsection~\ref{Sub:g1} discusses the special features of genus $g=1$.

We start by defining the poset of universal degeneracy subsets.

\begin{definition}\label{D:degsub}
\noindent 
\begin{enumerate}
    \item \label{D:degsub1} A \emph{(universal) degeneracy subset} of type $(g,n)$ is a subset $D\subset \Dgn$ that satisfies the following properties: 
    \begin{enumerate}
        \item \label{D:degsub1a} [Complement-closure] If $(e;h,A)\in D$ then $(e;h,A)^c\in D$;
        \item \label{D:degsub1b} [Triangle-closure] For any triangle $\Delta=[(e_1;h_1,A_1), (e_2;h_2,A_2), (e_3;h_3,A_3)]$ of $\Dgn$, if two among the elements of $\Delta$ belong to $D$, then so does the third.
    \end{enumerate}
    We denote by $\Deg_{g,n}$ the collection of all degeneracy subsets of type $(g,n)$.
    \item \label{D:degsub2} Let $D^1,D^2\in \Deg_{g,n}$. We say that $D^1\geq D^2$ if $D^1\subseteq D^2$ and there exists a subset $E\subseteq D^2-D^1$ such that 
    \begin{enumerate}
        \item \label{D:degsub2a}  We have $(e;h,A) \in E$ if and only if $(e;h,A)^c \in E^c = (D^2-D^1) - E$.
        \item \label{D:degsub2b} For any triangle $\Delta=[(e_1;h_1,A_1), (e_2;h_2,A_2), (e_3;h_3,A_3)]$ contained in $D^2$ such that exactly one element of $\Delta$ belongs to $D^1$, we have that 
        $$
        |\{i \in \{1,2,3\}: (e_i;h_i,A_i)\in E\}|=1.
        $$
        \item \label{D:degsub2c} For any triangle $\Delta=[(e_1;h_1,A_1), (e_2;h_2,A_2), (e_3;h_3,A_3)]$ contained in $D^2-D^1$,  we have that 
        $$
        |\{i \in \{1,2,3\}: (e_i;h_i,A_i)\in E\}|=1 \text{ or } 2.
        $$
    \end{enumerate}    
    We call such an element $E$ a \emph{witness}  for the relation $D^1\geq D^2$. 
\end{enumerate}
\end{definition}

In Equation~\eqref{E:degsigma} we defined the degeneracy subset $\D(\sigma)$ of a V-function $\sigma\in \Sigma_{g,n}$, and the two conditions of Definition~\ref{D:Sigmagn} ensure that $\D(\sigma)$ is an element of $\Deg_{g,n}$. This defines a \emph{degeneracy map} 
\begin{equation}\label{D:Deg-map}
    \begin{aligned}
        \D: \Sigma_{g,n}& \longrightarrow \Deg_{g,n}\\
        \sigma & \mapsto \D(\sigma),
    \end{aligned}
\end{equation}
which we will now investigate as a map of posets.  Note that $\D=\coprod_{\chi \in \ZZ} \D^\chi$, where $\D^{\chi}:=\D_{|\Sigma_{g,n}^\chi}$.

\begin{proposition}\label{P:Deg-map}
 The degeneracy map $\D$ satisfies the following properties:
 \begin{enumerate}
      \item \label{P:Deg-map1}
     $\D$ is invariant with respect to the action of $\wt \PR_{g,n}$ on $\Sigma_{g,n}$.
      \item \label{P:Deg-map2}
      $\D$ is order-preserving, i.e. $\sigma^1\geq \sigma^2\Rightarrow \D(\sigma^1)\geq \D(\sigma^2)$.
      \item \label{P:Deg-map3}
      $\D$ is upper lifting, i.e. for any $D^1\geq D^2$ in $\Deg_{g,n}$ and any $\sigma^2\in \Sigma_{g,n}$ such that $\D(\sigma^2)=D^2$, there exists $\sigma^1\in \Sigma_{g,n}$ such that $\sigma^1\geq \sigma^2$ and $\D(\sigma^1)=D^1$.
 
      \item \label{P:Deg-map4} $\D$ is conservative, i.e. if $\sigma^1\geq \sigma^2$ is such that $\D(\sigma^1)=\D(\sigma^2)$ then $\sigma^1=\sigma^2$.
  \end{enumerate}  
\end{proposition}
\begin{proof}
    Part \eqref{P:Deg-map1} follows immediately from the explicit description of the action of $\wt \PR_{g,n}$ on $\Sigma_{g,n}$, see Remark \ref{R:PR-Pol}.  Part \eqref{P:Deg-map4} follows from Remark \ref{R:Sigma-pos}.

    Part \eqref{P:Deg-map2}: consider two V-functions  $\sigma^1,\sigma^2\in \Sigma_{g,n}$ such that $\sigma^1\geq \sigma^2$ (which implies that they have the same characteristic  $|\sigma^1|=|\sigma^2|:=\chi$) and let us show that $D^1:=\D(\sigma^1)\geq D^2:=\D(\sigma^2)$. First of all, $D^1\subseteq D^2$ by Remark \ref{R:Sigma-pos}. We now show that the subset 
    $$
    E:=\{(e;h,A)\in D^2-D^1\: : \; \sigma_1(e;h,A)=\sigma_2(e;h,A)+1\}. 
    $$
    is a witness for the relation $D^1\geq D^2$, i.e. it satisfies the conditions of Definition \ref{D:degsub}\eqref{D:degsub2}:

    $\bullet$ Condition \eqref{D:degsub2a} follows from Remark \ref{R:Sigma-pos}.

    $\bullet$ Consider a triangle $\Delta=[(e_1;h_1,A_1), (e_2;h_2,A_2), (e_3;h_3,A_3)]$ such that $(e_1;h_1,A_1)\in D^1\subseteq D^2$ and $(e_2;h_2,A_2), (e_3;h_3,A_3)\in D^2-D^1$. By applying \eqref{E:triaUni} to $\sigma^1$ and $\sigma^2$ and the triangle $\Delta$ we get that 
    $$
    \begin{sis}
    & \sum_{i=1}^3\sigma^1(e_i;h_i,A_i)-\chi=1,\\
    & \sum_{i=1}^3\sigma^2(e_i;h_i,A_i)-\chi=0.
    \end{sis}
    $$
   This implies that $E$ contains exactly one among $(e_2;h_2,A_2)$ and $(e_3;h_3,A_3)$, and hence condition~\eqref{D:degsub2b} holds. 

   $\bullet$ Consider a triangle $\Delta=[(e_1;h_1,A_1), (e_2;h_2,A_2), (e_3;h_3,A_3)]$ such that $(e_i;h_i,A_i)\in D^1\subseteq D^2$ for any $1\leq i \leq 3$. 
   By applying \eqref{E:triaUni} to $\sigma^1$ and $\sigma^2$ and the triangle $\Delta$ we get that 
    $$
    \begin{sis}
    & \sum_{i=1}^3\sigma^1(e_i;h_i,A_i)-\chi\in \{1,2\},\\
    & \sum_{i=1}^3\sigma^2(e_i;h_i,A_i)-\chi=0.
    \end{sis}
    $$
   This implies that $E$ contains either $1$ or $2$ elements of $\Delta$, and  hence Condition \eqref{D:degsub2c} holds. 

    Part \eqref{P:Deg-map3}: consider two degeneracy subsets $D^1,D^2\in \Deg_{g,n}$ such that $D^1\geq D^2$ in $\Deg_{g,n}$ and a V-function $\sigma^2\in \Sigma_{g,n}$ such that $\D(\sigma^2)=D^2$. By Definition \ref{D:degsub}\eqref{D:degsub2}, there exists $E\subset D^2-D^1$ satisfying the conditions of loc. cit. 
    We now consider the function $\sigma^1:\Dgn\to \Z$ defined by 
    $$
    \sigma^1(e;h,A):=
    \begin{sis}
        \sigma^2(e;h,A) & \text{ if either } (e;h,A)\in D^1 \text{ or } (e;h,A)\not \in D^2 \text{ or } (e;h,A)\in E^c,\\
        \sigma^2(e;h,A)+1 & \text{ if } (e;h,A)\in E.
    \end{sis}
    $$
  We will now show that $\sigma^1$ is a V-function of characteristic $\chi:=|\sigma^2|$ with $\D(\sigma^1)=D^1$, which clearly concludes the proof since $\sigma^1\geq \sigma^2$ by construction. In order to show this, we will check the conditions of Definition \ref{D:Sigmagn}. 

  Equation \eqref{E:sumUni} holds for $\sigma^1$ by the property in \eqref{D:degsub2a} of $E$ and we also get that the $\sigma^1$-degenerate elements are exactly the elements of $D^1$. Hence, the first condition of Definition \ref{D:Sigmagn}\eqref{E:condUni2} holds true since $D^1$ is a degeneracy subset. 

In order to check Condition \eqref{E:triaUni}, consider a triangle $\Delta=[(e_1;h_1,A_1), (e_2;h_2,A_2), (e_3;h_3,A_3)]$ and we will distinguish several cases:

$\bullet$ If either $(e_i;h_i,A_i)\in D^1$ for all $i$, or $(e_i;h_i,A_i)\not \in D^2$ for all $i$, or one of the elements of $\Delta$ belongs to $D^1$ and the other two do not belong to $D^2$, then we have that $\sigma^1(e_i;h_i,A_i)=\sigma^2(e_i;h_i,A_i)$ for every $i$ and hence Condition \eqref{E:triaUni} for $\sigma^2$ implies that the same condition holds for $\sigma^1$. 

  $\bullet$ If (up to permuting the indices) we have that $(e_1;h_1,A_1)\in D^2-D^1$ and $(e_2;h_2,A_2),(e_3;h_3,A_3)\not \in D^2$, then we compute
  $$
  \sum_{i=1}^3\sigma^1(e_i;h_i,A_i)-\chi=\sum_{i=1}^3\sigma^2(e_i;h_i,A_i)-\chi+|\{(e_1;h_1,A_1)\}\cap E|\in 1+\{0,1\}=\{1,2\}. 
  $$

 $\bullet$ If  we have that $(e_i;h_i,A_i)\in D^2-D^1$ for all $i$, then we compute 
  $$
  \sum_{i=1}^3\sigma^1(e_i;h_i,A_i)-\chi=\sum_{i=1}^3\sigma^2(e_i;h_i,A_i)-\chi+|\{(e_i;h_i,A_i)\in E\}|\in 0+\{1,2\}=\{1,2\},
  $$
where we have used Condition \eqref{D:degsub2c} for $E$. 

  $\bullet$ If (up to permuting the indices) we have that $(e_1;h_1,A_1)\in D^1$ and $(e_2;h_2,A_2),(e_3;h_3,A_3) \in D^2-D^1$, then we compute
  $$
  \sum_{i=1}^3\sigma^1(e_i;h_i,A_i)-\chi=\sum_{i=1}^3\sigma^2(e_i;h_i,A_i)-\chi+|\{(e_2;h_2,A_2), (e_3;h_3,A_3)\}\cap E|\in 0+\{0,1\}=\{0,1\},
  $$
  where we have used Condition \eqref{D:degsub2b} for $E$. 

In each of the above cases, Condition \eqref{E:triaUni} is satisfied for $\sigma^1$ with respect to the triangle $\Delta$, and we are done. 
\end{proof}

\begin{corollary}\label{C:witness}
   For any given $\sigma^2\in \Sigma_{g,n}^\chi$ and any $D^1\geq \D(\sigma^2)$ in $\Deg_{g,n}$, there is a bijection
      $$
      \begin{aligned}
       \left\{\text{Witnesses for } D^1\geq \D(\sigma^2) \right\}  &\xrightarrow{\cong} \left\{\sigma^1\in \Sigma_{g,n} : \text{ such that } \sigma^1\geq \sigma^2 \text{ and }\D(\sigma^1)=D^1 \right\} \\
       E & \mapsto \sigma^2+\chi_E
      \end{aligned}
      $$
      where $\chi_{E}$ is the characteristic function of $E$. 
\end{corollary}
\begin{proof}
    This follows by combining the proofs of Parts \eqref{P:Deg-map2} and \eqref{P:Deg-map3} of Proposition \ref{P:Deg-map}.
\end{proof}

\begin{corollary}\label{C:iso-Sigma}
 Let $\chi, \chi'\in \Z$. If either $n\geq 1$ or $\chi\equiv \pm \chi \mod (2g-2)$, then there exists an isomorphism $\Sigma_{g,n}^\chi\xrightarrow{\cong} \Sigma_{g,n}^{\chi'}$ commuting with the degeneracy maps. 
\end{corollary}
\begin{proof}
This follows from Proposition \ref{P:Deg-map}\eqref{P:Deg-map1}, together with the explicit action of $\wt \PR_{g,n}$ on  $\Sigma$ (see Remark~\ref{R:PR-Pol}).    
\end{proof}

We can decompose the degeneracy map into a separating and a non-separating part, as we now explain. 

First of all, using the decomposition $\Dgn=\Dgns\sqcup \Dgnns$, we get an isomorphism of posets
 \begin{equation}\label{E:dec-Deg}
\begin{aligned}
     \Deg_{g,n}& \xrightarrow{\cong} \Degs_{g,n}\times \Degns_{g,n},\\
     D &\mapsto (D^s:=D\cap \Dgns, D^{ns}:=D\cap \Dgnns).
\end{aligned}
 \end{equation}
 where $\Degs_{g,n}$ is the collection of all subsets of $\Dgns$ satisfying Condition \eqref{D:degsub1a} of Definition \ref{D:degsub} (note that there are no triangles in $\Dgns$) and $\Degns_{g,n}$ is the collection of all subsets of $\Dgnns$ satisfying \eqref{D:degsub1a} and \eqref{D:degsub1b} of Definition \ref{D:degsub}, and the poset structure is defined in both cases as in Definition \ref{D:degsub}\eqref{D:degsub2}. 

Definition \ref{D:Sigmagn2} provides two degeneracy maps, called respectively the \emph{separating} and \emph{non-separating} degeneracy map
\begin{equation}\label{E:Ds-Dns}
\Ds:\Sigmas_{g,n}\to \Degs_{g,n} \text{ and } \Dns:\Sigmans_{g,n}\to \Degns_{g,n}.
\end{equation}
For any $\chi \in \ZZ$, we set $\Ds^{\chi}:=\Ds_{|\Sigmas^\chi}$ and $\Dns^{\chi}:=\Dns_{|\Sigmans^\chi}$.

Finally, Lemma \ref{L:Sigma-s-ns} implies that 
\begin{equation}\label{E:Deg-dec}
    \D=\Ds\times \Dns.
\end{equation}

The separating degeneracy maps $\Ds^\chi$ are easy to describe, as we now show. 

\begin{proposition}\label{P:Dsep}
    Fix $\chi\in \ZZ$.
    \begin{enumerate}
        \item \label{P:Dsep1} The separating degeneracy map $\Ds^\chi$ factors as 
    $$\Ds^{\chi}:\Sigmas^\chi\twoheadrightarrow \Sigmas^\chi/W_{g,n}\cong \Im(\Ds^\chi)\hookrightarrow \Degs_{g,n}$$
    \item \label{P:Dsep2} The image of $\Ds^\chi$ is equal to 
    $$
    \Im(\Ds^\chi)= 
    \begin{cases}
     \Deg^s_{g,n}   \text{ if either } n\geq 1 \text{ or } g \text{ is odd,}\\ 
     \left\{D\in \Deg^s_{g,n}\: : \: (1;g/2,\emptyset)^c=(1;g/2, \emptyset)\in D \text{ if and only if } \chi \text{ is even}
     \right\}  \text{ otherwise.}
    \end{cases}
    $$
    \item \label{P:Dsep3} The restriction on $\Im(\Ds^\chi)$ of the order relation on $\Degs_{g,n}$ is the order relation given by the opposite of the inclusion.
    \item \label{P:Dsep4} The map 
    $\Ds^\chi:\Sigmas_{g,n}^\chi\twoheadrightarrow \Im(\Ds^\chi)$
    is:
    
    $\bullet$ upper lifting, i.e. for any $\sigma\in \Sigmas_{g,n}^{\chi}$ and any $E\in \Im(\Ds^\chi)$ such that $E\geq \Ds(\sigma)$, there exists $\tau\in \Sigmas_{g,n}^\chi$ such that $\tau\geq \sigma$ and $\D(\tau)=E$;
    
    $\bullet$ lower lifting, i.e. for any $\sigma\in \Sigmas_{g,n}^{\chi}$ and any $E\in \Im \Ds^\chi$ such that $\Ds(\sigma)\geq E$, there exists $\tau\in \Sigmas_{g,n}^\chi$ such that $\sigma\geq \tau$ and $\Ds(\tau)=E$.
    
    \end{enumerate}
    
\end{proposition}
\begin{proof}
Part \eqref{P:Dsep1} follows from  Lemma \ref{L:Sigma-ac}.

The other parts follows easily from the Definition \ref{D:Sigmagn2} of $\Sigmas_{g,n}^\chi$ and the above definition of $\Deg_{g,n}$ (using that there are no triangles in $\Dgns$). 
\end{proof}

The non-separating degeneracy maps $\Dns^\chi$ are much more difficult to study and we will give partial results in the subsections that follow. 

We end this subsection with some definitions and corollaries of the above results.

\begin{definition} \label{D:Deg-rea}
Let $D \in \Deg_{g,n}$. We say that $D$ is 
\begin{enumerate}[(i)]
\item \emph{realizable} (resp. \emph{$\chi$-realizable}) if $D=\D(\sigma)$ for some V-function $\sigma \in \Sigma_{g,n}^\chi$ for some $\chi\in \ZZ$  (resp. $\sigma \in \Sigma_{g,n}^\chi$); in this case we  say that $\sigma$ realizes $D$.
\item \emph{classical} (resp. \emph{$\chi$-classical}) if it is realized by some classical  V-function (resp. of characteristic $\chi$), that is, by $\sigma_L$ (see Equation~\eqref{E:for-deg}) for some $L \in \PicRel^{\mathbb{Z}}_{g,n}(\mathbb{R})$ (resp. $L\in \PicRel^{\chi}_{g,n}(\mathbb{R}))$.
\end{enumerate}
Similar definitions can be given for $D\in \Degs_{g,n}$ and $D\in \Degns_{g,n}$.

We will denote by 
\begin{enumerate}
    \item $\Deg^{re}_{g,n}$ (resp. $\Degs^{re}_{g,n}$, resp. $\Degns^{re}_{g,n}$) the subposet of $\Deg_{g,n}$ (resp. $\Degs_{g,n}$, resp. $\Degns_{g,n}$) consisting of realizable universal (resp. separating, resp. non-separating) degeneracy subsets of type $(g,n)$.

    \item $\Deg^{cl}_{g,n}$ (resp. $\Degs^{cl}_{g,n}$, resp. $\Degns^{cl}_{g,n}$) the subposet of $\Deg_{g,n}$ (resp. $\Degs_{g,n}$, resp. $\Degns_{g,n}$) consisting of classical universal (resp. separating, resp. non-separating) degeneracy subsets of type $(g,n)$.
\end{enumerate}

\end{definition}

\begin{remark}\label{R:Deg-rea}
Since the degeneracy map is upper lifting (see Proposition \ref{P:Deg-map}\eqref{P:Deg-map3}), we have that the $(\chi-)$realizability  is upward closed, i.e. if $D_1\geq D_2$ in $\Deg_{g,n}$ and $D_2$ is $(\chi$-)realizable, then $D_1$ is ($\chi$-)realizable. In other words, the subposet $\Deg^{re}_{g,n}\subseteq \Deg_{g,n}$ is upward-closed. And similarly for the subposets $\Degns^{re}_{g,n}\subseteq \Degns_{g,n}$ and $\Degs^{re}_{g,n}\subseteq \Degs_{g,n}$.

Moreover, if $n\geq 1$, then for any $D\in \Deg_{g,n}$ we have that: 
\begin{enumerate}[(i)]
    \item $D$ is realizable if and only if it is $\chi$-realizable for any $\chi \in \Z$ (by Corollary \ref{C:iso-Sigma}). In other words, the map $\D^\chi:\Sigma_{g,n}^\chi\to \Deg^{re}_{g,n}$ is surjective for every $\chi \in \ZZ$.
    \item $D$ is classical if and only if it is $\chi$-classical for any $\chi \in \Z$ (by Corollary \ref{C:iso-Sigma} and Remark \ref{R:PR-Pol}). In other words, the composition  $\Pol^\chi_{g,n}\xrightarrow{\sigma_-}\Sigma_{g,n}^\chi\xrightarrow{\D^{\chi}} \Deg^{cl}_{g,n}$ is surjective for every $\chi \in \ZZ$.
     \item  $D\in \Deg_{g,n}$ is realizable (resp. classical) if and only if $^sD$ and $^{ns}D$ are realizable (resp. classical). In other words, the isomorphism \eqref{E:dec-Deg} induces isomorphisms
     $$
     \begin{sis}
    &\Deg^{re}_{g,n}& \xrightarrow{\cong} \Degs^{re}_{g,n}\times \Degns^{re}_{g,n},\\
    &\Deg^{cl}_{g,n}& \xrightarrow{\cong} \Degs^{cl}_{g,n}\times \Degns^{cl}_{g,n}.\\
     \end{sis}
     $$
    \item  If $D \in \Degs_{g,n}$ then $D$ is classical, and hence realizable, as it follows from Proposition \ref{P:Dsep}\eqref{P:Dsep2} and Remark~\ref{R:class-s-ns}(i). In other words, we have that 
    $$
    \Degs_{g,n}=\Degs^{re}_{g,n}=\Degs^{cl}_{g,n}.
    $$
\end{enumerate}
\end{remark}
We will show in Remark \ref{R:oss-poset} that the subposet $\Deg_{g,n}^{cl}\subset \Deg_{g,n}$ is not, in general, upward-closed.

\vspace{0.1cm}

Recall  that if $(\P,\leq)$ is a finite poset, the \emph{height} of an element $x\in \P$, denoted by $h_{\P}(x)=h(x)$,  is the maximum length of an ascending chain starting from $x$, i.e. the maximum  integer $n$ such that there exists a chain $x_0>x_1>\ldots>x_{n-1}>x_n=x$. 
%In particular, the elements $x$ with $h(x)=0$ are the \emph{maximal elements} of $\P$, while the elements $x\in \P$ with $h(x)=1$ are called \emph{submaximal elements}. 

We will apply the above definitions to the posets $\Sigma_{g,n}^\chi$ (and its separating or non-separating subposets) and to $\Deg_{g,n}$ (and its separating or non-separating subposets and the realizable subposet).
Note that since all the above subsposets are upward-closed, we have that 
$$
\begin{sis}
  & h_{\Sigmas_{g,n}^\chi}= (h_{\Sigma_{g,n}^\chi})_{|\Sigmas_{g,n}^\chi} \text{ and } h_{\Sigmans_{g,n}^\chi}= (h_{\Sigma_{g,n}^\chi})_{|\Sigmans_{g,n}^\chi}, \\
  & h_{\Deg^{(re)}_{g,n}}= (h_{\Deg_{g,n}})_{|\Deg^{(re)}_{g,n}} \text{ and }
   h_{\Degns^{(re)}_{g,n}}= (h_{\Deg_{g,n}})_{|\Degns^{(re)}_{g,n}} \text{ and } h_{\Degs^{(re)}_{g,n}}= (h_{\Deg_{g,n}})_{|\Degs^{(re)}_{g,n}}. \\
\end{sis}
$$
Therefore, we can write $h(\sigma)$ for a (separating or non-separating) V-function $\sigma$ or $h(D)$ for a (separating or non-separating, realizable or not) universal degeneracy subset, we do not have to specify in which poset we are considering the height.

On the other hand, we will show in Remark \ref{R:oss-poset} that, in general, $h_{\Deg^{cl}_{g,n}}\neq (h_{\Deg_{g,n}})_{|\Deg^{cl}_{g,n}}$.

\begin{lemma}\label{C:height}
\noindent
 \begin{enumerate}
     \item \label{C:height1} For any $\sigma\in \Sigma_{g,n}^\chi$, we have that $h(\sigma)=h(\D(\sigma))$.
     \item \label{C:height2} We have that 
     $$
     \begin{sis}
      & h(\sigma)=h(\sigma^{ns})+h(\sigma^s) \text{ for any }  \sigma\in \Sigma_{g,n}^\chi, \\ 
      & h(D)=h(D^{ns})+h(D^s) \text{ for any }  D\in \Deg^{(re)}_{g,n}. \\
     \end{sis}$$
     \item \label{C:height3}  For $D\in \Degs_{g,n}$, we have that 
     $$\displaystyle h(D)=
     \begin{cases}
     \frac{|D|+1}{2} & \text{ if } n=0, \, g \text{ is even and } \left(1;\frac{g}{2},\emptyset\right)\in D, \\
     \frac{|D|}{2} & \text{ otherwise.}
     \end{cases}$$
 \end{enumerate}   
\end{lemma}
\begin{proof}
    Part \eqref{C:height1} follows from Proposition \ref{P:Deg-map}: the inequality $h(\sigma)\geq h(\D(\sigma))$ follows from the fact that $\D$ is upper-lifting, the inequality $h(\sigma)\leq h(\D(\sigma))$ follows from the fact that $\D$ is order-preserving and conservative. 

    Part \eqref{C:height2} follows from the isomorphisms of Lemma \ref{L:Sigma-s-ns} and \eqref{E:dec-Deg}.

    Part \eqref{C:height3} follows from Proposition \ref{P:Dsep}\eqref{P:Dsep3} and the fact that $\Degs_{g,n}$ is the collection of all subsets of $\Dgns$ satisfying Condition \eqref{D:degsub1a} of Definition \ref{D:degsub}.
\end{proof}

\subsection{The case $n=1$}\label{Sub:n1}

The aim of this subsection is to describe the poset $\Sigma_{g,1}^\chi$ of V-functions of type $(g,1)$. The main results here are:
\begin{enumerate}
\item A characterization of the classical V-functions as those with 'a uniform behavior', that is, their values and degeneracy subset only depend upon the value of the log-canonical degree on half-vine types without the marked point (and not on the particular half-vine type): this is Theorem~\ref{T:class-n1}.

\item A complete and explicit description of the poset of representable elements inside $\Deg_{g,1}$, with all its objects and  order relations: this culminates in Theorem~\ref{T:pos-Degn1}.

\item A proof that $\Sigma_{g,1}^{\chi}$ is connected through height $1$: this is Theorem~\ref{T:conn-h1}.
\end{enumerate}
Part~(2) is used in Subsection~\ref{sub:max-wall} -- readers only interested in that subsection  may skip the other parts.

By Lemma \ref{L:Sigma-s-ns} and Proposition \ref{P:Dsep}, we can restrict to the poset $\Sigmans^{\chi}_{g,1}$ of non-separating V-functions. Proposition \ref{P:Deg-map} and Remark  \ref{R:Deg-rea} provide a surjective, order preserving, upper lifting and conservative map
$$
\Dns:\Sigmans_{g,1}^{\chi}\twoheadrightarrow \Degns_{g,1}^{re}\subseteq \Degns_{g,1},
$$
where $\Degns_{g,1}^{re}$ is the subposet of realizable non-separating universal degeneracy subsets of type $(g,1)$. 
%The main result of this subsection is a complete description of the poset  $\Degns_{g,1}^{re}$.

%In order to achieve that, 
For later use, we define some further structure on the stability domain $\Dguno$. First of all, we can partition the stability domain $\Dguno$ into an \emph{empty}  and a \emph{full} stability subdomain  
\begin{equation}\label{E:mar-unmar}
\Dguno=\Dguno^e\sqcup \Dguno^f, \text{ where }
\begin{sis}
& \Dguno^e:=\{(e;h,\emptyset)=(e;h)\in \Dgn\},\\
& \Dguno^f:=\{(e;h,\{1\})\in \Dgn\}.\\    
\end{sis}
\end{equation}
Note that the complementary involution $(-)^c$ defines a bijection between $\Dguno^e$ and $\Dguno^f$. Moreover, the decomposition \eqref{E:mar-unmar} is compatible with the decomposition \eqref{E:Dgn-s-ns}, so that we get a decomposition
\begin{equation}\label{E:dec-Dg1-ns}
\Dgunons=\Dgunou\sqcup \Dgunom, 
\text{ where }
\begin{sis}
& \Dgunou:=\Dgunons\cap \Dguno^e:=\{(e;h,\emptyset)=(e;h)\in \Dgn : e\geq 2\},\\
& \Dgunom:=\Dgunons\cap \Dguno^f:=\{(e;h,\{1\})\in \Dgn : e\geq 2\}.\\    
\end{sis}
\end{equation}
and a similar one for $\Dgunos$. 

Recall the log-canonical degree \eqref{E:lcdeg}, which on $\Dguno^e$ takes the following form
$$
\begin{aligned}
   \delta: \Dguno^e & \longrightarrow [1,2g-2],\\
   (e;h)& \mapsto \delta(e;h):=2h-2+e.
\end{aligned}
$$
We now focus on $\Dgpr:=\Dgunou$, that we call the \emph{primitive} stability domain of type $(g,1)$. We introduce two further structures on $\Dgpr$ that will be useful: 

$\bullet$ For any $(e_1;h_1),(e_2;h_2)\in \Dgpr$ and any index $i$ such that 
$$\max\{1,h_1+h_2+e_1+e_2-g-2\}\leq i \leq \min\{e_1-1,e_2-1\},$$
we define the \emph{$i$-th composition} as
\begin{equation}\label{E:i-compo}
    (e_1;h_1)\circ_i (e_2;h_2):=(e_1+e_2-2i;h_1+h_2+i-1).
\end{equation}
The partial binary operation $\circ_i$ is commutative and associative (when defined) and it satisfies
\begin{equation}\label{E:sum-comp}
\delta(x_1\circ_i x_2)=\delta(x_1)+\delta(x_2).
\end{equation}
Moreover, we have that:
\begin{equation}\label{E:tria-comp}
\text{ every triangle in } \Dguno \text{ is of the form }    \Delta=[x_1,x_2, (x_1\circ_i x_2)^c],
\end{equation}
for some $x_1,x_2\in \Dgpr$ such that $\delta(x_1)+\delta(x_2)\leq 2g-2$, and some $i$.

$\bullet$ We introduce the following order relation on $\Dgpr$: 
\begin{equation}\label{E:order-D}
 (e_1;h_1)\leq (e_2;h_2) \Longleftrightarrow h_1\leq h_2 \text{ and }  e_1+h_1\leq e_2+h_2.
\end{equation}
Note that, by the definition of $\leq$ and the fact that $\delta(e;h)=h+(h+e)-2$, we have that 
\begin{equation}\label{E:deg-leq}
\begin{sis}
& (e_1;h_1)\leq (e_2;h_2)\Rightarrow \delta(e_1;h_1)\leq \delta(e_2;h_2), \\
& (e_1;h_1)< (e_2;h_2)\Rightarrow \delta(e_1;h_1)< \delta(e_2;h_2). \\
\end{sis}
\end{equation}
 It turns out that 
\begin{equation}\label{E:ord-comp}
     (e_1;h_1)< (e_2;h_2)\Longleftrightarrow (e_2;h_2)=(e_1;h_1)\circ_{i} (\wt e;\wt h) \text{ for some } (\wt e;\wt h)\in \Dgpr \text{ and some } i.
\end{equation}
Indeed, the implication $\Leftarrow$ follows from \eqref{E:i-compo} using that $1\leq i \leq \wt e-1$, $\wt h\geq 0$ and that $(\wt e;\wt h)\neq (2,0)$. Conversely, if $(e_1;h_1)<(e_2;h_2)$ then we get the right hand side of \eqref{E:ord-comp} by taking 
$$(\wt e,\wt h, i)=
\begin{cases}
(e_1-e_2+2,h_2+e_2-h_1-e_1, e_1-e_2-1) & \text{ if } e_2\leq e_1, \\
(e_2-e_1+2,h_2-h_1,1) & \text{ if } e_1\leq e_2.
\end{cases}$$

In what follows, we will adopt the following notation: for any $D\subseteq \Dgpr=\Dgunou$ we set 
\begin{equation} \label{e:dtilde}
\wt{D}:=\{x^c : x\in D\}\subseteq \Dgunom.\end{equation}
\begin{remark}\label{R:Deg-Sig-pr}
\noindent
\begin{enumerate}
    \item \label{R:Deg-Sig-pr1} Given $D\in \Degns_{g,1}$, we define the primitive part of $D$ to be 
    $$D^{pr}:=D\cap \Dgpr\subset \Dgpr.$$
    Indeed, $D^{pr}$ uniquely determines $D$ since 
    \begin{equation}\label{E:D-Dpr}
    D=D^{pr}\sqcup \wt{D^{pr}}.
    \end{equation}
    The triangle closure for $D$ (see Definition \ref{D:degsub}\eqref{D:degsub1b}) is equivalent to the following property for $D^{pr}$
    \begin{equation}\label{E:Dpr-prop}
    |[x_1,x_2, x_1\circ_i x_2]\cap D^{pr}|\geq 2 \Rightarrow [x_1,x_2, x_1\circ_i x_2] \subseteq D^{pr},
    \end{equation}
    for any partial composition $(e_1;h_1)\circ_i (e_2;h_2)$. 

    Conversely, any subset $D^{pr}\subseteq \Dgpr$ satisfying \eqref{E:Dpr-prop} defines an element $D\in \Degns_{g,1}$ via \eqref{E:D-Dpr}.

     \item \label{R:Deg-Sig-pr2} Given $\sigma \in \Sigmans_{g,1}^\chi$, we define its primitive part to be
     $$
     \sigma^{pr}:=\sigma_{|\Dgpr}:\Dgpr\to \ZZ.
     $$
    The pair $(\sigma^{pr},\D(\sigma)^{pr})$ uniquely determines $\sigma$ since 
    \begin{equation}\label{E:sig-sigpr}
    \sigma(x)=
    \begin{cases}
    \sigma^{pr}(x) & \text{ if } x\in \Dgpr, \\
    \chi+1-\sigma^{pr}(x^c) & \text{ if } x\not \in \Dgpr \text{ and } x^c\not \in \D(\sigma)^{pr},\\
    \chi-\sigma^{pr}(x^c) & \text{ if } x\not \in \Dgpr \text{ and } x^c\in \D(\sigma)^{pr}.\\
    \end{cases}
    \end{equation}
     The property in \eqref{E:triaUni} for $\sigma$ implies, by Remark \ref{R:triang} and \eqref{E:tria-comp}, that $\sigma^{pr}$ satisfies
     \begin{equation}\label{E:sigmapr-prop}
        \sigma^{pr}(x_1\circ_i x_2)-\sigma^{pr}(x_1)-\sigma^{pr}(x_2)=
        \begin{cases}
          0 & \text{ if either } x_1\in \D(\sigma)^{pr} \text{ or }   x_2\in \D(\sigma)^{pr}, \\
          -1 & \text{ if  } x_1, x_2\not \in \D(\sigma)^{pr} \text{ and }   x_1\circ_i x_2\in \D(\sigma)^{pr}, \\
          \{0,-1\} & \text{ otherwise.}
        \end{cases}
     \end{equation}
     
     Conversely, any pair $(\sigma^{pr},D^{pr})$ such that $\sigma^{pr}$ satisfies \eqref{E:sigmapr-prop} and $D^{pr}$ satisfies \eqref{E:Dpr-prop} defines an element $\sigma\in \Sigmans_{g,n}^\chi$ via \eqref{E:sig-sigpr}.
\end{enumerate}   
\end{remark}

Our first result is a characterization of the V-functions of type $(g,1)$ that are classical.

\begin{theorem}\label{T:class-n1}
A V-function $\sigma\in \Sigma_{g,1}^\chi$ is classical if and only if the following conditions hold:
\begin{enumerate}[(i)]
    \item \label{T:class-n1i} its non-separating degeneracy subset $\Dns(\sigma)=\D(\sigma^{ns})$ is either empty or it is equal to 
    $$
W_{\delta}:=\{x\in \Dgpr : \delta \mid \delta(x)\}\sqcup \{x^c : x\in \Dgpr \text{ and } \delta \mid \delta(x)\} \text{ for } 1\leq \delta\leq 2g-2.$$
\item \label{T:class-n1ii} $\sigma$ is \emph{uniform}, i.e. $\sigma(e;h)=\sigma(e';h')$ for each $(e;h),(e';h')\in\Dgpr$ such that $\delta(e;h)=\delta(e';h')$.
\end{enumerate}
Moreover, each subset of the form $W_\delta$  arises as the degeneracy subset of some classical V-functions.
\end{theorem}
\begin{proof}
Observe that: $\sigma$ is uniform if and only if $\sigma^{ns}$ is uniform by definition, $\sigma$ is  classical if and only if $\sigma^{ns}$ is  classical by Remark \ref{R:class-s-ns}, and that $\Dns(\sigma)=\D(\sigma^{ns})$. Hence, we can assume for the rest of the proof that $\sigma\in \Sigmans_{g,n}^\chi$.

Let us first show the only if implication. According to \eqref{E:for-deg}, the classical non-separating V-functions of characteristic $\chi$ are given by (for any $x=(e;h,A)\in \Dgunons$)
\begin{equation}\label{E:sigma-beta}
\sigma_{\beta}(x):=\sigma_{\beta\omega_{\pi}+(\chi-(2g-2)\beta)\Sigma_1}(x)=
\begin{cases}
    \lceil \delta(x)\beta\rceil & \text{ if } x\in \Dgunou=\Dgpr, \\
     \lceil \chi-\beta\delta(x^c)\rceil & \text{ if } x\in \Dgunom,
\end{cases} 
\end{equation}
as $\beta$ varies in $\RR$. The above equation shows that any classical (non-separating) V-function is uniform and that 
\begin{equation}\label{E:D-sigbeta}
 \D(\sigma_{\beta})^{pr}=\left\{x=(e;h)\in \Dgpr: \: \beta\,\delta(x)\in\mathbb Z\right\}.
\end{equation}
Therefore, if $\beta\not\in \QQ$ then $\D(\sigma_\beta)^{pr}=\emptyset$, which implies that  $\D(\sigma_\beta)=\emptyset$ by Remark \ref{R:Deg-Sig-pr}\eqref{R:Deg-Sig-pr1}. Otherwise, we write $\beta=\frac{p}{q}$ in lowest terms with $q>0$. Then $\beta\,\delta(x)\in\mathbb Z$ is equivalent to $q\mid \delta(x)$. If $q>2g-2$  then again $\D(\sigma_\beta)=\emptyset$. On the other hand, if $q\leq 2g-2$ then $\D(\sigma_\beta)^{pr}=W_q^{pr}$, which implies that $\D(\sigma_\beta)=W_q$ by Remark \ref{R:Deg-Sig-pr}\eqref{R:Deg-Sig-pr1}. This computation also proves the last statement since 
$$
\begin{sis}
& \emptyset=\D(\sigma_\beta) \text{ for any  } \beta\not \in \QQ, \\
& W_\delta=\D(\sigma_{1/\delta}) \text{ for } 1\leq \delta\leq 2g-2.
\end{sis}
$$

We now show the if implication. Let $\sigma\in\Sigmans^\chi_{g,1}$ be a uniform V-function, and denote by $\sigma(\delta):=\sigma(e;h)$ for each $(e;h)\in\Dgpr$ such that $\delta(e;h)=\delta\in [1,2g-2]$. 
We now distinguish two cases according to $\D(\sigma)$. 

$\un{\text{Case I:}}$ $\Dns(\sigma)=W_\delta$ for some $1\leq \delta\leq 2g-2$. 

We will prove that (see \eqref{E:sigma-beta} for the notation) 
$\displaystyle \sigma=\sigma_{\frac{\sigma(\delta)}{\delta}}$,
which amounts to showing, by \eqref{E:sigma-beta} and \eqref{E:D-sigbeta} together with Remark \ref{R:Deg-Sig-pr}\eqref{R:Deg-Sig-pr2}, the following two properties:
    \begin{equation}\label{E:betaWdelta1}
     \sigma(\delta')=\bigg\lceil\frac{\sigma(\delta)\delta'}{\delta}\bigg\rceil  \text{ for any } 1\leq \delta'\leq 2g-2.
    \end{equation}
    \begin{equation}\label{E:betaWdelta2} 
     \frac{\sigma(\delta)\delta'}{\delta}\in \ZZ  \Longleftrightarrow  \delta\vert \delta', \text{ for any } 1\leq \delta'\leq 2g-2.
    \end{equation}
Observe that  \eqref{E:betaWdelta2} is equivalent to the condition 
\begin{equation}\label{E:coprime}
 \gcd(\sigma(\delta),\delta)=1, 
\end{equation}
 which we are now going to check.
Consider any decomposition $\delta=k\cdot h$ with $h,k\in \NN_{>0}$ such that $h\neq 1$ (so that $k\neq \delta$).
Property \eqref{E:sigmapr-prop} and the assumption  $\D(\sigma)=W_{\delta}$ imply that 
$$
\begin{sis}
& -1\leq \sigma(ik)-\sigma((i-1)k)-\sigma(k)\leq 0 \text{ for any } 2\leq i \leq h-1,\\
& \sigma(\delta)=\sigma(hk)=\sigma((h-1)k)+\sigma(k)-1.
\end{sis}
$$
By summing the above inequalities, we get that 
\begin{equation}\label{E:inequa1}
-(h-1)\leq \sigma(\delta)-h\sigma(k)\leq -1.
\end{equation}
We deduce that $h$ does not divide $\sigma(\delta)$ and, since this is true for all divisors $1<h$ of $\delta$, we infer that \eqref{E:coprime} holds true. 

We now check \eqref{E:betaWdelta1} by distinguishing three cases.

$(a)$ If $\delta'\vert \delta$ then \eqref{E:betaWdelta1} follows from \eqref{E:inequa1} if $\delta'<\delta$ and it is trivial if $\delta'=\delta$. 

$(b)$ If $1\leq \delta'<\delta$, we proceed by induction on  $\delta'$. We can assume that $\delta\geq 2$, for otherwise there is nothing to prove. The base case of the induction is $\delta'=1$, which follows from Case (a). 

We now prove the inductive step. Write $\delta-\delta'=k\delta'+l$ for some $k,l\in \NN$ such that $0\leq l<\delta'$. We can assume that $1\leq l$, for otherwise $\delta'\vert \delta$ and the result follows from Case (a). Hence, by induction and the definition of $l$, we have that
\begin{equation}\label{E:indu-l}
\sigma(l)=\bigg\lceil\frac{\sigma(\delta)l}{\delta}\bigg\rceil=\bigg\lceil\frac{\sigma(\delta)(\delta-(k+1)\delta')}{\delta}\bigg\rceil=\sigma(\delta)+\bigg\lceil-\frac{\sigma(\delta)(k+1)\delta'}{\delta}\bigg\rceil=\sigma(\delta)-\bigg\lceil\frac{\sigma(\delta)(k+1)\delta'}{\delta}\bigg\rceil+1,
\end{equation}
where in the last equality we have used 
$$\frac{\sigma(\delta)(k+1)\delta'}{\delta}\not \in \ZZ
\text{ by  \eqref{E:betaWdelta2}  and } \delta=(k+1)\delta'+l>(k+1)\delta'$$
and the property $\lceil -x\rceil=-\lceil x\rceil +1$ for any $x\not \in \ZZ$.

Using this, a repeated application of  \eqref{E:sigmapr-prop} (together with the assumption that $\D(\sigma)=W_\delta$) yields the following
\begin{equation}\label{E:tante-in}
\begin{sis}
& \sigma(\delta)-\sigma(\delta')-\sigma(\delta-\delta')=-1,\\
& -1\leq\sigma(\delta-\delta')-\sigma(k\delta')-\sigma(l)\leq 0 \quad \text{ if } k\geq 1,\\
&-(k-1)\leq\sigma(k\delta')-k\sigma(\delta')\leq 0 \quad \text{ if } k\geq 1. 
\end{sis}    
\end{equation}
By summing all the above inequalities, we get the following inequality (note that this also holds true when $k=0$)
\begin{equation}\label{E:final-in}
-(k+1)\leq \sigma(\delta)-(k+1)\sigma(\delta')-\sigma(l)\leq -1.
\end{equation}
From this inequality and \eqref{E:indu-l}, we deduce that
\begin{equation}\label{E:bounddelta'}
 \begin{sis}
 & \sigma(\delta')\geq \frac{1}{k+1}\bigg(\sigma(\delta)-\sigma(l)+1\bigg)=\frac{1}{k+1}\bigg\lceil\frac{\sigma(\delta)(k+1)\delta'}{\delta}\bigg\rceil\geq \bigg\lceil\frac{\sigma(\delta)\delta'}{\delta}\bigg\rceil-1+\frac{1}{k+1},\\
& \sigma(\delta')\leq\frac{1}{k+1}\bigg(\sigma(\delta)-\sigma(l)+k+1\bigg)=\frac{1}{k+1}\bigg(\bigg\lceil\frac{\sigma(\delta)(k+1)\delta'}{\delta}\bigg\rceil+k\bigg)\leq \bigg\lceil\frac{\sigma(\delta)\delta'}{\delta}\bigg\rceil+1-\frac{1}{k+1},
 \end{sis}
\end{equation}
where we have used that
$$
\lceil x\rceil -1+\frac{1}{h}\leq \frac{\lceil hx\rceil}{h} \leq \lceil x\rceil \quad \text{ for any } x\not \in \ZZ \text { and any } h\in \NN_{>0}.$$
The two inequalities in \eqref{E:bounddelta'} force the equality $\displaystyle \sigma(\delta')=\bigg\lceil\frac{\sigma(\delta)\delta'}{\delta}\bigg\rceil$, i.e.  \eqref{E:betaWdelta1} holds true for $\delta'$. 
%$\sigma(\delta')=\sigma_{\frac{\sigma(\delta)}{\delta}}(\delta')$ and that $\frac{\sigma(\delta)\delta'}{\delta}\notin\ZZ$.

$(c)$ For an arbitrary $1\leq \delta'\leq 2g-2$, write $\delta'=k\delta+l$, with $0\leq l<\delta$ and $k\geq 0$. If $l=0$ then we compute 
    $$
\sigma(\delta')=\sigma(k\delta)=k\sigma(\delta)= \frac{\sigma(\delta)\delta'}{\delta}= \bigg\lceil \frac{\sigma(\delta)\delta'}{\delta}\bigg\rceil,
$$
where we have used \eqref{E:sigmapr-prop} and the assumption that $\D(\sigma)=W_\delta$ in the second equality.

    Otherwise, if $1\leq l <\delta$ then we compute
$$
\begin{aligned}
\sigma(\delta')&=\sigma(k\delta+l)=k\sigma(\delta)+\sigma(l) & \text{ by \eqref{E:sigmapr-prop} using that } \D(\sigma)=W_{\delta}\\
&= k\sigma(\delta)+\bigg\lceil \frac{\sigma(\delta)l}{\delta}\bigg\rceil & \text{ by Case (b) applied to $l$}\\
&= \bigg\lceil \frac{\sigma(\delta)(k\delta+l)}{\delta}\bigg\rceil=\bigg\lceil \frac{\sigma(\delta)\delta'}{\delta}\bigg\rceil.
\end{aligned}
$$
which concludes the proof of \eqref{E:betaWdelta1}.
    
$\un{\text{Case II:}}$  $\Dns(\sigma)=\emptyset$.

   We will show  that there exists $\beta\in \RR-\QQ$ such that $\sigma=\sigma_\beta$. The latter amounts to showing, by \eqref{E:sigma-beta} and \eqref{E:D-sigbeta} together with Remark~\ref{R:Deg-Sig-pr}\eqref{R:Deg-Sig-pr2}, the equality 
    \begin{equation}\label{E:beta2}
    \sigma(\delta)=\lceil\beta\delta\rceil,  \quad\textup{for any}\quad 1\leq\delta\leq2g-2.
    \end{equation}
The above condition \eqref{E:beta2} for $\beta$ is equivalent to 
\begin{equation}\label{E:beta3}
    \frac{\sigma(\delta)-1}{\delta}<\beta <  \frac{\sigma(\delta)}{\delta},  \quad\textup{for any}\quad 1\leq\delta\leq2g-2.
    \end{equation}
Hence, we have to show that 
    \begin{equation}\label{E:beta4}
    \bigcap_{1\leq \delta\leq 2g-2} \left(\frac{\sigma(\delta)-1}{\delta}, \frac{\sigma(\delta)}{\delta}\right)\neq \emptyset,
    \end{equation}
    where $(a,b)$ denotes as usual the open interval of real numbers larger than $a$ and smaller than $b$.  
Condition~\eqref{E:beta4} is equivalent to showing that the lower bound of each interval is smaller than the upper bound of every other interval, which translates into the following inequalities 
$$
\frac{\sigma(\delta)-1}{\delta}< \frac{\sigma(\delta')}{\delta'} \text{ for any } 1\leq \delta, \delta'\leq 2g-2, 
$$
or equivalently into 
\begin{equation}\label{E:beta5}
-\delta+1\leq \delta'\sigma(\delta)-\delta\sigma(\delta')\leq \delta'-1 \text{ for any } 1\leq \delta'<\delta\leq 2g-2.
\end{equation}
We will now check \eqref{E:beta5} by distinguishing two cases:

$(a)$ If $\delta'\vert \delta$, i.e. $\delta=k\delta'$ for some $k\in \NN_{>0}$, then the iteration of Property \eqref{E:sigmapr-prop} and the assumption  $\D(\sigma)=\emptyset$ imply that
$$
-(k-1)\leq \sigma(\delta)-k\sigma(\delta')\leq 0.
$$
By multiplying the above inequality times $\delta'$ and using that $\delta'\geq 1$, we obtain \eqref{E:beta5} in this case.

$(b)$ In the general case $1\leq \delta'<\delta$, we proceed by induction on the pairs $(\delta'<\delta)$, lexicographically ordered. The base case of the induction is the case $(\delta'=1< \delta)$ where the result follows from Case (a). Write $\delta=k\delta'+l$ with $k\in \NN_{>0}$ and $0\leq l <\delta'$. If $l=0$, then $\delta'\vert \delta$ and the result follows from Case (a). So we can assume that $1\leq l <\delta'$. 

Again, Property \eqref{E:sigmapr-prop} and the assumption  $\D(\sigma)=\emptyset$ imply that
\begin{equation}\label{E:ineq-bet1}
-k\leq \sigma(\delta)-k\sigma(\delta')-\sigma(l)\leq 0.
\end{equation}
On the other hand, the induction assumption applied to the pair $(l<\delta')$ (which is lexicographically smaller than $(\delta'< \delta)$) implies  
\begin{equation}\label{E:ineq-bet2}
-\delta'+1\leq l\sigma(\delta')-\delta'\sigma(l)\leq l-1.
\end{equation}
By multiplying \eqref{E:ineq-bet1} for $\delta'$ and subtracting \eqref{E:ineq-bet2}, we deduce that \eqref{E:beta5} holds true. 
\end{proof}

We now turn our attention to the description of the poset $\Degns_{g,1}^{re}$ of realizable non-separating universal degeneracy subsets of type $(g,1)$.

First we define the \emph{log-canonical degree} of an element of $D\in \Degns_{g,1}$ to be 
\begin{equation}\label{E:lc-deg-D}
\delta(D):= \min_{x\in D^{pr}} \{\delta(x)\}\in [1,2g-2]\cup\{+\infty\}
\end{equation}
with the convention that $\delta(\emptyset)=+\infty$.
%A similar definition applies to elements of $ \Degns_{g,n}$ or $\Degs_{g,n}$.

We first classify the elements of $\Degns_{g,1}^{re}$ whose log-canonical degree is at most $g-1$. 

\begin{proposition}\label{P:deg-piccoli}
    Let $D\in \Degns_{g,1}$ with $\delta(D)\leq g-1$.
    Then 
    $$
    D \text{ is realizable } \Longleftrightarrow D=W_{\delta(D)}.
    $$
    \end{proposition}
   
\begin{proof}
If $D=W_{\delta(D)}$ then by Theorem \ref{T:class-n1} $D$ is classical, and hence it is realizable. 

Conversely, assume that $D$ is realizable, i.e. $D=\D(\sigma)$ for some $\sigma\in \Sigmans_{g,1}^\chi$, and let us show that $D=W_{\delta(D)}$. By Remark \ref{R:Deg-Sig-pr}\eqref{R:Deg-Sig-pr1}, it is enough to show that $D^{pr}=W_{\delta}^{pr}$. We divide the proof in three steps.

\un{Step I:} If $x=(e;h)\in \Dgpr$ is such that $\delta(x)=\delta(D)$ then $x\in D^{pr}$. 

Indeed, by contradiction suppose that $x\not \in D^{pr}$. By definition of $\delta(D)$, there exists $\wt x=(\wt e;\wt h)\in D^{pr}$ with $\delta(\wt x)=\delta(D)$. Consider the two compositions
$$
(e;h)\circ_{e-1} (e;h)=(2;\delta(D))=(\wt e;\wt h)\circ_{\wt e-1} (\wt e;\wt h).
$$
Since $\wt x\in D^{pr}$, we have that $(2;\delta(D))\in D^{pr}$ by \eqref{E:Dpr-prop}. We now apply \eqref{E:sigmapr-prop} to the two compositions above in order to get 
$$
\begin{sis}
 \sigma(2;\delta(D))-2\sigma(\wt x)=0,\\
 \sigma(2;\delta(D))-2\sigma(x)=-1.
\end{sis}
$$
This is absurd since the first equation implies that $\sigma(2;\delta(D))$ is even and the second equation implies that $\sigma(2;\delta(D))$ is odd. 

\un{Step II:} $W_{\delta(D)}^{pr}\subset D^{pr}$. 

Indeed, we will show, by induction on $k\geq 1$, that for any $x\in \Dgpr$ we have
$$\delta(x)=k\delta(D)\Rightarrow x\in D^{pr}.$$
The base case $k=1$ of the induction follows from Step I. 
Assume that $k\geq 2$. We claim  that $x=(e;h)> (\delta(D)+2;0)$ since $h\geq 0$ and 
$$2h+2e\geq (2h+e-2)+4=k\delta(D)+4\geq 2\delta(D)+4>2(r+2), $$
where we used that $e\geq 2$ and that $\delta(x)=k\delta(D)\geq 2\delta(D)$. Therefore, \eqref{E:ord-comp} implies that we can write 
\begin{equation}\label{E:x-tria}
x=(e;h)=(\delta(D)+2;0)\circ_i \wt x \text{ for some } \wt x \in \Dgpr \text{ and some } i.
\end{equation}
By \eqref{E:sum-comp} and \eqref{E:div-delta}, we get that $\delta(\wt x)=\delta(x)-\delta(\delta(D)+2;0)=(k-1)\delta(D)$, which implies that $\wt x\in D^{pr}$ by our induction hypothesis. By applying \eqref{E:Dpr-prop} to the composition \eqref{E:x-tria}, we get that $x\in D^{pr}$, as required. 

\un{Step III:} $W_{\delta(D)}^{pr}\supseteq D^{pr}$.

Indeed, by contradiction, assume that there exists $x=(e;h)\in D^{pr}-W_{\delta(D)}^{pr}$. By definition of $\delta(D)$, we must have $\delta(x)\geq \delta(D)$. Then, using that $x\not \in W_{\delta(D)}^{pr}$, we can write 
\begin{equation}\label{E:div-delta}
    \delta(x)=k\delta(D)+r \text{ for some } 0<r<\delta(D) \text{ and } k\geq 1. 
    \end{equation}
We now claim that $x=(e;h)> (r+2;0)$ since $h\geq 0$ and 
$$2h+2e\geq (2h+e-2)+4=k\delta(D)+r+4\geq \delta(D)+r+4>2r+4, $$
where we used that $e\geq 2$ and \eqref{E:div-delta}. Therefore, \eqref{E:ord-comp} implies that we can write 
\begin{equation}\label{E:trian-x}
x=(e;h)=(r+2;0)\circ_i \wt x \text{ for some } \wt x \in \Dgpr \text{ and some } i.
\end{equation}
By \eqref{E:sum-comp} and \eqref{E:div-delta}, we get that $\delta(\wt x)=\delta(x)-\delta(r+2;0)=k\delta(D)$, which implies that $\wt x\in W_{\delta(D)}^{pr}\subseteq D^{pr}$ by Step II. By applying \eqref{E:Dpr-prop} to the composition \eqref{E:trian-x}, we get that $(r+2;0)\in D^{pr}$ which is an absurd since $\delta(r+2;0)=r<\delta(D)$.
\end{proof}

\begin{example}\label{E:non-real}
     Note that there exist non-realizable elements $D\in \Degns_{g,1}$ with $\delta(D)\leq g-1$. 
     
     %The above Proposition \ref{P:deg-piccoli} is false without the assumption that $D$ is realizable. 
     For example:  $D_1=\{(4;0),(2;2),(4;0)^c,(2;2)^c\}$ and $D_2=\{(2;1),(2;2),(2;1)^c,(2;2)^c\}$ are elements of $\Degns_{3,1}$ with $\delta(D_1)=\delta(D_2)=2=g-1$, but $D_1,D_2\subsetneq W_{2}$ since $(2;1)\in W_2-D_1$ and $(4;0)\in W_2-D_2$. 
\end{example}

We now classify the elements of $\Degns_{g,1}$ whose log-canonical degree is larger than $g-1$. 
In order to state our classification result, we define the poset 
$$
\Dgbig:=\{x\in \Dgpr : \delta(x)\geq g\},
$$
endowed with the order relation in \eqref{E:order-D}. 
Recall that an \emph{antichain} in a poset is a subset whose elements are pairwise incomparable. 

\begin{proposition}\label{P:deg-grandi}
    \noindent
    \begin{enumerate}[(i)] 
        \item \label{P:deg-grandi1}
        For every antichain $A$ of $\Dgbig$, the subset $D(A):=A\sqcup \wt A$  (see \eqref{e:dtilde}) is a realizable element of $\Degns_{g,1}$ with $\delta(D(A))\geq g$. 
        \item \label{P:deg-grandi2}
        Every $D\in \Degns_{g,1}$ with $\delta(D)\geq g$ is equal to $D(A)$ for a unique antichain $A$ of $\Dgbig$. 
    \end{enumerate}
\end{proposition}
\begin{proof}
Part \eqref{P:deg-grandi2}: let $D\in \Degns_{g,1}$ with $\delta(D)\geq g$. Since $D=D^{pr}\sqcup \wt{D^{pr}}$ by \eqref{E:D-Dpr}, it is enough to prove that $D^{pr}\subseteq \Dgbig$ is an antichain. Indeed, by contradiction, suppose that $x_1,x_2\in D^{pr}$ are such that $x_1< x_2$. Then \eqref{E:ord-comp} gives that  
$$
x_2=x_1\circ_i y \text{ for some } y\in \Dgpr.
$$ 
Applying \eqref{E:Dpr-prop} to the above composition, we deduce that $y\in D^{pr}$. Moreover, \eqref{E:sum-comp} and the assumption $\delta(D)\geq g$ imply that $$\delta(y)=\delta(x_2)-\delta(x_1)\leq 2g-2-g=g-2<\delta(D),$$ 
which is an absurd by the definition of $\delta(D)$. 

Part \eqref{P:deg-grandi1}: let $A$ be an antichain of $\Dgbig$. 
We claim that 
\begin{equation}\label{E:inter-A}
|[x_1,x_2,x_1\circ_i x_2]\cap A|\leq 1,
\end{equation}
for any two elements $x_1,x_2\in \Dgpr$ such that $x_1\circ_i x_2$ is well-defined. Indeed, since the composition $x_1\circ_i x_2$ is well-defined, then at least one of the two elements $\{x_1,x_2\}$, say $x_1$, does not belong to $\Dgbig$ and hence does not belong to $A$. On the other hand, since $x_2< x_1\circ_i x_2$ by \eqref{E:ord-comp} and $A$ is an antichain, it follows that $A$ cannot contain both $x_2$ and $x_1\circ_i x_2$, and \eqref{E:inter-A} is proved. 
Property \eqref{E:inter-A} implies that $A$ satisfies \eqref{E:Dpr-prop} and, hence, $D(A)=A\sqcup \wt A$ belongs to $\Degns_{g,1}$ by Remark \ref{R:Deg-Sig-pr}\eqref{R:Deg-Sig-pr1}. Moreover, we have  that $\delta(D(A))\geq g$ since $D(A)^{pr}=A\subset \Dgbig$.

It remains to prove that any such $D(A)$ is realizable. Consider the function 
\begin{equation}\label{E:tauA}
   \begin{aligned}
     \tau_A: \Dgpr& \longrightarrow \ZZ\\  
     x &\mapsto 
     \begin{cases}
     2 & \text{ if } x>y \text{ for some } y\in A,\\
     1 & \text{ otherwise.}
     \end{cases}
   \end{aligned} 
\end{equation}
Consider now two elements $x_1, x_2\in \Dgpr$ such that there exists $x_1\circ_i x_2\in \Dgpr$. One of the two elements $\{x_1,x_2\}$, say $x_1$, must have log-canonical degree less than or equal to $g-1$, which then implies that $\tau_A(x_1)=1$ because $x_1$ cannot dominate any element of $A$ since $A\subseteq \Dgbig$. We now distinguish the following cases:

$\bullet$ If $x_2\in A$ then we have that $\tau_A(x_2)=1$ because $x_2$ cannot strictly dominate another element of $A$ since $A$ is an antichain and $\tau_A(x_1\circ_i x_2)=2$ because $x_1\circ_i x_2> x_2\in A$ by \eqref{E:ord-comp}.

$\bullet$ If $x_1\circ_i x_2\in A$ (which then implies that $x_2\not \in A$ since $x_1\circ_i x_2> x_2\in A$ and $A$ is an antichain) then we have that $\tau_A(x_2)=\tau_A(x_1\circ_i x_2)=1$ because neither $x_1\circ_i x_2$ nor $x_2$ (which is smaller than $x_1\circ_i x_2$) can strictly dominate an element of $A$ since $A$ is an antichain. 

$\bullet$ If $x_2,x_1\circ_i x_2\not\in A$ then we must have that $(\tau_A(x_2),\tau_A(x_1\circ_i x_2))\neq (2,1)$ because if $x_2$ strictly dominate an element of $A$ then the same is true for $x_1\circ_i x_2$ (which is bigger then $x_2$).  Hence, we must have that $(\tau_A(x_2),\tau_A(x_1\circ_i x_2))=(1,1), (1,2),(2,2)$. 

By putting all the cases together, we deduce that 
\begin{equation}\label{E:tauA-prop}
        \tau_A(x_1\circ_i x_2)-\tau_A(x_1)-\tau_A(x_2)=
        \begin{cases}
          0 & \text{ if either } x_1\in A \text{ or }   x_2\in A, \\
          -1 & \text{ if  } x_1, x_2\not \in A  \text{ and }   x_1\circ_i x_2\in A, \\
          \{0,-1\} & \text{ otherwise.}
        \end{cases}
     \end{equation}
We can now extend $\tau_A$ to a function 
$$
\begin{aligned}
    \sigma_A: \Dgunons& \longrightarrow \ZZ\\
    x & \mapsto 
    \begin{cases}
    \tau_A(x) & \text{ if } x\in \Dgpr, \\
    \chi+1-\tau_A(x^c) & \text{ if } x\not \in \Dgpr \text{ and } x^c\not \in A,\\
    \chi-\tau_A(x^c) & \text{ if } x\not \in \Dgpr \text{ and } x^c\in A.\\
    \end{cases}
\end{aligned}
$$
Using that $\tau_A$ satisfies \eqref{E:tauA-prop} with respect to $A=D(A)^{pr}$, we get that $\sigma_A\in \Sigmans^\chi_{g,1}$ with $\D(\sigma_A)=D(A)$ by Remark \ref{R:Deg-Sig-pr}\eqref{R:Deg-Sig-pr2}. 
\end{proof}

\begin{example}\label{Ex:antichain}
For any $g\leq \delta \leq 2g-2$, the subset 
$$
A_{\delta}:=\{x\in \Dgbig : \delta(x)=\delta\}
$$
is an antichain of $\Dgbig$ by \eqref{E:deg-leq}. Moreover, we have that 
$$
D(A_{\delta})=W_{\delta}$$.
Indeed we have that, if $x\in \Dgpr$ is such that $\delta\mid \delta(x)$, then $\delta=\delta(x)$ since $g\leq \delta$ and $\delta(x)\leq 2g-2$. Therefore, by Theorem \ref{T:class-n1}, we recover some classical elements of $\Degns_{g,1}$.

However, there are elements of $\Degns_{g,1}$ with $\delta(D)\geq g$ that are not classical (but yet they are realizable by Proposition \ref{P:deg-grandi}\eqref{P:deg-grandi1}). For example:
\begin{itemize}
    \item If $A\subsetneq A_{\delta}$ for some $g\leq \delta\leq 2g-2$ (such subsets exist for all $g\geq 4$), then $A$ is an antichain and its associated $D(A)$ is properly contained in $W_{\delta}$, and hence it is not classical. 
    \item  $\displaystyle A_0=\left\{(g-3k;2k+1)\in \Dgpr: 0\leq k\leq \lfloor \frac{g-2}{3}\rfloor\right\}$ is an antichain such that $D(A_0)$ is not contained in any of the $W_{\delta}$ with $g\leq \delta$, for any $g\geq 5$.
\end{itemize}
\end{example}

We now describe the poset of realizable elements.

\begin{theorem}\label{T:pos-Degn1}
The poset $\Degns_{g,1}^{re}$ consists of the following elements
$$
\Degns_{g,1}^{re}= \{\emptyset\} \cup \bigcup_{1\leq \delta \leq g-1} \{W_\delta\}\bigcup_{\substack{A\subseteq \Dgbig\\ \text{antichain}}} \{D(A)\} \ ,
$$
subject to the following non-trivial order relations:
  \begin{enumerate}[(i)]
      \item \label{T:pos-Degn1a}
      $W_{\delta}<\emptyset$ for any $1\leq \delta\leq g-1$. Moreover, the only two witnesses for $W_{\delta}<\emptyset$ are $W_{\delta}^{pr}$ and $\wt{W_{\delta}^{pr}}$.
      \item \label{T:pos-Degn1b}
      $D(A)< D(B)$ for any antichains $A,B\subseteq \Dgbig$ such that $A\supsetneq B$. Moreover, the witnesses for $D(A)< D(B)$ are all the subsets of the form $\mathcal{E}\sqcup \wt{\mathcal{E}^c}$ for $\mathcal{E}\subseteq A-B$.
  \end{enumerate}
\end{theorem}
\begin{proof}
The description of the elements of $\Degns_{g,1}^{re}$ follows from Propositions \ref{P:deg-piccoli} and \ref{P:deg-grandi}. It remains to prove that the unique non-trivial order relations are the ones given in \eqref{T:pos-Degn1a} and \eqref{T:pos-Degn1b}.

Let us, first, prove Part~\eqref{T:pos-Degn1b}. First of all, if $A,B$ are two antichains of $\Dgbig$ such that $A\supsetneq B$, then $D(A)\supsetneq D(B)$. Then the subsets $E$ of $D(A)-D(B)$ that satisfy Condition~\eqref{D:degsub2a} of Definition \ref{D:degsub} are all the ones of the form $E:=\E\sqcup \wt{\E^c}$ for a subset $\E\subseteq A-B$. Such subsets are also witnesses for $D(A)\leq D(B)$ because Conditions \eqref{D:degsub2b} and \eqref{D:degsub2c} of Definition \ref{D:degsub} holds trivially true since there are no triangles entirely contained in $D(A)$. Therefore, Part \eqref{T:pos-Degn1b} is proved. 

Let us now prove Part \eqref{T:pos-Degn1a}. First of all, we have that $W_{\delta}<\emptyset$ with $W_{\delta}^{pr}$ and $\wt{W_{\delta}^{pr}}$ being two witnesses since $W_{\delta}=W_{\delta}^{pr}\sqcup \wt{W_{\delta}^{pr}}$ and any triangle contained in $W_{\delta}$ contains always two elements of $W_{\delta}^{pr}$ and one element of $W_{\delta}^{pr}$, so that the conditions of Definition \ref{D:degsub}\eqref{D:degsub2} hold true. 

It remains to show that $W_{\delta}^{pr}$ and $\wt{W_{\delta}^{pr}}$ are the only two witness for $W_{\delta}<\emptyset$. So suppose that $E$ is a witness for $W_{\delta}<\emptyset$. Up to switching $E$ with $E^c$, we can assume that $(\delta+2;0)\in E$. Then we have to show that $E=W_{\delta}^{pr}$. Using Condition~\eqref{D:degsub2a} of Definition \ref{D:degsub}, it is enough to show that $W_{\delta}^{pr}\subseteq E$. We will show this in three steps:

$\star$ $W_{2\delta}\subseteq E$.

Given $(e;h)\in W_{2\delta}$ we can write $(e;h)=(\delta+2;0)\circ_{h+1} (\delta+2;0)$ and then consider the triangle $\Delta=[(\delta+2;0),(\delta+2;0),(e;h)^c)]$ in $W_{\delta}$ (see \eqref{E:tria-comp}): since the first two elements belong to $E$,  Condition~\eqref{D:degsub2c} of Definition \ref{D:degsub} implies that the last element does not belong to $E$, which is equivalent to  $(e;h)\in E$.

$\star$ $W_{\delta}\subseteq E$.

Given $(e;h)\in W_{\delta}$ we can form $(e;h)\circ_{e-1} (e;h)=(2;2h+e-1)\in W_{2\delta}$ and then consider the triangle $\Delta=[(e;h),(e;h),(2;2h-2+e)^c)]$ in $W_{\delta}$ (see \eqref{E:tria-comp}): since the last element belongs to $\wt{E}$ (and hence does not belong to $E$) and the first two elements are equal, Condition \eqref{D:degsub2c} of Definition \ref{D:degsub} implies that  $(e;h)\in E$.

$\star$ $W_{m\delta}\subseteq E$ for any $m\geq 1$ such that $m\delta\leq 2g-2$.

We prove this by induction on $m$.  The base case $m=1$ has already been proved. Assume that $m\geq 2$. Any element $(e;h)\in W_{m\delta}$ is such that $(e;h)> (\delta+2;0)$ since $h\geq 0$ and 
$$2(e+h)=\delta(e;h)+2+e\geq 2\delta+4=2(\delta+2).$$
Then, \eqref{E:ord-comp} implies that we can write $(e;h)=(\delta+2;0)\circ_i (\wt e;\wt h)$ for some $(\wt e;\wt h)$ which belongs to $W_{(m-1)\delta}$ by \eqref{E:sum-comp} and hence to $E$ by induction. Consider the triangle $\Delta=[(\delta+2;0),(\wt e;\wt h),(e;h)^c)]$ in $W_{\delta}$ (see \eqref{E:tria-comp}): since the first two elements belong to $E$, Condition~\eqref{D:degsub2c} of Definition \ref{D:degsub} implies that the last element does not belong to $E$, which is equivalent to  $(e;h)\in E$. 

\vspace{0.1cm}

It remains to prove that there are no further strict order relations between the elements of $\Degns_{g,1}^{re}$. Using the classification of Propositions \ref{P:deg-piccoli} and \ref{P:deg-grandi}, we will divide the proof in three cases:

\un{Case I:} $D(A)<D\in \Degns_{g,1}^{re}\Rightarrow D=D(B)$  for some $A\supsetneq B$.

This follows from the fact that if $D(A)<D$ then $D(A)\supsetneq D$ together with the fact that:

$\bullet$ 
$D(A)\not\supseteq W_{\delta}$ for any $1\leq \delta \leq g-1$;

$\bullet$ $D(A)\supsetneq D(B)\Leftrightarrow A\supsetneq B$.

\un{Case II:} $W_{\delta}\not\leq  W_{\wt \delta}$ for any $1\leq \delta\neq  \wt \delta\leq g-1$. 

Indeed, this is trivially true if $\wt \delta$ is not a  multiple of $\delta$,  for otherwise $W_{\delta}\not\supseteq W_{\wt \delta}$. Hence, we can assume that $\wt \delta=k\delta$ for some $k\geq 2$. 
Suppose by contradiction that $W_{\delta}\leq  W_{k \delta}$ and let $E\subset W_{\delta}-W_{k\delta}$ a witness. Up to switching $E$ with $E^c$, we can assume that $(\delta+2;0)\in E$.

We will now prove by induction on $m$ that 
\begin{equation}\label{E:inclu-E}
W_{m\delta}\subset E \text{ for any } 2\leq m \leq k-1.
\end{equation}
Indeed, for the base case $m=2$ of the induction we argue as follows. Given $(e;h)\in W_{2\delta}$ we can write $(e;h)=(\delta+2;0)\circ_{h+1} (\delta+2;0)$ and then consider the triangle $\Delta=[(\delta+2;0),(\delta+2;0),(e;h)^c)]$ (see \eqref{E:tria-comp}): the first two elements belong to $E$ and the last element belong to $W_{2\delta}\subset W_{\delta}-W_{k\delta}$; hence  Condition \eqref{D:degsub2c} of Definition \ref{D:degsub} implies that the last element does not belong to $E$, which is equivalent to  $(e;h)\in E$.  
Now, for the induction step assume that $3\leq m \leq k-1$. Any element $(e;h)\in W_{m\delta}$ is such that $(e;h)> (\delta+2;0)$ since $h\geq 0$ and 
$$2(e+h)=\delta(e;h)+2+e>2\delta+4=2(\delta+2).$$
Then, \eqref{E:ord-comp} implies that we can write $(e;h)=(\delta+2;0)\circ_i (\wt e;\wt h)$ for some $(\wt e;\wt h)$ which belongs to $W_{(m-1)\delta}$ by \eqref{E:sum-comp} and hence to $E$ by induction. Consider the triangle $\Delta=[(\delta+2;0),(\wt e;\wt h),(e;h)^c)]$ (see \eqref{E:tria-comp}): the first two elements belong to $E$ and the last element belong to $W_{m\delta}\subset W_{\delta}-W_{k\delta}$; hence  Condition~\eqref{D:degsub2c} of Definition \ref{D:degsub} implies that the last element does not belong to $E$, which is equivalent to  $(e;h)\in E$. The proof of \eqref{E:inclu-E} is now complete. 

We now finish the proof by finding the desired contradiction. Since $(\delta+2;0)\circ_i ((k-1)\delta+2;0)=(k\delta+2;0)$, we can look at the triangle $\Delta=[(\delta+2;0),((k-1)\delta+2;0),(k\delta+2;0)^c]$ of $W_{\delta}$ (see \eqref{E:tria-comp}): the first two elements of $\Delta$ belong to $E$ by the assumption that $(\delta+2;0)\in E$ and \eqref{E:inclu-E}, while the last element belongs to $W_{k\delta}$, and this violates Condition \eqref{D:degsub2b} of Definition \ref{D:degsub}.

\un{Case III:} $W_{\delta}\not\leq  D(A)$ for any non-empty antichain $A$ of $\Dgbig$.

Indeed, first of all, using \eqref{T:pos-Degn1b}, we can assume that $A=\{x\}$ for some $x\in \Dgbig$, so that $D(A)=\{x,x^c\}$. Moreover, the statement is trivially true if $\delta(x)$ is not a multiple of $\delta$, for otherwise $W_{\delta}\not\supseteq \{x, x^c\}$. Hence, we can assume that $\delta(x)=k\delta$ for some $k\geq 2$ (since $\delta(x)\geq g$ while $\delta\leq g-1$). Suppose by contradiction that $W_{\delta}\leq \{x\}$ and let $E\subset W_{\delta}-\{x,x^c\}$ be a witness. Up to switching $E$ and $E^c$, we can assume that $(\delta+2;0)\in E$.  

Observe that \eqref{E:inclu-E} holds also in this case with the same proof of the Case II. 

We can now find the desired contradiction. Write $x=(e;h)$ and observe that $(e;h)>(\delta+2;0)$ since $h\geq 0$ and 
$$
2(h+e)=\delta(e;h)+e+2=k\delta+e+2\geq 2\delta+4=2(\delta+2).
$$
Then, \eqref{E:ord-comp} implies that we can write $(e;h)=(\delta+2;0)\circ_i (\wt e;\wt h)$ for some $(\wt e;\wt h)$ which belongs to $W_{(k-1)\delta}$ by \eqref{E:sum-comp} and hence to $E$ by \eqref{E:inclu-E}. Consider the triangle $\Delta=[(\delta+2;0),(\wt e;\wt h),(e;h)^c=x^c)]$ (see \eqref{E:tria-comp}): the first two elements belong to $E$ and the last element belongs to $\{x,x^c\}$, and this violates Condition~\eqref{D:degsub2b} of Definition \ref{D:degsub}.
\end{proof}

We immediately deduce:

\begin{corollary}\label{C:height-n1}
    The height of the elements of $\Degns_{g,1}^{re}$ is given by 
    $$
    \begin{sis}
     & h(\emptyset)=0,\\ 
     & h(W_{\delta})=1 & \text{ with } 1\leq \delta \leq g-1, \\
     & h(D(A))=|A| & \text{ for any antichain } A\subseteq \Dgbig.
    \end{sis}
    $$
    In particular, $\emptyset$ is the maximum element of $\Degns_{g,1}^{re}$.
    Moreover, the minimal elements of $\Degns_{g,1}^{re}$ are:
        \begin{itemize}
            \item $W_{\delta}$ for any $1\leq \delta \leq g-1$;
            \item $D(A)$ for any maximal antichain $A$ of $\Dgbig$. 
        \end{itemize}
\end{corollary}

\begin{remark}\label{R:oss-poset}
  \noindent 
  \begin{enumerate}[(i)]
      \item The poset $\Degns_{g,1}^{re}$ (and hence also $\Sigma_{g,1}^\chi$ by Proposition \ref{P:Deg-map}) is \emph{upper-graded}, i.e. all ascending maximal chains starting from an element $D\in \Degns_{g,1}^{re}$ have length equal to $h(D)$. 
       Indeed, this is obvious for the elements of height $0$ or $1$, i.e. $\emptyset$ and $W_{\delta}$ for $1\leq \delta \leq g-1$, while for the elements $D(A)$ it follows from Theorem \ref{T:pos-Degn1}\eqref{T:pos-Degn1b}.

       However, the poset $\Degns_{g,1}^{re}$ is not graded if $g\geq 4$ since it has maximal chains of different lengths. For example:
      
      $\bullet$ $\emptyset >W_{\delta}$ is a maximal chain of length $1$, for any $1\leq \delta\leq g-1$.
      
      $\bullet$ $\emptyset > D((g;1)) >D((g;1),(g-2;2))>\ldots >D((g;1),(g-2;2),\ldots,(g-2\lfloor g/2\rfloor+2;\lfloor g/2\rfloor)=W_{g}$ is a maximal chain of length $\lfloor g/2\rfloor$. 
      
      \item The non-empty classical elements $\{W_{k}\}_{1\leq k\leq 2g-2}$ (see Theorem \ref{T:class-n1}) are all minimal elements of $\Degns_{g,1}^{re}$. 
      This follows from Corollary \ref{C:height-n1} using that, for $g\leq \delta\leq 2g-2$, $W_{\delta}^{pr}$ is a maximal antichain since all the elements $x\in \Dgbig$ with $\delta(x)>\delta$ are bigger than some element of $W_{\delta}^{pr}$ while all the  elements $x\in \Dgbig$ with $\delta(x)<\delta$ are smaller than some element of $W_{\delta}^{pr}$.

      However, there are other minimal elements of $\Degns_{g,1}^{re}$: for example, $A_0$ of Example \ref{Ex:antichain} is a maximal antichain of $\Dgbig$ so that $D(A_0)$ is a minimal element of $\Degns_{g,1}^{re}$. 
     
      \item The subposet $\Degns^{cl}_{g,1}\subset \Degns^{re}_{g,1}$ is such that:
      \begin{itemize}
          \item $h_{\Degns^{cl}_{g,1}}\neq (h_{\Degns^{re}_{g,1}})_{|\Degns^{cl}_{g,1}}$. Indeed, since $W_{\delta}^{pr}=\{(2g+2-\delta-2k;\delta-g+k)\}_{k=1}^{g-\lceil \delta/2\rceil}$, we have that $$h_{\Degns^{re}_{g,1}}(W_{\delta})=g-\left\lceil \frac{\delta}{2}\right\rceil \text{ and } h_{\Degns^{cl}_{g,1}}(W_{\delta})=1.$$
          
          \item $\Degns^{cl}_{g,1}$ is not upward-closed if $g\geq 4$: since $|W_{\delta}^{pr}|\geq 2$ for $g\leq \delta\leq 2g-4$, there are  non-classical elements dominating $W_{\delta}$ in this range. 
      \end{itemize}
  \end{enumerate}  
\end{remark}

We conclude this subsection by showing that the poset $\Sigma^\chi_{g,1}$ is connected through height $1$, in the following sense.

\begin{definition}\label{D:conn-h1}
  A poset $(\P,\leq)$ is said to be connected through height one if any two maximal elements $x,x'\in \P$ are connected through height one, i.e. there exists a sequence of elements $x=x_1,y_1,x_2,y_2,\ldots,y_{k-1}, x_k=x'$ such that 
  \begin{itemize}
        \item $x_i$ is maximal for $i=1,\ldots,k$;
        \item $y_i$ is submaximal (i.e. it has height $1$) for $i=1,\ldots,k-1$;
        \item $y_i<x_i,x_{i+1}$ for each $i=1,\ldots,k-1$.
    \end{itemize}
\end{definition} 
Note that if $(\P,\leq)$ is connected through height one, then it is also connected.

\begin{theorem}\label{T:conn-h1}
  The poset $\Sigma_{g,1}^\chi$ is connected through height one. 
\end{theorem}
\begin{proof}
By Lemma \ref{L:Sigma-s-ns} and Proposition \ref{P:Dsep}, it is enough to show that $\Sigmans_{g,1}^{\chi}$ is connected through height one. This follows from Lemmas \ref{L:connec1} and \ref{L:connec2} below, using that the relation of being connected through height one is an equivalence relation on the set of maximal elements of $\Sigmans_{g,1}^\chi$. 
\end{proof}

\begin{lemma}\label{L:connec1}
    Let $\sigma,\sigma'\in\Sigmans^\chi_{g,1}$ be two classical maximal elements. Then they are connected through height $1$.
\end{lemma}
\begin{proof}
    Since the poset of (non-separating) classical  V-functions is the poset of regions of an hyperplane arrangement in a real affine space by Lemma \ref{L:arr-Uni}, then is connected through height $1$. Thus it is enough to show that any two non-separating classical V-functions $\sigma$ and $\sigma'$, which are adjacent in the poset of regions, are connected through height $1$ in the poset $\Sigmans^\chi_{g,1}$.

    Let $\ov\sigma\in\Sigmans_{g,1}^\chi$ be a classical V-function such that $\sigma,\sigma'>\ov\sigma$ and let $D=\D(\ov\sigma)$. If $D$ is a submaximal element of $\Deg_{g,1}$, we are done. Otherwise, $D=W_{\delta}$, for some $\delta$ such that $g\leq\delta\leq 2g-4$ by  Theorem~\ref{T:class-n1}\eqref{T:class-n1i} and Corollary \ref{C:height-n1}. In particular, $|D|=2k\geq 2$.

    We observe that, since $\sigma>\ov\sigma<\sigma'$ and $\D(\ov\sigma)=D$, we have that $\sigma_{|_{D^c}}=\sigma'_{|_{D^c}}$. 
    Let $\{(e_i;h_i)\}_{i=1,\ldots,k}$  be such that $D=\bigcup_{i=1}^k\{(e_i;h_i),(e_i;h_i)^c\}$, and let $D_i:=\{(e_i;h_i),(e_i;h_i)^c\}$ for each $i=1\ldots,k$. Since both $\sigma$ and $\sigma'$ are maximal non-separating V-functions, for any $i=1,\ldots,k$, the function $\sigma_i$ defined as
    $$
    \sigma_i(e;h,A):=\begin{cases}
        \sigma'(e;h,A),\textup{ if }(e;h,A)\in\bigcup_{j\leq i}D_j;\\
        \sigma(e;h,A), \textup{ otherwise},
    \end{cases}
    $$
    is also a maximal non-separating V-function. In particular, $\sigma=\sigma_0$ (trivially) and $\sigma'=\sigma_k$.

    Similarly, for each $i=1,\ldots,k-1$, the function $\ov\sigma_i$, defined as
    $$
    %\ov\sigma_i(e;h,A):=\begin{cases}
     %   \sigma'_i(e;h,A), \textup{ if } (e;h,A)\in\bigcup_{j<i}D_j \textup{ or } (e;h,A)=(e_i;h_i,\emptyset)^c;\\
     %   \sigma_i(e;h,A), \textup{ if } (e;h,A)\in\bigcup_{j>i}D_j\textup{ or } (e;h,A)=(e_i;h_i,\emptyset),\\
    %\end{cases}
    \ov\sigma_i(e;h,A):=\min (\sigma_i(e;h,A),\sigma_{i+1}(e;h,A)) \quad \text{for all } (e;h,A) \in \Dgunons
    $$
    is a non-separating V-function such that $\D(\ov\sigma_i)=D_i$ (in particular $\ov \sigma_i$ is submaximal) and $\sigma_i>\ov\sigma_i<\sigma_{i+1}$.

    Therefore, we conclude that $\sigma$ and $\sigma'$ are connected through height $1$ via the sequence\\$\sigma=\sigma_0,\ov\sigma_0,\sigma_1,\ldots,\ov\sigma_{k-1},\sigma_k=\sigma'$.
\end{proof}
Indeed, the result of Lemma~\ref{L:connec1} holds true in $\Sigmans_{g,n}^\chi$ for any $n\geq 1$ with an analogous proof.

\begin{lemma}\label{L:connec2}
    Let $\sigma\in\Sigmans^\chi_{g,1}$ be a maximal element. Then there exists a classical maximal element $\tau\in\Sigmans^\chi_{g,1}$, such that $\sigma$ and $\tau$ are connected through height $1$.
\end{lemma}
\begin{proof}
    If $\sigma$ is classical there is nothing to prove. Otherwise, by Theorem~\ref{T:class-n1}, there exists $\delta'$ such that $\sigma$ is not constant on $\DD_{\delta'}:=\{x\in \Dgpr: \delta(x)=\delta'\}$. Let $\delta$ be the minimum such degree $\delta'$. We want to show that $\sigma$ is connected through height $1$ to a non-separating V-function that is constant on $\DD_{\delta}$.

    Since $\sigma$ is a V-function, by \eqref{E:sigmapr-prop} and minimality of $\delta$, for each $(e;h),(e';h')\in \DD_{\delta}$, we have $|\sigma(e;h)-\sigma(e';h')|\leq~1$. Thus, we can write $\DD_{\delta}=D_1\sqcup D_2$ with $D_1,D_2\neq \emptyset$ such that 
    \begin{equation}\label{E:sigmaD1D2}
    \sigma(e;h)=\sigma(e';h')-1\quad\textup{for all } (e;h)\in D_1,(e';h')\in D_2.
    \end{equation}
    We observe, in particular, that $\sigma$ is constant on $D_1$ and $D_2$.
    %and $(e;h)^c\in D_i$ if and only if $(e;h)\in D_i$, for $i=1,2$.

    Let $(e_0;h_0)\in D_1$ with $(e_0;h_0)$ minimal with respect to the order in~\eqref{E:order-D}. 
    
    \un{\text{Claim:}}  The function $\ov\sigma$ defined by
    $$
    \ov\sigma(e;h,A):=\begin{cases}
        \sigma(e;h,A)-1,\textup{ if }(e;h,A)=(e_0;h_0)^c;\\
        \sigma(e;h,A), \textup{ otherwise},
    \end{cases}
    $$
    is a V-function of type $(g,1)$ such that $\D(\ov\sigma)=\{(e_0;h_0),(e_0;h_0)^c\}$ and $\ov\sigma<\sigma$.

    Indeed, since $\sigma$ is a V-function with $\D(\sigma)=\emptyset$ and since there are no triangles contained in $\{(e_0;h_0),(e_0;h_0)^c\}$, by \eqref{E:triaUni}, we only need to check that, for each triangle $\Delta=[(e_0;h_0)^c,(e_1;h_1),(e_2;h_2)]$, we have 
    $$
    \ov\sigma((e_0;h_0)^c)+\ov\sigma(e_1;h_1)+\ov\sigma(e_2;h_2)-\chi=1,
    $$
    that is
    $$
    \sigma((e_0;h_0)^c)+\sigma(e_1;h_1)+\sigma(e_2;h_2)-\chi=2.
    $$
    Suppose, by contradiction, that there exists $\Delta$ such that $ \sigma((e_0;h_0)^c)+\sigma(e_1;h_1)+\sigma(e_2;h_2)-\chi=1$. By the assumptions of minimality on $\delta$ and $(e_0;h_0)$, $\sigma$ is constant on $\DD_{\delta_i}$, where $\delta_i:=\delta(e_i;h_i)$, for $i=1,2.$ Clearly, we can choose $(e'_1;h'_1),(e'_2;h'_2)$ with $\delta(e'_i;h'_i)=\delta_i$, such that $\Delta=[(e;h)^c,(e'_1;h'_1),(e'_2;h'_2)]$ is a triangle with $(e;h)\in D_2$. So, by \eqref{E:sigmaD1D2}, we have 
    $$
\sigma((e;h)^c)+\sigma(e'_1;h'_1)+\sigma(e'_2;h'_2)-\chi=0,
    $$
    a contradiction, which finishes the proof of the Claim.
    
    Consider now the function $\sigma'$ defined as
    $$
    \sigma'(e;h,A):=\begin{cases}
        \ov\sigma(e;h,A)+1,\textup{ if }(e;h,A)=(e_0;h_0);\\
        \ov\sigma(e;h,A), \textup{ otherwise},
    \end{cases}
    $$
    which is a V-function such that $\sigma'>\ov\sigma<\sigma$, hence it is connected to $\sigma$ through height $1$. Moreover $\sigma'$ satisfies:
    $$
    \sigma'(e;h)=\sigma'(e';,h')-1\quad\textup{for all } (e;h)\in D_1-\{(e_0;h_0)\},\textup{ and }(e';,h')\in D_2\cup\{(e_0;h_0)\}.
    $$

    By inductively repeating the same argument for each element of $D_1$, we obtain a V-function that is constant on $\DD_{\delta}$ and is connected to $\sigma$ through height $1$.
    Finally, repeating again the same argument for each $\DD_{\delta'}$ such that $\sigma$ is not constant on $\DD_{\delta'}$, we obtain a V-function $\tau$ that is connected through height $1$ to $\sigma$ and that is uniform, hence classical by Theorem~\ref{T:class-n1}.
\end{proof}

\subsection{Maximal and submaximal elements }\label{sub:max-wall} 

The purpose of this subsection is to classify the \emph{maximal} elements, i.e., those of height zero, and the \emph{submaximal} elements, i.e., those of height one, of $\Deg^{re}_{g,n}$ and $\Sigma_{g,n}^\chi$. We will focus primarily on the case $n\geq 1$, since the case $n=0$ follows easily from Theorem \ref{T:class-n0} and Proposition \ref{P:Dsep}.

The main results here describe the properties of maximal and submaximal elements of $\Sigma_{g,n}^{\chi}$ with $n \geq 1$:

\begin{enumerate} 
\item Corollary~\ref{C:max-vfun}, which characterizes the maximal element as the general V-functions, 
\item Corollary~\ref{C:walls-Vfun}, which gives a complete description of the submaximal elements, and 
\item Corollary~\ref{C:2max-wall}, which shows that each submaximal element is dominated by exactly two maximal elements. 
\end{enumerate} 

Moreover, in Remark~\ref{R:classic-walls} we compare the submaximal elements of $\Sigma_{g,n}^{\chi}$ with the walls in the stability space (a hyperplane arrangement) of classical compactified universal Jacobians, described in \cite{Kass_2019}.

\begin{remark}\label{R:Sigma-Deg}
 It follows from Lemma \ref{C:height}\eqref{C:height1} that $\sigma\in \Sigma_{g,n}^\chi$ is maximal (resp. submaximal) if and only if $\D(\sigma)\in \Deg^{re}_{g,n}$ is  maximal (resp. submaximal).
\end{remark}

Maximal elements of $\Deg_{g,n}$ and of $\Sigma_{g,n}^\chi$ are easy to describe, if $n\geq 1$.

\begin{proposition}\label{prop:empty-is-maximum-deg}
If $n\geq 1$, then the empty set $\emptyset$ is the maximum element of $\Deg_{g,n}$.
\end{proposition}
Note  that the above result is false for $n=0$: for example, every subset of $\mathbb{D}_{g,0}^{\text{half}}:=\{(e;h,\emptyset)\: : 2h-2+e=g-1\}\subset \mathbb{D}_{g,0}$ is a maximal element of $\Deg_{g,0}$ since the complementary involution is the identity on $\mathbb{D}_{g,0}^{\text{half}}$ and there are no triangles contained in $\mathbb{D}_{g,0}^{\text{half}}$.  
\begin{proof} 
Let $D\in\Deg_{g,n}$. We have to prove that $\emptyset\ge D$ in $\Deg_{g,n}$.  Since $\emptyset\subseteq D$, it remains to exhibit a witness for the relation $\emptyset \geq D$, i.e. a subset $E\subseteq D$ satisfying the conditions of Definition \ref{D:degsub}\eqref{D:degsub2}. 

Fix an index $j\in[n]$ and set
\[
E \;:=\; \{\, (e;h,A)\in D \;:\; j\in A \,\}.
\]
We check the required conditions.

\smallskip\noindent
\eqref{D:degsub2a}:
If $(e;h,A)\in D$, then its complement is $(e;g-h-e+1,A^c)\in D$ (since $D$ is a
degeneracy subset). Exactly one of $A$ and $A^c$ contains $j$, hence
$(e;h,A)\in E$ if and only if $(e;g-h-e+1,A^c)\notin E$, i.e.\ the complement of
$(e;h,A)$ lies in $E^c$.

\smallskip\noindent
\eqref{D:degsub2b}:
This condition is vacuous because it concerns triangles $\Delta\subseteq D$
satisfying $|\Delta\cap D_1|=1$, and $D_1=\emptyset$.

\smallskip\noindent
\eqref{D:degsub2c}:
Let $\Delta=[(e_i;h_i,A_i)]_{i=1}^3$ be a triangle contained in $D$.
By definition of triangle we have a disjoint union
$A_1\sqcup A_2\sqcup A_3=[n]$, so exactly one of the $A_i$ contains $j$.
Therefore $|\Delta\cap E|=1$, and \textup{(2c)} is satisfied.
\end{proof}

\begin{corollary}\label{C:max-vfun}
If $n\geq 1$ then the maximal elements of $\Sigma_{g,n}$ are exactly the general V-functions of type $(g,n)$.      
\end{corollary}
Note that the above result is false for $n=0$, since in this case Theorem \ref{T:class-n0} implies that there are general V-functions of characteristic $\chi$ if and only if $\gcd(\chi, 2g-2)=1$. 
\begin{proof}
 This follows from Proposition \ref{prop:empty-is-maximum-deg} and Remark \ref{R:Sigma-Deg}, together with the fact that $\sigma\in \Sigma_{g,n}$ is general if and only if $\D(\sigma)=\emptyset$. 
\end{proof}

In the rest of this subsection, we focus on the classification of the submaximal elements in $\Sigma_{g,n}^\chi$, via the study of submaximal elements in $\Deg_{g,n}$.  

We first introduce the following 

\begin{definition}\label{D:mix-un}
\noindent 
\begin{enumerate}[(i)]
    \item An element $(e;h,A)\in \Dgn$ is called \emph{unmixed} if $A=\emptyset$ or $[n]$, and  \emph{mixed} if $\emptyset \subsetneq A \subsetneq [n]$. The set of unmixed (resp. mixed) elements of $\Dgn$ is denoted by $\Dgn^u$ (resp. $\Dgn^m$). We have a partition 
    $$
    \Dgn=\Dgn^u\sqcup \Dgn^m,
    $$
   which is stable under the complement operation.
    \item  An element $D\in \Deg_{g,n}$ is called \emph{unmixed} if $D\subseteq \Dgn^u$ and \emph{mixed} otherwise, i.e. it contains some mixed element. 
    We denote by $\Deg^u_{g,n}$ (resp. $\Deg_{g,n}^{re,u}$) the subsposet of $\Deg_{g,n}$ (resp. $\Deg_{g,n}^{re}$) consisting of unmixed (resp. and realizable) universal degeneracy subsets.  
\end{enumerate}   
\end{definition}
Note that $\Deg_{g,n}^{re,u}$ and $\Deg^u_{g,n}$ are upward-closed subposets of $\Deg_{g,n}$; hence their height functions are the restriction of the height function on $\Deg_{g,n}$.
One can also define a subposet of mixed elements, but this is not so useful for us since it is not upward-closed. 

The poset of unmixed degeneracy subsets can be completely described. Consider the injection $\alpha_{n} \colon \mathbb{D}_{g,1}=\DD_{g,1}^u \to \Dgn$ defined by
\[
\alpha_n((e;h,\emptyset)) = (e;h, \emptyset), \quad \alpha_n((e;h,\{1\}))= (e;h,[n]).
\]
The image of $\alpha_n$ is equal to $\Dgn^u$.

\begin{proposition} \label{P:reduceto1}
    Fix $n\geq 1$.
    \begin{enumerate}[(i)]
        \item  \label{P:reduceto1a} There is an isomorphism of posets given by taking the direct image:
        $$
        \begin{aligned}
            \alpha_{n*}:\Deg_{g,1}=\Deg_{g,1}^u & \xrightarrow{\cong} \Deg^u_{g,n}\\
            D & \mapsto \alpha_n(D). 
        \end{aligned}
        $$
        Moreover, taking the direct image via $\alpha_n$ induces a bijection between the witnesses for $D_1\geq D_2$ and the witnesses for $\alpha_{n*}(D_1)\geq \alpha_{n*}(D_2)$. 
          \item  \label{P:reduceto1b} An element $D\in \Deg_{g,1}$ is realizable if and only if $\alpha_{n*}(D)$ is realizable. In other words, $\alpha_{n*}$ restricts to isomorphism of subposets 
          $$
            \alpha_{n*}:\Deg_{g,1}^{re}  \xrightarrow{\cong} \Deg^{re,u}_{g,n}. 
          $$
    \end{enumerate}
\end{proposition}
\begin{proof} 
We can assume that $n\geq 2$, for otherwise the statement is a tautology. 

Let us first prove Part \eqref{P:reduceto1a}.
Since the elements of $\Deg_{g,n}^u$ are exactly the elements $D'\in \Deg_{g,n}$ such that $D'\subseteq \Dgn^u=\Im(\alpha_n)$, the fact that $\alpha_{n*}$ is a well-defined bijection is equivalent to showing that a subset $D\subset \mathbb{D}_{g,1}$ belongs to $\Deg_{g,1}$, i.e. it satisfies the two conditions of Definition~\ref{D:degsub}\eqref{D:degsub1}, if and only if $\alpha_n(D)$ belongs to $\Deg_{g,n}$.
The complement-closure for $D$ is equivalent to the complement-closure for $\alpha_{n*}(D)$ since $\alpha_n(x^c)=\alpha_n(x)^c$ for any $x\in \mathbb{D}_{g,1}$. The triangle-closure for $\alpha_{n*}(D)$ implies the triangle-closure for $D$ since $\alpha_n$ induces a bijection between the triangles of $\mathbb{D}_{g,1}$ and the triangles of $\Dgn$ entirely contained in $\Im(\alpha_n)$. On the other hand, the triangle-closure for $D$ implies the triangle-closure for $\alpha_n(D)$ using  that if a triangle of $\Dgn$ is such that two of its elements belong to $\Im(\alpha_n)$ then also the third one belongs to $\Im(\alpha_n)$ and hence it is the image of a triangle of $\mathbb{D}_{g,1}$ via $\alpha_n$.

\smallskip

The partial order on each $\Deg_{g,k}$ (Definition~\ref{D:degsub}\eqref{D:degsub2} is defined using inclusions and the existence of a ``witness'' subset $E$, via conditions that are expressed only in terms of complements and triangles. If $D_1\ge D_2$ in $\Deg_{g,1}$ with witness $E\subseteq D_2- D_1$, then $\alpha_n(E)\subseteq \alpha_n(D_2)- \alpha_n(D_1)$ satisfies the corresponding witness conditions for $\alpha_n(D_1)\ge \alpha_n(D_2)$ because $\alpha_n$ preserves complements and identifies triangles as above. Conversely, any witness for $\alpha_n(D_1)\ge \alpha_n(D_2)$ pulls back via $\alpha_n$ to a witness for $D_1\ge D_2$.
Thus $\alpha_{n*}$ is an order-isomorphism onto its image.

Let us now prove Part \eqref{P:reduceto1b} considering  the two implications separately. 

($\Leftarrow$) Let $D\in \Deg_{g,1}$ be such that $\alpha_{n*}(D)$ is realizable, i.e. $\alpha_n(D)=\D(\sigma)$ for some $\sigma\in \Sigma_{g,n}^\chi$. The composition 
$$
\alpha_n^*(\sigma):\mathbb{D}_{g,1}\xrightarrow{\alpha_n} \Dgn \xrightarrow{\sigma} \ZZ
$$
defines an element of $\Sigma_{g,1}^\chi$ (using that $\alpha_n$ commutes with taking complements 
and it sends triangles of $\mathbb{D}_{g,1}$ into triangles of $\Dgn$) such that 
$$\D(\alpha_n^*(\sigma))=\alpha_n^{-1}(\D(\sigma))=\alpha_n^{-1}(\alpha_n(D))=D,$$
which shows that $D$ is realizable. 

($\Rightarrow$) Let $D\in \Deg_{g,1}$ be realizable, i.e. $D=\D(\sigma)$ for some $\sigma\in \Sigma_{g,1}^\chi$.  
In order to show that $\alpha_{n*}(D)$ is realizable, we partition the set $\Dgn$ into four disjoint subsets
$$
\Dgn=\Dgn^u\sqcup \Dgn^i \sqcup \Dgn^< \sqcup \Dgn^>,
$$
where $\Dgn^u$ is the set of \emph{unmixed} elements (see Definition \ref{D:mix-un}),
$\Dgn^i$ is the set of \emph{intermediate} elements:
$$
\Dgn^i:=\{(e;h,A)\in \Dgn : \emptyset\subsetneq A \subsetneq [n] \text{ and  } 0<2h-2+e+|A-\{2,\ldots, n\}|<2g-1\},
$$
$\Dgn^<$ is the set of \emph{small} elements
$$\begin{aligned}
&\Dgn^<:=\{(e;h,A)\in \Dgn : 2h-2+e+|A-\{2,\ldots, n\}|\leq 0\}= \\
& =\{(e;h,A)\in \Dgn :\text{ either } (e,h)=(1,0) \text{ or } (e,h)=(2,0) \text{ and } 1\not \in A\}, 
\end{aligned}$$
$\Dgn^>$ is the set of \emph{big} elements
$$\begin{aligned}
&\Dgn^>:=\{(e;h,A)\in \Dgn : 2h-2+e+|A-\{2,\ldots, n\}|\geq 2g-1\}= \\
& =\{(e;h,A)\in \Dgn :\text{ either } (e,h)=(1,g) \text{ or } (e,h)=(2,g-1) \text{ and } 1 \in A\}, 
\end{aligned}$$
Note that there exists a forgetful map 
$$
\begin{aligned}
\varpi: \Dgn^u\sqcup \Dgn^i & \longrightarrow \mathbb{D}_{g,1}\\
(e;h,A)& \mapsto (e;h,A-\{2,\ldots, n\})
\end{aligned}
$$
such that $\varpi\circ \alpha_n=\id$. 

Consider the following function 
\begin{equation}
    \begin{aligned}
       \wt \sigma: \Dgn & \longrightarrow \ZZ \\
       (e;h,A) & \mapsto 
       \begin{sis}
       & 0 & \text{ if } (e;h,A)\in \Dgn^< \text{ and } n\not\in A,\\
         & 1 & \text{ if } (e;h,A)\in \Dgn^< \text{ and } n\in A, \\
         & \chi & \text{ if } (e;h,A)\in \Dgn^> \text{ and } n\not\in A, \\
         & \chi+1 & \text{ if } (e;h,A)\in \Dgn^> \text{ and } n\in A, \\
         & \sigma(\varpi(e;h,A)) & \text{ if } (e;h,A)\in \Dgn^u,\\
         & \sigma(\varpi(e;h,A)) & \text{ if } (e;h,A)\in \Dgn^i  \text{ and either } \varpi(e;h,A)\not \in \D(\sigma) \text{ or } n\not \in A,\\
         & \sigma(\varpi(e;h,A))+1 & \text{ if } (e;h,A)\in \Dgn^i  \text{ and  } \varpi(e;h,A) \in \D(\sigma) \text{ and } n\in A.\\
       \end{sis}
    \end{aligned}
\end{equation}
We claim that 
\begin{equation*}
\wt \sigma\in \Sigma_{g,n}^\chi \text{ with } \D(\wt \sigma)=\alpha_{n*}(D),
\end{equation*}
which will conclude the proof. 

We need to  check the two conditions of Definition~\ref{D:Sigmagn} for $\wt \sigma$.

$\bullet$ Condition \eqref{E:condUni1}: since the complementary involution on $\Dgn$ maps $\Dgn^>$ isomorphically onto $\Dgn^<$ and it preserves $\Dgn^u$ and $\Dgn^i$, and since the forgetful map commutes with the complementary involutions on the source and target, we deduce that $\wt \sigma$ satisfies Equation \eqref{E:sumUni} with 
$$\D(\wt \sigma)=\varpi^{-1}(\D(\sigma))\cap \Dgn^u=\alpha_{n*}(\D(\sigma))=\alpha_{n*}(D).$$

$\bullet$ Condition \eqref{E:condUni2}: the first Condition (2a) holds true since we already know that $\D(\wt \sigma)=\alpha_{n*}(D)\in \Deg_{g,n}$. In order to check Condition (2b), consider a triangle $\Delta=[x_1,x_2,x_3]$ of $\Dgn$, with $x_i:=(e_i;h_i,A_i)$ for $i=1,2,3$, and consider the following different possibilities:

(a) Assume that $\Delta$ contains some elements of $\Dgn^>\sqcup \Dgn^<$. Then there are two possibilities:

(a1) If $\Delta$ contains at least two elements of $\Dgn^>\sqcup \Dgn^<$, then we must have (up to reordering) that $x_1=(2;0,A_1),x_2=(2;0,A_2)\in \Dgn^<$ and $x_3=(2;g-1,A_3)\in \Dgn^>$, and $n$ must belong to exactly one among $A_1, A_2, A_3$, so that 
$$
\wt \sigma(x_1)+\wt \sigma(x_2)+\wt \sigma(x_3)=0+0+\chi+1=\chi+1,
$$
and \eqref{E:triaUni} is satisfied since all the elements of $\Delta$ are $\wt \sigma$-nondegenerate.

(a2) If $\Delta$ contains exactly one element of $\Dgn^>\sqcup \Dgn^<$, then we must have  (up to reordering) that $x_1=(2;0,A_1)\in \Dgn^<$ and $x_2,x_3\in \Dgn^u\sqcup \Dgn^i$ with $\varpi(x_2)^c=\varpi(x_3)$. Moreover, $A_1$ must be non-empty and it cannot contain $1$, so that if we assume that $1\in A_2$ (up to reordering) then we have $x_2\in \Dgn^i$, while $x_3\in \Dgn^u$ if and only if $A_3=\emptyset$.  We now compute 
$$
\wt \sigma(x_1)+\wt \sigma(x_2)+\wt \sigma(x_3)=
\begin{cases}
  0/1+\chi+1 & \text{ if }\varpi(x_2), \varpi(x_3)\not \in \D(\sigma),\\
   \chi+1 & \text{ if }\varpi(x_2), \varpi(x_3)\in \D(\sigma) \text{ and } x_3\in \Dgn^i,\\
\chi+1 & \text{ if }\varpi(x_2), \varpi(x_3)\in \D(\sigma) \text{ and } x_3\in \Dgn^u,\\
\end{cases}
$$
and \eqref{E:triaUni} is satisfied since in the first two cases all the elements of $\Delta$ are $\wt \sigma$-nondegenerate, while in the third case only $x_3\in \D(\wt \sigma)$. 

(b) Assume that $\Delta$ does not contain elements of $(\Dgn^>\sqcup \Dgn^<)$ but it does contain some element of $\Dgn^u$. We are going to use that $\sigma$ verifies \eqref{E:triaUni} for the triangle $\varpi(\Delta)=[\varpi(x_1),\varpi(x_2),\varpi(x_3)]$ of $\mathbb{D}_{g,1}$. There are two possibilities:

(b1) If $\Delta$ contains at least two elements of $\Dgn^u$, then $\Delta$ is entirely contained in $\Dgn^u$ and we have that 
$$
\sum_{i=1}^3 \wt \sigma(x_i)=\sum_{i=1}^3\sigma(\varpi(x_i))
$$
and hence \eqref{E:triaUni} for $\wt \sigma$ follows from \eqref{E:triaUni} for $\sigma$. 

(b2) If $\Delta$ contains exactly one element of $\Dgn^u$, say $x_1$, then we must have that $x_2,x_3\in \Dgn^i$ (which implies that $x_2,x_3\not \in \D(\wt \sigma)$) and that $A_1=\emptyset$ (which implies that $n\in A_2\sqcup A_3$). Now there are four cases:

$\star$ If $\varpi(x_1)\in \D(\sigma)$ (which implies that  $x_1\in \D(\wt \sigma)$) and $\varpi(x_2),\varpi(x_3)\not \in \D(\sigma)$ then 
$$
\sum_{i=1}^3 \wt \sigma(x_i)-\chi=\sum_{i=1}^3\sigma(\varpi(x_i))-\chi=1,
$$
and hence \eqref{E:triaUni} holds true.

$\star$ If $\varpi(x_1)\in \D(\sigma)$ (which implies that  $x_1\in \D(\wt \sigma)$) and one among $\varpi(x_2)$ and $\varpi(x_3)$ belongs to $\D(\sigma)$ (which then implies that both of them belong to $\D(\sigma)$) then 
$$
\sum_{i=1}^3 \wt \sigma(x_i)-\chi=\sum_{i=1}^3\sigma(\varpi(x_i))-\chi+1=1,
$$
and hence \eqref{E:triaUni} holds true.

$\star$ If $\varpi(x_1)\not \in \D(\sigma)$ (which implies that  $x_1\not \in \D(\wt \sigma)$) and $\varpi(x_2),\varpi(x_3)\not \in \D(\sigma)$ then 
$$
\sum_{i=1}^3 \wt \sigma(x_i)-\chi=\sum_{i=1}^3\sigma(\varpi(x_i))-\chi\in \{1,2\},
$$
and hence \eqref{E:triaUni} holds true.

$\star$ If $\varpi(x_1)\not \in \D(\sigma)$ (which implies that  $x_1\not \in \D(\wt \sigma)$) and one among $\varpi(x_2)$ and $\varpi(x_3)$ belongs to $\D(\sigma)$ (which then implies that exactly one of them belong to $\D(\sigma)$) then 
$$
\sum_{i=1}^3 \wt \sigma(x_i)-\chi=\sum_{i=1}^3\sigma(\varpi(x_i))-\chi+0/1=1+0/1\in \{1,2\},
$$
and hence \eqref{E:triaUni} holds true.

(c) Assume that $\Delta$ is entirely contained in $\Dgn^i$, so that none of the elements of $\Delta$ belongs to $\D(\wt \sigma)$. We are going to use that $\sigma$ verifies \eqref{E:triaUni} for the triangle $\varpi(\Delta)=[\varpi(x_1),\varpi(x_2),\varpi(x_3)]$ of $\mathbb{D}_{g,1}$. There are three cases:

$\star$ If $\varpi(x_1),\varpi(x_2),\varpi(x_3)\in \D(\sigma)$ 
then 
$$
\sum_{i=1}^3 \wt \sigma(x_i)-\chi=\sum_{i=1}^3\sigma(\varpi(x_i))-\chi+1=1,
$$
and hence \eqref{E:triaUni} holds true.

$\star$ If only one among $\{\varpi(x_1),\varpi(x_2),\varpi(x_3)\}$ belongs to $\D(\sigma)$ 
then 
$$
\sum_{i=1}^3 \wt \sigma(x_i)-\chi=\sum_{i=1}^3\sigma(\varpi(x_i))-\chi+0/1=1+0/1\in \{1,2\},
$$
and hence \eqref{E:triaUni} holds true.

$\star$ If none of the elements of $\{\varpi(x_1),\varpi(x_2),\varpi(x_3)\}$ belongs to $\D(\sigma)$ 
then 
$$
\sum_{i=1}^3 \wt \sigma(x_i)-\chi=\sum_{i=1}^3\sigma(\varpi(x_i))-\chi\in \{1,2\},
$$
and hence \eqref{E:triaUni} holds true.
\end{proof}

\begin{corollary}\label{C:walls-to1}
  For any $n\geq 1$, there is a bijection
  $$
  \alpha_{n*}:\{\text{(realizable) submaximal of } \Deg_{g,1} \} \xrightarrow{\cong} \{\text{(realizable) unmixed submaximal  of } \Deg_{g,n}\}.
  $$
\end{corollary}
\begin{proof}
 This follows from the previous Proposition together with the fact that the $\Deg^u_{g,n}$ is upward-closed in $\Deg_{g,n}$.
\end{proof}

The rest of this subsection is devoted to the classification of (realizable) mixed submaximal elements of $\Deg_{g,n}$.

\begin{proposition}\label{P:good}
 The mixed submaximal elements of $\Deg_{g,n}$ are given by 
 $$
 \{x,x^c\} \text{ for any } x=(e;h,A)\in \Dgn \text{ with } \emptyset \subsetneq A \subsetneq [n].
 $$
\end{proposition}
\begin{proof}
We divide the proof in two steps.

\un{Step I:} $\{x,x^c\}$ is a (mixed) submaximal element for any $x=(e;h,A)\in \Dgn$  with $\emptyset \subsetneq A \subsetneq [n]$.

\vspace{0.1cm}

Indeed, let us first show that $\{x,x^c\}\in \Deg_{g,n}$, i.e. that $\{x,x^c\}$ satisfies the two conditions of Definition~\ref{D:degsub}\eqref{D:degsub1}.
Complement-closure (i.e. Condition~\ref{D:degsub}\eqref{D:degsub1a}) is clear. 
 For the triangle-closure  (i.e. Condition~\ref{D:degsub}\eqref{D:degsub1b}), we show that no triangle in $\Dgn$ contains two elements of $\{x,x^c\}$. 
The only way for two elements of $\{x,x^c\}$ to appear in a triangle $\Delta$ is:
\begin{enumerate}
\item $\Delta=[x,x,y]$ for some $y\in \Dgn$, but this is absurd since $A\neq \emptyset$;
\item $\Delta=[x^c,x^c,y]$ for some $y\in \Dgn$, but this is absurd since $A^c\neq \emptyset$; 
%or $x^c$ appears twice (similarly impossible),
\item $\Delta=[x,x^c,y]$ for some $y\in \Dgn$, but this absurd since \eqref{E:for-lcdeg} would imply that 
$$
\delta(y)=2g-2+n-\delta(x)-\delta(x^c)=0,
$$
which is impossible since $\delta(y)\in (0,2g-2+n)$.
\end{enumerate}
Now it is clear that $\{x,x^c\}$ is a submaximal element of $\Deg_{g,n}$ since it is dominated only by $\emptyset$.

 \un{Step II:} If $D$ is a (mixed) submaximal element of $\Deg_{g,n}$ then $D=\{x,x^c\}$ for some $x=(e;h,A)\in \Dgn$  with $\emptyset \subsetneq A \subsetneq [n]$.
 
 \vspace{0.1cm}

 Choose $x_0=(e_0;h_0,A_0)\in D$ with $\emptyset\subsetneq A_0\subsetneq [n]$ by minimizing lexicographically first $|A|$, then $h$, then $e$  among all  $(e;h,A)$ occurring in $D$.
By Definition~\ref{D:degsub}\eqref{D:degsub1a}, $D$ is closed under complements, hence $x_0^c\in D$.
Set $D_0:=\{x_0,x_0^c\}$. By Step I, we have that $D_0 \in \Deg_{g,n}$.

We will show  that $D \leq D_0$, which then will force the equality $D=D_0$ (and hence the conclusion) because $D$ is a submaximal element and $D_0$ is not a maximal element by Proposition \ref{prop:empty-is-maximum-deg}.  

In order to show that $D_0\leq D$, using that $D_0 \subseteq D$, it suffices to find a witness  $E \subseteq D - D_0$.
Since $A_0\subsetneq[n]$, there exists $j\in A_0^c$. We define
\[
E\ :=\ \{(e;h,A)\in D- D_0 \ :\ j\in A\}.
\]
We will check that $E$ satisfies the conditions of Definition~\ref{D:degsub}\eqref{D:degsub2}.

\medskip
\noindent\eqref{D:degsub2a}.
Let $(e;h,A)\in D- D_0$. Then its complement is $(e;h,A)^c=(e;g-h-e+1,A^c)\in D- D_0$.
Exactly one of $A$ and $A^c$ contains $j$, so exactly one element in the complementary pair lies in $E$, which settles Condition~\eqref{D:degsub2a}.

\medskip
\noindent\eqref{D:degsub2b}.
Let $\Delta \subseteq D$ be a triangle such that
$|\Delta\cap D_0|=1$.
We will show that $|\Delta\cap E|=1$. There are two cases.

\smallskip
\emph{Case 1: $\Delta\cap D_0=\{x_0^c\}$.}
 We will show that this case cannot occur. Write
\[
x_0=(e_0;h_0,A_0)\in D_0,\qquad x_0^c=(e_0;g-h_0-e_0+1,A_0^c)\in D_0,
\]
and let the other two elements of $\Delta$ be
\[
x_2=(e_2;h_2,A_2) \in D - D_0,\qquad x_3=(e_3;h_3,A_3)\in D- D_0.
\]
Since $\Delta$ is a triangle, the marking sets satisfy
\[
A_0^c\sqcup A_2\sqcup A_3=[n]\qquad\text{hence}\qquad A_2\sqcup A_3=A_0.
\]
If both $A_2$ and $A_3$ are nonempty, then one of them is a proper nontrivial subset of $A_0$,
contradicting the minimality of $|A_0|$.  Thus it suffices to rule out the subcase
\[
(A_2,A_3)=(A_0,\emptyset)\quad\text{or}\quad(\emptyset,A_0).
\]
Without loss of generality, assume $(A_2,A_3)=(A_0,\emptyset)$ for a contradiction.

The triangle genus equation gives
\[
g=(g-h_0-e_0+1)+h_2+h_3+\frac{e_0+e_2+e_3}{2}-2,
\]
equivalently
\begin{equation}\label{eq:case1}
h_2+h_3+\frac{e_2+e_3-e_0}{2}=h_0+1.
\end{equation}
Set $\displaystyle n_{23}:=\frac{e_2+e_3-e_0}{2}$.  The edge inequalities for a triangle imply $n_{23}\ge 1$,
so \eqref{eq:case1} yields $h_2=h_0+1-h_3-n_{23}\le h_0$.  By the defining choice of $x_0$
(minimality of the marking set $|A_0|$ and then of $h_0$),
we must have $h_2\ge h_0$, hence $h_2=h_0$ and then $h_3+n_{23}=1$.  Therefore
$h_3=0$ and $n_{23}=1$, i.e.\ $e_2+e_3=e_0+2$.  Since $x_3\in \Dgn$ and $x_3$ belongs to a triangle, one has $e_3\ge 2$,
so $e_2\le e_0$; by the minimality assumption and the fact that $h_2=h_0$, we deduce $e_2\ge e_0$, hence $e_2=e_0$ and $e_3=2$.
Thus $x_3=(2,0,\emptyset)$, contradicting the stability inequality $2h-2+e+|A|>0$ for elements of $\Dgn$, so Case~1 cannot occur.

\smallskip
\emph{Case 2: $\Delta\cap D_0=\{x_0\}$.}
Then one element of $\Delta$ has marking set $A_0$, so the other two marking sets partition $A_0^c=A_i\sqcup A_k$. Since $j\in A_0^c$, exactly one of $A_i$ and $A_k$ contains $j$.
By definition of $E$, exactly one of the two elements in $\Delta-\{x_0\}$ lies in $E$.
Also $x_0\notin E$ because $E\subseteq D- D_0$,
therefore $|\Delta\cap E|=1$.

This completes the proof of \eqref{D:degsub2b}.

\medskip
\noindent\eqref{D:degsub2c}.
Let $\Delta\subset D- D_0$ be a triangle.
Then $A_1\sqcup A_2\sqcup A_3=[n]$, so exactly one of the marking sets contains $j$.
Hence exactly one element of $\Delta$ lies in $E$, which means $|\Delta\cap E|=1$. Condition \eqref{D:degsub2c} is satisfied.
\end{proof}

We now prove that all mixed submaximal elements are realizable. 

\begin{proposition}\label{prop:good-wall-realizable}
Let  $W\in\Deg_{g,n}$ be a mixed submaximal element. 
Then $W$ is realizable, i.e.\ there exists $\sigma\in\Sigma_{g,n}$ such that $\D(\sigma)=W$.
\end{proposition}

\begin{proof}
By Proposition~\ref{P:good},  $W=\{x,x^c\}$ for some $x=(e_0;h_0,A_0)\in D_{g,n}$ with $\emptyset\subsetneq A_0\subsetneq [n]$.

Set   $\chi:=0$, choose $j\not \in A_0$ and define $\sigma:\Dgn\to\mathbb Z$ as follows:

\begin{itemize}
\item $\sigma(x)=1$ and $\sigma(x^c)=-1$.
\item If $y=(e;h,A)$ satisfies $j\notin A$ and $y\neq x$, set
\[
\sigma(y):=\left\lceil\frac{\delta(y)}{\delta(x)}\right\rceil \quad \text{ and } \quad \sigma(y^c):=1-\sigma(y).
\]

\end{itemize}

We now check that  $\sigma\in \Sigma_{g,n}^0$ with $\D(\sigma)=\{x,x^c\}$, by verifying the conditions of Definition~\ref{D:Sigmagn}.

Condition \eqref{E:condUni1} is clearly satisfied with $\chi=0$ and the $\sigma$-degenerate elements are $\{x,x^c\}$. 

Let us now check Condition \eqref{E:condUni2}. First of all, Condition (2a) holds vacuously because no triangle can contain two elements of $\{x,x^c\}$ as we already observed in Step I of the proof of Proposition \ref{P:good}.  Thus, it remains to verify Condition (2b).  Let $\Delta=[u,v,w^c]$ be any triangle in $\mathbb{D}_{g,n}$, arranged so that $j$ belong to the marked set of $w^c$ (and hence it does not belong to the marked set of $u$, $v$ and $w$).
By \eqref{E:for-lcdeg}, we have $\delta(w)=\delta(u)+\delta(v)$.
We check~\eqref{E:triaUni} (with $\chi=0$) for $\Delta$ by distinguishing the following three exhausting cases.

\smallskip\noindent
\emph{Case 1:} $\Delta\cap\{x,x^c\}=\emptyset$.
Using $\delta(w)=\delta(u)+\delta(v)$ and the definition of $\sigma$, we get 
\[
\sigma(u)+\sigma(v)-\sigma(w)
=\left\lceil \frac{\delta(u)}{\delta(x)}\right\rceil
+\left\lceil \frac{\delta(v)}{\delta(x)}\right\rceil
-\left\lceil \frac{\delta(u)+\delta(v)}{\delta(x)}\right\rceil\in\{0,1\}.
\]
From this and the fact that $\sigma(w^c)=1-\sigma(w)$, we compute 
\[
\sigma(u)+\sigma(v)+\sigma(w^c)
=
1+\sigma(u)+\sigma(v)-\sigma(w)\in \{1,2\},
\]
which verifies Equation \eqref{E:triaUni} since the three elements of $\Delta$ are all $\sigma$-nondegenerate. 

\smallskip\noindent
\emph{Case 2:} $w^c=x^c$ (so $w=x$).
Using that $\delta(x)=\delta(w)=\delta(u)+\delta(v)$ and that $\delta(u),\delta(v)>0$, we get that $\sigma(u)=\sigma(v)=1$. Hence, using also $\sigma(w^c)=\sigma(x^c)=-1$, we compute 
$$
\sigma(u)+\sigma(v)+\sigma(w^c)=1+1-1=1,
$$
which verifies Equation \eqref{E:triaUni} since $x^c$ is the unique element of $\Delta$ that is $\sigma$-degenerate. 

\smallskip\noindent
\emph{Case 3:} $u=x$ or $v=x$ (say $u=x$) and $w\neq x$.
Since $\delta(w)=\delta(u)+\delta(v)=\delta(x)+\delta(v)$, we have
\[
\sigma(w)=\left\lceil\frac{\delta(x)+\delta(v)}{\delta(x)}\right\rceil
=1+\left\lceil\frac{\delta(v)}{\delta(x)}\right\rceil
=1+\sigma(v).
\]
Using this and the fact that $\sigma(w^c)=1-\sigma(w)$ (because $w^c\neq x^c$), we compute 
$$
\sigma(x)+\sigma(v)+\sigma(w^c)=1+\sigma(v)+1-\sigma(w)=1,
$$
which verifies Equation \eqref{E:triaUni} since $x$ is the unique element of $\Delta$ that is $\sigma$-degenerate. 
\end{proof}

We now put all the above results together in order to describe the submaximal elements of $\Sigma_{g,n}^\chi$ for $n\geq 1$.

\begin{corollary}\label{C:walls-Vfun}
If $n\geq 1$, then the submaximal elements  of $\Sigma_{g,n}^\chi$ are those V-functions $\sigma$ such that $\D(\sigma)$ is equal to one of the following 
    \begin{enumerate}[(a)]
        \item  $\{(e;h,A),(e;h,A)^c\}$ with either $e=1$, or $\emptyset\subsetneq A\subsetneq [n]$, or $A=\emptyset$ and $2h-2+e\geq g$, or $A=[n]$ and $2h+e\leq g$;
        \item $\displaystyle W_{\delta}:=\bigcup_{\substack{\delta \mid (2h-2+e) \\ e\geq 2}} \{(e;h,\emptyset), (e;h,\emptyset)^c\}$ for some  $1\leq \delta\leq g-1$. 
    \end{enumerate}
    Moreover, each of the above subsets is a universal degeneracy subset of some element of $\Sigma_{g,n}^\chi$.
\end{corollary}
The submaximal elements $\sigma$ or $\D(\sigma)$  of the first type are called \emph{simple}.
\begin{proof}
    By Remark \ref{R:Sigma-Deg}, an element $\sigma\in \Sigma_{g,n}^\chi$ is submaximal if and only if $\D(\sigma)\in \Deg_{g,n}$ is  submaximal (and realizable). The realizable submaximal elements of $\Deg_{g,n}$ are of the following type:
    \begin{itemize}
    \item separating (mixed or unmixed), which are of the form $\{(1;h,A),(1;g-h,A^c)\}$ and are realizable by Proposition \ref{P:Dsep};
        \item mixed (separating or not), which are of the form $\{(e;h,A),(e;h,A)^c\}$ with $\emptyset\subsetneq A \subsetneq [n]$ by Proposition~\ref{P:good}. All such elements are realizable by Proposition \ref{prop:good-wall-realizable};
        \item unmixed, non-separating and simple, which are of the form $\{(e;h,\emptyset),(e;h,\emptyset)^c\}$ with $\delta(e;h,\emptyset)=2h-2+e\geq g$ and $e\geq 2$ by Corollary \ref{C:walls-to1}, Proposition \ref{P:deg-grandi}\eqref{P:deg-grandi2} and Theorem \ref{T:pos-Degn1}\eqref{T:pos-Degn1b}. All such elements are realizable by Corollary \ref{C:walls-to1} and Proposition \ref{P:deg-grandi}\eqref{P:deg-grandi1}.
        \item unmixed, non-separating and non-simple, which are of the form $W_{\delta}$ for some $1\leq \delta\leq g-1$ by Corollary \ref{C:walls-to1} and Proposition \ref{P:deg-piccoli}. All such elements are realizable by Corollary \ref{C:walls-to1} and by Theorem \ref{T:pos-Degn1}.
    \end{itemize}
\end{proof}

It is instructive to compare the submaximal elements of $\Sigma_{g,n}^\chi$ with the classical walls, i.e. the codimension one regions of $\PicRel^{\ZZ}(\RR)$ with respect to $\A_{g,n}$ (see Lemma \ref{L:arr-Uni}).

\begin{remark}\label{R:classic-walls}
If $n\geq 1$ the classical walls of $\PicRel^{\ZZ}(\RR)$ with respect to $\A_{g,n}$ are given by the regions $[L]$ of one of the following types:
    \begin{enumerate}[(a)]
        \item $[L]$ is a separating wall if 
        $$\D(\sigma([L]))=\{(1;h,A),(1;h,A)^c\}$$ 
        for some $(1;h,A)\in \Dgn$.
    \item $[L]$ is a mixed non-separating wall (which is called a good wall in \cite{APag}) if
    %good classical walls
    $$\D(\sigma([L]))= W_{A,\delta}:=\bigcup_{\substack{\delta-|A| =2h-2+e \\ e\geq 2}} \{(e;h,A), (e;h,A)^c\},$$
    for some $\emptyset\subsetneq A\subsetneq [n]$ and some $|A|\leq \delta \leq 2g-2+|A|$.
        \item\label{R:classic-wallsC} $[L]$ is an unmixed non-separating wall (which is called a bad wall in \cite{APag})
        %bad classical walls
        $$\D(\sigma([L])= W_{\delta}:=\bigcup_{\substack{\delta \mid (2h-2+e) \\ e\geq 2}} \{(e;h,\emptyset), (e;h,\emptyset)^c\},$$ 
        for some  $1\leq \delta\leq 2g-2$. 
    \end{enumerate}  
    This follows easily from the explicit Equations \eqref{E:Agn} of the hyperplane arrangement $\A_{g,n}$, and it was observed in \cite[Sec. 3.b]{APag}.

    By comparing the classical walls with the submaximal elements of $\Sigma_{g,n}^\chi$, we see that almost  all the mixed non-separating walls (namely those $W_{A,\delta}$ such that $2\leq\delta-|A|\leq2g-4$) and (roughly) half of the unmixed non-separating walls (namely those $W_{\delta}$ with $g\leq \delta\leq 2g-4$) are not submaximal elements in $\Sigma_{g,n}^\chi$. On the other hand, all the separating walls and the remaining (roughly) half of the bad classical walls (namely those $W_{\delta}$ with $1\leq \delta\leq g-1$ or $\delta=2g-2,2g-3$) are still submaximal elements of $\Sigma_{g,n}^\chi$.
\end{remark}

\begin{corollary}\label{C:2max-wall}
If $n\geq 1$, then every submaximal element  of $\Sigma_{g,n}^\chi$ is dominated by exactly two maximal elements $\Sigma_{g,n}^\chi$. 
%    Let $W$ be a wall. Then, for each $\sigma$ such that $\D(\sigma)=W$, the set 
%        $$
%        \{\sigma'\in\Sigma_{g,n}:\sigma'>\sigma\}
%        $$
%        has exactly two elements.
\end{corollary}
\begin{proof}
If $\sigma\in \Sigma_{g,n}^\chi$ is a submaximal element, then all the maximal elements of $\Sigma_{g,n}^\chi$ that dominates $\sigma$ are in bijection with the witnesses for $\emptyset >\D(\sigma)$ by Corollary \ref{C:witness}. 
Hence, we have to show that if $D$ is a realizable submaximal element of $\Deg_{g,n}$, then there are exactly two witnesses for $\emptyset>D$. This is clear for simple submaximal elements and for the elements $W_{\delta}$ it follows by combining Theorem \ref{T:pos-Degn1}\eqref{T:pos-Degn1a} and Proposition \ref{P:reduceto1}\eqref{P:reduceto1a}.
\end{proof}

\subsection{Example: $g=1$}\label{Sub:g1}

In this subsection, we focus on the poset $\Sigmans_{g,n}^\chi$ and $\Degns_{g,n}$ for $g=1$ (and hence $n\geq 1$).

Observe that the non-separating domain in genus $1$ is equal to 
$$
\DDns_{1,n}=\DDns_{1,n}^m=\{(2;0,A):\: \emptyset \subsetneq A \subsetneq [n]\},
$$
with the complement of an element of $\DDns_{1,n}$ given by 
$$
(2;0,A)^c=(2;0,A^c),
$$
and the triangles in $\DDns_{1,n}$ given by 
$$
\Delta=[(2;0,A_1),(2;0,A_2),(2;0,A_3)] \text{ with } [n]=A_1\sqcup A_2\sqcup A_3.
$$

This implies that the collection of all non-separating degeneracy subsets of type $(1,n)$ (as in Definition~\ref{D:degsub}~\eqref{D:degsub1}) admits a natural bijection with the collection $\Dyn([n])$ of all \emph{Dynkin systems} on $[n]$:
\begin{equation}\label{E:Deg-g1}
\begin{aligned}
    \Degns_{1,n}& \xrightarrow{\cong} \Dyn([n])\\
    D & \mapsto \ov D:=D\cup \{\emptyset, [n]\},
\end{aligned}
\end{equation}
where a Dynkin system $E$ on $[n]$ is a collection of subsets of $[n]$ such that 
\begin{itemize}
    \item $\emptyset\in E$,
    \item $A\in E\Rightarrow A^c\in E$,
    \item If $A_1,A_2\in E$ with $A_1\cap A_2=\emptyset$ then $A_1\sqcup A_2\in E$.
\end{itemize}
The number of Dynkin systems on $[n]$ form the sequence \cite[A380571]{oeis_encyclopedia}, while the number of Dynkin systems on $[n]$ up to the action of the symmetric group $S_n$ form the sequence \cite[A381471]{oeis_encyclopedia}. 

The bijection \eqref{E:Deg-g1} becomes an isomorphism of posets (with respect to the poset structure on $\Deg_{1,n}$ as in Definition \ref{D:degsub}\eqref{D:degsub2}) if we put the following poset structure on $\Dyn([n])$: given $D^1,D^2\in \Dyn([n])$, we say that $D^1\geq D^2$ if $D^1\subseteq D^2$ and there exists a subset $E\subset D^2-D^1$ such that 
    \begin{itemize}
        \item  $D^2-D^1=E\sqcup E'$, where $E':=\{A^c\: : A\in E\}$.
        \item  For any decomposition  $[n]=A_1\sqcup A_2\sqcup A_3$ such that $A_1,A_2,A_3\in D^2$ and such that  $|\{A_i\in D^1\}|=1$, we have that 
        $$
        |\{A_i\in E\}|=1.
        $$
        \item   For any decomposition  $[n]=A_1\sqcup A_2\sqcup A_3$ such that $A_1,A_2,A_3\in D^2-D^1$,  we have that 
        $$
        |\{A_i\in E\}|=1 \text{ or } 2.
        $$
        \end{itemize}

Moreover, a non-separating V-function $\sigma$ of characteristic $\chi$ determines (and it is uniquely determined by) a function (called a \emph{mildly superadditive} function of characteristic $\chi$, see \cite{PTgenus1})
\begin{equation}\label{E:fun-f}
    \begin{aligned}
    f: 2^{[n]} & \longrightarrow \Z\\
 A &\mapsto f(A):=\begin{sis}
     \sigma(2;0,A) & \quad \text{ if } A\neq \emptyset, [n],\\
     0 & \quad \text{ if } A=\emptyset,\\
     \chi & \quad  \text{ if } A=[n],\\
 \end{sis}
 \end{aligned}
 \end{equation}
which satisfies the following properties:
\begin{enumerate}
\item  for any $A\subseteq [n]$, we have 
\begin{equation}\label{E:sum-f}
f(A)+f(A^c)-\chi\in \{0,1\}.
\end{equation}

An element $A\in 2^{[n]}$ is said to be \emph{$f$-degenerate} if $f(A)+f(A^c)=\chi$, and \emph{$f$-nondegenerate} otherwise.

\item   for any decomposition $[n]=A_1\sqcup A_2\sqcup A_3$, we have that:
\begin{enumerate}
 \item if two among the elements of $\{A_1,A_2,A_3\}$ are $f$-degenerate, then so is  the third. 
            \item the following holds
            \begin{equation}\label{E:tria-f}
            \sum_{i=1}^{3}f(A_i)-\chi
            \in \begin{cases}
                \{1,2\} \textup{ if $A_i$ is $f$-nondegenerate for all $i$};\\
                \{1\} \textup{ if there exists a unique $i$ such that } \\
                \hspace{2cm} \textup{ $A_i$ is $f$-degenerate};\\
                \{0\} \textup{ if $A_i$ is $f$-degenerate for all $i$}.
            \end{cases}
        \end{equation}
\end{enumerate}
\end{enumerate}
Note that the degeneracy subset of $f$ defined as
\begin{equation}\label{E:deg-f}
\D(f):=\{A\in [n]: A \text{ is $f$-degenerate}\},
\end{equation}
is a Dynkin system on $[n]$, which is equal to $\ov{\D(\sigma)}$ if $\sigma$ is the non-separating V-function that corresponds to $f$.

By combining the above description of non-separating V-functions in terms of mildly superadditive functions with Proposition \ref{P:Dsep} and Theorem \ref{T:thmA}, we recover the classification of fine compactified Jacobians over $\ov{\M}_{1,n}$ in \cite[Thm. 6.5]{PTgenus1} and we extend it beyond the fine case. 

\begin{remark}\label{R:g1-Kn}
It follows from the above description that there is a canonical isomorphism of posets
$$
\Sigmans_{1,n}^\chi\cong \VStab^\chi(X(K_n)),
$$
where $X(K_n)$ is any nodal curve whose dual graph is the complete graph $K_n$ on $n$ vertices.     
\end{remark}

The following seems to be an interesting problem.

\begin{question}\label{Q:surj-D}
 Is the non-separating degeneracy map $\Dns:\Sigmans_{1,n}\to \Deg^{ns}_{1,n}$ surjective? Or, in other words, is every Dynkin system on $[n]$ realizable by some V-function of type $(1,n)$?
\end{question}

For $n\leq 5$, it can be checked that the above Question has a positive answer and that, moreover, all non-separating V-functions are classical. This last fact is not true for $n\geq 6$, as we now indicate. 

\begin{example}\label{Ex:g1-n6}
Consider the classical non-separating V-function 
$\sigma^2_1\left[(\frac{1}{3},\ldots, \frac{1}{3}\right)]$ of characteristic $2$  (in the notation of \eqref{E:sigma-cl}), whose associated mildly superadditive function as in \eqref{E:fun-f} is given by  
$$
\phi(A)=
\begin{cases}
  2 & \text{ if } |A|\geq 4, \\
  1 & \text{ if } 0<|A|\leq 3,\\
  0 & \text{ if } A=\emptyset. 
\end{cases}  
$$
The degeneracy subset of $\phi$ is the Dynkin system
$$
\D(\phi)=\binom{[6]}{3}\cup\{\emptyset, [6]\},
$$
where $\binom{[6]}{3}$ is the collection of all subsets of $[6]$ of cardinality $3$. 

All the mildly superadditive functions  that dominate $\phi$ in the poset structure of $\Sigmans_{1,n}$, 
are given by the following construction. For any subset $\S\subseteq \binom{[6]}{3}$ such that $\S\cap \ov \S=\emptyset$, where $\ov \S:=\{A^c\: : A\in \S\}$, we consider the mildly superadditive function $\phi^\S$ of characteristic $2$ given by 
$$
\phi^\S(A)=
\begin{cases}
  2 & \text{ if } |A|\geq 4 \text{ or } |A|=3 \text{ and } A\in \S, \\
  1 & \text{ if } 0<|A|\leq 2 \text{ or } |A|=3 \text{ and } A\not \in \S,\\
  0 & \text{ if } A=\emptyset. 
\end{cases}  
$$
The degeneracy subset of $\phi^\S$ is the Dynkin system
$$
\D(\phi^S)=\binom{[6]}{3}-(\S\cup \ov \S)\cup\{\emptyset, [6]\}.
$$
Moreover, we have that $\phi^{\S_1}\geq \phi^{\S_2}$ if and only if $\S_1\supseteq \S_2$. Hence, the poset of all such $\{\phi^\S\}$ is graded by the function $|\S|$ and it has rank equal to $10$.

On the other hand, the poset of all classical non-separating V-functions that dominate $\sigma^2_1\left[(\frac{1}{3},\ldots, \frac{1}{3}\right)]$ is the poset of regions, cut out by the arrangement of hyperplanes $\A_{1,6}$ of \eqref{E:Agn} on $\Sigmans_{1,6}$, that contains the point $\sigma^2_1\left[(\frac{1}{3},\ldots, \frac{1}{3}\right)]$ in their closure, and hence it is a graded poset of rank equal to $5$. This implies that many of the elements in $\{\phi^{\S}\}$ are not classical and also that many of the Dynkin systems $\D(\sigma^\S)$ are not the degeneracy subset of some classical non-separating V-function. 
\end{example}

    \bibliographystyle{alpha}	
    \bibliography{bibtex}
\end{document}